\documentclass[11pt,letterpaper,reqno]{amsart}

\usepackage[margin=1in,marginparwidth=24mm, marginparsep=1mm]{geometry}

\usepackage{caption}
\usepackage{graphics}
\usepackage{graphicx}
\usepackage{float}
\usepackage{color}
\usepackage{marginnote}
\usepackage[dvipsnames]{xcolor}
\usepackage{tcolorbox} 
\usepackage{amsfonts}
\usepackage{amsmath, bm, bbm}
\usepackage{amstext}
\usepackage[shortlabels]{enumitem}
\usepackage{amsopn}
\usepackage{amsbsy}
\usepackage[numbers,sort]{natbib} 
\usepackage{empheq}
\usepackage{amscd}
\usepackage{amsxtra}
\usepackage{amsthm}
\usepackage{amssymb}
\usepackage{graphicx}
\usepackage{esint}
\usepackage{geometry}
\usepackage{abstract,lipsum,mathrsfs}
\usepackage[new]{old-arrows}
\usepackage{environ,etoolbox}
\newcommand{\h}{\hspace}
\newcommand{\p}{\partial}

\numberwithin{equation}{section}
\usepackage[hidelinks]{hyperref}
\usepackage[makeroom]{cancel}
\allowdisplaybreaks

\makeatletter
\def\l@subsection{\@tocline{2}{0pt}{2.5em}{3.5em}{}}
\makeatother
\makeatletter
\newcommand\Const[3]{%
	\@ifundefined{#1-#2}%
	{\stepcounter{#3}\expandafter\xdef\csname #1-#2\endcsname{\arabic{#3}}}%
	{}%
	\ifnum\pdfstrcmp{#1}{eps}=0 
	\varepsilon_{\csname #1-#2\endcsname}%
	\else
	#1_{\csname #1-#2\endcsname}%
	\fi
}
\makeatother

\newcommand\C[1]{\Const{C}{#1}{Ccnt}} 
\newcommand\R[1]{\Const{R}{#1}{Rcnt}} 
\newcommand\eps[1]{\Const{eps}{#1}{Ecnt}} 
\usepackage{scalerel,stackengine}
\newcommand\pig[1]{\scalerel*[5.5pt]{\big#1}{ 
		\ensurestackMath{\addstackgap[1.15pt]{\big#1}}}}
\newcommand\pigl[1]{\mathopen{\pig{#1}}}
\newcommand\pigr[1]{\mathclose{\pig{#1}}}

\makeatletter 
\newcommand{\customlabel}[2]{%
	\protected@write \@auxout {}{\string \newlabel {#1}{{#2}{\thepage}{#2}{#1}{}} }%
	\hypertarget{#1}{}
}
\makeatother

\newcommand{\mres}{\mathbin{\vrule height 1.6ex depth 0pt width
		0.13ex\vrule height 0.13ex depth 0pt width 1.3ex}}

\newcommand\specialcomment[3]{%
	\newtoggle{#1}\toggletrue{#1}
	\NewEnviron{#1}{\iftoggle{#1}{\BODY}{}}
}
\newcommand{\excludecomment}[1]{\togglefalse{#1}}

\specialcomment{com}{}{} 
\excludecomment{com} 

\newtheorem{thm}{Theorem}[section]

\newtheorem{lem}[thm]{Lemma}
\newtheorem{prop}[thm]{Proposition}

\newtheorem{rmk}[thm]{Remark}

\newcommand{\pcm}[1]{{\color{purple}#1}}

\newcommand{\cm}[1]{}  

\def\Xint#1{\mathchoice  
	{\XXint\displaystyle\textstyle{#1}}%
	{\XXint\textstyle\scriptstyle{#1}}%
	{\XXint\scriptstyle\scriptscriptstyle{#1}}%
	{\XXint\scriptscriptstyle\scriptscriptstyle{#1}}%
	\!\int}
\def\XXint#1#2#3{{\setbox0=\hbox{$#1{#2#3}{\int}$}
		\vcenter{\hbox{$#2#3$}}\kern-.5\wd0}}

\def\dashint{\Xint-}

\title{Planar Degenerate Anchoring in Landau--de Gennes Energy}

\author[H.~M.~Tai ]{\textnormal{Ho Man Tai}\ \ \ }
\address{School of Mathematics and Statistics, The University of Sydney, Sydney, NSW 2006, Australia}
\email{homan.tai@sydney.edu.au}

\author[Y.~Yu]{\ \ \textnormal{Yong Yu}}
\address{Department of Mathematics, The Chinese University of Hong Kong, Hong Kong}
\email{yongyu@cuhk.edu.hk}

\date{\today}


\begin{document}
	\maketitle
    \thispagestyle{empty}
	\begin{abstract}
		\noindent{\bf Abstract:} The aim of this article is twofold. First, in the large-body limit and when the temperature is below the nematic–isotropic transition threshold, we verify that the $\mathbb S^2$-valued energy-minimizing harmonic map on a bounded smooth domain $\Omega \subset \mathbb R^3$ with tangential boundary condition is a singular limit of the Landau–de Gennes energy minimizers subject to the Fournier-Galatola planar degenerate anchoring \cite{FP05}. This harmonic map is referred to as the canonical harmonic map. Our second aim is to address the local structure of the canonical harmonic map near the boundary singularities, which we call boojums. We show that the tangent map of the canonical harmonic map near a boojum is uniquely characterized by a half bubble with a hedgehog or an anti-hedgehog structure, up to a planar rotation. Comparing to the interior counterpart studied by Brezis-Coron-Lieb in \cite{BCL86}, for which the full $\mathrm{SO}(3)$ group action can be applied to the tangent map near an interior singularity, we can only apply planar rotations to the tangent map near a boojum to maintain the tangential boundary condition. The degeneracy of the group action from $\mathrm{SO}(3)$ to $\mathrm{SO}(2)$ makes it challenging to investigate the local structure of the boojum singularity. On the other hand, the boundary condition for the half bubble is Dirichlet on the curved boundary and tangential on the flat boundary. We need to extend the Schoen-Uhlenbeck bubbling analysis in \cite{SU82, SU83} for energy-minimizing harmonic maps with Dirichlet boundary conditions to our current case with the mixed-type boundary conditions.
	\end{abstract}

	\tableofcontents
	
	\section{Introduction} The Landau--de Gennes (LdG for short) theory is a comprehensive continuum theory for the mathematical modeling of nematic liquid crystals \cite{de1993physics}. It takes into account both uniaxial and biaxial phases. In addition, it can describe higher-dimensional disclinations, which may induce many significant optical properties of liquid crystals.

	\subsection{Motivations and a review on related works}\label{lit. rev}
The radial hedgehog solution of the LdG equation was first discussed by Schopohl--Sluckin  \cite{SS87} in 1988. It admits an isotropic point disclination at the origin. In certain parameter regimes, the hedgehog solution is unstable. Higher-dimensional disclinations may arise and stabilize the critical points of the LdG energy. For example, the half-degree biaxial ring and the split-core segment of strength 1 are two topologically different disclinations. They were identified through numerical simulations by Penzenstadler--Trebin \cite{PT89} in 1989 and by Gartland--Mkaddem \cite{MG00} in 2000, respectively. Rigorous mathematical proofs for the existence of ring and split-core disclinations in LdG theory have been given by the second author in \cite{Y20} for the limiting case ($a \to \infty$) and by Tai--Yu in \cite{TY23} for the case with large but finite $a$. We also refer the reader to \cite{DMP21, DMP24a, DMP24b} by Dipasquale--Millot--Pisante for the LdG theory with Lyuksyutov constraint.  More studies on LdG disclinations include, but are not limited to, those of Canevari \cite{C17}, Bauman--Park--Phillips \cite{BPP12}, Di Fratta--Robbins--Slastikov--Zarnescu \cite{DRJSZ16}, and Ignat--Nguyen--Slastikov--Zarnescu \cite{INSZ13, INSZ15, INSZ16}. \vspace{0.2pc}
	
	In the previously mentioned articles, the disclinations are located in the interior of the bulk domain. In practice, disclinations on the bounding interface are also of significant interest for both fundamental research and technological applications. A boojum singularity is a topologically stable point disclination that forms on the bounding interface. To study the boojums, we need appropriate boundary conditions for the order parameters. In general, the alignment of the director field on the bounding interface depends on the anchoring effect, which can be prescribed in a strong or weak sense. Strong anchoring is given by the Dirichlet conditions imposed on the interface. Weak anchoring can be realized by penalizing the free energy with a surface integral. Usually, weak anchoring conditions are more realistic than strong ones. Boojums are believed to occur under the weak anchoring condition of Nobili--Durand \cite{nobili1992disorientation} or the weak planar degenerate anchoring of Fournier--Galatola \cite{FP05}. This article focuses on the latter degenerate case. The bulk is a simply connected, bounded, smooth domain of dimension 3. Particular attention will be paid to the large-body limit. Our motivations primarily stem from the boojum-type singularities observed in the physics literature. See Volovik--Lavrentovich \cite{VL83} and Lavrentovich \cite{L98}. \vspace{0.2pc}
	
	Before introducing the weak planar degenerate anchoring condition used in this article, let us first discuss some existing works on the non-degenerate weak anchoring condition, specifically the Nobili--Durand anchoring condition. In this case, the surface integral in the total energy measures the $L^2$-distance between the order parameter and a given map defined on the boundary. It is non-degenerate in the sense that the corresponding director field favors some particular directions prescribed by the given map. Alama--Bronsard--Galv\~{a}o--Sousa \cite{ABG15} studied the LdG model with non-degenerate weak anchoring in 2D domains. The original problem is reduced to the Ginzburg--Landau model with a complex-valued order parameter. As the length scale $\epsilon$ tends to $0$ and the anchoring strength scales with $\epsilon$ at a specific rate, vortices can appear on the boundary of the domain. The topological degrees of the vortices are $+1$ for bounded simply connected domains, and are $-1$ for annular domains or exterior domains. Bauman--Phillips--Wang \cite{BPW19} examined a similar problem in higher-dimensional domains. They showed that the minimizers converge to a generalized harmonic map smoothly in the interior, away from a $(n-2)$-rectifiable set. Under a suitable  $\eta$--smallness condition, they eliminated the possibility of forming vortices on the boundary. There are also some results for the region surrounding a spherical colloidal particle. In \cite{ABL16}, Alama--Bronsard--Lamy verify the existence of a Saturn ring disclination in the LdG theory under the non-degenerate weak anchoring condition. The coexistence of boojums and Saturn rings in the axially symmetric LdG theory is shown in a recent preprint \cite{huangyu} by Huang--Yu.  \vspace{0.2pc}
	
	In this article, we adopt the weak planar degenerate anchoring introduced by Fournier--Galatola \cite{FP05}. The surface energy density is quartic and favors planar anchoring. Consequently, limiting minimizers must be discontinuous on the boundary due to the Poincar\'{e}--Hopf theorem. A theoretical investigation in this direction was carried out by Alama--Bronsard--Golovaty in \cite{ABG20}. They studied the Ginzburg–Landau energy in a 2D thin-film domain. To promote oblique alignments, they introduced a surface energy density up to the fourth order. The authors identified a critical rate for the anchoring strength as the length scale $\epsilon$ tends to $0$. When the anchoring strength converges faster than this critical rate, interior vortices form. For slower convergence, boojum pairs appear on the boundary, and no interior vortices are present. Ignat--Kurzke \cite{IGNAT2021108928,IgantKurzke23} incorporated a boundary term into the 2D Ginzburg--Landau energy  that is quadratic in the normal component of the order parameter. They derived the renormalized energy for boundary vortices and characterized their asymptotic locations in the ground states, using the global Jacobian method and $\Gamma$-convergence. On closed Riemannian surfaces, Ignat--Jerrard \cite{ignat2021renormalized} determined the structure of canonical harmonic unit tangent vector fields with prescribed vortices, including the relevant topological constraints. They also derived the corresponding renormalized energy governing the interaction of these vortices. \vspace{0.2pc}

    In the recent work \cite{BCS25} by Bronsard--Colinet--Stantejsky, the Oseen--Frank model is examined in the 3D unit ball under the tangential boundary condition. The set of defects is shown to have finitely many boundary points. Under additional assumptions, these defects can be precisely located. In our work, harmonic maps into $\mathbb S^2$ with tangential boundary conditions serve as canonical harmonic maps of the LdG energy with Fournier--Galatola surface potential. In the large-body limit, our problem reduces to one closely related to \cite{BCS25}. However, their arguments use the symmetry of the unit ball to construct reflection-based extensions, which are crucial for establishing their regularity and monotonicity formulas. In contrast, our present work develops the bubbling analysis of minimizing harmonic maps with tangential boundary conditions in general domains of $\mathbb{R}^3$. We also determine an explicit half bubble at a boojum, which is uniquely characterized by a hedgehog or an anti-hedgehog profile, up to a planar rotation around the inner normal direction at the boojum. 
    
	\subsection{Model setting and main contributions}\label{sec. Model Setting}
	In this article, $\Omega$ is a simply connected, bounded, smooth domain in $\mathbb{R}^3$ with $\p \Omega$ diffeomorphic to $\mathbb S^2$. The order parameter takes values in the space:  \begin{equation*}\mathcal{S}_0:=\left\{ \h{1pt}Q \in \mathbb{R}^{3\times 3}: Q=Q^\top \h{5pt} \text{and} \h{5pt} \textup{tr}(Q)=0 \h{1pt}\right\}.\end{equation*} Given $Q\in H^1\big(\Omega;\mathcal{S}_0\big)$, we consider the LdG energy in the form:
	\begin{align}
		\mathcal{E}_L(Q) := \int_{\Omega} 
		\left(\dfrac{1}{2}\h{1pt}\big|\nabla Q \big|^2
		+\dfrac{1}{L^{2}} \h{1pt}f_B(Q) \right)
		+ \dfrac{1}{L} \int_{\p\Omega} f_S(Q) \h{1.5pt}\mathrm d\h{.5pt}\mathscr{H}^2.
		\label{def ldg energy}
	\end{align} 
	Here, the bulk energy density $f_B:\mathcal{S}_0 \to \mathbb{R}$ is read as  
	$$f_B(Q):=f^*_B(Q) -\displaystyle\min_{ \mathcal{S}_0}f^*_B,\quad\text{where}\h{5pt} f^*_B(Q):=\dfrac{a}{2} \textup{tr}(Q^2)
	-\dfrac{b}{3}\textup{tr}(Q^3)
	+\dfrac{c}{4}\pigl[\textup{tr}(Q^2)\pigr]^2.$$
	It penalizes the order parameters for their non-uniaxiality. In the definition of $f_B^*$, the parameter $a \in \mathbb R$ depends on both temperature and the material. $b$ and $c$ are positive material constants for nematic rod-like molecules. We assume throughout the article that $b^2 > 27 a c$ to ensure that the minimum of $f_B$ is attained at a uniaxial state. In this regime, the temperature is below the nematic-isotropic transition threshold. We point out that, as discussed in \cite{gartlandscaling}, the energy \eqref{def ldg energy} can be obtained by scaling a domain with the size of $L^{-1}$. The limit problem where $L \to 0$ is interpreted as the large-body limit. \vspace{0.2pc}
	
	For the surface energy density $f_S$, we use the Fournier–Galatola surface potential introduced in \cite{FP05}. More specifically, letting $Q \in  \mathcal{S}_0$ and denoting by $\mathbf{I}_3$ the $3 \times 3$ identity matrix, we define
	\begin{equation}\label{defn s_0}Q^{(s_0)}:=Q+ \frac{s_0 }{3} \hspace{1pt}\mathbf{I}_3, \quad\text{where} \h{5pt} s_0:=\frac{b + \sqrt{b^2 - 24ac}}{4c}.\end{equation}
	Denoting by $\nu(x)$ the outward unit normal vector at $x \in \p\Omega$ and by $\widetilde{P}(x):=\mathbf{I}_3-\nu(x)\otimes\nu(x)$ the orthonormal projection onto the tangent plane at $x$, we define, in the sense of trace, that $$Q^\perp(x):= \widetilde{P}(x)\h{1pt}Q(x), \quad \text{for $x \in \p\Omega$ and $Q \in  H^1\big(\Omega;\mathcal{S}_0\big)$}. $$  The surface energy density $f_S(Q)$ is then given by
	\begin{align*}
		f_S(Q) :&= 
		s_1 \left|\h{1pt}Q^{(s_0)} -\big[Q^{(s_0)} \big]^\perp \h{1pt}\right|^2
		+s_2 \pigl(\h{1pt}\big|\h{1pt}Q^{(s_0)} \h{1pt}\big|^2-s_0^2\h{1pt}\pigr)^2,
	\end{align*}
	where $s_1,s_2>0$ are two parameters. Unlike the surface energy density used in \cite{BPW19, ABG15, CLR15, HW18}, which forces the energy minimizers to coincide with a given map on the boundary as the anchoring strength tends to infinity, the surface energy density $f_S$ only confines the director field to the tangent plane of the boundary in the limit situation, allowing discontinuities of the director field on the boundary. See Lemma \ref{lem H1 conv}. \vspace{0.2pc}

	We now summarize the main results of this work.
	\begin{thm} Suppose $b$, $c$, $s_1$ and $s_2$ are given positive material constants. The constant $a\in \mathbb{R}$ satisfies $b^2>27ac$. Then the following holds: \begin{itemize}
			\item[$(1).$] For each $L>0$, there exists at least one minimizer, denoted by $Q_L \in H^1\big(\Omega;\mathcal{S}_0\big)$, such that \begin{equation}\label{min pro ldg}\mathcal{E}_L(Q_L) = \inf\Big\{\h{1.5pt}\mathcal{E}_L(Q) : Q \in H^1\big(\Omega;\mathcal{S}_0\big)\h{1.5pt} \Big\}.\end{equation}
			\item[$(2).$] Up to a subsequence, $\big\{Q_L\big\}$ converges to $s_0\big(u^* \otimes u^* - \frac{1}{3} \mathbf{I}_3\h{1pt}\big)$ strongly in $H^1\big(\Omega;\mathcal{S}_0\big)$ as $L \to 0^+$.  Here, the constant $s_0$ is given in \eqref{defn s_0}. $u^*$ is a harmonic map which satisfies \begin{equation}\label{var. pro. limi} \int_\Omega \big|\h{0.5pt} \nabla u^* \h{0.5pt}\big|^2 = \inf \left\{\h{1.5pt} \int_\Omega |\h{0.5pt} \nabla u \h{0.5pt}|^2 : u \in H^1\big(\Omega;\mathbb{S}^2\big) \h{6pt}\text{and}\h{7pt} u \cdot \nu = 0 \h{4pt}\text{on}\h{5pt}\partial \Omega \h{1.5pt}\right\}.
			\end{equation}
			\item[$(3).$] Any energy minimizer, still denoted by $u^*$, of the variational problem \eqref{var. pro. limi} has only finitely many singularities on $\overline{\Omega}$. The number of singularities of $u^*$ is even, both in $\Omega$ and on $\partial \Omega$. Note that $u^*$ might have no singularity in $\Omega$, but it must have singularities on $\partial \Omega$.\vspace{0.2pc}
			\item[$(4).$] Let $z_*$ be a singularity of $u^*$ on $\partial \Omega$. There is a diffeomorphism $\varphi$ between a neighborhood of $z_*$ in $\overline{\Omega}$ and an upper half ball centered at $0$ such that $\varphi(z_*) = 0$. Moreover, the tangent map of $u^* \circ \varphi^{-1} (x)$ near $0$ must be of degree $\pm 1$ and equal to either $\frac{x}{| \h{0.5pt}x \h{0.5pt}|}$ or $- \frac{x}{| \h{0.5pt}x \h{0.5pt}|}$, up to a rotation around the $x_3$-axis. Here, the degree is defined by extending the tangent map across the equatorial plane $\big\{ x_3 = 0\big\}$, using the even-even-odd extension (see Section \ref{sin str}).  
		\end{itemize}   
		\label{thm main result}
	\end{thm} The boundary singularities in Item (3) of Theorem \ref{thm main result} are guaranteed to exist since the tangent bundle of $\mathbb{S}^2$ is not parallelizable. So, it is impossible to have a continuous unit vector field tangent to $\partial \Omega$ at all points. In the physics paper \cite{VL83}, it is shown that the tangential anchoring on a nematic droplet leads to boojum singularities. Moreover, two unit-index boojums are predicted in \cite{VL83} for the spherical case. Our results in Items (3) and (4) of Theorem \ref{thm main result} provide rigorous justifications for the predictions in \cite{VL83}. 

\subsection{Challenges and summary of main ideas in our proofs}\label{cha}
The energy monotonicity and partial regularity of minimizing harmonic maps associated with \eqref{var. pro. limi} have been established by Scheven  \cite{S06tangentplane}. In fact, more general boundary conditions are discussed in \cite{S06tangentplane}. To study the local structure of the boojums, we need to extend the bubbling analysis developed by Schoen--Uhlenbeck \cite{SU82,SU83}. See also Hardt--Lin \cite{HL89free}, Duzaar--Steffen \cite{DS89, DS89optimal}, Duzaar--Grotowski \cite{Duzaar1996,duzaar1994energy} and Moser--Roberts \cite{moser2022partial} for related works on free boundary problems. We now explain the main challenges arising from the tangential constraint that do not occur under other boundary conditions.    \vspace{0.2pc}

\noindent\textit{Challenge 1}: The first difficulty is to extend the energy minimizers across the boundary and to identify the weak equation satisfied by the extended maps. This step differs substantially from the reflection arguments used in \cite{moser2022partial,HL89free,DS89,DS89optimal} for free boundary problems. For example, \cite{DS89} constructs a reflection using the projection of the nearest point onto a fixed target submanifold. The projections and the distance functions are fixed in the neighborhood of a given boundary point. In contrast, under tangential boundary conditions, both the relevant projections and the associated distance functions depend on the boundary points. In the strip region near the boundary, this dependence generates additional terms involving the derivatives along the boundary's tangential directions. Instead of the nearest-point projection, in Lemma \ref{lem existence of pi}, a carefully chosen projection \(P\) is introduced, using the method of characteristics, to cancel these additional terms. Moreover, the projection and the induced extended maps adapt well to our tangential boundary conditions. We refer the reader to Lemma \ref{lem eq of extended u} for the precise discussions, and \eqref{def ext of tilde u} for the construction of the extended maps. \vspace{0.2pc}


\noindent\textit{Challenge 2}: The second difficulty is to construct perturbations of minimizers for bubbling analysis. These perturbations must agree with the prescribed data in the interior while satisfying tangential boundary conditions on $\partial\Omega$. Our construction is genuinely new compared to the free-boundary problems in \cite{HL89free,DS89,DS89optimal}. In these works, the boundary values lie in a fixed target submanifold, allowing constant extensions in certain regions. However, in our tangential setting, the boundary value at $x$ must lie in $T_{x}\p\Omega$, the tangent plane of $\p \Omega$ at $x$. It varies along the boundary. Constant extensions fail to preserve the tangential constraint. The Luckhaus-type extension used in Duzaar--Grotowski \cite{Duzaar1996,duzaar1994energy} does not have a trivial modification in our case. To tackle this difficulty, in Proposition \ref{lem ext from sphere to ball}, we extend the arguments of Schoen--Uhlenbeck \cite{SU82, SU83} to obtain our extension map and the associated estimates. Specifically, we introduce some novel and carefully designed projections onto the relevant tangent planes. See the proofs of Lemmas \ref{lem compact of rescaled map} and \ref{lem prop of tangent map}. For these projections to be well defined, one must control the distance from the image of a minimizer to the corresponding tangent plane in terms of Dirichlet energy. This requires us to control some additional terms when we estimate the extension map in Proposition \ref{lem ext from sphere to ball}. We note that the additional terms involve the $L^2$-integral of the distance function to the unit circle on the tangent plane. This distance function, denoted by $\textup{dist}\left(p,\mathbb{S}^2\cap T_{x}\p\Omega\right)$, depends not only on the location $p$, but also on the boundary point $x$. It is quite different from the classical free-boundary setting, where the corresponding distance function depends only on \(p\). Controlling these additional terms is a new ingredient in our analysis. \vspace{0.2pc}  

\noindent\textit{Challenge 3}: The third difficulty is to classify the precise local structure of the boojum singularity. 
For the tangent map near an interior singularity, Brezis--Coron--Lieb utilize the Haar measure on SO(3) and an averaging argument to rule out all bubbles with $|\h{0.5pt}d\h{0.9pt}| \geq 2$. Here, $d$ is the topological degree of the bubble. See \cite[Section VII]{BCL86}. Consequently, the local structure near an interior singularity is uniquely characterized by the hedgehog or anti-hedgehog profile, up to a rigid rotation in $\mathrm{SO}(3)$. Our problem is different. Up to a diffeomorphism, the tangent map at the boojum singularity is given by a minimizing $0$-homogeneous half bubble on the upper-half ball $B_1^+$. On the flat boundary of $B_1^+$, it satisfies a free boundary condition with the image lying in the equator of $\mathbb{S}^2$. The free boundary condition breaks the SO(3)-symmetry. To maintain the boundary condition, only planar rotations preserving the equator remain admissible. However, through averaging over $\mathrm{SO}(2)$ and taking the maximum over $[-1, 1\h{0.5pt}]$ for the degenerate mapping variable, we can only rule out the half bubbles with \(| \h{.5pt} d \h{.9pt}|\ge 3\). See Lemma \ref{lem deg<4}. A new strategy is needed to show \(| \h{.5pt} d \h{.9pt}|\neq 2\). First, we express the tangent map by Blaschke products, using the stereographic charts. Then, the moving-center variation eliminates the non-minimizing possibilities, showing that there are only two standard forms in the $d = 2$ case. More specifically, up to a planar rotation and some M\"{o}bius transformations, only $z^2$ or $z + z^{-1}$ can serve as our degree 2 half bubble. See Lemma \ref{lem: deg 2 explicit form}. Then we rule out the degree 2 case (hence the degree -2 case), using a novel stability estimate. Eventually, the tangent map near a boojum singularity is shown to be either a hedgehog or an anti-hedgehog, up to a planar rotation. This classification of tangent maps already covers a particular case of free boundary problems, which appears to be open in the existing literature. Our approach can also be extended to classify tangent maps that arise in general free boundary problems for harmonic maps.

	\subsection{Notations}\label{notation}
	We list the notation used in this article. 
	\begin{itemize}
		\item[$\mathrm{(1).}$] As a convention, repeated indices in formulas and equations are summed. Given a vector field $X$, the notation $X_j$ denotes its $j$-th component. \vspace{0.4pc}
		\item[$\mathrm{(2).}$] Conventionally, we write $A \lesssim_{\,c_1,c_2,\ldots,c_n}B$ to mean that there exists a positive constant $c$ depending on $c_1,c_2,\ldots,c_n$ such that $A\leq c B$. When the implicit constant $c$ is universal, we simply write $A \lesssim B$.\vspace{0.4pc}
		\item[$\mathrm{(3).}$] $\Omega$ is a simply connected, bounded, smooth domain in $\mathbb{R}^3$. The boundary $\p \Omega$ is diffeomorphic to the unit sphere $\mathbb S^2$. Given $x \in \p \Omega$, the space $T_x\h{0.5pt}\p\Omega$ consists of all tangent vectors of $\p\Omega$ at $x$. We use $T(\p\Omega)$ to denote the set $$\Big\{ f : \p \Omega \to \mathbb R^3 : f \h{3pt}\text{is measurable and $f(x) \in T_x \h{0.5pt}\p \Omega$ \h{1pt}for $\mathscr H^2$-a.e.\h{3pt}$x \in \p \Omega$}\Big\}. $$Here and throughout the following, $\mathscr{H}^d$ is the $d$-dimensional Hausdorff measure. In addition, we define $H_T^1(\Omega; \mathbb S^2)$ to be the configuration space consisting of all $\mathbb S^2$-valued maps with finite Dirichlet energy and tangential boundary conditions. \vspace{0.4pc}
		
		\item[$\mathrm{(4).}$] The set $B_r^+(x) := \big\{y\in B_r(x):y_3>0\big\}$ is the open upper part of the open ball $B_r(x)$. Its curved and flat boundaries are denoted by $\p^+ B_r^+(x)$ and $\p^0 B_r^+(x)$, respectively. $B_r^-(x) := B_r(x) \setminus \overline{B_r^+(x)}$ represents the open lower part of $B_r(x)$. Given two radii $r < R$ and $x \in \mathbb R^3$, the annulus $B_R(x) \setminus \overline{B_{r}(x)}$ and its upper part $B_R^+(x) \setminus \overline{B_{r}^+(x)}$ are denoted by $A_{r,R}(x)$ and $ A_{r,R}^+(x)$, respectively. In all these notations, if $x = 0$, then we simply use the notations $B_r^+$, $B_r^-$, $A_{r, R}$ and $A_{r, R}^+$ to denote $B_r^+(0)$, $B_r^-(0)$, $A_{r, R}(0)$ and $A_{r, R}^+(0)$, respectively. \vspace{0.4pc}
		
		\item[$\mathrm{(5).}$] Denote by $P_0$ the plane $\big\{y \in \mathbb R^3 : y_3 = 0\big\}$. Letting $y=(y_1,y_2,y_3)^\top \in \mathbb{R}^3$, we use $y^0$, $(y)^0$ and $[\h{0.5pt}y\h{0.5pt}]^0$ interchangeably to represent the projection $(y_1,y_2,0)^\top \in P_0$. In addition, we use $\widehat{y}$ to denote the normalized vector of $y$ if $y \neq 0$. $y' = (y_1, y_2)$ denotes the tangential variables. We also introduce the reflection of a vector $y \in \mathbb R^n$. That is $y^\star:=(y_1, ... , y_{n-1}, -y_n)^\top$, where $y = (y_1, ..., y_{n-1}, y_n)$. The dimension $n$ is equal to $2$ or $3$.  \vspace{0.4pc}
		
		\item[$\mathrm{(6).}$] The Dirichlet energy of a vector field $u$ on $U \subset \mathbb R^3$ is defined by $$\mathcal{E}^*_{u, \h{0.5pt}U}:=\int_U \big|\nabla u \big|^2.$$ The normalized Dirichlet energy of $u$ on $B_r(x)$ and $B_r^+(x)$ are respectively denoted by $$\widehat{\mathcal{E}}^*_{r, \h{0.5pt}x}(u) :=\frac{1}{r}\int_{B_r(x)} \big|\nabla u \big|^2  \h{15pt}\text{and}\h{15pt}\widehat{\mathcal{E}}^*_{r, \h{0.5pt}x; \h{0.5pt}+}(u):=\frac{1}{r}\int_{B_r^+(x)} \big|\nabla u \big|^2.$$ Given a function $v$ on $\p^+ B_\rho^+$, its surface energy is defined and evaluated by \begin{align*}
			\mathcal{E}^*_{v, \h{0.5pt}\p^+B^+_\rho} := \int_0^{\frac{\pi}{2}} \mathrm d \xi \int_0^{2 \pi} \sin \xi \h{1pt}\big| \h{0.5pt} \p_\xi v \h{0.5pt}\big|^2 +  \frac{1}{\sin \xi} \h{1pt}\big| \h{0.5pt} \p_\theta v \h{0.5pt}\big|^2 \mathrm d \theta,
		\end{align*}where $\theta$ is the azimuthal angle. $\xi$ is the polar angle. $v = v(\rho, \xi, \theta)$ is expressed in terms of the polar coordinates.
		\vspace{0.4pc}
		
		\item[$\mathrm{(7).}$] Given a point $p\in \mathbb{R}^3$ and a subset $U \subseteq\mathbb{R}^3$, the function $\textup{dist}(p,U)$ gives the shortest distance between $p$ and $U$ measured in the standard Euclidean metric in $\mathbb{R}^3$. In particular, if $x \in \p \Omega$, then the distance function $\textup{dist}\left(p,\mathbb{S}^2\cap T_{x}\p\Omega\right)$ is simply denoted by $\textup{dist}_{x}(p)$. For a subset $F \subseteq \mathbb R^3$ and a positive constant $\delta$, we use $U_{\delta}(F)$ to denote the $\delta$-neighborhood of $F$ in $\mathbb{R}^3$ such that $\textup{dist}(p, F)<\delta$ for any $p\in U_{\delta}(F)$. \vspace{0.4pc}
		
		\item[$\mathrm{(8).}$] Let $\varphi$ be the boundary flattening map introduced in Section \ref{bdry flatten}. It is a diffeomorphism between $\varphi^{-1}\big(B_{2R_0'}^+\big)$ and $B_{2R_0'}^+$. Given $p \in \mathbb R^3$ and $y \in \p^0 B_{2R_0'}^+$, the vector $\Pi\left(y,p\right) \in \mathbb R^3$ is the shortest distance projection of $p$ onto the unit circle in $T_{\varphi^{-1}(y)}\h{0.5pt}\p\Omega$ centered at $0$.   \vspace{0.4pc}
		
		\item[$\mathrm{(9).}$]  
		Suppose $U$ is a subset of $\mathbb{R}^3$ satisfying $\mathscr{H}^k(U) < \infty$. Here, $k=1, \,2, \,3$. Letting $q \in \mathbb R^3$ and $v \in L^2(U; \mathbb R^3)$, we define $$W^{\h{0.5pt}q}_{v, \h{0.5pt}U}:=\int_{U} |\h{.5pt}v - q\h{.5pt}|^2 \h{1.5pt} \mathrm d \mathscr{H}^k. $$
		Suppose in addition that $\psi$ is a map from $U$ to $\p \Omega$. Then, we define $$    W_{v, \h{0.5pt}\psi, \h{0.5pt}U}:=\int_{U} \left(\textup{dist}_{\h{0.5pt}\psi(x)}\big(v(x)\big) \right)^2\mathrm d \mathscr{H}^k_x.$$ 
		The notation $W^{\h{0.5pt}q}_{v,\h{0.5pt}\psi, \h{0.5pt}U}$  denotes the sum $W^{\h{0.5pt}q}_{v, \h{0.5pt}U}+W_{v,\h{0.5pt}\psi, \h{0.5pt}U}$. If $U$ is an open subset in $\mathbb R^3$, then $$M_U(v):=\dashint_U v : =\frac{1}{|\h{0.5pt}U\h{0.5pt}|}\int_U v$$ is the average of $v$ over $U$. Here, $|\h{0.5pt}U\h{0.5pt}|$ is the Lebesgue measure of $U$. \vspace{0.4pc}
		
		\item[$\mathrm{(10).}$] Define $$D_r:= \Big\{y\in\mathbb{R}^2:y_1^2+y_2^2<r^2 \Big\} \hspace{15pt} \text{and} \hspace{15pt} D_r^+:=D_r\cap\Big\{ y_2>0\Big\}.$$ The curved and flat boundaries of $D_r^+$ are denoted by $\p^+D_r^+$ and $\p^0D_r^+$, respectively.
		
		If $\mathcal D$ is an open set in the $(x, y)$-plane, then we denote by $\overline{\nabla}$ the gradient operator $(\partial_x, \p_y)$. For a function $v$ on $\mathcal D$, we let $$\mathcal{E}^*_{v, \h{0.5pt}\mathcal D} := \int_\mathcal{D}  \big|\h{0.5pt}\overline{\nabla} v \h{0.5pt}\big|^2. $$ We also define the following energies for functions on $\p D_\rho$ and $\p^+ D^+_\rho$, respectively: $$\mathcal{E}^*_{v, \h{0.5pt}\p D_\rho}:= \frac{1}{\rho}\int_0^{2\pi} \big|\h{0.5pt}\p_\theta v \h{0.5pt}\big|^2  \mathrm d\theta \h{20pt}\text{and} \h{20pt} \mathcal{E}^*_{v, \h{0.5pt}\p^+ D^+_\rho}:=\frac{1}{\rho}\int_0^{\pi} \big|\h{0.5pt}\p_\theta v \h{0.5pt}\big|^2  \mathrm d\theta.$$ Here, $v = v(\rho, \theta)$. \vspace{0.6pc}
	\end{itemize}
	
	\section{Minimizers of \texorpdfstring{$\mathcal{E}_L$}{TEXT}-energy}
	Throughout the remainder of the article, we fix the constants $a$, $b$, $c$, $s_1$, and $s_2$. Moreover, we assume $b^2 > 27 ac$. In this section, we study some basic properties of the $\mathcal E_L$-energy minimizers. \begin{lem}[\bf Existence of minimizer]\label{ex. min}
		Given $L>0$, there is a $Q_L \in H^1(\Omega;\mathcal{S}_0)$ satisfying \eqref{min pro ldg}.
	\end{lem} The proof of this lemma follows from the standard direct method in the calculus of variations and an estimate of the lower bound of $f_B(Q)$ as in the proof of \cite[Corollary 4.4]{DG98}. We omit it here. The minimizer $Q_L$ is a weak solution to the boundary value problem in the following remark:
	\begin{rmk}We define \begin{equation*}
			f_B'(Q) := a\h{0.5pt}Q
			-b\left(Q^2-\dfrac{1}{3}\h{0.5pt}|\h{0.5pt}Q\h{0.5pt}|^2\h{0.5pt} \mathbf{I}_3\right) 
			+c\h{0.5pt}|\h{0.5pt}Q\h{0.5pt}|^2Q
		\end{equation*}and let 
        \medmuskip=1.5mu\thickmuskip=1.5mu\begin{equation*}
			f_S'(Q) := s_1  \left(Q
			\left(\nu\otimes\nu\right)
            +\left(\nu\otimes\nu\right)Q
            -\dfrac{ 2\textup{tr} \pig[Q\left(\nu\otimes\nu\right)\pig]}{3}\mathbf{I}_3\right) 
			+\dfrac{2}{3} s_0s_1\left(\nu\otimes\nu-\dfrac{1}{3}\mathbf{I}_3\right)+4s_2  \left(|\h{0.5pt}Q\h{0.5pt}|^2-\dfrac{2}{3} s_0^2\right)Q.
		\end{equation*}
		Then $Q_L$ satisfies the weak formulation of the boundary value problem:
		\begin{equation}\label{bdv}
			\left\{\h{-3pt}\begin{aligned}
				L^{2}\h{1pt}\Delta Q&\h{1.5pt}=\h{1.5pt}f_B'(Q)  & \h{15pt} &\text{in $\Omega$};\\[1.5mm]
				-L\h{1pt}\p_\nu Q  &\h{1.5pt}=\h{1.5pt}f_S'(Q) 
				& \h{15pt} &\text{on $\p\Omega$.}
			\end{aligned}\right.
		\end{equation}
		More specifically,	we have\begin{align}
			\int_{\Omega} \nabla Q_L \colon \nabla \psi+\dfrac{1}{L^{2}} \h{1pt}f'_B(Q_L)\h{0.5pt}\cdot \h{0.5pt}\psi + \dfrac{1}{L}\int_{\p\Omega}
			f'_S(Q_L)\h{0.5pt}\cdot\h{0.5pt}\psi \h{1.5pt} = \h{1.5pt}0, \h{15pt}\text{for any $\psi \in H^1\big(\Omega;\mathcal{S}_0\big)$.} 
			\label{eq. weak sol.}
		\end{align}  
		Here, given two symmetric matrices $A$ and $B$, their inner product is defined by $A \cdot B := \mathrm{tr}(AB)$. For two $\mathcal S_0$-valued tensor fields $Q_1$ and $Q_2$, we define $$\nabla Q_1 \colon \nabla Q_2:=\sum^3_{i=1}\p_{i}\h{0.5pt}Q_1 \cdot \p_{i}\h{0.5pt}Q_2.$$ 
	\end{rmk}

	In the next, we investigate the uniform upper bound of $|\h{0.5pt}Q_L\h{0.5pt}|$. 
	\begin{lem}
		There is a positive constant $\C{1}$ independent of $L$ such that for any $L > 0$, we have $
		|\h{0.5pt}Q_L\h{0.5pt}|\leq \C{1}$ at almost every point in $\Omega$. In addition, $|\h{0.5pt}Q_L\h{0.5pt}|\leq \C{1}$ on $\p\Omega$ in the sense of trace.
		\label{lem existence of min. and L^infty bound}
	\end{lem}
	
	\begin{proof}Recall $Q^{(s_0)}$ defined in \eqref{defn s_0}. It turns out $\big|\h{0.5pt}Q^{(s_0)}\h{0.5pt}\big|^2=|\h{0.5pt}Q\h{0.5pt}|^2+\frac{1}{3}s_0^2$. The surface energy density $f_S(Q)$ can be expressed as
		\begin{align}
			f_S(Q) 
			= s_1\left(\nu^\top Q^2\nu+\frac{2}{3} s_0 \h{0.5pt}\nu^\top Q\h{0.5pt}\nu+\dfrac{s_0^2}{9}\right)
			+s_2\left(\h{0.5pt}|\h{0.5pt}Q\h{0.5pt}|^2-\frac{2}{3}s_0^2\right)^2.
			\label{192}
		\end{align}	Here, $\nu$ is understood as a column vector.	
		
		Fix $k_1>0$, we define the truncation, denoted by $\widetilde{Q_L}$, of $Q_L$ as follows: $$\widetilde{Q_L} := Q_L \h{15pt}\text{on $\Big\{\h{0.5pt} \big|\h{0.5pt}Q_L \h{0.5pt}\big|\leq k_1\Big\}$}; \h{20pt} \widetilde{Q_L} := k_1\h{0.5pt}\frac{Q_L}{\big|\h{0.5pt}Q_L\h{0.5pt}\big|}  \h{15pt}\text{on $\Big\{\h{0.5pt} \big|\h{0.5pt}Q_L \h{0.5pt}\big| > k_1\Big\}$}.$$  
		On the set $\big\{x\in \p\Omega :  |\h{0.5pt}Q_L\h{0.5pt}|>k_1 \big\}$, we can apply \eqref{192} to obtain
		\begin{align*}
			&\dfrac{f_S(Q_L)  -f_S(\widetilde{Q_L})}{|\h{0.5pt}Q_L\h{0.5pt}|-k_1}  \\[1.5mm]
			&\h{15pt}= 
			s_1 \left(\Big(|\h{0.5pt}Q_L\h{0.5pt}| +   k_1 \Big)\h{1pt} \left|\h{1pt} \frac{ Q_L }{|\h{0.5pt}Q_L\h{0.5pt}|} \h{1pt}\nu \h{1pt}\right|^2
			+\dfrac{2}{3}s_0\h{1pt} \nu^\top \frac{Q_L}{|\h{0.5pt}Q_L\h{0.5pt}|} \h{1pt}\nu \right)
			+s_2 \Big(|\h{0.5pt}Q_L\h{0.5pt}|  +   k_1 \Big) \left(|\h{0.5pt}Q_L\h{0.5pt}|^2+k_1^2-\dfrac{4}{3}s_0^2\right).
		\end{align*}
		The last equality then yields
		\begin{align*}
			\dfrac{f_S(Q_L)  -f_S(\widetilde{Q_L})}{|\h{0.5pt}Q_L\h{0.5pt}|-k_1}  
			&\geq s_1  \left(  k_1 \left|\h{1pt} \frac{ Q_L }{|\h{0.5pt}Q_L\h{0.5pt}|} \h{1pt}\nu \h{1pt}\right|^2
			- \dfrac{2}{3}s_0\h{1pt}  \left|\h{1pt} \frac{Q_L}{|\h{0.5pt}Q_L\h{0.5pt}|} \h{1pt}\nu \h{1pt}\right| \h{1pt} \right)   +2 s_2 \Big(|\h{0.5pt}Q_L\h{0.5pt}| +   k_1 \Big) \left(k_1^2-\dfrac{2}{3}s_0^2\right). \end{align*}
		Let $k^2_1 \geq \frac{2}{3} s_0^2$. The last estimate can be reduced to \begin{align*}
			\dfrac{f_S(Q_L)  -f_S(\widetilde{Q_L})}{|\h{0.5pt}Q_L\h{0.5pt}|-k_1}  
			&\geq - \frac{s_0^2}{9\h{0.5pt}k_1}\h{1pt}s_1  + 2s_2  k_1\left(k_1^2-\dfrac{2}{3}s_0^2\right) \geq - \frac{1}{6}\h{1pt}s_1\h{0.5pt}k_1   + 2s_2  k_1\left(k_1^2-\dfrac{2}{3}s_0^2\right).\end{align*} If we take $k_1$ sufficiently large with largeness depending on $s_0, s_1, s_2$, then it holds $$f_S(Q_L) > f_S(\widetilde{Q_L})  \h{15pt} \text{on $\big\{x\in \p\Omega : |\h{0.5pt}Q_L\h{0.5pt}|>k_1\big\}$}.$$ On $\big\{x\in \Omega : |\h{0.5pt}Q_L\h{0.5pt}|>k_1\big\}$, we compute that
		\begin{align*}
			\big|\nabla \widetilde{Q_L}\big|^2 
			= k_1^2\sum^3_{i=1}
			\dfrac{\big|\h{0.5pt}\p_iQ_L\h{0.5pt}\big|^2}{|\h{0.5pt}Q_L\h{0.5pt}|^2}
			-\left(\dfrac{\mathrm{tr}\big(Q_L\h{0.5pt} \p_iQ_L\big)}{|\h{0.5pt}Q_L\h{0.5pt}|^2}\right)^2
			\leq\dfrac{k_1^2}{|\h{0.5pt}Q_L\h{0.5pt}|^2} \h{1pt}\big|\nabla Q_L\big|^2
			< \big|\nabla Q_L\big|^2.
		\end{align*}
		Moreover, following the same arguments as in \cite[Theorem 1]{M10}, we have $$f_B(Q_L) > f_B(\widetilde{Q_L}) \h{15pt}\text{ on $\big\{x\in \Omega : |\h{0.5pt}Q_L\h{0.5pt}|>k_1\big\}$},$$ provided that $k_1$ is suitably large. We now fix the large constant $k_1$. If $\big\{x\in \Omega : |\h{0.5pt}Q_L\h{0.5pt}|>k_1\big\}$ has positive Lebesgue measure or $\big\{x\in \p\Omega : |\h{0.5pt}Q_L\h{0.5pt}|>k_1\big\}$ has positive $\mathscr{H}^2$-measure, then it follows that $ \mathcal{E}_L(Q_L) > \mathcal{E}_L(\widetilde{Q_L})$. This contradicts the fact that $Q_L$ is an $\mathcal E_L$-energy minimizer. This concludes the uniform boundedness of $|\h{0.5pt}Q_L \h{0.5pt}|$.
	\end{proof}
	
	Lemma \ref{lem existence of min. and L^infty bound} induces that $Q_L$ solves the boundary value problem \eqref{bdv} in the classical sense.
	\begin{prop}[\bf Regularity of minimizers]
		Suppose $Q_L$ is the $\mathcal E_L$-energy minimizer. Then $Q_L \in C^{1, \frac{1}{2}}\big (\overline{\Omega};\mathcal{S}_0\big)\cap C^{2,\frac{1}{2}} \big(\Omega;\mathcal{S}_0\big)$. It is a classical solution to the boundary value problem \eqref{bdv}.
	\end{prop}
	
	\begin{proof}
		Since $Q_L \in H^1\big(\Omega;\mathcal{S}_0\big)$, we have $f'_S(Q_L)\in H^{\frac{1}{2}}\big(\p\Omega;\mathcal{S}_0\big)$ by the trace theorem and the $L^\infty$-boundedness in Lemma \ref{lem existence of min. and L^infty bound}. In addition, we can rewrite the equation of $Q_L$ in \eqref{bdv} as follows: \begin{equation*}
			L^{2}\h{1pt}\Delta Q_L - Q_L =f_B'(Q_L) - Q_L   \h{15pt} \text{in $\Omega$}.
		\end{equation*}It then follows that $L^{2} \Delta Q_L - Q_L \in L^\infty(\Omega)$ still by the $L^\infty$-boundedness of $Q_L$ in Lemma \ref{lem existence of min. and L^infty bound}. We now apply Theorem 2.4.2.7 in \cite{G85} and obtain $Q_L \in H^2(\Omega; \mathcal S_0)$. 
		
		By the Sobolev embedding and Morrey's inequality, it follows that $Q_L \in W^{1, 6}(\Omega; \mathcal S_0) \hookrightarrow C^{\frac{1}{2}}(\overline{\Omega}; \mathcal S_0)$. Therefore, $Q_L \in C^{2, \frac{1}{2}}(\Omega; \mathcal S_0)$ by the standard Schauder estimate and the equation of $Q_L$ in \eqref{bdv}. 
		
		Note that the trace operator maps $W^{1, 6}(\Omega; \mathcal S_0)$ to $W^{\frac{5}{6}, 6}(\p\Omega; \mathcal S_0)$. This fact, combined with the $L^\infty$-boundedness in Lemma \ref{lem existence of min. and L^infty bound}, implies that $f'_S(Q_L)\in W^{\frac{5}{6}, 6}\big(\p\Omega;\mathcal{S}_0\big)$. Applying Theorem 2.4.2.7 in \cite{G85} again yields $Q_L \in W^{2, 6}(\Omega; \mathcal S_0)$, which, together with Morrey's inequality, implies that $\nabla Q_L \in C^{\frac{1}{2}}(\overline{\Omega}; \mathcal S_0)$. This completes the proof.
	\end{proof}
	
	\section{Strong \texorpdfstring{$H^1$}{TEXT}-limit of \texorpdfstring{$Q_L$}{TEXT}}
	We first characterize all $ Q \in H^1(\Omega; \mathcal S_0)$ which satisfy $f_B(Q) = 0$ in $\Omega$ almost everywhere and $f_S(Q) =0$ on $\partial \Omega$ in the sense of trace. Due to \cite{M10}, the zero set of $f_B$ consists of all the matrices in
	\begin{align}
		\mathcal{S}_* := \left\{ Q \in \mathcal{S}_0 : Q = s_0 \left(n \otimes n - \frac{1}{3} \mathbf{I}_3\right)\h{5pt} \text{for some} \h{5pt} n \in \mathbb{S}^2\right\}.
		\label{def +ve uniaxial Q tensor}
	\end{align} Therefore, we only need to find all $Q \in H^1(\Omega; \mathcal S_*)$ with $f_S(Q) = 0$ on $\p \Omega$ in the sense of trace.
	\begin{lem} Given $Q \in H^1(\Omega;\mathcal{S}_*) $, we represent it by \begin{equation}\label{rep of Q}Q =s_0\left(n^Q  \otimes n^Q  - \frac{1}{3} \mathbf{I}_3\right), \h{20pt}\text{for some $n^Q\in H^1(\Omega;\mathbb{S}^2)$.}\end{equation}  Then in the sense of trace, $f_S(Q)=0$ on $\p\Omega$ is equivalent to $n^Q\in T (\p\Omega)$.
		\label{lem. f_s}
	\end{lem}
	We point out that the representation  \eqref{rep of Q} is proved by Ball--Zarnescu in \cite[Theorem 2]{BZ11}.
	\begin{proof}[\bf Proof of Lemma \ref{lem. f_s}] 
		By the representation of $Q$ in \eqref{rep of Q}, we have $|\h{0.5pt}Q\h{0.5pt}|^2=\frac{2}{3} s^2_0$ on $\p \Omega$ in the sense of trace. Recall \eqref{192}. It turns out  $$f_S(Q) = s_1 s_0^2 \h{1pt} \big(n^Q  \cdot \nu \big)^2 \h{20pt}\text{$\mathscr H^2$-a.e. on $\p \Omega$}.$$ Here we also use \cite[Proposition 3]{BZ11}. The proof is completed.
	\end{proof}
	
	Next, we prove the strong $H^1$-convergence of $Q_L$ as $L \to 0^+$. 	\begin{prop}\label{lem H1 conv}
		There exists $u^* \in H_T^1(\Omega; \mathbb S^2)$ such that, up to a subsequence, $Q_L$ converges strongly to \begin{equation}\label{rep. Q*}Q^* := s_0\Big(u^* \otimes u^* - \frac{1}{3} \mathbf{I}_3\Big)\end{equation} in $H^1$ as $L \to 0^+$. In addition, $u^*$ is a Dirichlet energy minimizer in  $H_T^1(\Omega; \mathbb S^2)$.
	\end{prop}
	\begin{proof}We claim $H_T^1(\Omega; \mathbb S^2) \neq \emptyset$. By \cite[Theorem 6.2]{HL87}, if there is a map $u \in H^{\frac{1}{2}}(\p \Omega; \mathbb S^2) \cap T(\p \Omega)$, then $u$ can be extended to $\Omega$ with the extension in $H_T^1(\Omega; \mathbb S^2)$. Since $\p \Omega$ is diffeomorphic to the sphere $\mathbb S^2$, we only need to construct a map $u \in H^{\frac{1}{2}}(\mathbb S^2; \mathbb S^2)$ with $u(x)$ tangent to $\mathbb S^2$ at $x$, for $\mathscr H^2$-a.e. $x$ in $\mathbb S^2$. Let $(r,\xi,\theta)$ be the spherical coordinates in $\mathbb{R}^3$, where $r$ is the radial variable, $\theta \in [\h{0.5pt}0,2\pi)$ is the azimuthal angle, and $\xi \in [\h{0.5pt}0,\pi]$ is the polar angle. At the point
\(
(\sin\xi \cos\theta,\ \sin\xi \sin\theta,\ \cos\xi)^\top,
\)
we assign the tangent vector
\[
(-\cos\xi \cos\theta,\ -\cos\xi \sin\theta,\ \sin\xi)^\top.
\]
The vector field obtained is in $H^{\frac{1}{2}}(\mathbb S^2; \mathbb S^2)$. It has two singularities at the north and south poles.

		Fix $u \in H_T^1(\Omega; \mathbb S^2)$ and denote by $Q_u$ the matrix $s_0\Big( u  \otimes u  - \frac{1}{3} \mathbf{I}_3\Big)$. The minimality of  $Q_L$ and Lemma \ref{lem. f_s} then infer
		\begin{align}
			\dfrac{1}{2} \int_{\Omega} |\h{0.5pt}\nabla Q_L \h{0.5pt}|^2  
			\leq \mathcal{E}_L( Q_L)
			\leq \mathcal{E}_L( Q_u) 
			= \dfrac{1}{2} \int_{\Omega} |\nabla Q_u|^2
			= s_0^2 \int_{\Omega} |\nabla u |^2.
			\label{506}
		\end{align}
		Therefore, $\|Q_L\|_{H^1(\Omega)}$ is uniformly bounded in $L$ due to the last estimate and Lemma \ref{lem existence of min. and L^infty bound}. There exists  $Q^* \in H^1(\Omega;\mathcal{S}_0)$ such that, up to a subsequence, $Q_L$ converges to $Q^*$ weakly in $H^1(\Omega;\mathcal{S}_0)$ as $L \to0^+$. Without loss of generality, we assume that along the same subsequence, $Q_L$ converges strongly to $Q^*$ in both $L^2(\Omega;\mathcal{S}_0)$ and $L^2(\p\Omega;\mathcal{S}_0)$. Furthermore, $Q_L$ converges to $Q^*$ almost everywhere in $\Omega$ and $\mathscr H^2$-a.e. on $ \p\Omega$ along the same subsequence.
		
		In addition, \eqref{506} also implies that $$
		L^{-1}\int_{\Omega}
		f_B(Q_L)
		+  \int_{\p\Omega}  f_S(Q_L) \h{2pt}\mathrm{d} \mathscr H^2
		\leq s_0^2 \h{1pt} L \int_{\Omega} |\nabla u|^2 \to 0 \h{20pt}\text{as $L \to 0^+$.}$$  Lebesgue's dominated convergence theorem then implies $f_B(Q^*)=0$ a.e. in $\Omega$ and $f_S(Q^*)=0$ on $ \p\Omega$ in the sense of trace. Using the fact that $f_B(Q^*) = 0$ a.e. in $\Omega$ and \cite[Theorem 2]{BZ11}, we get the representation of $Q^*$ in \eqref{rep. Q*} for some $u^* \in H^1(\Omega; \mathbb S^2)$. In addition, by \cite[Proposition 3]{BZ11} and Lemma \ref{lem. f_s}, the fact that $f_S(Q^*) = 0$ on $\p \Omega$ in the sense of trace implies that $u^* \in T(\p \Omega)$. Therefore, $u^* \in H_T^1(\Omega; \mathbb S^2)$.
		
		Due to the weak lower semi-continuity and \eqref{506},
		\begin{align*}
			2s_0^2\int_{\Omega} |\nabla u^*|^2 =\int_{\Omega} |\nabla Q^*|^2
			\leq \liminf_{L \h{0.5pt}\to \h{0.5pt}0^+}\int_{\Omega}  |\nabla Q_L|^2 
			\leq \limsup_{L \h{0.5pt}\to \h{0.5pt} 0^+}\int_{\Omega}  |\nabla Q_L|^2 
			\leq 2s_0^2 \int_{\Omega} |\nabla u|^2.
		\end{align*}Here, $Q_L$ is the same subsequence as before. Hence, $u^*$ is a Dirichlet energy minimizer in $H_T^1(\Omega; \mathbb S^2)$. Replacing $u$ on the rightmost side above by $u^*$, we obtain \begin{align*}
			\liminf_{L \h{0.5pt}\to \h{0.5pt} 0^+}\int_{\Omega}  |\nabla Q_L|^2 
			= \limsup_{L \h{0.5pt}\to \h{0.5pt} 0^+}\int_{\Omega}  |\nabla Q_L|^2 
			=\int_{\Omega} |\nabla Q^*|^2, 
		\end{align*}
		which implies that $Q_L$ converges to $Q^*$ strongly in $H^1(\Omega;\mathcal{S}_0)$ as $L \to 0^+$ along the subsequence.
	\end{proof} 
	
	\section{Boundary flattening map and the associated projection}\label{bdry flatten}
	Fix an arbitrary $x_0 \in \p \Omega$. Through translation and rotation, we assume $V_0 \cap \p \Omega$ is represented by $x_3 = h(x_1, x_2)$, where $V_0$ is a small neighborhood of $x_0$. $h$ is the smooth height function of the surface $V_0 \cap \p \Omega$. The points in $V_0 \cap \Omega$ satisfy $x_3 > h(x_1, x_2)$. Without loss of generality, we can also assume the first two coordinates of $x_0$ are equal to $0$. Therefore, $x_0 = (0, 0, h(0, 0))^\top$. 
	
	We define $$\varphi(x):=\big(x_1,x_2,x_3-h(x_1,x_2)\big)^\top \h{20pt}\text{for any $x \in V_0 \cap \Omega$}.$$ This map is a diffeomorphism between $V_0 \cap \Omega$ and $\varphi(V_0 \cap \Omega)$. Moreover, $V_0 \cap \p \Omega$ is mapped to  \begin{equation*}P_0 := \Big\{(x', 0)^\top : x' \in \mathbb R^2\Big\},\end{equation*} while $\varphi(V_0 \cap \Omega)$ lies above the plane $P_0$.  Take $R_0'$ sufficiently small. The map $\varphi$ then induces a diffeomorphism between $U_0 := \varphi^{-1}\big(B_{2R_0'}^+\big)$ and $B_{2R_0'}^+$. This diffeomorphism is our boundary flattening map.  
	
	Suppose $u$ is a Dirichlet energy minimizer in the configuration space $H_T^1(\Omega; \mathbb S^2)$. If we define
	\begin{align}
		\widetilde{u}\left(y\right) := u\left(\varphi^{-1}(y)\right) \h{15pt}\text{for any $y \in B_{2 R_0'}^+$,}
		\label{def transformed u}
	\end{align} then $\widetilde{u}$ is a minimizer of the following energy functional \begin{align}
		\widetilde{\mathcal{E}}_{v,\h{0.5pt} B_{2R_0'}^+} := \int_{B_{2R_0'}^+} a_{jl} \h{1.5pt}
		\p_{y_j} v \cdot
		\p_{y_l} v  
		\label{def dirichlet energy on transformed domain}
	\end{align}
	over the configuration space $$\left\{v\in H^1\big(B_{2R_0'}^+;\mathbb{S}^2\big) : v = \widetilde{u} \h{7pt}\text{$\mathscr H^2$-a.e. on $\p^+ B^+_{2R_0'}$}\h{6pt}\text{and}\h{7pt}v(\cdot)\in T_{\varphi^{-1}\left(\cdot\right)}\left(\p\Omega\right)\h{5pt}\text{$\mathscr H^2$-a.e. on}\h{4pt}  \p^0 B^+_{2R_0'}\right\}.$$ The coefficient matrix $(a_{jl})$ in \eqref{def dirichlet energy on transformed domain} is defined with its entries given by $$ a_{jl} (y):= \nabla_x \h{0.5pt} \varphi_j \cdot \nabla_x \h{0.5pt}\varphi_l \h{1.5pt} \Big|_{x \h{0.8pt}=\h{0.8pt} \varphi^{-1}(y)} \h{15pt}\text{for any $j, l = 1, 2, 3$ and $y \in B_{2R_0'}^+$.}$$ Since the boundary $\p \Omega$ is smooth, there are positive constants $\lambda_{\Omega}$ and $\Lambda_{\Omega}$ such that
	\begin{align}
		\lambda_{\Omega} \leq a_{jl}\left(y\right)p_j\h{0.5pt}p_l \leq \Lambda_{\Omega} \h{15pt}\text{for any $y\in B_{2R_0'}^+$ and $p \in \mathbb S^2$.}
		\label{ineq bdd of aij}
	\end{align} 
	Consequently, the Jacobian $\textup{det}\big(\nabla_y \h{1pt}\varphi^{-1}\big)$ is also bounded from below and above on $B_{2 R_0'}^+$ by positive constants depending only on $\Omega$. 
	
	Next, using the method of characteristics, we introduce a crucial projection map to $P_0$.  
	
	\begin{lem}\label{lem existence of pi}
		\sloppy There exist  $R_0'' \in \big(0, R_0'\h{1pt}\big]$ and a $C^2$ projection map $P : \overline{B^+_{R_0''}} \to P_0$ such that \begin{itemize}
			\item[$\mathrm{(1).}$] For each $i = 1, 2, 3$, it satisfies \begin{equation*} \big[\h{1.5pt}\nabla\varphi^{-1}\big|_{P(y)}\nabla P (y)\h{1.5pt}\big]_{i\h{0.5pt}k}\h{1.5pt}a_{k\h{0.5pt}3}(y)=0, \h{15pt}\text{ for all $y \in B^+_{R_0''}$}. \vspace{0.2pc}\end{equation*}
			\item[$\mathrm{(2).}$]  $P(y)=y$ for any  $y \in \p^0B^+_{R_0''}$. \vspace{0.4pc}
			\item[$\mathrm{(3).}$] There exists a positive constant $C_{\Omega}$ depending only on $\Omega$ such that $$\big|\h{0.5pt}\nabla P \h{0.5pt}\big| + \big|\h{0.5pt}\nabla^2 P\h{0.5pt}\big| \leq C_{\Omega} \h{15pt}\text{on $B^+_{R_0''}$.}$$
		\end{itemize}  
	\end{lem}
	\begin{proof}
		For any $y \in B_{2 R_0'}^+$, direct computations yield \begin{align*} &a_{13}(y)=-\p_{x_1}h \h{1pt}\big|_{x\h{0.5pt}=\h{0.5pt}\varphi^{-1}(y)}, \h{16pt}a_{23}(y)=-\p_{x_2}h \h{1pt}\big|_{x\h{0.5pt}=\h{0.5pt}\varphi^{-1}(y)}, \\[2mm] &a_{33}(y)=1+ \big(\p_{x_1}h\big)^2 \Big|_{x\h{0.5pt}=\h{0.5pt}\varphi^{-1}(y)} +\big(\p_{x_2}h\big)^2 \Big|_{x\h{0.5pt}=\h{0.5pt}\varphi^{-1}(y)}. \end{align*} Define $n(x):= \big(-\p_{x_1}h, -\p_{x_2}h,1\big)^\top$ to be the normal to $\p \Omega$ at $x$. It then turns out\vspace{0.2pc}
		\begin{align*}
			&\nabla\varphi^{-1}\big|_{P(y)}\nabla P (y) \begin{pmatrix}
				a_{13}(y) \\[1mm]
				a_{23}(y) \\[1mm]
				a_{33}(y)
			\end{pmatrix}  \\[2mm]
			&\h{15pt}= \begin{pmatrix}
				1 & 0 & 0 \\[1mm]
				0 & 1 & 0 \\[1mm]
				-n_1\big(P(y)\big) & -n_2\big(P(y)\big) & \h{4pt}1
				\h{4pt}\end{pmatrix}
			\begin{pmatrix}
				\p_{y_1}P_1  & \p_{y_2}P_1  & \p_{y_3}P_1  \\[1mm]
				\p_{y_1}P_2  & \p_{y_2}P_2  & \p_{y_3}P_2  \\[1mm]
				0 & 0 & 0 
			\end{pmatrix}
			\begin{pmatrix}
				n_1 \big(\varphi^{-1}(y)\big) \\[1mm]
				n_2 \big(\varphi^{-1}(y)\big)\\[1mm]
				|\h{0.5pt}n \h{0.5pt}|^2\big(\varphi^{-1}(y)\big) 
			\end{pmatrix}\\[2mm]
			&\h{15pt}=
			\begin{pmatrix}
				F\big(y,\nabla P_1\big) \\[1mm]
				F\big(y,\nabla P_2\big) \\[1mm]
				-n_1\big(P(y)\big)F\big(y,\nabla P_1\big)
				-n_2\big(P(y)\big)F \big(y,\nabla P_2\big)
			\end{pmatrix},			
		\end{align*} 
		where $P_1$, $P_2$ are the first two components of $P$. For a given function $\eta$, $$F\big(y,\nabla \eta\big):= n_1 \big(\varphi^{-1}(y)\big)\h{1pt}\p_{y_1} \eta  + n_2 \big(\varphi^{-1}(y)\big)\h{1pt}\p_{y_2}\eta  
		+| \h{0.5pt}n \h{0.5pt}|^2\big(\varphi^{-1}(y)\big) \h{1pt}\p_{y_3}\eta.$$ Applying \cite[Theorem 2, Section 3.2]{E10}, we can find a $R_0''\in \big(0,R_0'\h{0.5pt}\big]$ such that for $i = 1, 2$, the boundary value problems: $$F\big(y,\nabla P_i\big) =0 \h{15pt}\text{on $B_{R^{''}_0}^+$}; \h{20pt} P_i = y_i \h{15pt}\text{on $\p^0 B_{R^{''}_0}^+$}$$ have solutions. The proof is then completed.
	\end{proof}
	
	Recall items (7)-(8) in Section \ref{notation}. Now, we assume \begin{equation}\label{range of yp}(y, p) \in \p^0 B_{4R_0^*}^+ \times U_{4\delta_0} \left(\h{1pt} \mathbb{S}^2\cap T_{x_0}\p\Omega \h{1pt}\right) \h{15pt}\text{for some $\delta_0 < \frac{1}{10}$ and $R_0^*< \min\left\{\frac{1}{10}, \frac{\delta_0}{2}, \frac{R_0^{''}}{2} \right\}$.}\end{equation}Here, $R_0^*$ is chosen sufficiently small with the smallness depending on $\delta_0$ and $\Omega$. The shortest-distance projection $\Pi(y,\cdot)$ is therefore uniquely well defined in $ U_{4\delta_0}\left(\h{1pt}\mathbb{S}^2\cap T_{x_0}\p\Omega\h{1pt}\right)$ with its image equal to $\mathbb{S}^2 \cap T_{\varphi^{-1}(y)}\p\Omega$. Note that $\varphi$ is the boundary flattening map. 
	\begin{lem}\label{lem. regularity of proj. and dist}
		If we keep taking $R_0^*$ small enough with the smallness depending on $\delta_0$ and $\Omega$, then the following properties also hold for the projection $\Pi$: \begin{itemize}
			\item[$\mathrm{(1).}$] $\Pi$ is smooth at any $(y, p)$ satisfying \eqref{range of yp}. In addition, the $C^2$-norm of $\Pi$ on 
			\begin{equation}\label{defn of set K}U_0^* := \p^0 B_{4R_0^*}^+ \times U_{4\delta_0} \left(\h{1pt} \mathbb{S}^2\cap T_{x_0}\p\Omega \h{1pt}\right)\end{equation}  is bounded from above  by a positive constant depending only on $\Omega$. As a consequence, $$\big|\h{1pt}\Pi(y,p)-\Pi(z,p)\h{1pt}\big| \leq C_\Omega \h{1pt} |\h{0.5pt}y - z\h{0.5pt}| \leq \frac{1}{2} $$  for any
			$y, z \in \p^0B_{2R_0^*}^+$ and $p  \in U_{2\delta_0}\left(\h{1pt}\mathbb{S}^2 \cap T_{x_0} \p\Omega\h{1pt}\right)$.\vspace{0.4pc}
			\item[$\mathrm{(2).}$] The $L^\infty$-norms of 
			$$\nabla_p \h{1pt} \textup{dist}_{ \varphi^{-1}(y)}(p) \h{15pt} \text{and}\h{15pt}  \nabla_y \h{1pt} \textup{dist}_{ \varphi^{-1}(y)}(p)$$ on $U_0^*$ are bounded from above by a positive constant depending only on $\Omega$. Moreover, \begin{equation*}\sup 
				\Big\{\textup{dist}_{ \varphi^{-1}(y)}(p) :  p \in \mathbb{S}^2 \cap T_{\varphi^{-1}(w)} \partial \Omega \h{1pt}\Big\} \leq C_\Omega \h{1.5pt}|\h{0.5pt}y - w\h{0.5pt}| \leq \frac{1}{2} \h{15pt}\text{for any 
					$y, w \in \partial^0 B_{4R_0^*}^+$.} \end{equation*}
			\item[$\mathrm{(3).}$] It holds $$ \mathbb{S}^2 \cap T_{\varphi^{-1}(y)}\h{1pt}\p\Omega \subseteq U_{\delta_0\h{0.5pt}/\h{0.5pt}4}\left(\h{1pt}\mathbb{S}^2 \cap T_{x_0} \p\Omega \h{1pt} \right) \h{15pt}\text{ for any $ y \in \p^0B_{4R_0^*}^+ $}.$$
		\end{itemize}
		
		Applying the continuity of $P$ (see Lemma \ref{lem existence of pi}), we can find a $R_0 \leq R_0^*$ such that \begin{equation}\label{set relation of P}P(y) \in \p^0 B^+_{R_0^*/4} \h{15pt}\text{
				for any $y \in \overline{B_{4R_0}^+}$.}\end{equation} The smallness of $R_0$ depends only on $\Omega$.
	\end{lem}
	The proof of this lemma is straightforward. We omit it here. Using this projection, we can extend the minimizer $\widetilde{u}$ defined on the upper half-ball $B^+_{R_0}$ to the whole ball $B_{R_0}$. Let $\kappa \in \big(\big(2\delta_0\big)^{-1}, \infty\big)$ be a fixed constant throughout this paper. We then choose $\varphi_\kappa \in C^2\big(\h{1pt}[0,\infty), [0,1]\h{1pt}\big)$ satisfying the following properties:
	\begin{align}
		&\varphi_\kappa = 1 \h{15pt} \text{on } \left[ \h{1pt} 0, (4\kappa)^{-1}\h{1pt}\right] , \h{25pt} \varphi_\kappa \h{2pt}= \h{2pt}0 \h{30pt} \text{on } \left[\h{1pt} (2\kappa)^{-1},\infty\right),\nonumber\\[1mm]
		-8\kappa\h{1pt}\leq\h{2pt} &\varphi_\kappa' \leq 0 \h{15pt} \text{on } \h{2.5pt}[\h{1pt}0,\infty), \h{44pt} \big|\h{1pt}\varphi_\kappa''\h{1pt}\big| \leq 128\kappa^2 \h{8.5pt} \text{on } \h{2.5pt} [\h{1pt}0,\infty).
		\label{prop of varphi}
	\end{align}
	We define, for $y \in B_{R_0}\setminus B_{R_0}^+$, the following map
	\begin{align}\label{defn of ex}
		&\widetilde{u}^{\,\textup{ex}}(y):=\\[1mm]
        &\varphi_\kappa\pig( \textup{dist}_{ \varphi^{-1}\circ P(y^\star) }(\widetilde{u}(y^\star))\pig)
		\pig[2\Pi(P(y^\star),\widetilde{u}(y^\star)) - \widetilde{u}(y^\star)\pig]+ \left[1 - \varphi_\kappa\pig( \textup{dist}_{ \varphi^{-1}\circ P(y^\star) }(\widetilde{u}(y^\star))\pig)\right]\widetilde{u}(y^\star). \nonumber  
	\end{align} Note that if $y\in B^+_{R_0}$ satisfies $\varphi_\kappa\pig( 
    \textup{dist}_{ \varphi^{-1}\circ P(y) }(\widetilde{u}(y))\pig) > 0$, then it follows that $$\textup{dist}_{ \varphi^{-1}\circ P(y) }(\widetilde{u}(y))\leq (2\kappa)^{-1}< \delta_0.$$ Thus the projection in $\widetilde{u}^{\,\textup{ex}}$ is well-defined, due to the choice of $\delta_0$ and $R_0^*$ in Lemma \ref{lem. regularity of proj. and dist}. Now, we define the extension of $\widetilde{u}$ by
	\begin{equation}
		\mathcal{U}(y):=\widetilde{u}(y)\mathbbm{1}_{B^+_{R_0}}(y)+\widetilde{u}^{\,\textup{ex}}(y)\mathbbm{1}_{B_{R_0}\setminus B_{R_0}^+}(y).
		\label{def ext of tilde u}
	\end{equation}  
	\sloppy On $\p^0B^+_{R_0}$, we note that $\mathcal{U}(y)=\widetilde{u}^{\,\textup{ex}}(y)=\widetilde{u}(y)$ since $\Pi(P(y^\star),\widetilde{u}(y^\star))=\widetilde{u}(y)$ in the sense of trace. We also observe that
	\begin{align}
		&\left|\h{1pt}\varphi_\kappa\pig( \textup{dist}_{ \varphi^{-1}\circ P(y) }(p)\pig)\big[\Pi( P (y),p) - p\big]\h{1pt}\right|\leq \delta_0<\dfrac{1}{10} \quad
		\text{and}\quad \nonumber\\[1.5mm]
		&\left|\h{1pt}\varphi_\kappa'\pig( \textup{dist}_{ \varphi^{-1}\circ P (y) }(p)\pig)\big[\Pi( P (y),p) - p\big]\h{1pt}\right|\leq 4
		\label{ineq. varphi' dist(y,p)}, \h{20pt}\text{for any $y \in \overline{B^+_{R_0}}$ and $p \in \mathbb{R}^3$. }
	\end{align}
	In the following lemma, we derive the equation satisfied by the extended map $\mathcal{U}$.
	\begin{lem}
		There is a constant $\C{6}=\C{6}(\Omega)>0$ such that for any Dirichlet energy minimizer $u$ in the configuration space $H_T^1(\Omega; \mathbb S^2)$ and any $\psi \in H^1_0(B_{R_0};\mathbb{R}^3)
		\cap L^\infty(B_{R_0};\mathbb{R}^3)$, the corresponding extension $\mathcal{U}$ belongs to $ H^1(B_{R_0};\mathbb{R}^3)$ and solves the equation: 
        	\begin{align}
			\int_{B_{R_0}}\overline{a}_{ij}\h{1pt}\p_{y_i} \mathcal{U} 
			\cdot \p_{y_j} \psi +\int_{B_{R_0}}J\big[y,\widetilde{u}(\cdot)\big] \cdot  \psi(y) \h{1.5pt}\mathrm d\h{.5pt}y=0.
			\label{eq. of widetilde of u, whole ball}
		\end{align}Furthermore, for any $r>0$ and $a\in \mathbb{R}^3$ satisfying $B_r(a)\subset B_{R_0}$, it holds
		\begin{align}\int_{B_r(a)} \big|\nabla\mathcal{U} \big|^2
			\leq
			\int_{B_r^+(a)} \big|\nabla\widetilde{u}\big|^2
			+\C{6}
			\int_{B_r^+(a^\star)}\big|\nabla\widetilde{u}\big|^2
			+\C{6} \h{1pt}r^3.
			\label{ineq. of widetilde of u whole ball<half ball}
		\end{align} 
		\sloppy Here, $\overline{a}_{ij}$ is the extension of $a_{ij}$ such that it is even across the plane $\{y_3=0\}$ when both $i$ and $j$ take values in $\{1,2\}$ or $i=j=3$. The remaining components are odd across the plane $\{y_3=0\}$. The operator $J$ in \eqref{eq. of widetilde of u, whole ball} is defined in \eqref{1827} and satisfies \begin{align}\label{J}\pig|J\big[y,\widetilde{u}(\cdot)\big]\pig|\leq \C{6} \Big(1+\big|\h{0.5pt}\nabla\widetilde{u}(y) \h{0.5pt}\big|^2 \Big)\mathbbm{1}_{B_{R_0}^+}(y)
		+\C{6}\h{0.5pt}\kappa \Big(1+ \big|\h{0.5pt}\nabla\widetilde{u}(y^\star) \h{0.5pt}\big|^2\Big)\mathbbm{1}_{B_{R_0}^-}(y)\end{align} for almost every $y\in B_{R_0}$. 
		\label{lem eq of extended u}
	\end{lem}
	\begin{proof}We  restrict \( \Pi(z, p) \) on 
		\( \p^0B_{4R_0^*}^+ \times U_{4\delta_0}\big(\mathbb{S}^2 \cap T_{x_0} \p\Omega\big) \) and simply write $$\varphi_\kappa=\varphi_\kappa\pig( \textup{dist}_{ \varphi^{-1}\circ P(y) }(\widetilde{u}(y))\pig), \h{20pt} \varphi'_\kappa=\varphi'_\kappa\pig( \textup{dist}_{ \varphi^{-1}\circ P(y) }(\widetilde{u}(y))\pig).$$
		
		\noindent{\bf Part 1. Proof of \eqref{ineq. of widetilde of u whole ball<half ball}:} The definition of $\mathcal{U}$ in \eqref{def ext of tilde u} yields
		\begin{align*}
			 \nabla \mathcal{U}(y^\star)
			=\,&  
			2\varphi_\kappa'
			\pig[\Pi( P (y),\widetilde{u}(y)) - \widetilde{u}(y)
			\pig]
			\left[ 	\nabla_y\textup{dist}_{ \varphi^{-1}\circ P(y) }(p)\big|_{p=\widetilde{u}(y)}\right]^\top  \nonumber\\[1mm]
			&\h{10pt}+2\varphi_\kappa' 
			\pig[\Pi( P (y),\widetilde{u}(y)) - \widetilde{u}(y)
			\pig]
			\left\{ \big[\nabla\widetilde{u}(y)\big]^\top \nabla_p\textup{dist}_{ \varphi^{-1}\circ P(y) }(\widetilde{u}(y))\right\}^\top
			\nonumber\\[1mm]
			&\h{10pt}+2\varphi_\kappa 
			\pig[\nabla_z\Pi( P (y),\widetilde{u}(y))
			\nabla P (y)
			+\nabla_p\Pi( P (y),\widetilde{u}(y))
			\nabla\widetilde{u}(y)
			- \nabla\widetilde{u}(y) \pig] 
			+ \nabla\widetilde{u}(y).
		\end{align*}
        Here, $$\nabla_z\Pi( P (y),\widetilde{u}(y))=\nabla_z\Pi(z,p)\Big|_{z= P (y), \h{1.5pt}p=\widetilde{u}(y)} \h{15pt}\text{and} \h{15pt} \nabla_p\Pi( P (y),\widetilde{u}(y))=\nabla_p\Pi(z,p)\Big|_{z= P (y), \h{1.5pt}p=\widetilde{u}(y)}.$$ 
        For $y\in B^+_r(a)$ such that $B_r(a)\subset B_{R_0}$, we use \eqref{ineq. varphi' dist(y,p)} and Lemma \ref{lem. regularity of proj. and dist} to observe that
        		\begin{align}
			\big| \nabla \mathcal{U}(y^\star)\big|^2 
			\lesssim\,&_\Omega \,
			1+ \big|\nabla\widetilde{u}(y)\big|^2.
			\label{1443}
		\end{align}
		\eqref{ineq. of widetilde of u whole ball<half ball} follows from this estimate. \vspace{0.4pc}
		
		\noindent{\bf Part 2. Proof of \eqref{eq. of widetilde of u, whole ball}:} We let $\varepsilon > 0$ and define $\eta_{\varepsilon}^*(y)= \eta_{\varepsilon}^*(y_3): B_{R_0} \to [\h{1pt}0,\infty)$ as follows:
        \begin{align*}
\eta_{\varepsilon}^*(y)= 0\h{10pt}\text{if $|\h{0.5pt}y_3\h{0.5pt}|\le \varepsilon$}; \h{20pt}\eta_{\varepsilon}^*(y)= \varepsilon^{-1}|\h{0.5pt}y_3\h{0.5pt}|-1 \h{10pt}\text{if $\varepsilon \le |y_3|\le 2\varepsilon$}; \h{20pt}\eta_{\varepsilon}^*(y)= 1\h{10pt}\text{if $|\h{0.5pt}y_3\h{0.5pt}|\ge 2\varepsilon$}.             
        \end{align*}
For any $\psi \in C^1_0(B_{R_0};\mathbb{R}^3)$, it holds
		\begin{align}
			\int_{B_{R_0}} \overline{a}_{ij} \h{1pt}\partial_{y_i} \mathcal{U} \cdot \partial_{y_j}\big(\eta_{\varepsilon}^* \psi\big)
			= \int_{B_{R_0}} \overline{a}_{ij} \h{1pt}\eta_{\varepsilon}^* \h{1pt}\partial_{y_i}  \mathcal{U} \cdot \partial_{y_j}\psi 
			+\int_{B_{R_0}} \overline{a}_{i3}\h{1.5pt}\partial_{y_3} \eta_{\varepsilon}^* \h{1.5pt}\partial_{y_i}  \mathcal{U} \cdot\psi.
			\label{eq:3.17}
		\end{align}
		Changing variables, we evaluate  
		\begin{align*}
			\int_{B_{R_0}\setminus B_{R_0}^+} \overline{a}_{i3}  \h{1.5pt}\partial_{y_3} \eta_{\varepsilon}^* \h{1.5pt}\partial_{y_i} \mathcal{U} \cdot \psi 
			=\,&\int_{B_{R_0}\setminus B_{R_0}^+}\overline{a}_{i3} \h{1.5pt} \partial_{y_3} \eta_{\varepsilon}^* \h{1.5pt} \partial_{y_i}\widetilde{u}^{\,\textup{ex}} \cdot \psi\\[1mm]
			=\,& \varepsilon^{-1}\int_{B_{R_0}\cap \h{1pt} \{\h{0.5pt}\varepsilon\h{0.5pt}<\h{0.5pt}y_3\h{0.5pt}<\h{0.5pt}2\varepsilon\h{0.5pt}\}} a_{i3}(y) \h{1.5pt}\psi(y^\star)\cdot \partial_{y_i} \big[\widetilde{u}^{\,\textup{ex}}(y^\star)\big] \h{1.5pt}\mathrm d\h{.5pt}y.
		\end{align*}
		Using the definition in \eqref{def ext of tilde u} and the last equality, we then obtain
		\begin{align}
			 &\int_{B_{R_0}} \overline{a}_{i3} \h{1.5pt}\partial_{y_3} \eta_{\varepsilon}^* \h{1.5pt}\partial_{y_i} \mathcal{U} \cdot \psi 
			= \varepsilon^{-1}\int_{B_{R_0}\cap \h{1pt}\{\h{.5pt}\varepsilon\h{.5pt}<\h{.5pt}y_3\h{.5pt}<\h{.5pt}2\varepsilon\h{.5pt}\}} a_{i3}(y)\h{1pt}\Big\{\psi(y)\cdot \partial_{y_i} \widetilde{u}(y)
			+\psi(y^\star)\cdot \partial_{y_i} \big[\widetilde{u}^{\,\textup{ex}}(y^\star)\big]\h{1pt}\Big\} \h{1.5pt}\mathrm d\h{.5pt}y\nonumber\\[1mm]
			&= \varepsilon^{-1}\int_{B_{R_0}\cap \h{1pt}\{\h{.5pt}\varepsilon\h{.5pt}<\h{.5pt}y_3\h{.5pt}<\h{.5pt}2\varepsilon\h{.5pt}\}} a_{i3}(y) 
			\pig[\psi(y)-\psi(y^\star)\pig]\cdot \partial_{y_i} \Big\{\varphi_\kappa \pig[ \widetilde{u}(y) - \Pi( P (y),\widetilde{u}(y)) \pig]\Big\} \h{1.5pt}\mathrm d\h{.5pt}y\nonumber\\[1mm]
			&+ \varepsilon^{-1}\int_{B_{R_0}\cap \h{1pt} \{\h{.5pt}\varepsilon\h{.5pt}<\h{.5pt}y_3\h{.5pt}<\h{.5pt}2\varepsilon\h{.5pt}\}} a_{i3}(y) 
			\pig[\psi(y)+\psi(y^\star)\pig]\cdot \partial_{y_i} \pig[\left(1 - \varphi_\kappa\right) \widetilde{u}(y)+\varphi_\kappa\Pi( P (y),\widetilde{u}(y)) \pig]\h{1.5pt}\mathrm d\h{.5pt}y.
			\label{eq:3.18}
		\end{align}
		
		\noindent{\bf Part 2.1. Estimates of $\textup{I}(\varepsilon)$ and $\textup{II}(\varepsilon)$:} Define the two integrals in the last equality of \eqref{eq:3.18} by $\textup{I}(\varepsilon)$ and $\textup{II}(\varepsilon)$, respectively. In this part, we estimate these two integrals. \vspace{0.2pc}
        
		First, the fact that $\psi \in C^1_0(B_{R_0};\mathbb{R}^3)$ implies
		\begin{align*}
			\big|\h{0.5pt}\textup{I}(\varepsilon)\h{0.5pt}\big|
			\lesssim\,&_{\Omega}\,  \|\nabla \psi\|_{L^\infty} \int_{B_{R_0}\cap \h{1pt}\{\h{0.5pt}\varepsilon\h{0.5pt}<\h{0.5pt}y_3\h{0.5pt}<\h{0.5pt}2\varepsilon\h{0.5pt}\}}\left|\h{1pt} \nabla_{y} \Big\{\varphi_\kappa \big[\h{0.5pt} \widetilde{u}(y) - \Pi( P (y),\widetilde{u}(y)) \h{0.5pt}\big]\right\}\Big| \h{2pt}\mathrm d\h{.5pt}y.
		\end{align*}
        Then, direct computations and inequalities in \eqref{ineq. varphi' dist(y,p)} imply
		\begin{align*}
			\big|\h{0.5pt}\textup{I}(\varepsilon)\h{0.5pt}\big|
			\lesssim\,&_{\Omega, \h{1pt}\psi}\,   \int_{B_{R_0}\cap \h{1pt}\{\h{0.5pt}\varepsilon\h{0.5pt}<\h{0.5pt}y_3\h{0.5pt}<\h{0.5pt}2\varepsilon\h{0.5pt}\}}\Big|\h{1pt}\varphi_\kappa'
			\h{1pt}\nabla_{y}\textup{dist}_{ \varphi^{-1}\circ P(y) }(\widetilde{u}(y)) \h{1pt}\Big|\h{2pt}
			\Big|\h{1pt} \widetilde{u}(y)-\Pi(P(y),\widetilde{u}(y)) \h{1pt}\Big| \nonumber\\[1mm]
			&+ \int_{B_{R_0}\cap \h{1pt}\{\h{0.5pt}\varepsilon\h{0.5pt}<\h{0.5pt}y_3\h{0.5pt}<\h{0.5pt}2\varepsilon\h{0.5pt}\}} \varphi_\kappa
			\Big[\h{1.5pt}\big| \nabla_{y}\widetilde{u} \big|
			 + \big| \nabla_{y} P\big| \h{1pt}\big| \h{0.5pt}\nabla_z\Pi( P (y),\widetilde{u}(y)) \h{0.5pt}\big|
			+ \big| \nabla_{y}\widetilde{u} \big| \h{1pt} \big|\h{0.5pt} \nabla_p\Pi( P (y),\widetilde{u}(y)) \h{0.5pt}\big| \h{1.5pt} \Big]\nonumber\\[1mm]
			\lesssim\,&_{\Omega}\,  \int_{B_{R_0}\cap \h{1pt}\{\h{.5pt}\varepsilon\h{.5pt}<\h{.5pt}y_3\h{.5pt}<\h{.5pt}2\varepsilon\h{.5pt}\}}1+ \big|\h{1pt}\nabla\widetilde{u}\h{1pt}\big|.
		\end{align*}
		
        Second, we estimate $\textup{II}(\varepsilon)$ by  	\begin{align*}
			&\big|\h{1pt}\textup{II}(\varepsilon)\h{1pt}\big|
			\lesssim_{\,\Omega} \h{1pt} \varepsilon^{-1}\left|\h{1pt}\int_{B_{R_0}\cap \h{1pt}\{\h{0.5pt}\varepsilon\h{0.5pt}<\h{0.5pt}y_3\h{0.5pt}<\h{0.5pt}2\varepsilon\h{0.5pt}\}}a_{i3}(y)\h{1pt}\varphi_\kappa\h{1pt}\Big[\psi(y)+\psi(y^\star)\Big]\cdot \p_{y_i}\Big[\h{1pt}\Pi( P (y),\widetilde{u}(y))\h{1pt}\Big]\h{1pt}\right|   \nonumber\\[1mm]
			&\h{5pt} +  \varepsilon^{-1} \| \h{0.5pt}\psi \h{0.5pt}\|_{L^\infty}  \int_{B_{R_0}\cap \h{1pt}\{\h{0.5pt}\varepsilon\h{0.5pt}<\h{0.5pt}y_3\h{0.5pt}<\h{0.5pt}2\varepsilon\h{0.5pt}\}} \left(1-\varphi_\kappa\right) \big|\h{.5pt}\nabla_{y}\widetilde{u}\h{.5pt}\big| + \Big|\h{1pt}\varphi_\kappa'
			\h{1pt}\nabla_{y}\textup{dist}_{ \varphi^{-1}\circ P(y) }(\widetilde{u}(y)) \h{1pt}\Big|\h{2pt}
			\Big|\h{1pt} \widetilde{u}(y) - \Pi( P (y),\widetilde{u}(y)) \h{1pt}\Big|.
		\end{align*}\normalsize
       Denoting by $\textup{III}(\varepsilon)$ the first term on the right-hand side above and applying the Cauchy-Schwarz inequality, we obtain	\begin{align*}
			\big|\h{0.5pt}\textup{II}(\varepsilon)\h{0.5pt}\big|
			&\lesssim_{\,\Omega, \h{1pt}\psi\,} \textup{III}(\varepsilon) + \varepsilon^{-1}\left(\int_{B_{R_0}\cap \h{1pt} \{\h{0.5pt}0\h{0.5pt}<\h{0.5pt}y_3\h{0.5pt}<\h{0.5pt}2\varepsilon\h{0.5pt}\}} \big|\h{0.5pt}1-\varphi_\kappa \h{0.5pt}\big|^2\right)^{\frac{1}{2}}
			\left(\int_{B_{R_0}\cap \h{1pt}\{\h{0.5pt}\varepsilon\h{0.5pt}<\h{0.5pt}y_3\h{0.5pt}<\h{0.5pt}2\varepsilon\h{0.5pt}\}} \big|\h{0.5pt}\nabla\widetilde{u} \h{0.5pt}\big|^2 \right)^{\frac{1}{2}}\nonumber\\[1mm]
            &+ \varepsilon^{-1} \int_{B_{R_0}\cap \h{1pt} \{\h{0.5pt}0\h{0.5pt}<\h{0.5pt}y_3\h{0.5pt}<\h{0.5pt}2\varepsilon\h{0.5pt}\}}\Big|\h{1pt}\varphi_\kappa'
			\big[\h{1pt} \widetilde{u}(y) - \Pi( P (y),\widetilde{u}(y)) \h{1pt}\big] \h{1pt}\Big| \nonumber\\[1mm]
			&+ \varepsilon^{-1}\left(\int_{B_{R_0}\cap \h{1pt} \{\h{0.5pt}0\h{0.5pt}<\h{0.5pt}y_3\h{0.5pt}<\h{0.5pt}2\varepsilon\h{0.5pt}\}}\Big|\h{1pt}\varphi_\kappa'
			\big[\h{1pt} \widetilde{u}(y) - \Pi( P (y),\widetilde{u}(y)) \h{1pt}\big] \h{1pt}\Big|^2\right)^{\frac{1}{2}}
			\left(\int_{B_{R_0}\cap \h{1pt}\{\h{0.5pt}\varepsilon\h{0.5pt}<\h{0.5pt}y_3\h{0.5pt}<\h{0.5pt}2\varepsilon\h{0.5pt}\}} \big|\h{0.5pt}\nabla\widetilde{u} \h{0.5pt}\big|^2 \right)^{\frac{1}{2}}.
		\end{align*}\normalsize
		As $\widetilde{u} - \Pi( P,\widetilde{u} )=0$ and $1-\varphi_\kappa=0$ on $\p^0B_{R_0}^+$ in the sense of trace, by \eqref{prop of varphi} and the Poincar\'e inequality with respect to $y_3$, it turns out that
		\begin{align*}
			\big|\h{1pt}\textup{II}(\varepsilon) \h{1pt}\big|
			\lesssim_{\,\Omega,\h{1pt}\psi\,} \textup{III}(\varepsilon) &+ \kappa \h{1pt}\varepsilon^{\frac{1}{2}} 
			\left(\int_{B_{R_0}\cap \h{1pt}\{\h{0.5pt}0\h{0.5pt}<\h{0.5pt}y_3\h{0.5pt}<\h{0.5pt}2\varepsilon\h{0.5pt}\}}1+ \big|\h{1pt}\p_{y_3}\widetilde{u}\h{1pt}\big|^2 \right)^{\frac{1}{2}} \nonumber\\[1mm]
            &+\kappa \left(\int_{B_{R_0}\cap \h{1pt} \{\h{0.5pt}0\h{0.5pt}<\h{0.5pt}y_3\h{0.5pt}<\h{0.5pt}2\varepsilon\h{0.5pt}\}}1+ \big|\h{1pt}\p_{y_3}\widetilde{u}\h{1pt}\big|^2 \right)^{\frac{1}{2}}
			\left(\int_{B_{R_0}\cap \h{1pt} \{\h{0.5pt}\varepsilon\h{0.5pt}<\h{0.5pt}y_3\h{0.5pt}<\h{0.5pt}2\varepsilon\h{0.5pt}\}} \big|\h{1pt}\nabla\widetilde{u} \h{1pt}\big|^2 \right)^{\frac{1}{2}}.
		\end{align*}
        
		Third, for the test function $\eta_\varepsilon^* \h{1pt} \Phi $ where $$\Phi \in H^1\big(B^+_{R_0};\mathbb{R}^3\big)
		\cap L^\infty\big(B^+_{R_0};\mathbb{R}^3\big)\cap\Big\{\psi:\psi = 0 \h{4pt}\text{on $\p^+B^+_{R_0}$ }\Big\},$$ the equation of $\widetilde{u}$ gives 
         \begin{align*}
			\int_{B_{R_0}^+}  \eta_\varepsilon^* \Big[\h{1pt}a_{ij} \h{1pt}
			\p_{y_i} \widetilde{u} \cdot \p_{y_j} \Phi
            - \left(\widetilde{u} \cdot \Phi \right) \left(a_{ij}\h{1pt}\p_{y_i} \widetilde{u} \cdot
			\p_{y_j} \widetilde{u} \right) \Big]
			+	\int_{B_{R_0}\cap \h{1pt} \{\h{0.5pt}\varepsilon\h{0.5pt}<\h{0.5pt}y_3\h{0.5pt}<\h{0.5pt}2\varepsilon\h{0.5pt}\}}\p_{y_j}\eta_\varepsilon^* \h{1pt}a_{ij} 
			\h{1pt}\Phi \cdot \p_{y_i} \widetilde{u}
			= 0.
		\end{align*}Therefore,  
        \begin{align}
        \lim_{\varepsilon \h{0.5pt}\to \h{0.5pt} 0^+} \varepsilon^{-1}\int_{B_{R_0}\cap \h{1pt}\{\h{0.5pt}\varepsilon\h{0.5pt}<\h{0.5pt}y_3\h{0.5pt}<\h{0.5pt}2\varepsilon\h{0.5pt}\}} a_{i3} \h{1pt}\Phi \cdot 
		\p_{y_i} \widetilde{u} = - \int_{B_{R_0}^+}  a_{ij} \h{1pt}
			\p_{y_i} \widetilde{u} \cdot \p_{y_j} \Phi
            - \left(\widetilde{u} \cdot \Phi \right) \left(a_{ij}\h{1pt}\p_{y_i} \widetilde{u} \cdot
			\p_{y_j} \widetilde{u} \right).
        \label{683}
        \end{align} 
        We choose $\Phi$ with its $j$-th component given by $\varphi_\kappa\big[\h{1pt}\psi(y)+\psi(y^\star)\h{1pt}\big]\cdot \p_{p_j}\Pi( P (y),p)\h{1pt}\Big|_{p\h{0.5pt}=\h{0.5pt}\widetilde{u}(y)}$. Note that  $$\Pi(y,p)=\frac{\Pi^*(x,p)}{\big|\h{1pt}\Pi^*(x,p)\h{1pt}\big|}\h{2pt}\bigg|_{x\h{1pt}=\h{1pt}\varphi^{-1}(y)} \quad\text{on} \quad
		 B_{4R_0^*}^+ \times U_{4\delta_0}\big(\mathbb{S}^2 \cap T_{x_0} \p\Omega\big),$$ where $\Pi^*(x,p):=p- \big(p\cdot \hat{n}(x) \big)\h{1pt} \hat{n}(x)$ with $n(x)= \big(-\p_{x_1}h(x),-\p_{x_2}h(x),1\big)^\top$. Then for $y\in \p^0B^+_{R_0}$,  
        \begin{align*}
            &\Phi(y)\cdot \hat{n}(\varphi^{-1}(y))\\[1mm]
            &\h{10pt}=2 \h{0.5pt}\varphi_\kappa \h{1pt}\psi_k(y) \h{1pt}  \p_{p_j}\Pi_k( y,p)
            \h{1pt}\hat{n}_j(x)\\[1mm]
            &\h{10pt}=
2\h{0.5pt}\varphi_\kappa \h{1pt} \psi_k(y)\left(\frac{\big(\delta_{jk}-\hat n_j(x)\h{1pt}\hat n_k(x)\big)\h{1pt}\hat n_j(x)}{\big|\h{0.5pt}\Pi^*(x,p)\h{0.5pt}\big|}
-\frac{\Pi_k^*(x,p)\,\Pi_\ell^*(x,p)\big(\delta_{j\ell}-\hat n_j(x)\hat n_\ell(x)\big)\hat n_j(x)}{\big|\h{0.5pt}\Pi^*(x,p)\h{0.5pt}\big|^3}\right) = 0.
        \end{align*}Here, the variable $p$ is evaluated at $\widetilde{u}(y)$ and $x=\varphi^{-1}(y)$.
        By \eqref{683} and the weak formulation of the equation of $\widetilde{u}$, the right-hand side of \eqref{683} is equal to $0$. Hence,
		\begin{align*}
			\lim_{\varepsilon\h{0.5pt}\to\h{0.5pt} 0^+} \textup{III}(\varepsilon)
			= \lim_{\varepsilon\h{0.5pt}\to\h{0.5pt} 0^+} \varepsilon^{-1}\left|\h{1pt}\int_{B_{R_0}\cap \h{1pt} \{\h{0.5pt}\varepsilon\h{0.5pt}<\h{0.5pt}y_3\h{0.5pt}<\h{0.5pt}2\varepsilon\h{0.5pt}\}}
			a_{i3} \h{1pt}
			\varphi_\kappa\big[\h{1pt}\psi(y)+\psi(y^\star)\h{1pt}\big]\cdot \p_{z_j}\Pi( P (y),\widetilde{u}(y))\h{1pt}\p_{y_i} P _j
			\h{2pt}\mathrm d\h{.5pt}y \h{1pt}\right|.
		\end{align*}
        Using Item (1) in Lemma \ref{lem existence of pi} yields \begin{align*}
         a_{i3}(y) \h{1pt}\p_{z_j}\Pi( P (y), \cdot) \h{1pt}\p_{y_i} P _j 
         = 
		a_{i3}(y) \h{1pt}
\p_{z_j}\big(\varphi^{-1}\big)_k\h{1pt}\Big|_{ z \h{1pt}= \h{1pt}P (y)}\p_{y_i} P _j \h{2pt} \p_{x_k}\frac{\Pi^*(x, \cdot)}{\big|\h{0.5pt}\Pi^*(x, \cdot)\h{0.5pt}\big|}\h{1.5pt}\bigg|_{x\h{1pt}=\h{1pt}\varphi^{-1}( P (y))} =0 \h{10pt}\text{in $B^+_{R_0}$.}
         \end{align*}
We therefore conclude that $\displaystyle\lim_{\varepsilon\h{0.5pt}\to \h{0.5pt} 0^+} \textup{III}(\varepsilon)=0$. \vspace{0.4pc}
		
		\noindent{\bf Part 2.2. Estimate of equation \eqref{eq:3.17}:} From the arguments in Part 2.1, we may pass $\varepsilon \to 0^+$ in  \eqref{eq:3.17} and obtain the following limit:
		\begin{align}
			\lim_{\varepsilon\h{0.5pt}\to\h{0.5pt}0^+}\int_{B_{R_0}} \overline{a}_{ij} \h{1pt}  \partial_{y_i} \mathcal{U}  \cdot \partial_{y_j}\big(\eta_{\varepsilon}^* \h{0.5pt}  \psi \big) = \int_{B_{R_0}} \overline{a}_{ij}\h{1pt}  \partial_{y_i} \mathcal{U}  \cdot \partial_{y_j}\psi  .
			\label{1767}
		\end{align}  
		By the equation of $\widetilde{u}$ and the definition of $\widetilde{u}^{\,\textup{ex}}$ in \eqref{def ext of tilde u}, we write the integral on the left-hand side as the sum of integrals of the following three terms
		\begin{align*}
		\int_{B_{R_0}^+} \big(\widetilde{u}\cdot \psi\big) \h{1pt}\eta_{\varepsilon}^* \h{1pt} a_{ij}\h{1pt} \partial_{y_i} \widetilde{u} \cdot \partial_{y_j}\widetilde{u}   
			+  a_{ij} \h{1pt} \partial_{y_i} \big(\widetilde{u}^{\,\textup{ex}}(y^\star)- \widetilde{u}\big)\cdot \partial_{y_j}\big(\eta_{\varepsilon}^*(y^\star) \psi(y^\star)\big) + a_{ij} \h{1pt}\partial_{y_i} \widetilde{u} \cdot \partial_{y_j}\big(\eta_{\varepsilon}^*(y^\star) \psi(y^\star)\big).
		\end{align*}
		Letting $F(y,p):=\varphi_\kappa\pig( \textup{dist}_{ \varphi^{-1}\circ P(y) }(p)\pig) \big(\Pi( P (y),p)-p\big)$, we rewrite the last integral by
        	\begin{align*}
		\int_{B_{R_0}^+} \Big(a_{ij} \h{1pt} \partial_{y_i} \widetilde{u} \cdot \partial_{y_j}\widetilde{u} \Big) \h{1pt}\Big(\widetilde{u} \cdot \big(\h{1pt}\eta_{\varepsilon}^* \h{1pt} \psi +\eta_{\varepsilon}^*(y^\star)\h{1pt} \psi(y^\star) \h{1pt}\big) \Big) + 2\int_{B_{R_0}^+} a_{ij} \h{1pt} \partial_{y_i} \pig(F(y,\widetilde{u}(y))\pig)\cdot \partial_{y_j}\big(\h{1pt}\eta_{\varepsilon}^*(y^\star) \psi(y^\star)\h{1pt}\big).
		\end{align*}
        Taking $\varepsilon \to 0^+$ and using \eqref{1767} then yield
		\begin{align}
			&\int_{B_{R_0}} \overline{a}_{ij} \h{1pt} \partial_{y_i} \mathcal{U}  \cdot \partial_{y_j} \psi \nonumber\\[1mm] 
			&\h{10pt}= \int_{B_{R_0}^+} \Big(a_{ij} \h{1pt} \partial_{y_i} \widetilde{u}  \cdot \partial_{y_j}\widetilde{u} \Big)\h{1pt}\Big( \widetilde{u}\cdot \big(\h{1pt}\psi +\psi(y^\star)\h{1pt}\big) \Big) +\lim_{\varepsilon\h{0.5pt}\to\h{0.5pt}0^+}2\int_{B_{R_0}^+} a_{ij} \h{1pt} \partial_{y_i} \pig(F(y,\widetilde{u}(y))\pig)\cdot \partial_{y_j} \Psi_{\varepsilon}.
			\label{1787}
		\end{align}
		Here, $\Psi_{\varepsilon}(y):=\eta_{\varepsilon}^*(y^\star) \psi(y^\star)$. We now consider the integral in the last term of \eqref{1787}. First, it holds
		\begin{align*}
\int_{B^+_{R_0}}a_{ij}\h{1pt}\p_{y_i}\pig(F(y,\widetilde{u}(y))\pig)\cdot\p_{y_j}\Psi_{\varepsilon} = \int_{B^+_{R_0}}a_{ij} \left(\p_{y_i}F(y, p) \h{1.5pt}\Big|_{p \h{1pt}=\h{1pt} \widetilde{u}(y)} + \p_{y_i}\widetilde{u}_k \h{1.5pt}\p_{p_k}F(y, p)\h{1.5pt}\Big|_{p \h{1pt}=\h{1pt}\widetilde{u}(y)} \right)\cdot\p_{y_j}\Psi_{\varepsilon}.
		\end{align*}
		Use the test function $\pig[\h{1pt}\nabla_p F(y,p)\h{1pt}\pigr]_{p \h{1pt} = \h{1pt}\widetilde{u}(y)}^\top \Psi_\varepsilon$ in the weak formulation of the equation of $\widetilde{u}$ and integrate by parts. It turns out that
		\begin{align*}
			&\h{-13pt}\int_{B^+_{R_0}}a_{ij}\h{1pt}\p_{y_i}\pig(F(y,\widetilde{u}(y))\pig)\cdot\p_{y_j}\Psi_{\varepsilon}\\[1mm]
			=\,&-\int_{B^+_{R_0}}\left\{\p_{y_j}a_{ij} \h{1pt}\p_{y_i}F(y, p)\h{1.5pt}\Big|_{p \h{1pt} = \h{1pt} \widetilde{u}(y)}
			+a_{ij}\h{1pt}\p_{y_i\h{0.5pt}y_j}F(y, p)\h{1.5pt}\Big|_{p \h{1pt}=\h{1pt} \widetilde{u}(y)}
			+2a_{ij} \h{1pt}\p_{y_j}\widetilde{u}_k \h{1pt}\p_{y_i \h{0.5pt}p_k}F(y,p)\h{1.5pt}\Big|_{p \h{1pt}=\h{1pt} \widetilde{u}(y)}\right\}\cdot\Psi_{\varepsilon}\\[1mm]
			&+\int_{B^+_{R_0}}a_{ij} \left\{\p_{y_i}\widetilde{u} \cdot\p_{y_j}\widetilde{u} \left(\widetilde{u}_k \h{1pt}\p_{p_k}F(y, p)\h{1.5pt}\Big|_{p \h{1pt}=\h{1pt} \widetilde{u}(y)} \right) 
			- \p_{y_i}\widetilde{u}_k \h{1pt}\p_{y_j}\widetilde{u}_l \h{1pt}
			\p_{p_l \h{0.5pt}p_k}F(y, p) \h{1.5pt}\Big|_{p \h{1pt}= \h{1pt}\widetilde{u}(y)}\right\}\cdot \Psi_{\varepsilon}.
		\end{align*}
		Passing $\varepsilon \to 0^+$ in the above equality and using \eqref{1787}, we then get
		\begin{align*}
			&\int_{B_{R_0}} \overline{a}_{ij} \h{1pt} \partial_{y_i} \mathcal{U}  \cdot \partial_{y_j} \psi = \int_{B_{R_0}^+} \Big(a_{ij} \h{1pt} \partial_{y_i} \widetilde{u}  \cdot \partial_{y_j}\widetilde{u} \Big)\h{1pt}\Big( \widetilde{u}\cdot \big(\h{1pt}\psi +\psi(y^\star)\h{1pt}\big) \Big) \\[1mm]&-2\int_{B^+_{R_0}}\left\{\p_{y_j}a_{ij} \h{1pt}\p_{y_i}F(y, p)\h{1.5pt}\Big|_{p \h{1pt} = \h{1pt} \widetilde{u}(y)}
			+a_{ij}\h{1pt}\p_{y_i\h{0.5pt}y_j}F(y, p)\h{1.5pt}\Big|_{p \h{1pt}=\h{1pt} \widetilde{u}(y)}
			+2a_{ij} \h{1pt}\p_{y_j}\widetilde{u}_k \h{1pt}\p_{y_i \h{0.5pt}p_k}F(y,p)\h{1.5pt}\Big|_{p \h{1pt}=\h{1pt} \widetilde{u}(y)}\right\}\cdot \psi(y^\star)\\[1mm]
			&+2\int_{B^+_{R_0}}a_{ij} \left\{\p_{y_i}\widetilde{u} \cdot\p_{y_j}\widetilde{u} \left(\widetilde{u}_k \h{1pt}\p_{p_k}F(y, p)\h{1.5pt}\Big|_{p \h{1pt}=\h{1pt} \widetilde{u}(y)} \right) 
			- \p_{y_i}\widetilde{u}_k \h{1pt}\p_{y_j}\widetilde{u}_l \h{1pt}
			\p_{p_l \h{0.5pt}p_k}F(y, p) \h{1.5pt}\Big|_{p \h{1pt}= \h{1pt}\widetilde{u}(y)}\right\}\cdot \psi(y^\star).
		\end{align*}
		Applying the change of variables induces
		\begin{align}\label{1827}
			& \int_{B_{R_0}} \overline{a}_{ij} \h{1pt} \partial_{y_i} \mathcal{U}  
			\cdot \partial_{y_j} \psi = - \int_{B_{R_0}} J\big[y,\widetilde{u}(\cdot)\big]\cdot \psi(y)\h{1.5pt}\mathrm d\h{.5pt}y := \int_{B_{R_0}^+} a_{ij} \h{1pt}
			 \partial_{y_i} \widetilde{u} \cdot \partial_{y_j}\widetilde{u} \left(\widetilde{u} 
			 \cdot \psi \right) \\[1mm]
			&+\int_{B_{R_0}^-} 
			\Big\{a_{ij} \h{1pt} \partial_{y_i} \widetilde{u}
			 \cdot \partial_{y_j}\widetilde{u} \h{2pt}
			 \widetilde{u}\Big\}_{y^\star}
			 \cdot \psi -2\int_{B_{R_0}^-}\left\{\p_{y_j}a_{ij} \h{1pt}\p_{y_i}F(y, p)\h{1.5pt}\Big|_{p \h{1pt} = \h{1pt} \widetilde{u}(y)} \right\}_{y^\star} \cdot \psi \nonumber\\[1mm]
             &-2\int_{B_{R_0}^-}\left\{a_{ij}\h{1pt}\p_{y_i\h{0.5pt}y_j}F(y, p)\h{1.5pt}\Big|_{p \h{1pt}=\h{1pt} \widetilde{u}(y)}
			+2a_{ij} \h{1pt}\p_{y_j}\widetilde{u}_k \h{1pt}\p_{y_i \h{0.5pt}p_k}F(y,p)\h{1.5pt}\Big|_{p \h{1pt}=\h{1pt} \widetilde{u}(y)}\right\}_{y^\star}\cdot \psi  \nonumber\\[1mm]
			&+2\int_{B_{R_0}^-}a_{ij} (y^\star) \left\{\p_{y_i}\widetilde{u} \cdot\p_{y_j}\widetilde{u} \left(\widetilde{u}_k \h{1pt}\p_{p_k}F(y, p)\h{1.5pt}\Big|_{p \h{1pt}=\h{1pt} \widetilde{u}(y)} \right) 
			- \p_{y_i}\widetilde{u}_k \h{1pt}\p_{y_j}\widetilde{u}_l \h{1pt}
			\p_{p_l \h{0.5pt}p_k}F(y, p) \h{1.5pt}\Big|_{p \h{1pt}= \h{1pt}\widetilde{u}(y)}\right\}_{y^\star}\cdot \psi. \nonumber
		\end{align}
		\sloppy By Lemma \ref{lem existence of pi}, it follows that \begin{align*}\pig|J\big[y,\widetilde{u}(\cdot)\big]\pig|\lesssim_{\,\Omega\,} \Big(1+ \big|\h{0.5pt}\nabla\widetilde{u}(y) \h{0.5pt}\big|^2\Big)\mathbbm{1}_{B_{R_0}^+}(y)
		+\kappa \Big(1+ \big|\h{0.5pt}\nabla\widetilde{u}(y^\star) \h{0.5pt}\big|^2\Big)\mathbbm{1}_{B_{R_0}^-}(y) \h{15pt}\text{for a.e. $y\in B_{R_0}$.}\end{align*} 
	The proof is complete.
\end{proof}

	\section{An extension proposition}
    
The following proposition, which extends a mapping defined on an upper hemisphere to an upper half-ball, will be used in Section \ref{bubbling analysis} to construct comparison mappings in our bubbling analysis. Recalling the diffeomorphism $\varphi$ in Section \ref{bdry flatten} and $R_0$ in Lemma \ref{lem. regularity of proj. and dist}, we prove
    \begin{prop}
		Fix $\big(q, \widetilde{a}\big) \in \mathbb{R}^3 \times \p^0 B^+_{R_0}$ and let $\{r_k\}$ be a positive sequence tending to zero as $k \to \infty$. Define  \begin{align}\label{defn of varphi_k}\varphi_k(y):=\varphi^{-1}\big(\widetilde{a}+r_k \h{1pt}y^0\big) \h{15pt}\text{for all $k \geq 100$}. \end{align} There exist constants $\eps{7}>0$, $\mu>2$, $\epsilon_0\in\left(0,\frac{1}{2}\right)$, $\delta'\in\left(0,\frac{\eps{7}}{2}\right)$, and $\C{18}>0$, all depending only on $\Omega$, such that the following holds: \vspace{0.2pc} 
        
        If $k\in\mathbb{N}$, $\rho\in \left(0, \frac{1}{2}\right)$, $\epsilon\in (0,\epsilon_0)$, $\delta\in(0,\delta')$, and $v \in H^1\big(\p^+B_{\rho}^{+};\mathbb{S}^2\big)$ satisfy \vspace{0.2pc}  
        \begin{align*}   
        &\mathrm{(a).} 	\h{10pt} v(\cdot) \in T_{\varphi_k(\cdot)}\p \Omega \h{3pt}\text{ on $\p B_{\rho} \cap \big\{y_3=0\big\}$ in the sense of trace}\h{0.5pt};\\[1.5mm]
		&\mathrm{(b).} \h{10pt} 	r_k^2\h{1pt}\epsilon^{2-\mu}\left( r_k^2 \left( \epsilon^{2-\mu}+\epsilon^{-2}   
			\right)
			+ \mathcal{E}^*_{v,\h{0.5pt}\p^+ B_\rho^+}  
			+ \rho^{-2}\h{1pt}\epsilon^{-2}
			\left(\epsilon^{2-\mu}  
			+  \epsilon ^{-2}\right) 
			W^{\h{0.5pt}q}_{ v,\h{0.5pt} \varphi_k,\h{0.5pt} \p^+ B_\rho^+ }\right)  \leq \delta^2; \\[1.5mm] 
        &\mathrm{(c).} \h{10pt}\epsilon^{2-\mu}\sqrt{\delta}
        +4 \h{0.5pt}\epsilon^{3-\mu} \h{1pt}\rho\h{.8pt}r_k\h{1pt}
				\leq \sqrt{\eps{7}}; \h{72pt}\mathrm{(d).} \h{10pt} \mathcal{E}^*_{v,\h{0.5pt}\p^+B_{\rho}^{+}} 
			\h{1pt}W^{\h{0.5pt}q}_{ v,\h{0.5pt} \varphi_k,\h{0.5pt} \p^+ B_\rho^+ }
			\leq \delta^2\h{0.5pt}\rho^2 \h{0.5pt}\epsilon^\mu,
		\end{align*}
		then there exists an extension $\overline{v} \in H^1\big(B^+_\rho; \mathbb{S}^2\big)$ with the following properties: 
		\begin{align*}
        &\mathrm{(1).} \h{10pt}
        \text{$\overline{v}=v$ on $\p^+B_{\rho}^{+}$ and $\overline{v}(\cdot) \in U_{\delta_0/2}\big(\mathbb{S}^2 \cap T_{\varphi_k(\cdot)}\p \Omega\big)$ on $\p^0 B^+_\rho$  in the sense of trace}\h{0.5pt};\\[1.5mm]
			&\mathrm{(2).} \h{10pt} \mathcal{E}^*_{\h{1pt}\overline{v},\h{0.5pt}B_{\rho}^{+}}  \leq \C{18} \left(\rho\h{1pt}\epsilon \h{1pt} \mathcal{E}^*_{v,\h{0.5pt}\p^+B^+_\rho}
			+\rho^{-1}\h{0.5pt}\epsilon^{-\mu} \h{1pt}
			W^{\h{0.5pt}q}_{ v,\h{0.5pt} \varphi_k,\h{0.5pt} \p^+ B_\rho^+ }
			+\rho \h{1pt}r_k^2 \left( \epsilon^{-1}
			\frac{ \log \epsilon}{\log \alpha} +\epsilon^{-2}
			+ \epsilon^{2-\mu} \right)\right);\\[1.5mm]
			&\mathrm{(3).} \h{10pt} W^{\h{0.5pt}q}_{\overline{v},\h{0.5pt}	\varphi_k,\h{0.5pt}B_{\rho}^{+}}   
			\leq \C{18}\left( \rho \h{1pt} \epsilon^{2-\mu} \h{1pt}
			W^{\h{0.5pt}q}_{ v,\h{0.5pt} \varphi_k,\h{0.5pt} \p^+ B_\rho^+ }
			+\rho^3 r_k^2 \left(1+\epsilon^{4-\mu}\right) 
			+\rho^3 \h{0.5pt}\epsilon \h{1pt} r_k^2 \left(\frac{ \log \epsilon}{\log \alpha}+1\right)\right) .
		\end{align*}The constant $\alpha \in (0,1)$ in  items (2) and (3) is a universal constant obtained in Lemma \ref{lem ext to annulus}.
		\label{lem ext from sphere to ball}
	\end{prop} 
    We first introduce the following lemma (proved in \cite[Lemma 4.1]{SU82}), which extends a mapping defined on the boundary of a cylinder to its interior. 
	\begin{lem}\label{lem ext to cylinder}
		Let $\sigma>0$ and $\rho \in[\h{0.5pt}\sigma,2\sigma\h{0.5pt}]$. Given \( v^\textup{top}, v^\textup{bot} \in H^1(D_\rho; \mathbb{S}^2) \) and \( v^\textup{lat} \in H^1(\p D_\rho; \mathbb{S}^2) \), if $$v^\textup{top} = v^\textup{bot} = v^\textup{lat} \h{15pt} \text{on $\p D_\rho$ in the sense of trace,}$$ then there exists an extension \( \overline{v} \in H^1\big(C_{\rho, \h{0.5pt}\sigma}; \mathbb{S}^2\big) \), where $C_{\rho, \h{0.5pt}\sigma}:=D_\rho\times (-\sigma,\sigma)$, such that $$  \overline{v}(\cdot, \sigma) = v^\textup{top}(\cdot) \h{15pt}\text{and}\h{15pt}   \overline{v}(\cdot, - \sigma)  = v^\textup{bot}(\cdot) \h{20pt}\text{on $D_\rho$, in the sense of trace.}$$  Moreover, $ \overline{v}(y, z) = v^\textup{lat}(y)$, for $(y, z) \in \p D_\rho \times [-\sigma,\sigma]$. We also have \begin{align*}
			&\mathrm{(1).} \h{10pt}\mathcal{E}^*_{\h{1pt}\overline{v}, \h{0.8pt}C_{\rho, \h{0.5pt}\sigma}}
			\leq C\sigma \left(\mathcal{E}^*_{v^\textup{top}, \h{0.5pt}D_\rho} 
			+ \mathcal{E}^*_{v^\textup{bot}, \h{0.5pt}D_\rho}  
			+ \sigma \h{1pt}\mathcal{E}^*_{v^\textup{lat}, \h{0.5pt}\p D_\rho}  \right);\\[1mm]
			&\mathrm{(2).} \h{10pt} W^{\h{0.5pt}q}_{\overline{v}, \h{0.8pt}C_{\rho, \h{0.5pt}\sigma}}
			\leq C\sigma \left(W^{\h{0.5pt}q}_{v^\textup{top}, \h{0.5pt}D_\rho} 
			+ W^{\h{0.5pt}q}_{v^\textup{bot}, \h{0.5pt}D_\rho}  
			+ \sigma \h{1pt} W^{\h{0.5pt}q}_{v^\textup{lat}, \h{0.5pt}\p D_\rho}  \right).
		\end{align*}In $\mathrm{(2)}$ above, $q$ is an arbitrary vector in $\mathbb R^3$. The constant $C$ is universal and independent of $q$. 
	\end{lem} 
	
	We modify the preceding lemma to deal with the tangential boundary conditions on $\p \Omega$. To simplify the notation, the half-cylinder $D_\rho^+ \times (-\sigma,\sigma)$ is denoted by $C_{\rho,\h{0.5pt}\sigma, \h{0.5pt}+}$. The closure of $C_{\rho,\h{0.5pt}\sigma, \h{0.5pt}+}$ is denoted by $\overline{C_{\rho,\h{0.5pt}\sigma, \h{0.5pt}+}}$. The top, bottom, lateral faces, and the flat boundary of $ C_{\rho,\h{0.5pt}\sigma,\h{0.5pt}+}$ are respectively denoted by $$C^\textup{top}_{\rho,\h{0.8pt}\sigma, \h{0.8pt}+}:=\overline{D^+_\rho} \times\{\sigma\}, \h{5pt} C^\textup{bot}_{\rho,\h{0.8pt}\sigma, \h{0.8pt}+}:=\overline{D^+_\rho} \times\{-\sigma\}, \h{5pt}
	C^\textup{lat}_{\rho,\h{0.8pt}\sigma,\h{0.8pt}+}:=\overline{\p^+D^+_\rho} \times[-\sigma,\sigma], \h{5pt}C^\textup{flat}_{\rho,\h{0.8pt}\sigma, \h{0.8pt}+} := \big(\p^0 D^+_\rho\big) \times [-\sigma,\sigma]. $$  Suppose $\varphi_0$ is a bi-Lipschitz homeomorphism from $\overline{C_{\rho,\h{0.5pt}\sigma, \h{0.5pt}+}}$ to a subset of $\overline{B^+_{R_0}}$, satisfying $\varphi_0\big(C^\textup{flat}_{\rho,\h{0.8pt}\sigma, \h{0.8pt}+}\big) \subseteq \p^0 B^+_{R_0}.$ Then, on $\overline{C_{\rho,\h{0.5pt}\sigma, \h{0.5pt}+}}$, we define the map $\varphi_* :=\varphi^{-1} \circ [\h{0.8pt}\varphi_0\h{0.5pt} ]^0$. 
	
	\begin{lem}\label{lem ext to cylinder tan}
		Fix $\sigma>0$ and $\rho\in[\h{0.5pt}3\sigma,4\sigma\h{0.5pt}]$. Let $v^\textup{top}, v^\textup{bot} \in H^1(D_{\rho}^+; \mathbb{S}^2)$   and   $v^\textup{lat} \in H^1(\p^+ D_\rho^+; \mathbb{S}^2)$ be the maps that satisfy
		$$v^\textup{top} = v^\textup{bot} = v^\textup{lat} \h{15pt} \text{on $\p^+ D_\rho^+$ in the sense of trace.}$$
		Denoting by $\mathbbm{1}_U$ the characteristic function of a set $U$, and $y' = (y_1, y_2)$ for $y \in \mathbb R^3$, we define \begin{align}\label{defn of v}v(y):=v^\textup{top}(y')\h{1pt}\mathbbm{1}_{C^\textup{top}_{\rho, \h{0.5pt}\sigma, \h{0.5pt}+}}(y)
			+v^\textup{bot}(y') \h{1pt}\mathbbm{1}_{C^\textup{bot}_{\rho,\h{0.5pt}\sigma,\h{0.5pt}+}}(y)
			+v^\textup{lat}(y') \h{1pt}\mathbbm{1}_{C^\textup{lat}_{\rho,\h{0.5pt}\sigma,\h{0.5pt}+}}(y)\end{align} for any $y \in C^\textup{top}_{\rho, \h{0.5pt}\sigma, \h{0.5pt}+}\cup 
		C^\textup{bot}_{\rho, \h{0.5pt}\sigma, \h{0.5pt}+}\cup
		C^\textup{lat}_{\rho, \h{0.5pt}\sigma, \h{0.5pt}+}$.  Recall $\delta_0$ in \eqref{range of yp} and let $\delta_1\in \big[\h{0.5pt}0, \frac{\delta_0}{4} \h{0.5pt}\big]$. If $v$ satisfies \begin{align}\label{assmption of v} v(\cdot) \in U_{\delta_1}\left(\mathbb{S}^2\cap T_{\varphi_{*}(\cdot)}\p \Omega\right) \h{20pt} \text{on $\partial \h{1pt}C^\textup{flat}_{\rho,\h{0.5pt}\sigma, \h{0.5pt}+}$, in the sense of trace,} \end{align}  then there exists an extension \( \overline{v} \in H^1\big(C_{\rho,\h{0.5pt}\sigma,\h{0.5pt}+}; \mathbb{S}^2\big) \) such that $$  \overline{v} = v \h{20pt}\text{on $C^\textup{top}_{\rho, \h{0.5pt}\sigma, \h{0.5pt}+}\cup 
			C^\textup{bot}_{\rho, \h{0.5pt}\sigma, \h{0.5pt}+}\cup
			C^\textup{lat}_{\rho, \h{0.5pt}\sigma, \h{0.5pt}+}$, in the sense of trace.}$$  Moreover, there is a constant $\C{15}=\C{15}(\Omega)>0$ such that 
		\begin{align*} &\mathrm{(1).} \h{10pt}\overline{v} (\cdot)\in U_{\delta_1+L_{\varphi_0}\sigma}\left(\mathbb{S}^2\cap T_{\varphi_{*}(\cdot)}\p \Omega\right) \h{20pt}\text{ on $C^\textup{flat}_{\rho,\h{0.5pt}\sigma,\h{0.5pt}+} \setminus \{0\}$, in the sense of trace};\\[1.5mm]
			&\mathrm{(2).} \h{10pt} \mathcal{E}^*_{\h{1pt}\overline{v},\h{0.8pt}C_{\rho,\h{0.5pt}\sigma,\h{0.5pt}+}}
			\leq \C{15}\h{0.5pt}\sigma \h{0.5pt}\left(\mathcal{E}^*_{v^\textup{top}, \h{0.5pt}D_\rho^+} 
			+ \mathcal{E}^*_{v^\textup{bot}, \h{0.5pt}D_\rho^+}  
			+ \sigma \h{0.5pt}\mathcal{E}^*_{v^\textup{lat},\h{0.5pt}\p^+ D_\rho^+}\right);\\[1.5mm]
			&\mathrm{(3).} \h{10pt} W_{\overline{v},\h{1pt}\varphi_*, \h{1pt} C_{\rho,\h{0.5pt}\sigma,\h{0.5pt}+}}^{\h{0.5pt}q}
			\leq \C{15}\h{0.5pt}\sigma\left(W_{v,\h{1pt}\varphi_*, \h{1pt}C^\textup{top}_{\rho,\h{0.8pt}\sigma, \h{0.8pt}+}}^{\h{0.5pt}q} 
			+ W_{v,\h{1pt}\varphi_*, \h{1pt}C^\textup{bot}_{\rho,\h{0.8pt}\sigma, \h{0.8pt}+}}^{\h{0.5pt}q}\right) \\[1.5mm]
			&\h{95pt}+ \C{15}\h{0.5pt}\sigma^2 \left( W_{v,\h{1pt}\varphi_*, \h{1pt}\p^+ D_\rho^+\times\{\sigma\}}^{\h{0.5pt}q}
			+ \sigma\displaystyle\sup_{y^1,\h{1pt}y^2\h{1.5pt}\in\h{1.5pt} C_{\rho,\h{0.5pt}\sigma, \h{0.5pt}+}}
			\left|\h{1pt} \varphi_0 (y^1)- \varphi_0 (y^2) \h{1pt}\right|^2\right). \end{align*}
		The constant $L_{\varphi_0}$ is defined by $\C{15}\h{1pt}\big\|\h{0.5pt}\nabla \varphi_0\h{0.5pt}\big\|_{L^\infty(C^\textup{flat}_{\rho,\h{0.5pt}\sigma, \h{0.5pt}+})}$. The $3$-vector $q$ is arbitrary.
	\end{lem}
	\begin{proof}
		There exists a bi-Lipschitz homeomorphism 
		$\varphi_1: \partial^+ B_\sigma^+ \to \p\h{1pt} C_{\rho,\h{0.5pt}\sigma, \h{0.5pt}+}\setminus C^\textup{flat}_{\rho,\h{0.5pt}\sigma,\h{0.5pt}+}$ such that the circle $\p B_\sigma\cap\big\{y_3=0\big\}$ is mapped onto $  \p \h{1pt} C^\textup{flat}_{\rho, \h{0.5pt}\sigma, \h{0.5pt}+} $. It can be extended to a bi-Lipschitz homeomorphism 
		$\overline{\varphi_1} :\overline{B_\sigma^+} \to \overline{C_{\rho,\h{0.5pt}\sigma, \h{0.5pt}+}}$ by setting $$
		\overline{\varphi_1}(y) := \frac{|\h{0.5pt}y\h{0.5pt}|}{\sigma} \h{1pt} \varphi_1(\sigma\h{0.5pt}\widehat{y}\h{0.5pt}), \h{20pt}\text{where $y \in \overline{B_\sigma^+}$}.$$   Accordingly, $\p^0B_\sigma^+$ is mapped to $C^\textup{flat}_{\rho,\h{0.5pt}\sigma, \h{0.5pt}+}$ under $\overline{\varphi_1}$. Suppose $\{ e_r, e_\xi, e_\theta\}$ is the standard spherical orthonormal basis of $\mathbb R^3$. We compute the Jacobian matrix \begin{align*}
			&\nabla\overline{\varphi_1}  =\dfrac{1}{\sigma}\h{0.5pt}\varphi_1 \h{0.5pt} e_r^\top
			+\nabla^{\textup{tan}} \varphi_1, \h{10pt}
			\text{where $\nabla^{\textup{tan}} \varphi_1
				= \frac{1}{\sigma}\,\partial_\xi \varphi_1\h{0.5pt} e_\xi^\top
				+ \frac{1}{\sigma\sin\xi}\,\partial_\theta \varphi_1\h{0.5pt} e_\theta^\top.$}		\end{align*} Scaling the domain and range spaces of $\overline{\varphi_1}$ and $\varphi_1$ by $\sigma$, we can find positive constants $\lambda$ and $\Lambda$ such that \begin{align}\label{bound of jacobian} \lambda\h{0.5pt} \mathbf{I}_3 \h{1.5pt}\leq \h{1.5pt} \nabla\overline{\varphi_1}\big(\nabla\overline{\varphi_1}\big)^\top, \h{3pt}\nabla^{\textup{tan}} \varphi_1 \big(\nabla^{\textup{tan}} \varphi_1\big)^{\top}  \h{1.5pt}\leq \h{1.5pt} \Lambda\h{0.5pt} \mathbf{I}_3 \h{15pt}\text{on $\overline{B_\sigma^+}$ and $\p^+ B_\sigma^+$, respectively.}  
		\end{align}Using $\overline{\varphi_1}$ and $\varphi_1$, we define $$
		\varphi_2
		:= \varphi_1 \left( \sigma \h{1pt} \widehat{\overline{\varphi_1}^{\,-1}}\right) \h{20pt}\text{on $C_{\rho,\sigma,+} \setminus \{0\}$}.$$ Recalling $v$ in \eqref{defn of v}, we then let $\overline{v} := v \circ \varphi_2$. The definitions of $\overline{\varphi_1}$ and $\varphi_2$ yield $$y=\frac{\left|\h{1pt}\overline{\varphi_1}^{\,-1}(y)\h{1pt}\right|}{\sigma} \h{1pt} \varphi_2(y) \h{20pt} \text{for any $y \in C_{\rho,\h{0.5pt}\sigma, \h{0.5pt}+} \setminus \{0\}$.}$$ Thus, $\varphi_2$ is the identity map on  $ \p \h{0.5pt} C_{\rho, \h{0.5pt}\sigma,\h{0.5pt}+}\setminus C^\textup{flat}_{\rho,\h{0.5pt}\sigma, \h{0.5pt}+}$, implying $\overline{v} =v $ on  $ \p \h{0.5pt} C_{\rho, \h{0.5pt}\sigma,\h{0.5pt}+}\setminus C^\textup{flat}_{\rho,\h{0.5pt}\sigma, \h{0.5pt}+}$. \vspace{0.2pc}
		
		It turns out that $$\widehat{\overline{\varphi_1}^{\,-1}} \in \p B_1 \cap\{y_3=0\} \h{20pt}\text{on $C^\textup{flat}_{\rho, \h{0.5pt}\sigma, \h{0.5pt}+}\setminus\{0\}$.}$$ Hence, $\varphi_2 : C^\textup{flat}_{\rho, \h{0.5pt}\sigma, \h{0.5pt}+}\setminus\{0\} \to \partial \h{1pt} C^\textup{flat}_{\rho, \h{0.5pt}\sigma, \h{0.5pt}+}$. The assumption \eqref{assmption of v} on $v$ then implies  $$\overline{v} = v \circ \varphi_2   \in  U_{\delta_1}\left(\mathbb{S}^2\cap T_{\varphi_{*}\circ \h{1pt}\varphi_2} \h{1pt}\p \Omega\h{1pt}\right) \h{20pt}\text{on $C^\textup{flat}_{\rho, \h{0.5pt}\sigma, \h{0.5pt}+} \setminus\{0\} $}.$$ Fix $y \in C^\textup{flat}_{\rho,\sigma,+}\setminus\{0\}$ and utilize Item (2) in Lemma \ref{lem. regularity of proj. and dist}. It follows
		\begin{align*}
			\textup{dist}_{\h{0.5pt} \varphi_*(y)}(\overline{v} (y))
			&\leq \inf \Big\{
			\big|\h{0.5pt}\overline{v} (y) - p \h{0.5pt}\big| : p\in \mathbb{S}^2\cap T_{\varphi_{*}\circ\h{1pt}\varphi_2(y)} \h{1pt}\p \Omega \Big\}\\[1mm]
			&\quad+\sup \Big\{\textup{dist}_{ \h{0.5pt} \varphi_*(y)}(p) : p\in \mathbb{S}^2\cap T_{\varphi_{*}\circ \h{1pt}\varphi_2(y)} \h{1pt}\p \Omega \Big\}\\[1mm]
			&\leq \delta_1
			+ C_\Omega\sup_{y\h{1pt}\in \h{1pt} C^\textup{flat}_{\rho,\h{0.5pt}\sigma, \h{0.5pt}+}}\big|\h{1pt}\varphi_0(y)-\varphi_0\circ\varphi_2(y) \h{1pt}\big| \h{2pt}\leq\h{2pt} \delta_1
			+ \sigma\h{1pt}C_\Omega \h{1pt}\big\|\h{0.5pt}\nabla \varphi_0\h{0.5pt}\big\|_{L^\infty(C^\textup{flat}_{\rho,\sigma,+})}.
		\end{align*}
		Then item (1) holds. \vspace{0.2pc}
		
		The relation  $\overline{v} \h{0.3pt} \circ \h{0.3pt} \overline{\varphi_1}(y) 
		=v \h{0.3pt}\circ \h{0.3pt} \varphi_1(\sigma\widehat{y}\h{0.5pt})$ for any $y\in B_\sigma^+ \setminus\{0\}$ implies $$\mathcal{E}^*_{\h{1pt}\overline{v}\h{1pt}\circ \h{1pt}\overline{\varphi_1}, \h{0.8pt}B^+_\sigma}
		= \sigma \h{1pt} \mathcal{E}^*_{v\h{1pt}\circ\h{1pt}\varphi_1, \h{1pt}\p^+ B^+_\sigma}.$$  Changing variables by $\overline{\varphi_1}$ and $\varphi_1$, we obtain the desired estimate in item (2), using \eqref{bound of jacobian}. \vspace{0.2pc} 
		
		Following the same arguments above yields \begin{align}\label{W control}W_{\overline{v},\h{0.5pt}C_{\rho,\h{0.5pt}\sigma, \h{0.5pt}+}}^{\h{0.5pt}q}
			\hspace{1.5pt}\lesssim\hspace{1.5pt}
			\sigma W_{v\h{1pt}\circ\h{1pt}\varphi_1, \h{0.5pt}\p^+ B^+_\sigma}^{\h{0.5pt}q} \hspace{1.5pt}\lesssim \hspace{1.5pt}
			\sigma \left(W_{v, \h{0.5pt}C^\textup{top}_{\rho,\h{0.8pt}\sigma, \h{0.8pt}+}}^{\h{0.5pt}q} 
			+ W_{v, \h{0.8pt}C^\textup{bot}_{\rho,\h{0.8pt}\sigma, \h{0.8pt}+}}^{\h{0.5pt}q} + \sigma \h{1pt} W_{v, \h{0.8pt}\p^+ D_\rho^+\times\{\sigma\}}^{\h{0.5pt}q}\right).  \end{align}For $W_{\overline{v},\h{0.8pt}\varphi_*, \h{0.8pt}C_{\rho,\h{0.5pt}\sigma, \h{0.5pt}+}}$, we use (1) in Lemma \ref{lem. regularity of proj. and dist} and estimate as follows:
		\begin{align*}
			W_{\overline{v},\h{0.8pt}\varphi_*, \h{0.8pt}C_{\rho,\h{0.5pt}\sigma, \h{0.5pt}+}}
			&\hspace{1.5pt}\lesssim\hspace{1.5pt} W_{v\left(\varphi_1(\sigma\,\widehat{\cdot}\,)\right),\h{1pt}\varphi_*\h{1pt}\circ\h{1.2pt}\overline{\varphi_1}, \h{1pt}B^+_\sigma}\nonumber\\[1mm]
			&\h{1.5pt}\lesssim_{\,\Omega\,} W_{v\left( \varphi_1(\sigma\,\widehat{\cdot}\,)\right),\h{1pt}\varphi_*\left(\varphi_1(\sigma\,\widehat{\cdot}\,)\right),\h{1pt}B^+_\sigma
			}
			+\sigma^3\h{1pt}\sup_{B^+_\sigma} \left|\h{1pt}[\h{0.5pt}\varphi_0\h{0.5pt}]^0\circ\varphi_1(\sigma\,\widehat{\cdot}\,) - [\h{0.5pt}\varphi_0\h{0.5pt}]^0\circ\overline{\varphi_1}(\cdot) \h{1pt}\right|^2. \end{align*}
		Since $\varphi_1(\sigma\,\widehat{\cdot}\,)$ is homogeneous of degree $0$, it turns out
		\begin{align}\label{3212}
			W_{\overline{v},\h{0.8pt}\varphi_*, \h{0.8pt}C_{\rho,\h{0.5pt}\sigma, \h{0.5pt}+}}
			&\hspace{1.5pt}\lesssim_{\, \Omega\,}\hspace{1.5pt} \sigma \h{1pt} W_{
				v \h{1pt}\circ\h{1pt} \varphi_1, \h{1pt}\varphi_*\h{1pt}\circ\h{1pt}\varphi_1, \h{1pt}\p^+B^+_\sigma}
			+\sigma^3			\displaystyle\sup_{y^1,\h{1pt}y^2\h{1pt}\in \h{1pt} C_{\rho,\sigma,+}}
			\left|\h{1pt} \varphi_0 (y^1)- \varphi_0 (y^2) \h{1pt}\right|^2\\[1mm]
			&\h{-35pt}\lesssim\sigma \h{1pt} W_{v,\h{0.8pt}\varphi_*, \h{0.8pt}C^\textup{top}_{\rho,\h{0.5pt}\sigma, \h{0.5pt}+}} 
			+ \sigma \h{1pt} W_{v,\h{0.8pt}\varphi_*, \h{0.8pt}C^\textup{bot}_{\rho,\h{0.5pt}\sigma, \h{0.5pt}+}} + \sigma \h{1pt} W_{v,\h{0.8pt}\varphi_*, \h{0.8pt}C^\textup{lat}_{\rho,\h{0.5pt}\sigma, \h{0.5pt}+}} +\sigma^3			\displaystyle\sup_{y^1,\h{1pt}y^2\h{1pt}\in \h{1pt} C_{\rho,\sigma,+}}
			\left|\h{1pt} \varphi_0 (y^1)- \varphi_0 (y^2) \h{1pt}\right|^2. \nonumber
		\end{align}
		Notice that $v$ is independent of the third spatial variable on $C^\textup{lat}_{\rho,\sigma,+}$. It follows that 
		\begin{align*}
			W_{v,\h{0.8pt}\varphi_*, \h{0.8pt}C^\textup{lat}_{\rho,\h{0.5pt}\sigma, \h{0.5pt}+}} &= \int_{C^\textup{lat}_{\rho,\sigma,+}} \left[ \textup{dist}_{\varphi_*(y)}\big(v^\textup{lat}(y')\big) \right] ^2\h{1.5pt}\mathrm d\h{.5pt}\mathscr{H}_y^2 \\
            &\lesssim\h{1.5pt} \sigma \int_{\p^+D^+_\rho\times\{\sigma\}} \left[ \textup{dist}_{\varphi_* (y)}\big(v^\textup{lat}(y')\big) \right] ^2\h{1.5pt}\mathrm d\h{.5pt}\mathscr{H}_y^1 \\[1mm]
			&\quad+\int_{C^\textup{lat}_{\rho,\sigma,+}} \left| \textup{dist}_{\varphi_* (y)}\big(v^\textup{lat}(y')\big) 
			-\textup{dist}_{\varphi_* (y)}\big(v^\textup{lat}(y')\big) \h{1pt}\pigr|_{y_3=\sigma}\right| ^2 
			\h{1.5pt}\mathrm d\h{.5pt}\mathscr{H}_y^2. \end{align*}
		Using Item (1) in Lemma \ref{lem. regularity of proj. and dist} infers
		\begin{align*}
			W_{v,\h{0.8pt}\varphi_*, \h{0.8pt}C^\textup{lat}_{\rho,\h{0.5pt}\sigma, \h{0.5pt}+}}&\lesssim_{\,\Omega\,} \int_{C^\textup{lat}_{\rho,\h{0.5pt}\sigma, \h{0.5pt}+}} \left|\h{1pt} \varphi_0 (y)- \varphi_0 (y', \sigma)\h{1pt} \right| ^2 \h{1.5pt}\mathrm d\h{.5pt}\mathscr{H}_y^2
			+ \sigma \int_{\p^+D^+_\rho\times\{\sigma\}} \left[ \textup{dist}_{\varphi_*(y)}\big(v^\textup{lat}(y')\big) \right] ^2\h{1.5pt}\mathrm d\h{.5pt}\mathscr{H}_y^1\nonumber\\[1mm]
			&\lesssim\h{1.5pt} \rho\h{1pt}\sigma\sup_{y\h{0.5pt}\in \h{0.5pt} C^\textup{lat}_{\rho,\h{0.5pt}\sigma,\h{0.5pt}+}}
			\left|\h{1pt} \varphi_0 (y)- \varphi_0 (y', \sigma) \h{1pt}\right|^2
			+ \sigma \int_{\p^+D^+_\rho\times\{\sigma\}} \left[ \textup{dist}_{\varphi_*(y)}\big(v^\textup{lat}(y')\big) \right] ^2\h{1.5pt}\mathrm d\h{.5pt}\mathscr{H}_y^1. 
		\end{align*}
		The desired estimate in item (3) follows from \eqref{W control}-\eqref{3212} and the last estimate.
	\end{proof}
	Our next step is to show an extension lemma from $\p^+ D_\rho^+$ to $D^+_\rho$. The proof proceeds as follows. We first extend a given map from $\p^+D^+_\rho$ to $\p D_\rho$ and then extend it from $\p D_\rho$ to $D_\rho$, using harmonic extension. A small energy assumption is imposed to ensure that the extended map remains within a small neighborhood of the tangent plane. \vspace{0.2pc}
	
	In addition to the boundary flattening map $\varphi$, the following maps are also used in Lemma \ref{lem ext for dim2}. $\varphi_3$ is a bi-Lipschitz homeomorphism from $\overline{D^+_\rho}$ to a closed subset in $\overline{B^+_{R_0}}$. It satisfies $$\varphi_3 \left(\p^0D^+_\rho\right) \subset  \p^0 B^+_{R_0}.$$ Using $\varphi$ and $\varphi_3$, we define  $\varphi_\dagger :=\varphi^{-1}\circ\varphi_3$ on $\overline{D^+_\rho }$. Letting $\varphi_4$ be a bi-Lipschitz homeomorphism from $D_\rho^+$ to $D_\rho$ which is the identity map when restricted to $\p^+D^+_\rho$, we define $\varphi_5:=\varphi_3 \circ \varphi_4^{-1}$. The map $\varphi_\ddagger$ from $D_\rho$ to $\p \Omega$ is given by $\varphi^{-1}\circ[\h{0.5pt}\varphi_5\h{0.5pt}]^0$. Moreover, we define $\varphi_{\ddagger,\,\star} : D_\rho \to \p \Omega$ as the map that satisfies $\varphi_{\ddagger,\,\star} (y) = \varphi_\ddagger (y^\star)$ for any $y \in D_\rho$. 
	\begin{lem}\label{lem ext for dim2} 
		There exist positive constants $\eps{6}$ and $\C{16}$ depending only on $\Omega$ with which we have \begin{enumerate} 
			\item[$\mathrm{(1).}$] Fix  $\rho>0$ and let $v \in H^1\big(\p D_{\rho}; \mathbb{S}^2\big)$ satisfying
			\begin{align}\label{small prod energy}\mathcal{E}^*_{v,\,\p D_{\rho}}
			W_{v,\,\p D_{\rho}}^{\h{0.5pt}q} \leq \eps{6},\end{align}
			then there exists an extension $\overline{v} \in H^1(D_\rho;\mathbb{S}^2)$ such that $\overline{v}= v$ on $\p D_\rho$ in the sense of trace. Moreover, the following two estimates hold $$\mathrm{(a).}\h{4pt}\pigl(\mathcal{E}^*_{\h{1pt}\overline{v},\,D_{\rho}}\pigr)^2  
			\leq \C{16} \h{1pt} \mathcal{E}^*_{v,\,\p D_{\rho}} W^{\h{0.5pt}q}_{v,\,\p D_{\rho}}; \h{20pt}\mathrm{(b).} \h{4pt}  W^{\h{0.5pt}q}_{\overline{v},\,D_{\rho}}   
			\leq \C{16} \h{1pt}\rho\h{1pt} W^{\h{0.5pt}q}_{v,\,\p D_{\rho}}.$$

			\item[$\mathrm{(2).}$] If, for some $\varepsilon \in \big(0, \eps{6}\h{0.5pt}\big]$ and $\rho>0$, the mapping $v \in H^1\big(\p^+D_{\rho}^{+}; \mathbb{S}^2\big)$ satisfies  \begin{align}\label{bdry of v}v(\cdot)  \in U_{\sqrt{\varepsilon}}\h{1pt}\big(\h{1pt}\mathbb{S}^2\cap T_{\varphi_\dagger(\cdot)}\p \Omega\h{1pt}\big)  \h{15pt} \text{on $\p D_{\rho} \cap \big\{y_2=0\big\}$}, \end{align}  and the estimate
			\[\left(\h{1pt}\rho \h{1pt}\big\|\nabla \varphi_5\big\|^2_{L^\infty(D_\rho)}\h{1pt}W_{v,\,\varphi_{\ddagger,\star}, \, \p^+D_\rho^{+}} \right)^{\frac{1}{2}}
			+	\left( \mathcal{E}^*_{v, \,\p^+D^+_\rho} \h{1pt}W_{v,\,\varphi_{\ddagger,\star}, \,\p^+D_\rho^{+}}^{\h{0.5pt}q}\right)^{\frac{1}{2}} \leq \varepsilon,
			\]
			then there exists an extension $\overline{v} \in H^1\big(D^+_\rho;\mathbb{S}^2\big)$ such that \begin{align*} \h{20pt}\mathrm{(a).} \h{5pt}\overline{v}=v \h{15pt}\text{ on $\p^+D_{\rho}^{+}$}; \h{20pt} \mathrm{(b).} \h{5pt}\overline{v}(\cdot) \in U_{\C{16}\sqrt{\varepsilon}}\h{1pt}\big(\h{1pt}\mathbb{S}^2\cap T_{\varphi_\dagger(\cdot)}\p \Omega\h{1pt}\big) \h{15pt} \text{on $\p^0 D^+_\rho$}, \end{align*}  in the sense of trace. Moreover, it holds that  
			\begin{align*}
				&\mathrm{(c).}\h{5pt}\left( \mathcal{E}^*_{\h{1pt}\overline{v},\,D_\rho^+}\right) ^2
				\leq
				\C{16} \,\mathcal{E}^*_{v,\,\p^+ D_\rho^+} \h{1.5pt}  W^{\h{0.5pt}q}_{v,\h{0.5pt}\varphi_{\ddagger,\star},\h{0.5pt}\p^+D_\rho^{+}};\\[1.5mm]
				&\mathrm{(d).}\h{5pt}W_{\overline{v},\,\varphi_{\ddagger,\star} \h{1pt}\circ\h{1pt}\varphi_4, \, D_\rho^+}^{\h{0.5pt}q}
				\leq \C{16}\bigg\{ \rho \h{1pt}  W^{\h{0.5pt}q}_{v,\,\varphi_{\ddagger,\,\star},\,\p^+D_\rho^{+}}
				+\int_{D_\rho^+} \left[ \h{1pt} \textup{dist}_{\varphi_{\ddagger,\,\star}\h{1pt}\circ\h{1.5pt} \varphi_4 }\left( q\right) \h{1pt} \right]^2
				\h{1.5pt}\mathrm d\h{.5pt}\mathscr{H}^2  \bigg\}.
			\end{align*}
            The $3$-vector $q$ is arbitrary.
		\end{enumerate}
	\end{lem}
	\begin{proof} We only need to prove item (2). The proof of item (1) can be found in \cite[Lemma 4.2]{SU82}. 
    
    First, we extend $v$ from $\p^+ D_\rho^+$ to $\p D_\rho$ so that the extension, denoted by $\mathbbm{v}$, is even with respect to the variable $y_2$. The map $\mathbbm{v}$  is continuous on $\p D_\rho$ as $\mathcal{E}^*_{v, \,\p^+D_{\rho}^{+}}$ is finite. On $\p^+D_\rho^{+}$ and in light of the boundary condition \eqref{bdry of v}, we can apply the fundamental theorem of calculus and (2) in Lemma \ref{lem. regularity of proj. and dist} to obtain
		\begin{align*}
\left[\h{1pt}\textup{dist}_{\varphi_{\ddagger,\,\star} }\left(v  \right)\h{1pt}\right]^2 
\h{1.5pt}\lesssim_{\,\Omega}\h{1.5pt} 
\varepsilon
+  \int_{\p^+D_\rho^{+}} \textup{dist}_{\varphi_{\ddagger,\,\star}\left(z\right)}\big(v(z)\big) 
			\Big(\h{1.5pt}\big|\nabla \varphi_5 (z^\star)\big| 
			+ \big|\nabla^{\textup{tan}} v(z)\big| \h{1.5pt} 
			\Big)\, \h{1.5pt}\mathrm d\h{.5pt}\mathscr{H}^1_z,
		\end{align*}
 where, without ambiguity, we still use $\nabla $ to denote the gradient on $D_\rho$. The tangential derivative of $v$ on $\p^+ D_\rho^+$ is also denoted by $\nabla^{\textup{tan}} v$, which is equal to $\rho^{-1}\,\partial_\theta v\h{1pt} e_\theta^\top$ with $\theta$ being the argument angle on $\mathbb{R}^2$. Note that in 2D, we use the same notation $e_\theta$ to denote the 2D angular unit vector.  Applying H\"{o}lder's inequality to the last estimate, we further get 
        \begin{align}
\left[\h{1pt}\textup{dist}_{\varphi_{\ddagger,\,\star}}(v)\h{1pt}\right]^2  &\h{1.5pt}\lesssim_{\,\Omega\,}\h{1.5pt} \varepsilon
+ \bigg( \int_{\p^+D_\rho^{+}} \pig[\textup{dist}_{\varphi_{\ddagger,\,\star}}(v)\pigr]^2 \h{1.5pt}\mathrm d\h{.5pt}\mathscr{H}^1 \bigg)^{\frac{1}{2}}
			\bigg( \int_{\p^+D_\rho^{+}}\big|\h{1pt}\nabla \varphi_5(z^\star)\h{1pt}\big|^2+ \big|\h{1pt}\nabla^{\textup{tan}} v(z)\h{1pt}\big|^2 \h{1.5pt}\mathrm d\h{.5pt}\mathscr{H}^1_z\bigg)^{\frac{1}{2}} \nonumber\\[1mm]
			&\h{1.5pt}\leq\h{1.5pt} \varepsilon
            + \left[\h{1pt}W_{v,\,\varphi_{\ddagger,\,\star},\,\p^+D_\rho^{+}} \, \left(\rho\h{1pt}\big\|\h{1pt}\nabla \varphi_5\h{1pt}\big\|^2_{L^\infty(D_\rho)}+\mathcal{E}^*_{v,\,\p^+D_\rho^{+}}\right)\h{1pt}\right]^{1/2}
			.
			\label{3166}
		\end{align}
		It then turns out $\mathbbm{v}(y) \in U_{C_\Omega\sqrt{\varepsilon}}
        \left(\mathbb{S}^2\cap T_{\varphi_\ddagger(y)}\p \Omega\right)$ for $y\in\p^-D_{\rho}^{-} $, where $C_\Omega>0$ is a constant depending on $\Omega$.\vspace{0.2pc}
		
		Now, we extend $\mathbbm{v}$ to the whole disk $D_\rho$ by the harmonic extension. We still denote this extension by $\mathbbm{v}$ without ambiguity. To estimate its Dirichlet energy, we first note that
		\begin{equation}
			\int_{\p D_\rho} |\h{0.5pt}\mathbbm{v} - q\h{0.5pt}|^2 
			= 2\int_{\p^+ D_\rho^+} |\h{0.5pt}v - q\h{0.5pt}|^2 
			=  2\h{1pt}W^{\h{0.5pt}q}_{v,\,\p^+D_\rho^{+}}.
			\label{2188}
		\end{equation}
Since $\mathbbm{v}$ is harmonic on $D_\rho$, we can write it in terms of a Fourier series as follows:
\[
\mathbbm{v}(r,\theta)
= a_0 + \sum_{n\,\ge\,1}r^n\left(a_n\cos n\theta + b_n \sin n\theta \right),
\qquad \text{where $r\in[\h{0.5pt}0,\rho\h{0.5pt}]$ and $\theta\in[\h{0.5pt}0,2\pi\h{0.5pt}]$.}
\]
For any $n \in \mathbb N$, $a_{n-1}$ and $b_n$ are constant vectors. Direct computations then show that
\begin{align*}
\mathcal{E}^*_{\mathbbm{v},\,D_\rho}=\int_{D_\rho} \big|\h{0.5pt}\nabla \mathbbm{v} \h{0.5pt}\big|^2
&= \pi\sum_{n \, \geq \, 1} n \h{1pt}\rho^{2n}\left[\h{1pt}\left|\h{0.5pt}a_n\h{0.5pt}\right|^2+\left|\h{0.5pt}b_n \h{0.5pt}\right|^2\right]. 
\end{align*}
We also have
\begin{align*}
\mathcal{E}^*_{\mathbbm{v},\,\p D_\rho}=\int_{\partial D_\rho} \big|\h{0.5pt}\nabla^{\textup{tan}}\mathbbm{v} \h{0.5pt}\big|^2
=  \pi  \sum_{n\geq 1} n^2\rho^{2n - 1}\,\left[\h{0.5pt}\left|\h{0.5pt} a_n \h{0.5pt}\right|^2+\left|\h{0.5pt} b_n\h{0.5pt}\right|^2\h{0.5pt}\right],\end{align*} and \begin{align*} 
\int_{\partial D_\rho} |\h{0.5pt}\mathbbm{v}-q\h{0.5pt}|^2
= 2\pi\rho\, |a_0-q|^2 + \pi \sum_{n\,\geq\,1}\rho^{2n + 1}\left[\h{0.5pt}\left| \h{0.5pt}a_n\h{0.5pt}\right|^2+ \left|\h{0.5pt}b_n\h{0.5pt}\right|^2\h{0.5pt}\right].
\end{align*}
By the Cauchy-Schwarz inequality, \eqref{2188} and the fact \(
		\mathcal{E}^*_{\mathbbm{v},\,\p D_\rho}
		= 2 \h{0.5pt} \mathcal{E}^*_{v,\,\p^+ D_\rho^+}
		\), it follows that
		\begin{equation}
			\left( \mathcal{E}^*_{\mathbbm{v},\,D_\rho}\right) ^2 
			\h{1.5pt}\leq\h{1.5pt} \mathcal{E}^*_{\mathbbm{v},\,\p D_\rho} \left(\int_{\p D_\rho} |\mathbbm{v} - q|^2\right)
			\h{1.5pt}\lesssim \h{1.5pt} \mathcal{E}^*_{v,\,\p^+ D_\rho^+} \h{1pt} W^{\h{0.5pt}q}_{v,\,\p^+D_\rho^{+}}.
			\label{3195}
		\end{equation}
        
		Next, we show that the normalization of $\mathbbm{v}$ is well-defined. As there is $y^1\in \p^+D^+_\rho$ such that $$\left|\h{0.5pt}v(y^1)-q \h{0.5pt}\right|^2 = \big(\pi\rho\big)^{-1}\int_{\p^+D^+_\rho}|\h{0.5pt}v-q \h{0.5pt}|^2,$$ according to the triangle inequality, it turns out that
		\begin{align*}
			\int_{\p^+D^+_\rho}\left|\h{0.5pt}v-v(y^1)\h{0.5pt}\right|^2
			\h{1.5pt}\leq \h{1.5pt} 2\int_{\p^+D^+_\rho}|\h{0.5pt}v-q\h{0.5pt}|^2
			+2\pi\rho \left|\h{0.5pt}v(y^1)-q \h{0.5pt}\right|^2
			\h{1.5pt}=\h{1.5pt} 4\int_{\p^+D^+_\rho}|\h{0.5pt}v-q \h{0.5pt}|^2.
		\end{align*}
		Therefore,
		\begin{align*}
			\max_{\p D_\rho } \left|\h{0.5pt}\mathbbm{v} -v(y^1)\h{0.5pt}\right|^2 
			= \sup_{ \p^+D^+_\rho } \left|\h{0.5pt}v -v(y^1)\h{0.5pt}\right|^2
			\h{1.5pt}\leq \h{1.5pt} 2\int_{\p^+D^+_\rho}\left|\h{0.5pt}v-v(y^1)\h{0.5pt}\right| \h{1pt}\left|\h{0.5pt}\nabla^{\textup{tan}} v \h{0.5pt}\right| \h{1.5pt}\lesssim \h{1.5pt}  	\left(\mathcal{E}^*_{v,\,\p^+D^+_\rho} \h{1.5pt} W^{\h{0.5pt}q}_{v,\,\p^+D^+_\rho}\right)^{\frac{1}{2}}.
		\end{align*} 
		Utilizing the last estimate and the maximum principle, we deduce that 
		\begin{align*}\displaystyle\sup_{D_\rho} \left|\h{0.5pt}\mathbbm{v} -\mathbbm{v}(y^1)\h{0.5pt}\right|^2
			=\displaystyle\max_{\p D_\rho} \left|\h{0.5pt}\mathbbm{v} -\mathbbm{v}(y^1) \h{0.5pt}\right|^2
			\h{1.5pt}\lesssim \h{1.5pt}  	\left(\mathcal{E}^*_{v,\,\p^+D^+_\rho} \h{1pt} W^{\h{0.5pt}q}_{v,\,\p^+D^+_\rho}\right)^{\frac{1}{2}}.
		\end{align*} 
		\sloppy Since $\mathbbm{v}(y^1) \in \mathbb{S}^2$, we can then define the normalized vector field $\widehat{\mathbbm v}$, choosing $\eps{6}$ small enough. The smallness depends only on $\Omega$. Using \eqref{3195}, we obtain $$\left( \mathcal{E}^*_{\,\widehat{\mathbbm{v}},\,D_\rho}\right) ^2
		\h{1.5pt}\lesssim\h{1.5pt} \left( \mathcal{E}^*_{\mathbbm{v},\,D_\rho}\right) ^2
		\h{1.5pt}\lesssim \h{1.5pt}   \mathcal{E}^*_{v,\,\p^+ D_\rho^+} \h{1pt}  W_{v,\,\p^+D_\rho^{+}}^{\h{0.5pt}q}.$$ Define $\overline{v}:=\widehat{\mathbbm{v}}\h{1pt}\circ\h{1pt}\varphi_4:D^+_\rho \to \mathbb{S}^2$ as our extension of $v$. We have $  \mathcal{E}^*_{\h{1pt}\widehat{\mathbbm{v}}\h{1pt}\circ\h{1pt}\varphi_4,\,D_\rho^+} 
		\h{1.5pt}\lesssim\h{1.5pt}  \mathcal{E}^*_{\h{1pt}\widehat{\mathbbm{v}},\h{0.5pt}D_\rho}$, which concludes the estimate (c) in Item (2) of this lemma.

		With regard to the estimate of \( W^{\h{0.5pt}q}_{\widehat{\mathbbm{v}},\,\varphi_{\ddagger,\star},\,D_\rho}\), we first use the triangle inequality to obtain
\begin{align*} W^{\h{0.5pt}q}_{\widehat{\mathbbm{v}},\,\varphi_{\ddagger,\star},\, D_\rho} 
			\h{1.5pt}\lesssim\h{1.5pt} \int_{D_\rho}  \left|\h{0.5pt} \widehat{\mathbbm{v}} - \mathbbm{v} \h{0.5pt} \right|^2 
			+| \h{0.5pt}\mathbbm{v}  - q \h{0.5pt}|^2 
			+ \Big[ \h{1pt} \textup{dist}_{\varphi_{\ddagger,\,\star}} ( \mathbbm{v} ) \h{1pt} \Big] ^2 \h{2pt}\mathrm d\h{.5pt}\mathscr{H}^2. \end{align*} Note that $\left|\h{0.5pt} \widehat{\mathbbm{v}} - \mathbbm{v} \h{0.5pt} \right| \leq \textup{dist}_{\varphi_{\ddagger,\,\star}} ( \mathbbm{v})$ and $\textup{dist}_{\varphi_{\ddagger,\,\star}} ( \mathbbm{v}) \leq | \h{0.5pt}\mathbbm{v}  - q \h{0.5pt}| + \textup{dist}_{\varphi_{\ddagger,\,\star}} (q).$ Moreover,  $$
		\int_{D_\rho} \big|\h{1pt}\mathbbm{v}  - q \h{1pt}\big|^2 \h{2pt}\mathrm d\h{.5pt}\mathscr{H}^2  \h{1pt}\leq \h{1pt} \frac{\rho}{2} \int_{\p D_\rho} |\mathbbm{v}  - q|^2 \h{1.5pt}\mathrm d\h{.5pt}\mathscr{H}^1 
		$$ since $\mathbbm v$ is harmonic. We then have \begin{align*}
			 W^{\h{0.5pt}q}_{\widehat{\mathbbm{v}},\,\varphi_{\ddagger,\star},\, D_\rho} \h{1.5pt}\lesssim \h{1.5pt}  \rho\int_{\p D_\rho} |\h{0.5pt}\mathbbm{v}  - q \h{0.5pt}|^2 \h{2pt}\mathrm d\h{.5pt}\mathscr{H}^1 
			+\int_{D_\rho} \Big[\h{1pt} \textup{dist}_{\varphi_{\ddagger,\,\star} } ( q ) \Big]^2
			\h{2pt}\mathrm d\h{.5pt}\mathscr{H}^2.
		\end{align*}
		By \eqref{2188}, it turns out that $$
		W_{\h{0.5pt}\widehat{\mathbbm{v}},\,\varphi_{\ddagger,\,\star},\,D_\rho}^{\h{0.5pt}q} \h{1.5pt}\lesssim\h{1.5pt} \rho \h{0.5pt} W_{\h{0.5pt}v,\,\varphi_{\ddagger,\,\star},\,\p^+D_\rho^{+}}^{\h{0.5pt}q}
		+\int_{D_\rho} \pig[ \h{1pt}\textup{dist}_{\varphi_{\ddagger,\,\star}} ( q ) \h{1pt}\pigr]^2
		\h{2pt}\mathrm d\h{.5pt}\mathscr{H}^2.
		$$ This estimate and the relation: $$
		W_{\h{0.5pt}\widehat{\mathbbm{v}}\h{1pt}\circ\h{1pt}\varphi_4,\,\varphi_{\ddagger,\,\star}\h{1pt}\circ\h{1.5pt}\varphi_4,\,D_\rho^+}^{\h{0.5pt}q}
		\h{1.5pt}\lesssim \h{1.5pt} W_{\h{0.5pt}\widehat{\mathbbm{v}},\,\varphi_{\ddagger,\,\star},\,D_\rho}^{\h{0.5pt}q}
		\h{1.5pt}\lesssim \h{1.5pt} \rho \h{0.5pt}  W_{v,\,\varphi_{\ddagger,\,\star},\,\p^+D_\rho^{+}}^{\h{0.5pt}q}
		+\int_{D_\rho^+} \pig[ \h{1pt} \textup{dist}_{\varphi_{\ddagger,\,\star}\h{1pt}\circ\h{1pt}\varphi_4} ( q ) \h{1pt} \pigr]^2
		\h{2pt}\mathrm d\h{.5pt}\mathscr{H}^2$$ complete the proof.
	\end{proof}

	Next, we extend a map on a hemisphere to a hemispherical shell. The key idea is to transform certain regions of the hemispherical shell to (half) cylinders. The top face of each (half) cylinder is diffeomorphic to a subset of the outer hemisphere. By these diffeomorphisms, the map defined on the outer hemisphere can be pulled to the tops of the (half) cylinders. We extend the map constantly on the top face to the lateral surface of each (half) cylinder, obtaining boundary values for the bottom face. Then we use harmonic extensions in Lemma \ref{lem ext for dim2} to define the values on the bottom faces of the (half) cylinders. With boundary values on the top, bottom, and lateral surfaces of the (half) cylinders, we then apply Lemmas \ref{lem ext to cylinder} and \ref{lem ext to cylinder tan} to induce extensions over the whole cylinders and half-cylinders, respectively. The extension from the hemisphere to the hemispherical shell can then be obtained by a covering argument. On the remaining regions of the hemispherical shell, those not covered by these (half) cylinders, the extended map is defined by constant extension. 
	
	\begin{lem}\label{lem ext to annulus} Fix $\big(q, \widetilde{a}\big) \in\mathbb{R}^3 \times \p^0 B^+_{R_0}$ and let $\{r_k\}_{k \h{0.5pt}\in\h{0.5pt}\mathbb{N}}$ be a sequence in $(0,1)$. In addition, we recall $\varphi_k$ in \eqref{defn of varphi_k} and $\eps{6}$ in \eqref{small prod energy}. Then, there are a universal constant $\alpha\in (0,1)$  and constants $$\eps{7}\in\left(0, \frac{\eps{6}\wedge \delta_0^2}{4}\right), \h{20pt}
		\beta\in(0,1), \h{20pt} \C{17}>2, \h{20pt}\text{depending only on $\Omega$, }$$
		such that for any $\rho\in(0,1)$, $\sigma \in \left(\eps{7}\rho, \frac{\rho}{4}\right)$, $\delta_2\in \big[\h{0.5pt}0,\sqrt{\eps{7}}\h{0.5pt}\big]$, and $v \in H^1\big(\p^+ B_\rho^+; \mathbb{S}^2\big)$ satisfying 
        \begin{align*}
        &\mathrm{(1).}\h{10pt}r_k^2 	\left( r_k^2 \h{1pt}\sigma^2+\mathcal{E}^*_{v,\h{0.5pt}\p^+ B_\rho^+}
		+ W_{v,\h{0.5pt}\varphi_k,\h{0.5pt}\p^+ B_\rho^+} \right) 
		+\sigma^{-2}\h{1pt}\mathcal{E}^*_{v,\h{0.5pt}\p^+ B_\rho^+} \h{1pt}
		W^{\h{0.5pt}q}_{v,\h{0.5pt}\varphi_k,\h{0.5pt}\p^+ B_\rho^+} 
		\leq \delta_2^4;\\[1mm]
        &\mathrm{(2).}\h{10pt} \text{$v(\cdot) \in U_{\delta_2}\big(\mathbb{S}^2\cap T_{\varphi_k(\cdot)}\p \Omega\big)$  on $\p B_{\rho}\cap \{y_3=0\}$ in the sense of trace, }
        \end{align*} 
		we can find an extension map $\overline{v} \in H^1\big(A_{\rho-\sigma, \h{1pt}\rho}^+; \mathbb{S}^2\big)$ with the following properties:
		\begin{align*}
			&\mathrm{(a).}\h{5pt} \text{$\overline{v}\h{1pt}\big|_{\p^+B_\rho^+} = v$ and $\overline{v}(\cdot)\in U_{\C{17} \left(\delta_2+r_k\h{0.5pt}\sigma\right)}\big(\mathbb{S}^2\cap T_{\varphi_k(\cdot)}\p \Omega\big)$ on $\p^0 A_{\rho-\sigma,\h{1pt}\rho}^+$ in the sense of trace};\\[1mm]
            &\mathrm{(b).}\h{5pt}	\mathcal{E}^*_{\h{1pt}\overline{v},\h{0.5pt} A_{\rho-\sigma,\rho}^+}
			\leq \C{17}\left( \sigma \mathcal{E}^*_{v,\h{0.5pt}\p^+B^+_\rho}
			+ \beta^{-1}\sigma^{-1}\h{1pt}W^{\h{0.5pt}q}_{v,\h{0.5pt}\varphi_k,\h{0.5pt}\p^+B^+_\rho} 
			+  \beta^{-1}\sigma^{-1}
			r_k^2\h{1pt}\rho^2\right);\\[1mm]   &\mathrm{(c).}\h{5pt}W^{\h{0.5pt}q}_{\overline{v},\h{0.5pt}\varphi_k,\h{0.5pt}A_{\rho-\sigma,\rho}^+}
			\leq \C{17} \sigma \left( W^{\h{0.5pt}q}_{v,\h{0.5pt}\varphi_k,\h{0.5pt}\p^+B^+_\rho}
			+ r_k^2\h{1pt}\rho^2\right). 
		\end{align*} 
		Furthermore, the boundary value $v_\textup{in}:=\overline{v}\h{1.5pt}\big|_{\p^+B_{\rho-\sigma}^+}$ satisfies 
		\begin{align*}
		&\mathrm{(d).}\h{5pt} v_\textup{in} \in H^1\big(\p^+B_{\rho-\sigma}^+; \mathbb{S}^2\big);\\[1mm]
        &\mathrm{(e).}\h{5pt} 
        \text{$v_\textup{in}(\cdot)\in U_{\C{17} (\delta_2+r_k\h{0.5pt}\sigma)}(\mathbb{S}^2\cap T_{\varphi_k(\cdot)}\p \Omega\big)$ on $\p B_{\rho-\sigma} \cap \big\{y_3=0\big\}$ in the sense of trace}\h{0.5pt};\\[1mm]
&\mathrm{(f).}\h{5pt}\mathcal{E}^*_{v_\textup{in},\h{0.5pt}\p^+B_{\rho-\sigma}^+}
			\leq \alpha \h{1pt}\mathcal{E}^*_{v,\h{0.5pt}\p^+B_{\rho}^+}
			+ \C{17} \h{1pt}\beta^{-1}\sigma^{-2}
			\left( W^{\h{0.5pt}q}_{v,\h{0.5pt}\varphi_k,\h{0.5pt}\p^+B^+_\rho} 
			+r_k^2\h{1pt}\rho^2\right); \\[1mm]
			&\mathrm{(g).}\h{5pt} W^{\h{0.5pt}q}_{ v_\textup{in},\h{0.5pt}\varphi_k,\h{0.5pt}\p^+B_{\rho-\sigma}^+}
			\leq \C{17}\Big(W^{\h{0.5pt}q}_{v,\h{0.5pt}\varphi_k,\h{0.5pt}\p^+B_{\rho}^+}
			+\sigma^2\h{1pt}r_k^2\Big).
		\end{align*}
	\end{lem}
	\begin{proof}
		Let $\rho \in (0,1)$ and $\sigma\in \big(\eps{7}\rho,\rho/4\big)$. $\big\{x^1, \dots, x^{n_x} \big\}\cup\big\{b^1, \dots, b^{n_b}\big\}$ is the maximal set of points on $\overline{\p^+B^+_\rho}$ satisfying the following:
		\begin{enumerate}[(i).]  \setlength\itemsep{.5em}
			\item  Denote by $\textup{dist}_{\p B_\rho}$ the geodesic distance on $\p B_\rho$. Then, $$\textup{dist}_{\p B_\rho}\left(x^i,\p B_\rho\cap\big\{y_3=0\big\}\right)\geq 2\sigma\h{15pt}\text{ for $i = 1,2, ..., n_x$}.$$ In addition, $\textup{dist}_{\p B_\rho}\big(x^i,x^j\big) \geq \sigma$ for $i,j=1,2,\ldots,n_x$ with $i \neq j$;
			
			\item $b^i \in \p B_\rho\cap\big\{y_3=0\big\}$ and $\textup{dist}_{\p B_\rho}\big(b^i,b^{\h{1pt}j}\big) \geq 3\sigma$ for $i,j=1,2,\ldots,n_b$ with $i \neq j$;
			
			\item \sloppy Let \(\mathcal{D}_{\sigma}(x^i)\) be the geodesic disk in \(\partial B_\rho\) with radius \(\sigma\) and center \(x^i\).  \(\mathcal{D}^{+}_{3\sigma}(b^i)\) is the half geodesic disk (intersection with \(\partial^{+}B_\rho^+\)) with radius \(3\sigma\) and center \(b^i\). Then, 
            $$x^i \not\in  \Bigg(\bigcup_{j \h{1pt} \neq \h{1pt} i}\mathcal{D}_{\sigma}(x^j) \Bigg)\h{2pt} \bigcup \h{2pt}\Bigg(\bigcup_\nu \mathcal{D}^+_{3\sigma}(b^\nu)\Bigg), \h{15pt}b^l \not\in  \Bigg(\bigcup_{j }\mathcal{D}_{\sigma}(x^j)\Bigg) \h{2pt} \bigcup \h{2pt}\Bigg(\bigcup_{\nu \h{1pt}\neq\h{1pt} l} \mathcal{D}^+_{3\sigma}(b^\nu)\Bigg), $$for any $i = 1,2, ..., n_x$ and $l = 1,2, ..., n_b$; 
			
			\item \(
			\p^+B^+_\rho \subseteq \Big(\bigcup_{i=1}^{n_x} \mathcal{D}_{\sigma}(x^i)\Big)
			\bigcup \h{1pt} \Big(\bigcup_{i=1}^{n_b}\mathcal{D}^+_{3\sigma}(b^i)\Big)
			\). \vspace{0.4pc}
		\end{enumerate}

		\noindent \textbf{Part 1. Estimates on half disks:} Fix $i =1,2,\ldots,n_b$. By Fubini's theorem, there is $\rho_i \in [3\sigma, 4\sigma]$ such that the restriction $v^{*,i} := v\h{1pt}|_{\partial^+ \mathcal{D}_i^+}$ lies in $H^1(\partial^+ \mathcal{D}_i^+;\mathbb{S}^2)$. Here,  we simply denote by $\mathcal{D}_i^+$ the half geodesic disk $\mathcal{D}_{\rho_i}^+(b^i)$. Letting $\varphi_{*,i} : \overline{D^+_{4\sigma}} \to \overline{\mathcal{D}_{4\sigma}^+(b^i)}$ be a bi-Lipschitz homeomorphism satisfying \begin{align*}&\varphi_{*,i}\big(\overline{D_{\rho_i}^+}\big)
		=\overline{\mathcal{D}_i^+}, \h{88pt} 
		\varphi_{*,i}\big(\p^0D^+_{4\sigma}\big) =\p^0\mathcal{D}_{4\sigma}^+(b^i),\\[1mm] 
        &\varphi_{*,i}\big(\p^+D^+_{4\sigma}\big) =\p^+\mathcal{D}_{4\sigma}^+(b^i),\h{40pt}
		\varphi_{*,i}\big(\p^0D^+_{\rho_i}\big) =\p^0\mathcal{D}_i^+, \h{40pt}\varphi_{*,i}\big(\p^+D^+_{\rho_i}\big) =\p^+\mathcal{D}_i^+,\end{align*} we also have the following estimates:
		\begin{align}
			&\mathcal{E}^*_{v^{*,i},\h{0.5pt}\partial^+ \mathcal{D}_i^+} 
			\lesssim  \sigma^{-1}\h{1pt} \mathcal{E}^*_{v,\h{0.5pt}\mathcal{D}_{4\sigma}^+(b^i)},\h{111pt}W_{v^{*,i},\h{0.5pt}\varphi_k,\h{0.5pt}\p^+\mathcal{D}_i^{+}}\lesssim \sigma^{-1}W_{v,\h{.5pt}\varphi_k,\h{.5pt}\mathcal{D}_{4\sigma}^+(b^i)},\nonumber\\[1.5mm]
			&W^{\h{0.5pt}q}_{v^{*,i},\h{1pt}\varphi_{\ddagger,\star}\h{1pt}\circ\h{1pt}\varphi_{*,i}^{-1},\h{1pt}\partial^+ \mathcal{D}_i^+}
			\lesssim
			\sigma^{-1} \h{1pt}W^{\h{0.5pt}q}_{v,\h{1pt}\varphi_{\ddagger,\star}\h{1pt}\circ\h{1pt}\varphi_{*,i}^{-1}, \h{1pt}\mathcal{D}_{4\sigma}^+(b^i)}, \h{20pt}W^{\h{0.5pt}q}_{v^{*,i},\h{1pt}
				\varphi_k, \h{1pt}\partial^+ \mathcal{D}_i^+} 
			\lesssim  \sigma^{-1} W^{\h{0.5pt}q}_{v,\h{1pt}
				\varphi_k,\h{1pt}\mathcal{D}_{4\sigma}^+(b^i)}.
			\label{3240}
		\end{align} 
		\sloppy Note that $v \in H^1\big(\p^+ B_\rho^+; \mathbb{S}^2\big)$ and $v(\cdot) \in U_{\delta_2}\big(\mathbb{S}^2\cap T_{\varphi_k(\cdot)}\p \Omega\big)$ on $\p B_{\rho}\cap \big\{y_3=0\big\}$ in the sense of trace. We can also assume that $\rho_i$ is chosen such that $$\widetilde{v}^{*,i}(\cdot) \in U_{\delta_2}\big(\mathbb{S}^2\cap T_{\varphi_k\h{1pt}\circ\h{1pt}\varphi_{*,i}(\cdot)}\p \Omega\big) \h{15pt}\text{ on $\p D_{\rho_i}\cap\big\{(y_1, y_2) : y_2=0\big\}$ in the sense of trace.}$$ 
		Here, $\widetilde{v}^{*,i}:=v^{*,i} \h{1pt}\circ\h{1pt} \varphi_{*,i}:\p^+D^+_{\rho_i} \to \mathbb{S}^2$. \vspace{0.4pc}

		\noindent \textbf{Part 1.1. Extensions of $\widetilde{v}^{*,i}$ to half disks:} Recall the mappings $\varphi_4$ and $\varphi_\ddagger$ in Lemma \ref{lem ext for dim2}. First, $\varphi_4$ induces a bi-Lipschitz homeomorphism from $D_{\rho_i}^+$ to $D_{\rho_i}$, which is the identity map when restricted to $\p^+D^+_{\rho_i}$. Second, using the triangle inequality and the definitions of $\varphi_k$, $\varphi_{\ddagger,\star}$, we obtain the following estimate: 
		\begin{align}
			\int_{\p^+\mathcal{D}_i^{+}} \Big[ \textup{dist}_{\varphi_{\ddagger,\star}\h{1pt}\circ\h{1pt}\varphi_{*,i}^{-1}(y)}(v^{*,i}(y)) \Big] ^2\h{1.5pt}\mathrm d
            \h{.5pt}\mathscr{H}^1_y &\h{1.5pt}\lesssim_{\,\Omega\,}\h{1.5pt}\int_{\p^+\mathcal{D}_i^{+}} \Big[ \textup{dist}_{\varphi_k(y)}(v^{*,i}(y)) \Big] ^2\h{1.5pt}\mathrm d
            \h{.5pt}\mathscr{H}^1_y \nonumber\\
			&\h{1.5pt}+\h{1.5pt}\sigma\sup_{y\h{1pt}\in\h{1pt}\p^+\mathcal{D}_i^{+}}
			\left|\h{1pt}\widetilde{a}+r_k\h{0.5pt}y
			-\varphi_3\h{1pt}\circ\h{1pt}\varphi_4^{-1}\h{1pt}\circ\h{1pt}(\varphi_{*,i}^{-1})^\star(y)\h{1pt}\right|^2.
			\label{3501}
		\end{align}	
        If we set $\varphi_3=\widetilde{a}+r_k\h{1pt}\varphi_{*,i}:D^+_{4\sigma} \to \varphi_3(D^+_{4\sigma} )\subseteq B^+_{R_0}$, then \eqref{3501} induces $$W_{v^{*,i},\h{1pt}\varphi_{\ddagger,\star}\h{1pt}\circ\h{1pt}\varphi_{*,i}^{-1},\h{1pt}\p^+\mathcal{D}_i^{+}} \h{1.5pt}\lesssim_{\,\Omega\,}\h{1.5pt}
		W_{v^{*,i},\h{.5pt}\varphi_k,\h{.5pt}\p^+\mathcal{D}_i^{+}}
		+\sigma\h{0.5pt} r^2_k.$$ 
        Together with $\varphi_5=\varphi_3 \h{1pt}\circ\h{1pt} \varphi_4^{-1}$ (see Lemma \ref{lem ext for dim2}) and the fact that $\|\nabla \varphi_5\|_{L^\infty(D_{4\sigma})}\lesssim r_k$, we have
		\begin{align*}
			&\rho_i\|\nabla \varphi_5\|^2_{L^\infty(D_{\rho_i})}
			W_{\widetilde{v}^{*,i},\h{.5pt}\varphi_{\ddagger,\star},\h{.5pt}\p^+D_{\rho_i}^{+}}  
			+  \mathcal{E}^*_{\h{1pt}\widetilde{v}^{*,i},\h{.5pt}\p^+D^+_{\rho_i}}
			W^{\h{0.5pt}q}_{\widetilde{v}^{*,i},\h{.5pt}\varphi_{\ddagger,\star},\h{.5pt}\p^+D_{\rho_i}^{+}} \nonumber\\[1mm]
			&\h{20pt}\lesssim \rho_i \h{1pt} r_k^2 \h{1pt} 	 W_{v^{*,i},\h{.5pt}\varphi_{\ddagger,\star}\h{1pt}\circ\h{1pt}\varphi_{*,i}^{-1},\h{.5pt}\p^+\mathcal{D}_i^{+}} 
			+\mathcal{E}^*_{v^{*,i},\h{.5pt}\p^+\mathcal{D}_i^{+}}
			W^{\h{0.5pt}q}_{v^{*,i},\h{.5pt}\varphi_{\ddagger,\star}\h{1pt}\circ\h{1pt}\varphi_{*,i}^{-1},\h{.5pt}\p^+\mathcal{D}_i^{+}} \nonumber\\[1mm]
			&\h{20pt}\lesssim \rho_i \h{1pt} r_k^2 \h{1pt}	 \Big(	
			W_{v^{*,i},\h{0.5pt}\varphi_k,\h{0.5pt}\p^+\mathcal{D}_i^{+}}
			+\sigma \h{0.5pt}r^2_k \Big) 	 
			+\mathcal{E}^*_{v^{*,i},\h{.5pt}\p^+\mathcal{D}_i^{+}}\Big(
			W^{\h{0.5pt}q}_{v^{*,i},\h{0.5pt}\varphi_k,\h{0.5pt}\p^+\mathcal{D}_i^{+}}
			+\sigma \h{0.5pt}r^2_k \Big).
		\end{align*}
		Using \eqref{3240} and the fact that $\rho_i \in [3\sigma, 4\sigma]$, we further induce that
		\begin{align}
			\rho_i\h{1pt}\|\nabla \varphi_5\|^2_{L^\infty(D_{\rho_i})}
			W_{\widetilde{v}^{*,i},\h{.5pt}\varphi_{\ddagger,\star},\h{.5pt}\p^+D_{\rho_i}^{+}}  
			&+  \mathcal{E}^*_{\h{1pt}\widetilde{v}^{*,i},\h{.5pt}\p^+D^+_{\rho_i}}
			W^{\h{0.5pt}q}_{\widetilde{v}^{*,i},\h{.5pt}\varphi_{\ddagger,\star},\h{.5pt}\p^+D_{\rho_i}^{+}} \nonumber\\[1mm]
			\lesssim r_k^2 	 
			\h{1pt}W_{v,\h{.5pt}\varphi_k,\h{.5pt}\mathcal{D}_{4\sigma}^+(b^i)}
			&+ \sigma^2 r^4_k	 
			+\sigma^{-2}\mathcal{E}^*_{v,\h{.5pt}\mathcal{D}_{4\sigma}^+(b^i)}
			W^{\h{0.5pt}q}_{v,\h{.5pt}\varphi_k,\h{.5pt}\mathcal{D}_{4\sigma}^+(b^i)}
			+r^2_k\h{1pt}\mathcal{E}^*_{v,\h{.5pt}\mathcal{D}_{4\sigma}^+(b^i)}.
			\label{3301}
		\end{align}
		The right-hand side of \eqref{3301} is smaller than $2^{-1}\eps{6}^2$ if we choose $\eps{7}$ in the assumption of this lemma to be small enough. By Lemma \ref{lem ext for dim2}, we obtain an extension of $\widetilde{v}^{*,i}$, denoted by $\overline{\widetilde{v}^{*,i}}\in H^1(D_{\rho_i}^+;\mathbb{S}^2)$, such that 
        \begin{align}
\text{$\overline{\widetilde{v}^{*,i}}=\widetilde{v}^{*,i}$ on $\p^+ D_{\rho_i}^+$, \h{15pt} $\overline{\widetilde{v}^{*,i}}(\cdot)\in 
		U_{C_\Omega \h{1pt} \delta_2}\big(\h{1pt}\mathbb{S}^2\cap T_{\varphi_k\h{1pt}\circ\h{1pt}\varphi_{*,i}(\cdot)}\p \Omega\h{1pt}\big)$ on $\p^0 D_{\rho_i}^+$,}
        \label{924}
		\end{align}
		  both in the sense of trace. Here the positive constant $C_\Omega $ depends only on $\Omega$ and may have varying values from line to line. 
		
		Moreover, for the mapping $\overline{v}^{*,i}:=\overline{\widetilde{v}^{*,i}}\h{1pt}\circ\h{1pt}\varphi_{*,i}^{-1} \in H^1(\mathcal{D}_i^+;\mathbb{S}^2)$, we can apply Lemma \ref{lem ext for dim2} to get 
		\[
		\begin{aligned}
			\pigl( \mathcal{E}^*_{\h{1pt}\overline{v}^{*,i},\h{.5pt}\mathcal{D}_i^+}\pigr) ^2  
			\lesssim
			\Big( \mathcal{E}^*_{\h{1pt}\overline{\widetilde {v}^{*,i}},\h{.5pt}D_{\rho_i}^+}\Bigr) ^2  
			&\lesssim_{\,\Omega\,} (\beta\sigma)^2
			\left(\mathcal{E}^*_{\h{1pt}\widetilde{v}^{*,i},\h{.5pt}\p^+D_{\rho_i}^+}\right)^2
			+(\beta\sigma)^{-2}\left( W^{\h{0.5pt}q}_{\widetilde{v}^{*,i},\h{0.5pt}\varphi_{\ddagger,\star},\h{0.5pt}\p^+D_{\rho_i}^+}\right)^2 \\[1mm]
			&\h{10pt}\lesssim (\beta\sigma)^2
			\left(\mathcal{E}^*_{v^{*,i},\h{.5pt}\p^+\mathcal{D}_i^+}\right)^2
			+(\beta\sigma)^{-2}\left( W^{\h{0.5pt}q}_{v^{*,i},\h{0.5pt}\varphi_{\ddagger,\star}\h{1pt}\circ\h{1pt}\varphi_{*,i}^{-1},\h{0.5pt}\p^+\mathcal{D}^+_i}\right)^2, \h{15pt} 
		\end{aligned}
		\] 
		for any $\beta>0$, and
		\begin{align*}
			W^{\h{0.5pt}q}_{\overline{v}^{*,i},\h{0.5pt}
				\varphi_{\ddagger,\star}\h{1pt}\circ\h{1pt}\varphi_4\h{1pt}\circ\h{1pt}\varphi_{*,i}^{-1},
				\h{0.5pt}\mathcal{D}_i^+}   
			&\lesssim
			W^{\h{0.5pt}q}_{\overline{\widetilde{v}^{*,i}},\h{0.5pt}\varphi_{\ddagger,\star}\h{1pt}\circ\h{1pt}\varphi_4,\h{0.5pt}D_{\rho_i}^+}   \nonumber\\[1mm]
			&\lesssim_{\,\Omega\,} \sigma \h{1pt}W^{\h{0.5pt}q}_{v^{*,i},\h{0.5pt}\varphi_{\ddagger,\star}\h{1pt}\circ\h{1pt}\varphi_{*,i}^{-1},\h{0.5pt}\p^+\mathcal{D}_i^+}
			+\int_{\mathcal{D}_i^+} \pig[ \textup{dist}_{\varphi_{\ddagger,\star}\h{1pt}\circ\h{1pt}\varphi_4\h{1pt}\circ\h{1pt}\varphi_{*,i}^{-1}(\cdot)}(q)\pigr]^2.
		\end{align*}Hence, the inequalities in \eqref{3240} imply
		\begin{align}
			\mathcal{E}^*_{\h{1pt}\overline{v}^{*,i},\h{.5pt}\mathcal{D}_i^+}
			&\lesssim_{\,\Omega\,} 
			\beta \h{1pt}\mathcal{E}^*_{v,\h{.5pt}\mathcal{D}_{4\sigma}^+(b^i)} 
			+\left(\beta\h{0.5pt}\sigma^2\right)^{-1} \h{1pt}W^{\h{0.5pt}q}_{v,\h{.5pt}\varphi_{\ddagger,\star}\h{1pt}\circ\h{1pt}\varphi_{*,i}^{-1},\h{.5pt}\mathcal{D}_{4\sigma}^+(b^i)} 			\label{3275a}
		\end{align}
and
\begin{align}&W^{\h{0.5pt}q}_{\overline{v}^{*,i},\h{0.5pt}
				\varphi_{\ddagger,\star}\h{1pt}\circ\h{1pt}\varphi_4\h{1pt}\circ\h{1pt}\varphi_{*,i}^{-1},
				\h{0.5pt}\mathcal{D}_i^+}  \nonumber\\[1mm] 
			&\lesssim_{\,\Omega\,} W^{\h{0.5pt}q}_{v,\h{0.5pt}\varphi_{\ddagger,\star}
				\h{1pt}\circ\h{1pt}\varphi_{*,i}^{-1},\h{0.5pt}\mathcal{D}_{4\sigma}^+(b^i)} +\sigma^2\sup_{\mathcal{D}_{4\sigma}^+(b^i)}
			\left|\h{1pt}\varphi_{\ddagger,\star}\h{1pt}\circ\h{1pt}\varphi_4\h{1pt}\circ\h{1pt}\varphi_{*,i}^{-1}
			-\varphi_{\ddagger,\star}\h{1pt}\circ\h{1pt}\varphi_{*,i}^{-1} \h{1pt}\right|^2 +\int_{\mathcal{D}_{4\sigma}^+(b^i)} \pig[ \textup{dist}_{\varphi_{\ddagger,\star}\h{1pt}\circ\h{1pt}\varphi_{*,i}^{-1}(\cdot)}(q)\pigr]^2\nonumber\\[1mm]
			&\lesssim_{\,\Omega\,} W^{\h{0.5pt}q}_{v,\h{0.5pt}\varphi_{\ddagger,\star}
				\h{1pt}\circ\h{1pt}\varphi_{*,i}^{-1},\h{0.5pt}\mathcal{D}_{4\sigma}^+(b^i)}
			+\sigma^4r_k^2.
			\label{3275}
		\end{align}\vspace{0.1pc}

		\noindent \textbf{Part 1.2. Extension to the cylinder:} \sloppy Define $e_3 := (0, 0, 1)^\top$ and let $$B^+_{\rho, \h{0.5pt}\sigma}:=\Big\{\h{1pt}y\in \mathbb{R}^3:y=\zeta+\lambda e_3\text{ for some $\zeta \in \overline{\p^+B^+_{\rho}}$ and $\lambda\in(-\sigma,\sigma)$} \h{1pt}\Big\}.$$ To apply Lemma \ref{lem ext to cylinder tan}, we introduce a bi-Lipschitz homeomorphism $\varphi_7:B^+_{\rho,\sigma}\to A^+_{\rho-\sigma,\rho}$ that maps $\overline{\p^+B^+_{\rho}}+\sigma e_3$ identically onto $\overline{\p^+B^+_{\rho}}$ in the sense that $$\varphi_7(\zeta+\sigma e_3)=\zeta \h{20pt}\text{for all $\zeta \in \overline{\p^+B^+_{\rho}}$.}$$  
		We can also extend $\varphi_{*,i}$ from $\overline{D^+_{4\sigma}}$ to $\overline{D^+_{4\sigma}}\times [-\sigma,\sigma]$ by letting $$\overline{\varphi_{*,i}}(y):=\varphi_{*,i}(y_1,y_2)+y_3 \h{1pt} e_3, \h{20pt}\text{where $y = (y_1, y_2, y_3) \in \overline{D^+_{4\sigma}} \times [-\sigma,\sigma]$.}$$ Therefore, the image of $\overline{\varphi_{*,i}}$ is a subset of $\overline{B^+_{\rho,\sigma}}.$ See Figure~\ref{fig 1B} for an illustration.\vspace{0.4pc}\captionsetup[figure]{font=footnotesize}
        \begin{figure}[h]
            \centering
            \includegraphics[scale=0.057]{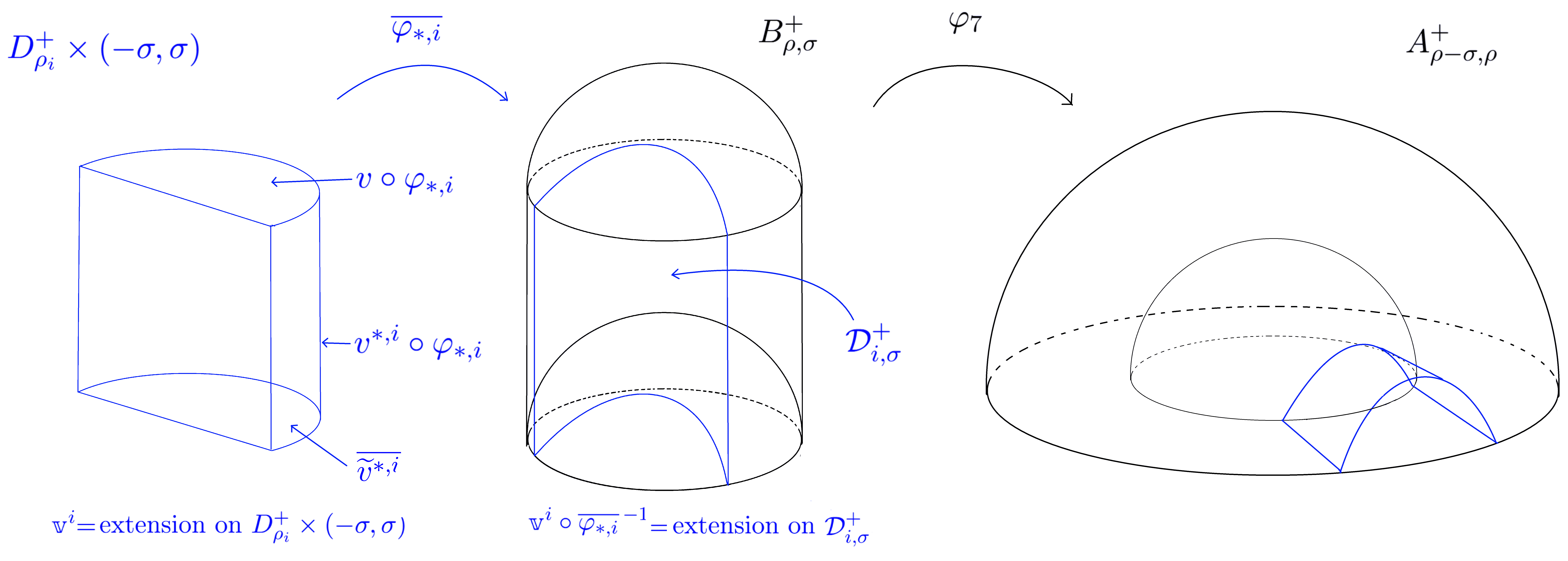}
            \caption{Extension in Part 1.2.}
            \label{fig 1B}
        \end{figure}
		
		To validate the assumptions of Lemma \ref{lem ext to cylinder tan}, we note that the assumptions of this lemma yield $$ v\big(\varphi_{*,i}(y')\big)=v\big(\varphi_7\h{1pt}\circ\h{1pt}\overline{\varphi_{*,i}}\left(y\right)\big)\in U_{\delta_2}\big(\mathbb{S}^2\cap T_{\varphi_k\h{1pt}\circ\h{1pt}\varphi_7\h{1pt}\circ\h{1pt} \overline{\varphi_{*,i}}\left(y\right)}\p \Omega\big), \h{15pt}\text{for any $y\in\p^0D^+_{\rho_i} \times\{\sigma\}$},$$  in the sense of trace. In addition,   \begin{align}\label{y11}\overline{\widetilde{v}^{*,i}}(y')\in U_{C_\Omega \left(\delta_2+r_k\h{0.5pt}\sigma\right)}\big(\mathbb{S}^2\cap T_{\varphi_k \h{1pt}\circ\h{1pt}\varphi_7\h{1pt}\circ\h{1pt}\overline{\varphi_{*,i}}\left(y\right)}\p \Omega\big), \h{15pt}\text{for any $y\in\p^0D^+_{\rho_i} \times\{-\sigma\}$,}\end{align} in the sense of trace, since, by \eqref{924} in Part 1.1, it holds
        $$\overline{\widetilde{v}^{*,i}}(\cdot) \in  U_{C_\Omega \h{1pt}\delta_2}\big(\mathbb{S}^2\cap T_{\varphi_k\h{1pt}\circ\h{1pt}\varphi_{*,i}(\cdot)}\p \Omega\big) \h{15pt}\text{on $\p^0D^+_{\rho_i}$.} $$ We can also prove \eqref{y11} by using the fact that   $$\overline{\widetilde{v}^{*,i}}(y')\in  U_{C_\Omega\h{0.5pt}\delta_2}(\mathbb{S}^2\cap T_{\varphi_k\h{1pt}\circ\h{1pt}\varphi_7\h{1pt}\circ\h{1pt}\overline{\varphi_{*,i}}\left(y\right)}\p \Omega\big), \h{15pt}\text{ for any $y\in\p^0D^+_{\rho_i}\times \{\sigma\}$.}$$ Similarly, $$\overline{\widetilde{v}^{*,i}}(y')
		=\widetilde{v}^{*,i}(y')\in U_{C_\Omega \left(\delta_2+r_k\h{0.5pt}\sigma\right)}\big(\mathbb{S}^2\cap T_{\varphi_k \h{1pt}\circ\h{1pt}\varphi_7\h{1pt}\circ\h{1pt}\overline{\varphi_{*,i}}\left(y\right)}\p \Omega\big), \h{5pt}\text{for any $y\in\big(\p D_{\rho_i}\cap\big\{y_2=0\big\} \big)\times[-\sigma,\sigma]$,} $$in the sense of trace. We now use the assumptions of this lemma to choose $\eps{7}$ sufficiently small such that $C_\Omega \left(\delta_2+r_k\h{0.5pt}\sigma\right)<2^{-2}\delta_0$. By Lemma \ref{lem ext to cylinder tan}, there is an extension over $D^+_{\rho_i} \times (-\sigma,\sigma)$, denoted by $\mathbbm{v}^i$, such that 
        \begin{align*}&\mathbbm{v}^i(y)
=v\circ\varphi_7\circ\overline{\varphi_{*,i}}(y) = v\big(\varphi_{*,i}(y')\big), \h{20pt}\text{for any $y \in D_{\rho_i}^+ \times \{\sigma\}$}; \\[1mm]
&\mathbbm{v}^i(y) = \overline{\widetilde{v}^{*,i}}(y'), \h{122pt}\text{for any $y \in D_{\rho_i}^+ \times \{-\sigma\}$}; \\[1mm]
&\mathbbm{v}^i(y) 
		=\widetilde{v}^{*,i}(y') 
		=\overline{\widetilde{v}^{*,i}}(y'), \h{75pt}\text{for any $ y \in \p^+ D_{\rho_i}^+ \times [-\sigma,\sigma]$}.
        \end{align*}
		
		Set $\varphi_0:=\widetilde{a}+r_k \h{1.5pt}\varphi_7\h{1pt}\circ\h{1pt}\overline{\varphi_{*,i}}$. By Items (1)-(3) in Lemma \ref{lem ext to cylinder tan}, the following estimates hold for $\mathbbm{v}^i$:
		\begin{align} \mathcal{E}^*_{\mathbbm{v}^i,\h{.5pt}D_{\rho_i}^+ \times (-\sigma,\sigma)}
			&\h{1.5pt}\lesssim_{\,\Omega\,}\h{1.5pt} \sigma\Big(\mathcal{E}^*_{v\h{1pt}\circ\h{1pt}\varphi_{*,i},\h{.5pt}D_{\rho_i}^+ } 
			+ \mathcal{E}^*_{\h{1pt}\overline{\widetilde{v}^{*,i}},\h{.5pt}D_{\rho_i}^+}  
			+ \sigma \h{1pt}\mathcal{E}^*_{\h{1pt}\widetilde{v}^{*,i},\h{.5pt}\p^+ D_{\rho_i}^+}  \Big), 			\label{3287a}
		\end{align}
        and
			\begin{align} &W^{\h{0.5pt}q}_{\mathbbm{v}^i,\h{.5pt}\varphi_k\h{1pt}\circ\h{1pt}\varphi_7\h{1pt}\circ\h{1pt}\overline{\varphi_{*,i}}, \h{.5pt}D_{\rho_i}^+\times (-\sigma,\sigma)} \nonumber\\[1.5mm]
			&\h{5pt}\lesssim_{\,\Omega\,}  \sigma\h{1pt}W^{\h{0.5pt}q}_{ v\h{1pt}\circ\h{1pt} \varphi_{*,i},\h{.5pt}
				\varphi_k\h{1pt}\circ\h{1pt}\varphi_7\h{1pt}\circ\h{1pt}\overline{\varphi_{*,i}},\h{.5pt}D_{\rho_i}^+\times\{\sigma\}}
			+ \sigma \h{1pt}W^{\h{0.5pt}q}_{\overline{\widetilde{v}^{*,i}},\h{.5pt}\varphi_k\h{1pt}\circ\h{1pt}\varphi_7\h{1pt}\circ\h{1pt}\overline{\varphi_{*,i}}
            ,\h{.5pt}D_{\rho_i}^+\times\{-\sigma\}}\nonumber
			\\[1.5mm]
			&\h{5pt}+ \sigma^2 \h{1pt}W^{\h{0.5pt}q}_{\widetilde{v}^{*,i},\h{.5pt}\varphi_k\h{1pt}\circ\h{1pt}\varphi_7\h{1pt}\circ\h{1pt}\overline{\varphi_{*,i}},\h{.5pt}\p^+  D_{\rho_i}^+\times\{\sigma\}} + \sigma^3 \h{1pt}r_k^2\displaystyle\sup_{y^1,\h{1pt}y^2 \h{1pt}\in \h{1pt}   D_{\rho_i}^+\times (-\sigma,\sigma)}
			\Big|\h{1pt} \varphi_7\h{1pt}\circ\h{1pt}\overline{\varphi_{*,i}} \left(y^1\right)- \varphi_7\h{1pt}\circ\h{1pt}\overline{\varphi_{*,i}} \left(y^2\right)\Big|^2.
			\label{3287}
		\end{align} Furthermore, it also satisfies the boundary condition: $$\mathbbm{v}^i (\cdot)\in U_{C_\Omega \left(\delta_2+r_k\h{.5pt}\sigma\right)}\big(\mathbb{S}^2\cap T_{\varphi_k\h{1pt}\circ\h{1pt}\varphi_7\h{1pt}\circ\h{1pt}\overline{\varphi_{*,i}}(\cdot)}\p \Omega\big) \h{15pt}\text{ on $ \p^0D^+_{\rho_i} \times [-\sigma,\sigma]$, in the sense of trace.} $$  
		
		\noindent \textbf{Part 2. Estimates on whole disks:} We repeat the extension procedure described in Part 1, but utilizing Lemmas \ref{lem ext to cylinder} and \ref{lem ext for dim2} to extend $v$ to the whole cylinder instead of the half cylinder. Similarly to \eqref{3240}, we can apply Fubini's theorem to obtain $\sigma_i\in[\sigma,2\sigma]$ which satisfies the inequality: 
        \begin{align}\label{1306}
            \mathcal{E}^*_{v^{i},\h{.5pt}\p \mathcal{D}_{i}}
		\h{1pt} \lesssim\h{1.5pt} \sigma^{-1}\mathcal{E}^*_{v,\h{0.5pt}\mathcal{D}_{2\sigma}(x^i)} ,\h{5pt}         W^{\h{0.5pt}q}_{v^{i},\h{0.5pt}\p \mathcal{D}_{i}}
		\h{1.5pt}\lesssim\h{1.5pt} \sigma^{-1} 
		W^{\h{0.5pt}q}_{v,\h{0.5pt}\mathcal{D}_{2\sigma}(x^i)},\h{8pt}\text{where $\mathcal{D}_i:=\mathcal{D}_{\sigma_i}(x^i)$ and $v^{i} := v \h{1pt}\big|_{\partial \mathcal{D}_i}$.}
        \end{align}
         Seen from (1) in the lemma's assumptions, the left-hand side is smaller than $\eps{6}$, provided that $\eps{7}$ is sufficiently small. Applying Item (1) of Lemma \ref{lem ext for dim2}, we obtain an extension $\overline{v}^{i}  \in H^1(\mathcal{D}_i;\mathbb{S}^2)$. Estimates similar to those in \eqref{3275a}-\eqref{3275} can be obtained by (1) of Lemma \ref{lem ext for dim2} and some estimates similar to \eqref{3240}. More precisely, we have the following:
		\begin{align}
			\mathcal{E}^*_{\h{1pt}\overline{v}^{i},\h{0.5pt}\mathcal{D}_i}
			\lesssim_{\,\Omega\,} 
			\beta \h{1pt}    \mathcal{E}^*_{v,\h{0.5pt}\mathcal{D}_{2\sigma}(x^i)} 
			+ \big(\beta\h{0.5pt}\sigma^2\big)^{-1} W^{\h{0.5pt}q}_{v,\h{0.5pt}\mathcal{D}_{2\sigma}(x^i)} 
			\h{15pt}\text{and}\h{15pt}
			W^{\h{0.5pt}q}_{\overline{v}^{i},\h{0.5pt}\varphi_k,\h{0.5pt}\mathcal{D}_i}    
			\lesssim_{\,\Omega\,} W^{\h{0.5pt}q}_{v,\h{0.5pt}\varphi_k,\h{0.5pt}\mathcal{D}_{2\sigma}(x^i)}. 
			\label{3575}
		\end{align}
		Estimates similar to \eqref{3287a}-\eqref{3287} remain valid in an analogous manner for the whole cylinders. \vspace{0.4pc}
		
		\noindent \textbf{Part 3. The extension $\overline{v}$ and its estimates:} Note that the family $\big\{\mathcal{D}_i \big\}_{i=1}^{n_x}\bigcup \big\{\mathcal{D}_i^+\big\}_{i=1}^{n_b}$ forms a cover of $\p^+B^+_{\rho}$. Using Besicovitch's covering theorem \cite[Theorem 2.7]{M99GA}, we can construct families $\mathscr{B}_1,\ldots,\mathscr{B}_{N_0}$ of sets such that each family $\mathscr{B}_i$ consists of disjoint elements from  $ \big\{\mathcal{D}_i\big\}_{i=1}^{n_x}\bigcup \big\{\mathcal{D}_i^+\big\}_{i=1}^{n_b}$. Every element in $\big\{\mathcal{D}_i\big\}_{i=1}^{n_x}\bigcup \big\{\mathcal{D}_i^+\big\}_{i=1}^{n_b}$ belongs to exactly one of $\big\{\mathscr{B}_i\big\}_{i=1}^{N_0}$. Here, $N_0$ is bounded by a universal constant, depending only on the dimension of the domain. $N_0$ is fixed in our study. Since $$\sum_{i=1}^{n_x} \mathcal{E}^*_{v,\h{0.5pt}\mathcal{D}_i}+\sum_{i=1}^{n_b} \mathcal{E}^*_{v,\h{0.5pt}\mathcal{D}_i^+}\geq\mathcal{E}^*_{v,\h{0.5pt}\p^+B^+_\rho},$$ there is an $i_0 \in \big\{1,2,\ldots,N_0\big\}$ such that \begin{align}\label{sum low bd} \sum_{\mathcal{D}_{i} \h{1pt}\in \h{1pt} \mathscr{B}_{i_0}} \mathcal{E}^*_{v,\h{0.5pt}\mathcal{D}_i}
		+\sum_{\mathcal{D}_{i}^+ \h{1pt}\in \h{1pt} \mathscr{B}_{i_0}} \mathcal{E}^*_{v,\h{0.5pt}\mathcal{D}_i^+}\geq  N_0^{-1} \h{1pt}\mathcal{E}^*_{v,\h{0.5pt}\p^+B^+_\rho}. \end{align} 
        
        \sloppy Define $\mathcal{O}_{i_0}$ as the union of all sets in $\mathscr{B}_{i_0}$ and let $$\mathcal{O}_{i_0,\sigma} := \Big\{\h{1pt}y\in \mathbb{R}^3:y=\zeta+\lambda\h{0.5pt} e_3 \h{2pt}\text{ for some $\zeta \in \overline{\mathcal{O}_{i_0}}$ and $\lambda\in(-\sigma,\sigma)$}\Big\}.$$ We extend $v$ from $\p^+B^+_\rho  $ to $B^+_{\rho,\sigma}$ so that the extension, denoted by $\overline{v}$, satisfies $$\overline{v}\h{1pt}\big|_{B^+_{\rho,\sigma}\setminus \h{1pt} \mathcal{O}_{i_0,\sigma}} 
        =v\h{1pt}\big|_{ \p^+B^+_{\rho} \setminus\h{1pt}\mathcal{O}_{i_0}} .$$ It is independent of the last coordinate on $B^+_{\rho,\sigma}\setminus \mathcal{O}_{i_0,\sigma}$. Given $\mathcal{D}_i^+\in\mathscr{B}_{i_0}$, we define the geodesic half-disk cylinder as follows: $$\mathcal{D}_{i,\sigma}^+:=\Big\{y\in \mathbb{R}^3:y=\zeta+\lambda \h{1pt} e_3\text{\h{1.5pt} for some $\zeta \in \overline{\mathcal{D}_i^+}$ and $\lambda\in(-\sigma,\sigma)$}\Big\}.$$
        Then, the extension $\overline{v}$ on $\mathcal{D}_{i,\sigma}^+$ is defined by $$\overline{v}\h{1pt}\big|_{\mathcal{D}_{i,\sigma}^+} := \mathbbm{v}^i\h{1pt}\circ\h{1pt} \overline{\varphi_{*,i}}^{\,-1}\h{1pt}\big|_{\mathcal{D}_{i,\sigma}^+}.$$ A similar extension is defined for $\overline{v}$ on the following geodesic disk cylinder: $$\mathcal{D}_{i,\sigma}:=\Big\{y\in \mathbb{R}^3:y=\zeta+\lambda \h{1pt}e_3\text{ for some $\zeta \in \overline{\mathcal{D}_i}$ and $\lambda\in(-\sigma,\sigma)$}\Big\}.$$ Note that we have the boundary condition: $$\overline{v}\h{1pt}\circ\h{1pt}\varphi^{-1}_7(\cdot)\in U_{ C_\Omega\left(\delta_2+ r_k\h{0.5pt}\sigma\right)}\big(\mathbb{S}^2\cap T_{\varphi_k(\cdot)}\p \Omega\big) \h{15pt}\text{on $\p^0 A_{\rho-\sigma,\rho}^+$, in the sense of trace,}$$  since it holds that $$\mathbbm{v}^i (\cdot)\in U_{ C_\Omega\left(\delta_2+ r_k\h{0.5pt}\sigma\right)}\big(\mathbb{S}^2\cap T_{\varphi_k\h{1pt}\circ\h{1pt}\varphi_7\h{1pt}\circ\h{1pt}\overline{\varphi_{*,i}}\left(\cdot\right)}\p \Omega\big) \h{15pt}\text{ on $\big( \p^0D^+_{\rho_i}\big)  \times [-\sigma,\sigma]$, for $i$ satisfying $ \mathcal{D}_{i}^+ \in\mathscr{B}_{i_0}$.} $$  In addition, using \eqref{3287} and the fact that $\varphi_7\h{1pt}\circ\h{1pt}\overline{\varphi_{*,i}}=\varphi_{*,i}$ on  $D_{\rho_i}^+\times\{\sigma\}$, we obtain 
		\begin{align*}
&W^{\h{0.5pt}q}_{\mathbbm{v}^i\h{1pt}\circ\h{1pt}\overline{\varphi_{*,i}}^{\,-1},\varphi_k\h{1pt}\circ\h{1pt}\varphi_7, \h{0.5pt}\mathcal{D}_{i,\sigma}^+} \\[1mm]&\h{12pt}\lesssim_{\,\Omega\,}\h{1.5pt}  \sigma \h{1pt}W^{\h{0.5pt}q}_{ v\h{1pt}\circ\h{1pt} \varphi_7,\h{0.5pt}\varphi_k\h{1pt}\circ\h{1pt}\varphi_7,\h{0.5pt}\mathcal{D}_i^++\sigma e_3} +  \sigma \h{1pt}W^{\h{0.5pt}q}_{\overline{\widetilde{v}^{*,i}}\h{1pt}\circ \h{1pt} \overline{\varphi_{*,i}}^{\,-1},\h{0.5pt}\varphi_k\h{1pt}\circ\h{1pt}\varphi_7,\h{0.5pt}\mathcal{D}_i^+ -\sigma e_3} \nonumber\\[1mm]
			&\h{12pt}
			+ \sigma^2 \h{1pt} W^{\h{0.5pt}q}_{v^{*,i}\h{1pt}\circ\h{1pt}\varphi_7,\h{0.5pt}\varphi_k\h{1pt}\circ\h{1pt}\varphi_7,\h{0.5pt}\p^+  \mathcal{D}_i^+ +\sigma e_3 } + \sigma^3\h{0.5pt}r_k^2\h{1pt}\displaystyle\sup_{y^1,\h{1pt}y^2\h{1pt}\in\h{1pt}   D_{\rho_i}^+\times (-\sigma,\sigma)}
			\Big| \h{1pt}\varphi_7\h{1pt}\circ\h{1pt}\overline{\varphi_{*,i}} \left(y^1\right)- \varphi_7\h{1pt}\circ\h{1pt}\overline{\varphi_{*,i}} \left(y^2\right)\Big|^2.
		\end{align*}
        As $\varphi_7(\zeta+\lambda e_3)=\zeta$ when $\zeta \in \overline{\p^+B^+_{\rho}}$, it can be simplified to
        \begin{align}
&W^{\h{0.5pt}q}_{\mathbbm{v}^i\h{1pt}\circ\h{1pt}\overline{\varphi_{*,i}}^{\,-1},\varphi_k\h{1pt}\circ\h{1pt}\varphi_7, \h{0.5pt}\mathcal{D}_{i,\sigma}^+} \,\lesssim_{\,\Omega\,}\h{1.5pt}    \sigma \h{1pt}W^{\h{0.5pt}q}_{ v,\h{0.5pt}\varphi_k,\h{0.5pt}\mathcal{D}_i^+}
			+  \sigma \h{1pt}W^{\h{0.5pt}q}_{\overline{\widetilde{v}^{*,i}} \h{1pt}\circ \h{1pt} \overline{\varphi_{*,i}}^{\,-1},\h{0.5pt}\varphi_k\h{1pt}\circ\h{1pt}\varphi_7,\h{0.5pt}\mathcal{D}_i^+ -\sigma e_3} \nonumber\\[1mm]
			&\h{20pt}+  \sigma^2 \h{1pt} W^{\h{0.5pt}q}_{v^{*,i},\h{0.5pt}\varphi_k,\h{0.5pt}\p^+  \mathcal{D}_i^+} +  \sigma^3\h{0.5pt}r_k^2\h{1pt}\displaystyle\sup_{y^1,\h{1pt}y^2 \h{1pt}\in \h{1pt}   D_{\rho_i}^+\times (-\sigma,\sigma)}
			\Big|\h{1pt} \varphi_7\h{1pt}\circ\h{1pt}\overline{\varphi_{*,i}} \left(y^1\right)- \varphi_7\h{1pt}\circ\h{1pt}\overline{\varphi_{*,i}} \left(y^2\right)\Big|^2.
			\label{3591}
		\end{align}
		
        By construction of $\overline{v}$, items (a), (d) and (e) of the lemma hold for the extension $\overline{v}\h{1pt}\circ\h{1pt}\varphi^{-1}_7$. We are left to verify the remaining estimates in the lemma.
        
        Due to \eqref{3287a}-\eqref{3287} and Lemma \ref{lem ext to cylinder}, it turns out that
		\begin{align*}
			\mathcal{E}^*_{\h{1pt}\overline{v},\h{0.5pt}B^+_{\rho,\sigma}}
			&=\mathcal{E}^*_{\h{1pt}\overline{v},\h{0.5pt}B^+_{\rho,\sigma}\setminus \h{0.5pt}\mathcal{O}_{i_0,\sigma}}
			+\mathcal{E}^*_{\h{1pt}\overline{v},\h{0.5pt} \mathcal{O}_{i_0,\sigma}} \h{1.5pt} \lesssim_{\,\Omega\,}\h{1.5pt} \sigma \h{1pt} \mathcal{E}^*_{v,\h{0.5pt}\p^+B^+_\rho \setminus \h{0.5pt} \mathcal{O}_{i_0}}
			\\[2mm]
			&+ \sigma \sum_{ \mathcal{D}_{i}^+ \h{.5pt}\in \h{1pt}\mathscr{B}_{i_0}} \Big(\mathcal{E}^*_{v,\h{0.5pt}\mathcal{D}_i^+} 
			+ \mathcal{E}^*_{\h{1pt}\overline{v}^{*,i},\h{.5pt}\mathcal{D}_i^+}  
			+ \sigma \h{1pt} \mathcal{E}^*_{v^{*,i},\h{0.5pt}\p^+ \mathcal{D}_i^+}  \Big) + \sigma \sum_{\mathcal{D}_{i} \h{0.5pt}\in \h{1pt}\mathscr{B}_{i_0}} \Big(\mathcal{E}^*_{v,\h{0.5pt}\mathcal{D}_i} 
			+ \mathcal{E}^*_{\h{1pt}\overline{v}^i,\h{0.5pt}\mathcal{D}_i}  
			+ \sigma \h{1pt}\mathcal{E}^*_{v^i,\h{0.5pt}\p \mathcal{D}_i}  \Big). 
		\end{align*} 
		It is evident that $n_x+n_b\lesssim (\rho/\sigma)^2$. Moreover, by modifying \cite[Lemma 2.6]{M99GA}, we observe that any point $y \in \p^+B^+_\rho$ belongs to at most a bounded number of disks in $\big\{\mathcal{D}_{2\sigma}(x^i)\big\}_{i=1}^{n_x}\bigcup \big\{\mathcal{D}_{4\sigma}^+(b^i)\big\}_{i=1}^{n_b}$. The number of disks to which $y$ belongs depends only on the dimension of the domain, which we fix in our remaining proof. Consequently, the sum of the characteristic functions: $$\sum_{ \mathcal{D}_{i}^+ \h{0.5pt}\in \h{1pt}\mathscr{B}_{i_0}}\mathbbm{1}_{\mathcal{D}_{4\sigma}^+(b^i)} 
		+\sum_{ \mathcal{D}_{i} \h{0.5pt}\in \h{1pt}\mathscr{B}_{i_0}} \mathbbm{1}_{\mathcal{D}_{2\sigma}(x^i)}$$ is bounded by a universal constant. Now, we apply \eqref{1306}, \eqref{3575}, \eqref{3275a}, and \eqref{3240}, followed by an estimate similar to \eqref{3501}, to obtain that
		\begin{align*}
			\mathcal{E}^*_{\h{1pt}\overline{v},\h{0.5pt}B^+_{\rho,\sigma}}
			&\lesssim_{\,\Omega\,} \sigma \left(1+\beta\right)\mathcal{E}^*_{v,\h{0.5pt}\p^+B^+_\rho}
			+ (\beta\sigma)^{-1}\sum_{\mathcal{D}_{i}^+ \h{1pt}\in \h{1pt}\mathscr{B}_{i_0}} 
W^{\h{0.5pt}q}_{v,\h{0.5pt}\varphi_{\ddagger,\star}\h{1pt}\circ\h{1pt}\varphi_{*,i}^{-1},\h{0.5pt}\mathcal{D}_{4\sigma}^+(b^i)} + (\beta\sigma)^{-1} \sum_{\mathcal{D}_{i} \h{1pt}\in\h{1pt}\mathscr{B}_{i_0}}  
			 W^{\h{0.5pt}q}_{v,\h{0.5pt}\mathcal{D}_{2\sigma} (x^i)} \\[1mm]
			&\lesssim_{\,\Omega\,} \sigma \left(1+\beta\right) \mathcal{E}^*_{v,\h{0.5pt}\p^+B^+_\rho}
			+(\beta\sigma)^{-1} \h{1pt}
			W^{\h{0.5pt}q}_{v,\h{0.5pt}\varphi_k,\h{0.5pt}\p^+B^+_\rho}
			+ n_b \h{1.5pt}\sigma \h{1pt} \beta^{-1} r_k^2 .
		\end{align*} 
		Similarly, we utilize \eqref{3591} and Lemma \ref{lem ext to cylinder}. It then follows that
		\begin{align*}
W^{\h{0.5pt}q}_{\overline{v},\h{0.5pt}\varphi_k\h{1pt}\circ\h{1pt}\varphi_7,\h{0.5pt}B^+_{\rho,\sigma}}
			&= 
			W^{\h{0.5pt}q}_{\overline{v},\h{0.5pt}\varphi_k\h{1pt}\circ\h{1pt}\varphi_7,\h{0.5pt}B^+_{\rho,\sigma}\setminus \h{0.5pt}\mathcal{O}_{i_0,\sigma}}
			+W^{\h{0.5pt}q}_{\overline{v},\h{0.5pt}\varphi_k\h{1pt}\circ\h{1pt}\varphi_7,\h{0.5pt} \mathcal{O}_{i_0,\sigma}} \\[1mm]
            &\lesssim_{\,\Omega\,} 
			\sigma \h{1pt}W^{\h{0.5pt}q}_{v,\h{0.5pt}\varphi_k\h{1pt}\circ\h{1pt}\varphi_7,\h{0.5pt}(\p^+B^+_\rho \h{0.5pt}\setminus \h{0.5pt}  \mathcal{O}_{i_0})+\sigma \h{0.5pt} e_3}  +\sigma \left(\rho \h{1pt} r_k\right)^2 \big(\sigma^2+\rho^2\big)\\[1mm]
			&+ \sigma\sum_{ \mathcal{D}_{i}^+ \h{1pt}\in \h{1pt}\mathscr{B}_{i_0}} \left(W^{\h{0.5pt}q}_{v,\h{0.5pt}\varphi_k\h{1pt},\h{0.5pt}\mathcal{D}_i^+} 
			+ W^{\h{0.5pt}q}_{\overline{\widetilde{v}^{*,i}}\h{1pt}\circ \h{1pt} \overline{\varphi_{*,i}}^{\,-1},\h{0.5pt}\varphi_k\h{1pt}\circ\h{1pt}\varphi_7,\h{0.5pt}\mathcal{D}_i^+ -\sigma e_3} 
			+ \sigma \h{1pt} W^{\h{0.5pt}q}_{v^{*,i},\h{0.5pt}\varphi_k ,\h{0.5pt}\p^+ \mathcal{D}_i^+ }
			+\sigma^4 \h{1pt}r_k^2  \right)\\[1mm]
			&+ \sigma\sum_{\mathcal{D}_{i} \h{1pt}\in \h{1pt}\mathscr{B}_{i_0}} \left(W^{\h{0.5pt}q}_{v,\h{0.5pt}\mathcal{D}_i} 
			+ W^{\h{0.5pt}q}_{\overline{v}^i,\h{0.5pt}\mathcal{D}_i}
			+ \sigma \h{1pt}W^{\h{0.5pt}q}_{v^i,\h{0.5pt}\p \mathcal{D}_i} + \sigma^{-1}\h{1pt}W_{\overline{v},\h{0.5pt}\varphi_k\h{1pt}\circ\h{1pt}\varphi_7,\h{0.5pt}\mathcal{D}_{i,\sigma} } \right).
		\end{align*}
		We handle the term $W_{\overline{v},\h{0.5pt}\varphi_k\h{1pt}\circ\h{1pt}\varphi_7,\h{0.5pt}\mathcal{D}_{i,\sigma} }$ by the same method as in \eqref{3212}. Hence, $$W_{\overline{v},\varphi_k\h{1pt}\circ\h{1pt}\varphi_7,\mathcal{D}_{i,\sigma}}
		\lesssim_{\,\Omega\,} \sigma \h{1pt} W_{v,\varphi_k,\mathcal{D}_{i} }
        +\sigma \h{1pt} W_{\overline{v}^i,\varphi_k\h{1pt}\circ\h{1pt}\varphi_7\left(\cdot\,-\sigma \h{0.5pt}e_3\right),\mathcal{D}_{i}}
        +\sigma^2 \h{1pt} W_{v^i,\varphi_k,\p\mathcal{D}_{i} }+\sigma^5 \h{1pt}r_k^2.$$ Using the above estimates, together with \eqref{3240}, \eqref{3275}, \eqref{1306}, and \eqref{3575}, we deduce
		\begin{align*}
W^{\h{0.5pt}q}_{\overline{v},\h{0.5pt}\varphi_k\h{1pt}\circ\h{1pt}\varphi_7,\h{0.5pt}B^+_{\rho,\sigma}}
			&\lesssim_{\,\Omega\,} 
			\sigma \h{1pt}W^{\h{0.5pt}q}_{v,\h{0.5pt}\varphi_k,\h{0.5pt}\p^+B^+_\rho \setminus \h{0.5pt} \mathcal{O}_{i_0}} 
            +\sigma \left(\rho \h{1pt} r_k\right)^2 \big(\sigma^2+\rho^2\big)\\[1mm]
		& + \sigma \sum_{\mathcal{D}_{i}^+ \h{1pt}\in \h{1pt}\mathscr{B}_{i_0}} \left(W^{\h{0.5pt}q}_{v,\h{0.5pt}\varphi_k\h{1pt},\h{0.5pt}\mathcal{D}_i^+} 
			+ W^{\h{0.5pt}q}_{\overline{\widetilde{v}^{*,i}} \h{1pt}\circ \h{1pt} \overline{\varphi_{*,i}}^{\,-1},\h{0.5pt}\varphi_k\h{1pt}\circ\h{1pt}\varphi_7,\h{0.5pt}\mathcal{D}_i^+ -\sigma \h{0.5pt} e_3} 
			+ \sigma \h{1pt} W^{\h{0.5pt}q}_{v^{*,i},\h{0.5pt}\varphi_k ,\h{0.5pt}\p^+ \mathcal{D}_i^+ }
			+\sigma^4 \h{1pt}r_k^2  \right)\\[1mm]
				&+ \sigma\sum_{\mathcal{D}_{i} \h{1pt}\in \h{1pt}\mathscr{B}_{i_0}} \left(W^{\h{0.5pt}q}_{v,\h{0.5pt}\varphi_k,\h{0.5pt}\mathcal{D}_i} 
			+ W^{\h{0.5pt}q}_{\overline{v}^i,\varphi_k\circ\varphi_7(\cdot\,-\sigma e_3),\mathcal{D}_{i}}  
			+ \sigma W^{\h{0.5pt}q}_{v^{i},\h{0.5pt}\varphi_k,\h{0.5pt}\p \mathcal{D}_i }
			\right)\\[1mm]
			&\lesssim_{\,\Omega\,} \sigma \h{1pt}W^{\h{0.5pt}q}_{v,\h{0.5pt}\varphi_k,\h{0.5pt}\p^+B^+_\rho }
			+\sigma^5\h{1pt}r_k^2 \left(n_b+n_x\right).
		\end{align*}
		
		For the estimate of $\overline{v}$ on $\p^+B^+_\rho  -\sigma e_3$, we use \eqref{3275a}, \eqref{3575} and \eqref{sum low bd} to deduce that
		\begin{align*}
			\mathcal{E}^*_{\h{1pt}\overline{v},\h{0.5pt}\p^+B^+_\rho - \sigma \h{0.5pt} e_3}
			&= \mathcal{E}^*_{v,\h{0.5pt}\p^+B^+_\rho \setminus \h{0.5pt}  \mathcal{O}_{i_0}}
			+\sum_{\mathcal{D}_{i} \h{1pt}\in\h{1pt}\mathscr{B}_{i_0}} \mathcal{E}^*_{\h{1pt}\overline{v}^i,\h{0.5pt}\mathcal{D}_{i}}
			+\sum_{\mathcal{D}_{i}^+ \h{1pt}\in\h{1pt}\mathscr{B}_{i_0}} \mathcal{E}^*_{ \overline{v}^{*,i},\h{0.5pt}\mathcal{D}_i^+}\\[1mm]
			&\leq \left(1- N_0^{-1} \right)\mathcal{E}^*_{v,\h{0.5pt}\p^+B^+_\rho}
			+C_\Omega\sum_{\mathcal{D}_{i} \h{1pt}\in\h{1pt}\mathscr{B}_{i_0}} \left( \beta\h{1pt}\mathcal{E}^*_{v,\h{0.5pt}\mathcal{D}_{2\sigma}(x^i)}
			+\big(\beta\h{0.5pt}\sigma^2\big)^{-1}\h{1pt}W^{\h{0.5pt}q}_{v,\h{0.5pt}\mathcal{D}_{2\sigma}(x^i)}\right)\\[1mm]
			& \h{63pt}	+C_\Omega\sum_{\mathcal{D}_{i}^+ \h{1pt}\in\h{1pt}\mathscr{B}_{i_0}} \left(\beta\h{1pt}\mathcal{E}^*_{v,\h{0.5pt}\mathcal{D}_{4\sigma}^{+}(b^i)}
			+\big(\beta\h{0.5pt}\sigma^2\big)^{-1} \h{1pt}W^{\h{0.5pt}q}_{v,\h{0.5pt}\varphi_{\ddagger,\star}\h{1pt}\circ\h{1pt}\varphi_{*,i}^{-1},\h{0.5pt}\mathcal{D}_{4\sigma}^+(b^i)} \right).
		\end{align*}
        Grouping the terms in the above and applying the triangle inequality yields
        		\begin{align*}
\mathcal{E}^*_{\h{1pt}\overline{v},\h{0.5pt}\p^+B^+_\rho - \sigma \h{0.5pt} e_3}
			&\leq \left(1- N_0^{-1}+ C_\Omega\h{1pt}\beta \right)\mathcal{E}^*_{v,\h{0.5pt}\p^+B^+_\rho}
			+C_\Omega \left(\beta\sigma^2\right)^{-1}
			W^{\h{0.5pt}q}_{v,\h{0.5pt}\varphi_k,\h{0.5pt}\p^+B^+_\rho}
			+ C_\Omega \h{1pt} n_b\h{1pt} \beta^{-1} r_k^2.
		\end{align*}We use the homeomorphism $\varphi_7$ to change the variables on the left side above. Then item (f) of the lemma follows from the extension $\overline{v}\h{1pt}\circ\h{1pt}\varphi^{-1}_7$, by choosing $\beta\in(0,1)$ small enough. The smallness of $\beta$ depends only on $\Omega$. 
        
        In the end, we note that
		\begin{align*}
			&W^{\h{0.5pt}q}_{ \overline{v},\h{0.5pt}\varphi_k\h{1pt}\circ\h{1pt}\varphi_7,\h{0.5pt}\p^+B^+_\rho - \sigma \h{0.5pt}e_3}\\[1mm]
			&\h{5pt}= W^{\h{0.5pt}q}_{\overline{v},\h{0.5pt}\varphi_k\h{1pt}\circ\h{1pt}\varphi_7,\h{0.5pt}(\p^+B^+_\rho \setminus \h{0.5pt} \mathcal{O}_{i_0} ) -\sigma \h{0.5pt} e_3}
			+\sum_{\mathcal{D}_{i} \h{1pt}\in \h{1pt}\mathscr{B}_{i_0}} W^{\h{0.5pt}q}_{ \overline{v},\h{0.5pt}\varphi_k\h{1pt}\circ\h{1pt}\varphi_7,\h{0.5pt}\mathcal{D}_i -\sigma \h{0.5pt}e_3}
			+\sum_{\mathcal{D}_{i}^+ \h{1pt}\in\h{1pt}\mathscr{B}_{i_0}} W^{\h{0.5pt}q}_{ \overline{v},\h{0.5pt}\varphi_k\h{1pt}\circ\h{1pt}\varphi_7,\h{0.5pt}\mathcal{D}_i^+ -\sigma \h{0.5pt} e_3}
			\\[1mm]
			&\h{5pt}\lesssim_{\,\Omega\,} W^{\h{0.5pt}q}_{v,\h{0.5pt}\varphi_k,\h{0.5pt}\p^+B^+_\rho \setminus \h{0.5pt} \mathcal{O}_{i_0}}
			+\sigma^4\h{1pt}r_k^2 \left(1+n_x+n_b\right)
			+\sum_{\mathcal{D}_{i} \h{1pt}\in\h{1pt}\mathscr{B}_{i_0}} 
			W^{\h{0.5pt}q}_{\overline{v}^{i},\h{0.5pt}\varphi_k,\h{0.5pt}\mathcal{D}_i}
			+\sum_{\mathcal{D}_{i}^+ \h{1pt}\in \h{1pt}\mathscr{B}_{i_0}} W^{\h{0.5pt}q}_{\overline{v}^{*,i},\h{0.5pt}\varphi_{\ddagger,\star}\h{1pt}\circ\h{1pt}\varphi_4\h{1pt}\circ\h{1pt}\varphi_{*,i}^{-1},\h{0.5pt}\mathcal{D}_i^+}.
		\end{align*}	
        Estimates \eqref{3275} and \eqref{3575} imply
        		\begin{align*}
			&W^{\h{0.5pt}q}_{ \overline{v},\h{0.5pt}\varphi_k\h{1pt}\circ\h{1pt}\varphi_7,\h{0.5pt}\p^+B^+_\rho - \sigma \h{0.5pt}e_3}\\[1mm]
			&\h{5pt}\lesssim_{\,\Omega\,} W^{\h{0.5pt}q}_{v,\h{0.5pt}\varphi_k,\h{0.5pt}\p^+B^+_\rho \setminus \h{0.5pt} \mathcal{O}_{i_0}}
			+\sigma^4\h{1pt}r_k^2\left(1+n_x+n_b\right)
			+\sum_{\mathcal{D}_{i} \h{1pt}\in\h{1pt}\mathscr{B}_{i_0}} 
			W^{\h{0.5pt}q}_{v,\h{0.5pt}\varphi_k,\h{0.5pt}\mathcal{D}_{2\sigma}(x^i)}
			+\sum_{\mathcal{D}_{i}^+ \h{1pt}\in\h{1pt}\mathscr{B}_{i_0}} W^{\h{0.5pt}q}_{v,\h{0.5pt}\varphi_k,\h{0.5pt}\mathcal{D}_{4\sigma}^+(b^i)}
			\\[1mm]
			&\h{5pt}\lesssim  W^{\h{0.5pt}q}_{v,\h{0.5pt}\varphi_k ,\h{0.5pt}\p^+B^+_\rho  }
			+\sigma^4\h{1pt}r_k^2 \left(n_b+n_x\right).
		\end{align*}	
		The desired estimates of the lemma are satisfied by the extension $\overline{v}\h{1pt}\circ\h{1pt}\varphi^{-1}_7$. 
	\end{proof}	
	Finally, we extend $v$ from the upper hemisphere to the upper half-ball.
	
	\begin{proof}[Proof of Proposition \ref{lem ext from sphere to ball}] Recall the constants $\alpha$, $\C{17}$, and $\eps{7}$ given by Lemma \ref{lem ext to annulus}. Let $\delta'$, $\epsilon_0\in(0,\frac{1}{2})$, and $\mu>2$ be parameters to be determined later. Moreover, we assume that $$ \text{$k\in\mathbb{N}$,\h{5pt} $\rho\in\big(0,\frac{1}{2}\big)$,\h{5pt} $\delta\in(0,\delta')$, \h{5pt} $\epsilon\in (0,\epsilon_0)$, \h{5pt}and\h{5pt} $v \in H^1\big(\p^+B_{\rho}^{+};\mathbb{S}^2\big)$} $$satisfy the assumptions stated in Proposition \ref{lem ext from sphere to ball}. Since $\alpha\in(0,1)$, we can choose $s \in \mathbb{N}$ such that $\alpha^s \in [\alpha\h{0.5pt}\epsilon,\epsilon]$. With a fixed constant $M_1 >0$, the remainder of the proof is divided into two cases. The constant $M_1$ depends only on $\Omega$ and will be determined later. \vspace{0.4pc}
		
		\noindent\textbf{Case 1:} In this case, we assume  \begin{align}\label{assum of case 1} M_1 ^{2s-1} \left( W^{\h{0.5pt}q}_{v,\h{0.5pt}\varphi_k,\h{0.5pt}\p^+ B_\rho^+}\right) ^2
			\leq \left(\epsilon\h{0.5pt}\rho\right)^4 \delta^2.\end{align} First, let us briefly outline the steps to construct the extension mappings. Let $\sigma=\frac{\epsilon\h{0.5pt}\rho}{2}$, and denote the upper half annuli by \[
		A_{\sigma}^{i,+} := A^+_{\rho - 2\h{0.5pt}i\h{0.5pt}\sigma,\h{1pt}\rho - 2\h{0.5pt}(i-1\h{.5pt})\sigma}, \h{20pt}\text{where}\h{2pt}i = 1,2, \dots, s.\]
       In the first step, we apply Lemma \ref{lem ext to annulus} to extend the mapping $v$ defined on the outer boundary of $A_{\sigma}^{1,+}$, obtaining the extension $\overline{v}^1$ on $A_{\sigma}^{1,+}$. The trace of $\overline{v}^1$ on the inner boundary of $A_{\sigma}^{1,+}$ is simply denoted by $v^{1}_\textup{in}$. Inductively, if $v^{i-1}_\textup{in}$ is defined on the outer boundary of $A_{\sigma}^{i,+}$, for some $i \geq 2$, then, using Lemma \ref{lem ext to annulus} again, we get the extension $\overline{v}^i$ on $A_{\sigma}^{i,+}$. Its trace on the inner boundary of $A_{\sigma}^{i,+}$ is denoted by  $v^{i}_\textup{in}$. As a convention, we set $v^{0}_\textup{in}=v$. Furthermore, we radially extend $v^{s}_\textup{in}$ defined on $\p^+B^+_{\rho-2\h{0.5pt}s\h{0.5pt}\sigma}$ into the half ball $B^+_{\rho-2\h{0.5pt}s\h{0.5pt}\sigma}$. The extension in $B^+_{\rho-2\h{0.5pt}s\h{0.5pt}\sigma}$ is $0$-homogeneous. \vspace{0.4pc} 
		
		\noindent\textbf{Part 1.1. Validity of Lemma \ref{lem ext to annulus} in the extension of $v$:} In this part, we prove by induction that we can apply Lemma \ref{lem ext to annulus} $s$ times, as described above, to extend $v$ defined on $\p^+ B_\rho^+$ into $A^+_{\rho - 2 \h{0.2pt}s\h{0.2pt}\sigma, \h{1pt}\rho}$. In what follows, the main task is to verify the assumptions of Lemma \ref{lem ext to annulus} at each iterative step to justify the validity of this construction. \vspace{0.2pc}
        
        In the first step, the assumptions of the proposition, together with the facts that $\sigma=\frac{\epsilon\h{0.5pt}\rho}{2}$, $\delta \in(0,\eps{7})$, and $\mu>2$, induce the following estimate:
		\begin{align*}
			r_k^2\left[\h{1pt}r_k^2\left(2\sigma\right)^2+ \mathcal{E}^*_{v,\h{0.5pt}\p^+ B_\rho^+}+ W_{v,\h{0.5pt}\varphi_k,\h{0.5pt}\p^+ B_\rho^+} \h{1pt}\right] 
			+(2\sigma)^{-2}\h{1pt}\mathcal{E}^*_{v,\h{0.5pt}\p^+ B_\rho^+}
			\h{1pt}W^{\h{0.5pt}q}_{v,\h{0.5pt}\varphi_k,\h{0.5pt}\p^+ B_\rho^+}
			\leq \delta^2
			+\left(\epsilon\rho\right)^{-2} \h{1pt}\delta^2\h{0.5pt}\rho^2\h{0.5pt}\epsilon^\mu
			< 2\delta^2.
		\end{align*}
		The rightmost side above is smaller than $\eps{7}^2$ if  $\sqrt{2}\h{0.5pt} \delta< \eps{7}$. Moreover, since $\alpha\epsilon\leq\alpha^s$, to ensure $$2\sigma<\frac{\rho-\left(i-1\right)\epsilon\h{0.5pt}\rho}{4},\h{15pt}\text{for any $i = 1,2, \dots, s$, }$$ it is sufficient to assume $4\epsilon+ \epsilon \h{1pt}\frac{\log \epsilon}{\log \alpha}  < \frac{1}{2}$.  
		Thus, we can apply Lemma \ref{lem ext to annulus} to obtain an extension $\overline{v}^1$ of $v$ on $A^{1,+}_\sigma$, satisfying the consequences there. In particular, the following conditions hold for the extension $\overline{v}^1 \in H^1\big(A_{\rho-2\h{0.5pt}\sigma,\h{1pt}\rho}^+; \mathbb{S}^2\big)$: 
        $$\overline{v}^1(\cdot)\in U_{2\h{0.5pt}\C{17}\left(\sqrt{\delta}\h{0.5pt}+\h{0.5pt}r_k\h{0.5pt}\sigma\right)}\big(\mathbb{S}^2\cap T_{\varphi_k(\cdot)}\p \Omega\big) \h{20pt}\text{on $\p^0 A_{\rho-2\h{0.5pt}\sigma,\h{1pt}\rho}^+$},$$  and $$v_\textup{in}^1(\cdot):=\overline{v}^1\h{1pt}\Big|_{\p^+B_{\rho-2\h{0.5pt}\sigma}^+}(\cdot) \in U_{2\h{0.5pt}\C{17} \left(\sqrt{\delta}+r_k\h{0.5pt}\sigma\right)}\big(\mathbb{S}^2\cap T_{\varphi_k(\cdot)}\p \Omega\big) \h{20pt}\text{on $\p B_{\rho-2\h{0.5pt}\sigma} \cap \big\{y_3=0\big\}$,}$$  both in the sense of trace.  
		
		Then, we assume without loss of generality that $s\geq 2$ and suppose that we can apply Lemma \ref{lem ext to annulus} up to $i-1$ times for $i=2,3,\ldots,s$. When we apply Lemma \ref{lem ext to annulus} to extend $v^{i-2}_\textup{in}$ from $\p^+B^+_{\rho-2\left(i-2\right)\sigma}$ to $A^{i-1,+}_{\sigma}$, we know that the extension $\overline{v}^{i-1}$ satisfies the following estimates:
		\begin{align*}
			&\mathcal{E}^*_{ v^{i-1}_\textup{in},\h{0.5pt}\p^+B^+_{\rho-2\left(i-1\right)\sigma}} 
			\leq \alpha\h{1pt}\mathcal{E}^*_{ v^{i-2}_\textup{in},\h{0.5pt}\p^+B^+_{\rho-2\left(i-2\right)\sigma}} 
			+ \C{17} \left(\beta\h{1pt}\sigma^2\right)^{-1}\left(
			W^{\h{0.5pt}q}_{ v^{i-2}_\textup{in},\h{0.5pt}\varphi_k,\h{0.5pt}\p^+B^+_{\rho-2\left(i-2\right)\sigma}}    
			+r_k^2\h{1pt}\rho^2\right),\\[1.5mm]
			&W^{\h{0.5pt}q}_{ 
				v^{i-1}_\textup{in},\h{0.5pt}
				\varphi_k,\h{0.5pt}\p^+B^+_{\rho-2(i-1)\sigma}}    
			\leq \C{17}\left( W^{\h{0.5pt}q}_{ v^{i-2}_\textup{in},\h{0.5pt}\varphi_k,\h{0.5pt}\p^+B^+_{\rho-2(i-2)\sigma}}   +\sigma^2r_k^2\right).
		\end{align*}
		Iterating these two inequalities induces the two estimates given below:
		\begin{align*}
			\mathcal{E}^*_{ v^{i-1}_\textup{in},\h{0.5pt}\p^+B^+_{\rho-2(i-1)\sigma}} 
			\leq\,  \alpha^{i-1} \h{1pt}\mathcal{E}^*_{v,\h{0.5pt}\p^+B^+_\rho}
			&+  \C{17}^{i-1} \left(\beta\h{0.5pt}\sigma^2\right)^{-1}\h{1pt}\sum^{i-2}_{j=0}\left(\frac{\alpha}{\C{17}}\right)^j \h{1pt}   
			W^{\h{0.5pt}q}_{v,\h{0.5pt}\varphi_k,\h{0.5pt}\p^+B^+_\rho}
			\nonumber\\[1mm]  
			&+ \C{17} \left(\beta\h{0.5pt}\sigma^2\right)^{-1}r_k^2\h{1pt}\rho^2\h{1pt}\sum^{i-2}_{j=0}\alpha^j
			+\C{17}^2 \h{1pt}\beta^{-1} \h{1pt} r_k^2\h{1pt}\sum^{i-2}_{j=0} 
			\left(\alpha^j \h{1pt} \dfrac{\C{17}^{i-2-j}-1}{\C{17}-1} \right),\end{align*}and            \begin{align*}
			W^{\h{0.5pt}q}_{ v^{i-1}_\textup{in},\h{0.5pt}\varphi_k,\h{0.5pt}\p^+B^+_{\rho-2(i-1)\sigma}}
			\leq\, \h{1pt} \C{17}^{i-1} \h{1pt}W^{\h{0.5pt}q}_{v,\h{0.5pt}\varphi_k,\h{0.5pt}\p^+B^+_\rho}
			+ \C{17}\h{1pt}\sigma^2\h{1pt}r_k^2 \h{1pt} \sum^{i-2}_{j=0} 
			\C{17}^{j}.
		\end{align*}
		Now, we choose $M_1\geq\C{17}$ sufficiently large so that the above estimates can be controlled by
		\begin{align}
			&\mathcal{E}^*_{ v^{i-1}_\textup{in},\h{0.5pt}\p^+B^+_{\rho-2(i-1)\sigma}} 
			\leq\,  \alpha^{i-1} \h{1pt}\mathcal{E}^*_{v,\h{0.5pt}\p^+B^+_\rho}
			+ M_1 ^{i}\h{1pt}\sigma^{-2} \h{1pt}
			W^{\h{0.5pt}q}_{v,\h{0.5pt}\varphi_k,\h{0.5pt}\p^+B^+_\rho}
			+ M_1 \h{1pt}\sigma^{-2}\h{1pt}r_k^2 \h{1pt}\rho^2 
			+M_1 ^i \h{1pt}r_k^2,  \nonumber\\[2mm]
			&W^{\h{0.5pt}q}_{ v^{i-1}_\textup{in},\h{0.5pt}\varphi_k,\h{0.5pt}\p^+B^+_{\rho-2(i-1)\sigma}}
			\leq\, M_1 ^{i-1}W^{\h{0.5pt}q}_{v,\h{0.5pt}\varphi_k,\h{0.5pt}\p^+B^+_\rho}
			+ M_1 ^{i}\h{1pt}\sigma^2\h{1pt}r_k^2 .
			\label{3649}
		\end{align}

		To apply Lemma \ref{lem ext to annulus} to obtain the extension $\overline{v}^{i}$ on $A^{i,+}_\sigma$ of $v^{i-1}_\textup{in}$, we need $J_i\leq \eps{7}^2$. Here,
		\begin{align}
			&J_i:= \left( 2\sigma\right)^{-2} \mathcal{E}^*_{v^{i-1}_\textup{in},\h{0.5pt}\p^+ B_{\rho-2(i-1)\sigma}^+}
			W^{\h{0.5pt}q}_{v^{i-1}_\textup{in},\h{0.5pt}	\varphi_k,\h{0.5pt}\p^+ B_{\rho-2(i-1)\sigma}^+}\nonumber\\[2mm] &\h{46pt}+ \,r_k^2\left(r_k^2\left(2\sigma\right)^2 +\mathcal{E}^*_{v^{i-1}_\textup{in},\h{0.5pt}\p^+ B_{\rho-2(i-1)\sigma}^+}
			+	W_{v^{i-1}_\textup{in},	\h{0.5pt}\varphi_k,\h{0.5pt}\p^+ B_{\rho-2(i-1)\sigma}^+}\right).
			\label{3590}
		\end{align} 
		Plugging \eqref{3649} into \eqref{3590} and setting $I_i:= M_1 \h{1pt}\sigma^{-2}\h{1pt} r_k^2 \h{1pt}\rho^2 
		+M_1 ^i \h{1pt} r_k^2$, we see that
		\begin{align*}
			J_i\leq\,& M_1^{i} \h{1pt}I_i \h{1pt} r_k^2 + \left(2\sigma\right)^{-2}  M_1^{i-1}\h{1pt}  \mathcal{E}^*_{v,\h{0.5pt}\p^+ B_\rho^+} 
			\left( M_1\sigma^2r_k^2 +  W^{\h{0.5pt}q}_{v,\h{0.5pt}\varphi_k,\h{0.5pt}\p^+ B_\rho^+} \right)  \\[2mm]
			&+ \left(2\sigma\right)^{-2}  M_1^{i-1} \h{1pt}W^{\h{0.5pt}q}_{v,\h{0.5pt}\varphi_k,\h{0.5pt}\p^+ B_\rho^+}
			\left( I_i  + M_1 ^{i+1}r_k^2 
			+    M_1 ^{i} \h{1pt} \sigma^{-2} \h{1pt}   W^{\h{0.5pt}q}_{v,\h{0.5pt}\varphi_k,\h{0.5pt}\p^+ B_\rho^+}  \right)\\[2mm]
            &+ r_k^2\left(\h{1pt}r_k^2\left(2\sigma\right)^2 + I_i  + M_1^{i} \h{1pt}\sigma^2\h{1pt}r_k^2 + \mathcal{E}^*_{v,\h{0.5pt}\p^+ B_\rho^+} 
			+ M_1 ^{i-1} \h{1pt} W^{\h{0.5pt}q}_{v,\h{0.5pt}\varphi_k,\h{0.5pt}\p^+ B_\rho^+} + M_1 ^{i}\h{1pt}  \sigma^{-2}\h{1pt}W^{\h{0.5pt}q}_{v,\h{0.5pt}\varphi_k,\h{0.5pt}\p^+ B_\rho^+}
			  \h{1pt}\right).
		\end{align*}
		It is sufficient to check $J_s \leq \eps{7}^2$. Define $\mu:=2- \frac{\log M_1}{\log \alpha}>2$. It turns out that
\begin{align}\label{1380}
M_1 ^s 
		= \alpha^{\frac{s\log M_1}{\log \alpha}} 
		\leq (\alpha\h{0.5pt}\epsilon)^{ \frac{\log M_1}{\log \alpha}}
		=M_1 \h{1pt} \epsilon^{ \frac{\log M_1}{\log \alpha}}
		=M_1 \h{1pt}\epsilon^{2-\mu}.
\end{align}
		Using \eqref{1380} and the fact that $I_s\leq 4\h{0.5pt}M_1 \h{0.5pt}  r_k^2\left(\epsilon^{-2} +\epsilon^{2-\mu}\right)  $, we further estimate $J_s$ by 
		\begin{align*}
			J_s\leq\, & 4 \h{0.5pt}M_1^2\h{1pt} \epsilon^{2-\mu} \h{1pt} r_k^4\left( \epsilon^{-2} + \epsilon^{2-\mu}
			\right) +  \epsilon^{-\mu} \h{.5pt}\rho^{-2}\h{1pt} 
			\mathcal{E}^*_{v,\h{0.5pt}\p^+ B_\rho^+}
			\left(\left(\epsilon \rho\right)^{2} M_1  r_k^2 + W^{\h{0.5pt}q}_{v,\h{0.5pt}\varphi_k,\h{0.5pt}\p^+ B_\rho^+}\right)
			\\[2mm]
			&\h{-13pt}+  M_1\h{0.5pt}r_k^2 \h{1pt} \rho^{-2}
			\left( 4\h{0.5pt}\epsilon^{-2-\mu}+\left(4+M_1\right)\epsilon^{2-2\mu} \right) W^{\h{0.5pt}q}_{v,\h{0.5pt}\varphi_k,\h{0.5pt}\p^+ B_\rho^+} +  4\h{0.5pt}M^{2s-1}_1\h{0.5pt} \left(\epsilon\h{0.5pt}\rho\right)^{-4} 
			\left( W^{\h{0.5pt}q}_{v,\h{0.5pt}\varphi_k,\h{0.5pt}\p^+ B_\rho^+}\right) ^2\\[2mm]
            &\h{-13pt}+ r_k^2\left(\h{1pt}r_k^2\left( \left(\epsilon\h{0.5pt}\rho\right)^2
			+ 4\h{0.5pt}M_1 \big(\epsilon^{-2}  
			+\epsilon^{2-\mu} + \epsilon^{4-\mu}\rho^2
			\big) \right) 
			+ \mathcal{E}^*_{v,\h{0.5pt}\p^+ B_\rho^+} 
			+ \big( \epsilon^{2-\mu} + 4\h{0.5pt}M_1 \h{0.5pt} \epsilon^{-\mu} \h{0.5pt}\rho^{-2} \big)\h{1pt} 
			W^{\h{0.5pt}q}_{v,\h{0.5pt}\varphi_k,\h{0.5pt}\p^+ B_\rho^+}
			\right).
		\end{align*}
		By \eqref{assum of case 1} and Assumption (d) in Proposition \ref{lem ext from sphere to ball}, it follows that
		\begin{align*}
			J_s
			\h{1.5pt}\lesssim\h{1.5pt}&_{\Omega\,}\h{1.5pt} \epsilon^{2-\mu} \h{1pt} r_k^4\left( \epsilon^{-2} + \epsilon^{2-\mu}\right) +  \epsilon^{2-\mu}\h{1pt}r_k^2 \h{1.5pt}  \mathcal{E}^*_{v,\h{0.5pt}\p^+ B_\rho^+}   + r_k^2 \h{1pt}\rho^{-2}
			\left(\epsilon^{-2-\mu} + \epsilon^{2-2\mu} \right) W^{\h{0.5pt}q}_{v,\h{0.5pt}\varphi_k,\h{0.5pt}\p^+ B_\rho^+} 
			\\[2mm]
            &\h{25pt} + \delta^2 + r_k^2\left(r_k^2\left(  \epsilon^{4-\mu}\rho^2
			+\epsilon^{-2}  
			+\epsilon^{2-\mu} \right) 
			+ \mathcal{E}^*_{v,\h{0.5pt}\p^+ B_\rho^+} 
			+ \left(\h{1pt} \epsilon^{2-\mu} + \epsilon^{-\mu}\rho^{-2} \right) 
			W^{\h{0.5pt}q}_{v,\h{0.5pt}\varphi_k,\h{0.5pt}\p^+ B_\rho^+}
			\right).
		\end{align*} 
        Simply grouping the terms and using Assumption (b) of this proposition, we obtain
        		\begin{align*}
			J_s
			\lesssim_{\,\Omega\,} \epsilon^{2-\mu} \h{1pt}r_k^4
			\left(   \epsilon^{-2} + \epsilon^{2-\mu} \right)
			+\epsilon^{2-\mu} \h{1pt}r_k^2 \h{1.5pt}  
			\mathcal{E}^*_{v,\h{0.5pt}\p^+ B_\rho^+}  
			+ r_k^2 \h{1pt} \rho^{-2}\left( \epsilon^{-2-\mu} + \epsilon^{2-2\mu}\right)  
			W^{\h{0.5pt}q}_{v,\h{0.5pt}\varphi_k,\h{0.5pt}\p^+ B_\rho^+} +\delta^2
            \h{1.5pt}\lesssim_{\,\Omega\,} \h{1.5pt} \delta^2.
		\end{align*} 

		On the other hand, Lemma \ref{lem ext to annulus} implies that  
        $$v^{i-1}_{\textup{in}}(\cdot) \in U_{\C{17}^{i-1}\sqrt{\delta}\h{2pt}+\h{2pt}\C{17} \h{0.5pt}\frac{\C{17}^{i-1}-1}{\C{17}-1} \h{0.5pt} \epsilon\h{.5pt}\rho\h{.5pt}r_k }\big(\mathbb{S}^2\cap T_{\varphi_k(\cdot)}\p \Omega\big) \h{15pt}\text{on $\p B_{\rho-2\left(i-1\right)\sigma}\cap \big\{y_3=0\big\}$ in the sense of trace.}$$
		Using \eqref{1380} and Assumption (c) of the proposition induces $$\C{17}^{s-1} \sqrt{\delta}+\C{17} \h{1pt} \frac{\C{17}^{s-1}-1}{\C{17}-1}\h{1pt}\epsilon \h{0.8pt}\rho\h{0.8pt}  r_k
        \h{1pt}\leq\h{1pt} \epsilon^{2-\mu} \sqrt{\delta}
        +4\h{0.5pt} \epsilon^{3-\mu}\h{1pt}\rho \h{0.8pt} r_k
        \h{1pt}\leq\h{1pt} \sqrt{\eps{7}}.$$ Therefore, $v^{i-1}_{\textup{in}}(\cdot)\in U_{\sqrt{\eps{7}}} \big(\mathbb{S}^2\cap T_{\varphi_k(\cdot)}\p \Omega\big)$ on $\p B_{\rho-2\left(i-1\right)\sigma}\cap\big\{y_3=0\big\}$ in the sense of trace when $i \in \big\{1,2,\ldots,s\big\}$. Lemma \ref{lem ext to annulus} can then be applied to obtain the extension $\overline{v}^i$. Applying the above arguments $s$ times, we get the desired extension on $A^+_{\rho-2\h{0.5pt}s\h{0.5pt}\sigma,\h{1pt}\rho}$, denoted by $\overline{v} $, such that $$\overline{v} := \overline{v}^i \h{20pt}\text{ on $A_{\sigma}^{i,+}$, for $i=1,2,\ldots,s$.}$$ Next, we extend  $\overline{v}$ homogeneously from $\p^+ B_{\rho-2\h{0.2pt}s\h{0.2pt}\sigma}^+$ to $B_{\rho-2\h{0.2pt}s\h{0.2pt}\sigma}^+=B_{\rho-s\epsilon\rho}^+$. That is $$\overline{v}(z) := v^s_{\textup{in}}\big(\left(\rho-s\epsilon\rho\right)\widehat{z}\h{1.5pt}\big) \h{20pt}\text{for any $ z \in B_{\rho-s\epsilon\rho}^+$.} $$ Thus, we obtain the mapping $\overline{v}$ on $B_\rho^+$ with $\overline{v}\h{1pt}\big|_{\p^+ B_\rho^+} = v$. Moreover, $$\overline{v}(\cdot)\in U_{\delta_0/2}\big(\mathbb{S}^2\cap T_{\varphi_k(\cdot)}\p \Omega\big) \h{20pt}\text{on $\p^0 B_{\rho}^+$ in the sense of trace},$$ provided that $\delta$ and $r_k$ are sufficiently small. \vspace{0.4pc}

		\noindent\textbf{Part 1.2. Estimates of the extension mapping:} Items (b)-(c) of Lemma \ref{lem ext to annulus} induce 
		\begin{align*}
			&\mathcal{E}^*_{\overline{v}^{i}, A^{i,+}_\sigma}
			\lesssim_{\,\Omega\,}  \sigma \h{1pt}  \mathcal{E}^*_{v^{i-1}_\textup{in},\p^+B^+_{\rho-2\left(i-1\right)\sigma}}
			+ \beta^{-1}\sigma^{-1}
			W^{\h{0.5pt}q}_{v^{i-1}_\textup{in},	\varphi_k,\p^+B^+_{\rho-2\left(i-1\right)\sigma}}  
			+ \beta^{-1}\sigma^{-1}r_k^2\h{0.8pt}\rho^2, \\[1.5mm]
			&W^{\h{0.5pt}q}_{\overline{v}^{i},\varphi_k,A^{i,+}_\sigma} 
			\lesssim_{\,\Omega\,} \sigma \h{1pt}
			W^{\h{0.5pt}q}_{v^{i-1}_\textup{in},	\varphi_k,\p^+B^+_{\rho-2(i-1)\sigma}}
			+ \sigma \h{0.8pt} r_k^2 \h{0.8pt}\rho^2, \h{80pt}\text{for $i=1,2,\ldots,s$.}
		\end{align*}
		 Applying \eqref{3649}, we obtain that
		\begin{align*}
			&\mathcal{E}^*_{\overline{v}^{i}, A^{i,+}_\sigma}
			\lesssim_{\,\Omega} \sigma \h{1pt}  \alpha^{i-1}\h{1pt} \mathcal{E}^*_{v,\h{0.5pt}\p^+B^+_\rho}
			+ M_1 ^i \h{1pt}\sigma^{-1}  
			W^{\h{0.5pt}q}_{v ,	\varphi_k,\p^+B^+_\rho} 
			+ \sigma^{-1}r_k^2 \h{1pt}\rho^2 
			+M_1^i \h{0.8pt} \sigma \h{0.8pt}r_k^2,\\[1.5mm]
			&W^{\h{0.5pt}q}_{\overline{v}^{i},\varphi_k,A^{i,+}_\sigma}  
			\lesssim_{\,\Omega} \sigma \h{1pt}
			M_1 ^{i-1} \h{1pt}W^{\h{0.5pt}q}_{v,	\varphi_k,\p^+B^+_\rho} 
			+ M_1^{i}\h{1pt}\sigma^3\h{1pt}r_k^2 
			+\sigma \h{1pt} r_k^2 \h{0.8pt}\rho^2, \h{45pt}\text{for $i=1,2,\ldots,s$.}
		\end{align*}
		As $s\leq  1+ \tfrac{ \log \epsilon}{\log \alpha}  $, we sum over $i=1,2,\ldots,s$ and then employ \eqref{1380} to yield the following estimates:
		\begin{align*}
			\mathcal{E}^*_{\overline{v},A^+_{\rho-2s\sigma,\rho}}
			\lesssim\,&_{\Omega\,}   \h{1.5pt} \sigma \h{1pt}  \mathcal{E}^*_{v,\h{0.5pt}\p^+B^+_\rho}
			+ \sigma ^{-1} M_1^{s} \h{1pt} 
			W^{\h{0.5pt}q}_{v,	\varphi_k,\p^+B^+_\rho} 
			+  s \h{0.8pt} \sigma^{-1}
			r_k^2 \h{0.8pt} \rho^2 
			+M_1^s\h{0.5pt} \sigma \h{0.5pt} r_k^2 \\[1.5mm]
			\lesssim\,&_{\Omega\,} \h{1pt}  \epsilon\h{0.5pt}\rho \h{1pt} \mathcal{E}^*_{v,\h{0.5pt}\p^+B^+_\rho}
			+ \epsilon ^{1-\mu} \rho^{-1}
			W^{\h{0.5pt}q}_{v,	\varphi_k,\p^+B^+_\rho} 
			+ \left( 1+\tfrac{ \log \epsilon}{\log \alpha}\right)r_k^2 \h{0.8pt}\rho \h{0.8pt}  \epsilon^{-1}
			+\epsilon^{3-\mu} \h{0.8pt} r_k^2 \h{0.8pt}\rho,\end{align*}
		and
		\begin{align*}
			W^{\h{0.5pt}q}_{\overline{v},\varphi_k,A^+_{\rho-2s\sigma,\rho}} 
			\lesssim\,&_{\Omega\,} \h{1pt} \sigma
			\h{1pt}M_1 ^{s} \h{1pt}W^{\h{0.5pt}q}_{v,	\varphi_k,\p^+B^+_\rho} 
			+ M_1^{s}\h{1pt}\sigma^3\h{1pt}r_k^2 
			+\sigma \h{0.8pt} r_k^2 \h{0.8pt}\rho^2\left( 1+\tfrac{ \log \epsilon}{\log \alpha}\right) \\[1.5mm]
			\lesssim\,&_{\Omega\,}  \h{1pt}\epsilon^{3-\mu}\rho \h{1pt} W^{\h{0.5pt}q}_{v,	\varphi_k,\p^+B^+_\rho} 
			+ \epsilon^{5-\mu}\rho^3\h{0.8pt}r_k^2 
			+ \epsilon \h{0.8pt} r_k^2 \h{0.8pt}\rho^3\left( 1+\tfrac{ \log \epsilon}{\log \alpha}\right).
		\end{align*}
		
		Next, we establish the estimates on $B^+_{\rho-2s\sigma}$. 
		Note that \eqref{3649} holds for $i = s + 1$. Hence,
		\begin{align*}
	&\mathcal{E}^*_{\overline{v},B_{\rho-2s\sigma}^{+}}
			\lesssim_{\,\Omega} \,\rho\h{1pt}\mathcal{E}^*_{v^{s}_\textup{in},\p^+ B_{\rho-2s\sigma}^+} \lesssim_{\,\Omega}\, \rho\h{1pt}\alpha^{s} \h{1pt} \mathcal{E}^*_{v,\h{0.5pt}\p^+B^+_\rho}
			+\rho \h{0.5pt}M_1 ^{s} \h{0.5pt}\sigma^{-2}\h{1pt} 
			W^{\h{0.5pt}q}_{v,\h{0.5pt}\varphi_k,\h{0.5pt}\p^+B^+_\rho}
			+ M_1 \h{0.5pt}\sigma^{-2}\h{0.5pt} r_k^2\h{1pt}\rho^3 
			+\rho \h{0.5pt} M_1^s \h{0.5pt} r_k^2, \\[2mm]
			 &W^{\h{0.5pt}q}_{\overline{v},	\varphi_k,B_{\rho-2s\sigma}^{+}} 
			\lesssim_{\,\Omega\,}  \rho \h{1pt}W^{\h{0.5pt}q}_{\overline{v},\varphi_k,\p^+ B_{\rho-2s\sigma}^{+}} 
			+\rho^3\h{1pt}r_k^2 \,\lesssim_{\,\Omega\,}  \rho \h{0.5pt} M_1 ^{s} \h{1pt}W^{\h{0.5pt}q}_{v,	\varphi_k,\p^+B^+_\rho} 
			+\rho \h{0.5pt}
			M_1 ^{s}\h{0.5pt}\sigma^2\h{0.5pt}r_k^2	+\rho^3\h{0.8pt}r_k^2.
		\end{align*}
        We can then apply \eqref{1380} to obtain
		\begin{align*}
			&\mathcal{E}^*_{\overline{v},B_{\rho-2s\sigma}^{+}}
			\lesssim_{\,\Omega\,} \h{1pt} \epsilon \h{0.5pt}\rho\h{1pt} \mathcal{E}^*_{v,\h{0.5pt}\p^+B^+_\rho}
			+\rho^{-1}\epsilon^{-\mu}\h{1pt}
			W^{\h{0.5pt}q}_{v,\h{0.5pt}\varphi_k,\h{0.5pt}\p^+B^+_\rho}
			+r_k^2\h{0.8pt}\rho \left(\epsilon^{-2}
			+ \epsilon^{2-\mu}\right),\\[2mm]
			&W^{\h{0.5pt}q}_{\overline{v},	\varphi_k,B_{\rho-2s\sigma}^{+}} 
			\lesssim_{\,\Omega\,}   \rho \h{0.8pt} \epsilon^{2-\mu}\h{1pt}W^{\h{0.5pt}q}_{v,	\varphi_k,\p^+B^+_\rho} 
			+\rho^3 \h{0.5pt} r_k^2 \left(1+\epsilon^{4-\mu} \right). 
		\end{align*}We conclude the desired estimates in the proposition by adding the corresponding terms above. \vspace{0.4pc}
		
		\noindent\textbf{Case 2.} We assume that \eqref{assum of case 1} fails and define $\overline{v}(z) := v(\rho \h{0.5pt} \widehat{z}\h{1pt})$ for any $z \in B_\rho^+ \setminus \{0\}$. The assumption $$ \mathcal{E}^*_{v,\p^+B_{\rho}^{+}}\h{1pt} 
		W^{\h{0.5pt}q}_{v,\h{0.5pt}\varphi_k,\h{0.5pt}\p^+ B_\rho^+} \h{1pt}\leq\h{1pt} \delta^2\rho^2\epsilon^\mu
		\leq  \epsilon^{\mu-4} \h{0.5pt} \rho^{-2} \h{0.5pt} M_1 ^{2s-1} \left( W^{\h{0.5pt}q}_{v,\h{0.5pt}\varphi_k,\h{0.5pt}\p^+ B_\rho^+}\right) ^2 $$ and \eqref{1380} imply that 
        $$ \mathcal{E}^*_{v,\p^+B_{\rho}^{+}}  \leq \epsilon^{\mu-4}\rho^{-2} M_1 ^{2s-1}    W^{\h{0.5pt}q}_{v,\h{0.5pt}\varphi_k,\h{0.5pt}\p^+ B_\rho^+}
        \h{1pt}\leq\h{1pt}  M_1 \h{.5pt} \epsilon^{-\mu} \rho^{-2} \h{1pt}   W^{\h{0.5pt}q}_{v,\h{0.5pt}\varphi_k,\h{0.5pt}\p^+ B_\rho^+}.$$ Thus, $$
		\mathcal{E}^*_{\overline{v},B_{\rho}^{+}}
		\lesssim_{\,\Omega} \rho\h{1pt}\mathcal{E}^*_{v,\p^+ B_{\rho}^{+}}
		\lesssim_{\,\Omega} \rho^{-1}\epsilon^{-\mu} \h{1pt} W^{\h{0.5pt}q}_{v,\h{0.5pt}\varphi_k,\h{0.5pt}\p^+ B_\rho^+}.$$
        It also holds that
        $$
		W^{\h{0.5pt}q}_{\overline{v},	\varphi_k,B_{\rho }^{+}} 
		\lesssim_{\,\Omega\,} \rho \h{1pt}W^{\h{0.5pt}q}_{v,	\varphi_k,\p^+ B_{\rho }^{+}} 
		+\rho^3\h{0.5pt}r_k^2.$$
        The desired estimates in the proposition follow, choosing any $\mu>2$. Moreover, $$\overline{v}(\cdot)\in U_{\delta_0/2}\big(\mathbb{S}^2\cap T_{\varphi_k(\cdot)}\p \Omega\big) \h{15pt} \text{on $\p^0 B_{\rho}^+$, in the sense of trace,}$$ provided that $\delta>0$ and $r_k$ are sufficiently small. The proof is complete.
	\end{proof}

	\section{Bubbling analysis at boundary singularities}\label{bubbling analysis}

	We now establish the partial regularity and bubbling analysis of the minimizing harmonic map $u \in H_T^1(\Omega; \mathbb S^2)$. The interior partial regularity follows from standard techniques for minimizing harmonic maps with fixed Dirichlet boundary conditions \cite{SU82, HL87}. However, addressing boundary regularity for our problem requires modifying arguments from Dirichlet or free boundary problems. Unlike classical free boundary problems \cite{HL89free, DS89optimal}, where $u(x)\in S$ on $\p \Omega$ for a fixed submanifold $S\subseteq\mathbb{S}^2$, our setting involves varying free boundary conditions:  $u(x)\in T_x\p \Omega$ which depend on each location $x \in \p \Omega$. This dependence introduces a technical issue: constructing suitable comparison mappings that preserve the mixed boundary conditions, incorporating both the tangential and Dirichlet ones.

	Our main result demonstrates that the singular set is discrete and finite. In Section \ref{cpt of resc}, we establish the strong $H^1$-convergence of the rescaled map of $u$ to the tangent map. Our approach is partly inspired by the overall strategy of \cite{SU82}, but non-trivial modifications are required to accommodate the tangential boundary conditions. Roughly speaking, as in \cite{SU82}, it suffices to show uniform convergence outside the singular set, since the energy is small near the singular set with zero one-dimensional Hausdorff measure. The uniform convergence can be established by using the partial regularity in Lemma \ref{lem Partial Regularity}. Next, we construct a suitable energy comparison mapping to show that the normalized energy is small. In settings with fixed free boundary conditions  studied in \cite{DS89, HL89free}, the comparison mapping is typically constructed by extending values from the boundary constantly along the flat boundary. However, this extension does not preserve the tangential conditions in our setting. To address this, we carefully construct an alternative correction mapping that maintains the tangential conditions. The tangent map at the boundary singularity is studied in Section \ref{sin str}. It is shown to be a half bubble with a hedgehog or an anti-hedgehog structure, up to a planar rotation.

	\subsection{Compactness of rescaled maps}\label{cpt of resc}
	In this section, we construct a correction map to adjust the boundary values of the extension obtained in Proposition \ref{lem ext from sphere to ball}, ensuring that the tangential conditions are satisfied. We then compare the resulting map with the rescaled map to demonstrate that the normalized energy is indeed small. Applying Lemma \ref{lem Partial Regularity}, we establish uniform convergence outside the singular set, leading to the strong $H^1$-convergence.

	\begin{lem} 
		\label{lem compact of rescaled map}
		Suppose $\widetilde{a}\in \p^0 B^+_{R_0/5}$, $\big\{r_k\big\}_{k \h{0.8pt}\in\h{0.8pt}\mathbb{N}}
        \subseteq(0,1)$ is a sequence satisfying $\displaystyle\lim_{k \to \infty}r_k=0$, and $u$ is a Dirichlet energy minimizer in the space $H_T^1(\Omega; \mathbb S^2)$ with $\widetilde{u}$ defined in \eqref{def transformed u}. \vspace{0.2pc}
        
        If on $B^+_{2}$, we define $\widetilde{u}_{\widetilde{a},r_k}(\cdot):=\widetilde{u}\left(\widetilde{a}+r_k \h{0.5pt}\cdot\right)$, then there is a $\widetilde{u}^{\h{.8pt}*} \in H^1(B_{2}^+; \mathbb{S}^2)$ such that \begin{align}\label{wk con}\widetilde{u}_{\widetilde{a},r_k} \to \widetilde{u}^{\h{.8pt}*}\quad
        \text{ weakly in $H^1\big(B_{2}^+; \mathbb{S}^2\big)$ as $k\to \infty$, up to a subsequence.}\end{align}
        Fixing the convergent subsequence obtained above, we define the concentration set by
		\begin{align*}
			\mathscr{S}(\overline{B^+_{1}})
			:=\left\{ y \in  \overline{B^+_{1}} : \lim_{\rho \h{.5pt}\to\h{.5pt} 0^+} \liminf_{k \h{0.5pt}\to\h{0.5pt} \infty} \h{1pt} 
			\dfrac{1}{\rho} \int_{\h{0.5pt}\overline{B_1^+} \h{1pt}\cap\h{1pt} B_\rho(y)} \big|\h{1pt}\nabla \widetilde{u}_{\widetilde{a},r_k}\h{1pt}\big|^2  >0 \right\}.
		\end{align*}Then, we have
        \begin{align*}
				&\mathrm{(1).}\h{5pt}\textup{$\mathscr{S}(\overline{B^+_{1}})$ is closed and  $\mathscr{H}^1\pigl( \mathscr{S}(\overline{B^+_{1}}) \pigr) =0$;}\\[1mm]
				&\mathrm{(2).}\h{5pt}\textup{$\widetilde{u}_{\widetilde{a},r_k} \to \widetilde{u}^{\h{.8pt}*}$ uniformly on any compact subset of $\overline{B^+_{1}} \setminus \mathscr{S}(\overline{B^+_{1}})$, as $k \to \infty$;}\\[1.5mm]
				&\mathrm{(3).}\h{5pt}\textup{$\widetilde{u}_{\widetilde{a},r_k} \to \widetilde{u}^{\h{.8pt}*}$ strongly in $H^1\big(B_{1}^+; \mathbb{S}^2\big)$, as $k \to \infty$.}
			\end{align*} 
	\end{lem}
	\begin{proof}
		Defining \begin{align}\label{defn of sacal}\text{$a^k_{ij}\left(\cdot\right):=a_{ij}\left(\widetilde{a}+r_k \h{0.5pt}\cdot\right)$ \h{15pt}and \h{15pt} $\widetilde{\mathcal{E}}^{\h{0.7pt}k}_{v,B_{r}^+(z)}:=\int_{B_{r}^+(z)}a^k_{ij}\h{1pt}\p_{y_i}v\cdot\p_{y_j}v$,\h{15pt}for any $B_{r}^+(z) \subseteq B_{2}^+$},\end{align} we compute that $$
		\frac{1}{r}\int_{B_{r}^+}\big|\nabla \widetilde{u}_{\widetilde{a},r_k}\big|^2 
		\lesssim_{\,\Omega\,}
		\frac{1}{r}\h{1.5pt}\widetilde{\mathcal{E}}_{\widetilde{u}_{\widetilde{a},r_k},B_{r}^+}^k
		=\frac{1}{rr_k} \h{1.5pt}
		\widetilde{\mathcal{E}}_{\widetilde{u},B_{rr_k}^+(\widetilde{a})}
		\lesssim_{\,\Omega\,} \frac{1}{rr_k}\int_{B_{rr_k}^+(\widetilde{a})}\big|\h{0.5pt}\nabla \widetilde{u}\h{0.5pt}\big|^2, \h{15pt}\text{for any $r\in(\h{0.5pt}0,2\h{1pt}]$.}$$ 
		By the energy monotonicity in Lemma \ref{lem. energy mono.}, it turns out that		\begin{align}
			\sup_{k\h{0.8pt}\in\h{0.8pt}\mathbb{N}}\h{1.5pt}\dfrac{1}{r}\int_{B_{r}^+}\big|\nabla \widetilde{u}_{\widetilde{a},r_k}\big|^2 
			\lesssim_{\,\Omega\,} 
			\R{1} + \dfrac{1}{\R{1}} \int_{B_{\R{1}}^+(\widetilde{a})}\big|\nabla \widetilde{u}\big|^2 
			\h{1pt}\lesssim_{\,\Omega\,}\h{1pt} 1, \h{15pt}\text{
		for any $r\in(\h{.5pt}0,2\h{1pt}]$. }
			\label{ineq. uniform energy bdd of rescaled map}
		\end{align}Thus, there is a limiting map $\widetilde{u}^{\h{.8pt}*} \in H^1\big(B_{2}^+; \mathbb{S}^2\big)$ such that \eqref{wk con} holds.\vspace{0.2pc}

        In the remainder of the proof, we prove Items (1)-(3) in the lemma. \vspace{0.4pc}
		
		\noindent{\bf Proof of Item (1):} The closedness of $\mathscr{S}(\overline{B^+_{1}})$ can be proved by the boundary partial regularity Lemma \ref{lem Partial Regularity} in the Appendix, and a similar partial regularity result in \cite{LW08} for interior points. In addition, we can find an $\varepsilon_0^*>0$ sufficiently small such that
		\begin{align*}
			\mathscr{S}(\overline{B^+_{1}})
			=\left\{ y \in  \overline{B^+_{1}} : \lim_{\rho \h{1pt}\to \h{1pt}0^+} \liminf_{k \h{1pt}\to\h{1pt} \infty} \h{1.5pt} 
			\dfrac{1}{\rho} \int_{\h{0.5pt}\overline{B_1^+} \cap \h{1pt}B_\rho^+(y)} \big|\nabla \widetilde{u}_{\widetilde{a},r_k}\big|^2  >4\varepsilon^*_0 \h{1pt}\right\}.
		\end{align*}		
		Let $\big\{B_{5\delta}(x^i)\h{0.5pt}\big\}_{i=1}^{N}$ be a finite cover of $\mathscr{S}(\overline{B^+_{1}})$, where $\delta\in(0,R_1)$ and $\big\{x^i\big\}_{i=1}^{N} \subseteq \mathscr{S}(\overline{B^+_{1}})$. Using the Vitali covering lemma, we can assume that $\big\{B_{\delta}(x^i)\big\}_{i=1}^{N}$ is a disjoint family. 
		In addition, there is a radius $\rho_0>0$ depending on $\big\{x^i\big\}_{i=1}^{N}$ and $\varepsilon_0^*$ such that
		$$ \liminf_{k \h{1pt}\to \h{1pt} \infty} \h{1pt} 
		\frac{1}{\rho} \int_{B_\rho^+(x^i)} \big|\nabla \widetilde{u}_{\widetilde{a},r_k}\big|^2
		\geq 2\varepsilon_0^*, \h{15pt}\text{for any $i=1,2,\ldots,N$ and $\rho \in \big(\h{0.5pt}0,\rho_0\h{0.5pt}\big]$. }$$
		Therefore, for some $k_0 \in \mathbb{N}$ depending on $\varepsilon_0^*$ and $\rho$, it holds
		$$\dfrac{1}{\rho}\int_{B^+_{\rho}(x^i)}\big|\nabla \widetilde{u}_{\widetilde{a},r_{k_0}}\big|^2\geq \varepsilon_0^*, \h{15pt}\text{for any $i=1,2,\ldots,N$.}$$
		By scaling and the energy inequality in Lemma \ref{lem int energy mono}, it turns out, for any $\rho^* \in \big[\h{0.5pt}\rho,R_1\h{0.5pt}\big)$,  that
		\begin{align*}
			&\dfrac{1}{\rho}\int_{B^+_{\rho}(x^i)}
			\big|\nabla \widetilde{u}_{\widetilde{a},r_{k_0}}\big|^2
			=\dfrac{1}{r_{k_0}\h{0.5pt}\rho}\int_{B^+_{r_{k_0}\h{0.5pt}\rho}\left(\h{0.5pt}\widetilde{a}+r_{k_0}x^i\h{0.2pt}\right)}\big|\nabla \widetilde{u} \big|^2 \\[1.5mm]
			&\h{40pt}\lesssim_{\,\Omega\,} r_{k_0}\h{0.5pt}\rho^* + \dfrac{1}{r_{k_0}\h{0.5pt}\rho^*}\int_{B^+_{r_{k_0}\rho^*}\left(\widetilde{a}+r_{k_0}x^i\right)} \big|\nabla \widetilde{u} \big|^2
			 = r_{k_0} \h{0.5pt}\rho^* + \dfrac{1}{\rho^*}
			\int_{B^+_{\rho^*}(x^i)}
			\big|\nabla \widetilde{u}_{\widetilde{a},r_{k_0}} \big|^2.
		\end{align*}
		We sum  the above estimate from $i = 1$ to $N$, and set $\rho^*=\delta \in \left(\rho, R_1\right)$. As $\big\{B_{\delta}(x^i)\big\}_{i=1}^{N}$ is disjoint, we obtain
		$$N\varepsilon_0^* \h{1pt}\leq\h{1pt}\dfrac{1}{\rho}\int_{\bigcup_i B^+_{\rho}(x^i)} \big|\nabla \widetilde{u}_{\widetilde{a},r_{k_0}}\big|^2
		\h{1pt}\lesssim_{\,\Omega\,} N\delta +  \dfrac{1}{\delta}
		\int_{\bigcup_i B^+_{\delta}(x^i)}
		\big|\nabla \widetilde{u}_{\widetilde{a},r_{k_0}}\big|^2.
		$$
		Let $\delta = \sigma \h{0.5pt}\varepsilon_0^*$, where $\sigma$ is small with the smallness depending only on the domain $\Omega$. It then follows from the above estimate that
		\begin{align*}
			N\varepsilon_0^*
			\h{1pt}\lesssim_{\,\Omega\,}  \dfrac{1}{\delta}
			\int_{\bigcup_i B^+_{\delta}(x^i)}
			\big|\nabla \widetilde{u}_{\widetilde{a},r_{k_0}}\big|^2
			\h{1pt}\leq \h{1pt}		\dfrac{1}{\delta}	\int_{B^+_1}
			\big|\nabla \widetilde{u}_{\widetilde{a},r_{k_0}}\big|^2, \h{10pt}\text{which implies $\displaystyle\lim_{\delta \h{0.5pt}\to \h{0.5pt} 0^+}\mathscr{H}^3\left(\h{1pt}\bigcup_i \h{1pt} B^+_{5\delta}(x^i)\right)=0$.}
		\end{align*}
		Here we use $\nabla \widetilde{u}_{\widetilde{a},r_{k_0}} \in L^2\big(B^+_1\big)$. Hence, the first inequality above deduces $\mathscr{H}^1\pigl( \mathscr{S}(\overline{B^+_{1}}) \pigr) =0$. \vspace{0.4pc}
		
		\noindent{\bf Proof of Item (2):} Fixing $y^* \in \p^0B_1^+ \setminus \mathscr{S}(\overline{B^+_{1}})$, we prove the uniform convergence of $\big\{\widetilde{u}_{\widetilde{a},r_k} \big\}$ in a neighborhood of $y^*$. The proof is similar when $y^*$ is an interior point. \vspace{0.2pc}
        
        For any  $\varepsilon>0$, there is $\rho_0=\rho_0\left(y^*,\varepsilon\right) \in (0,\frac{1}{2})$ such that
		\begin{align}
			\dfrac{C}{\rho^3} \int_{B_\rho^+(y^*)} \big|\h{1pt}\widetilde{u}^{\h{.8pt}*}-M_{B_\rho^+(y^*)}(\widetilde{u}^{\h{.8pt}*}) \h{1pt}\big|^2 
			\leq\dfrac{1}{\rho} \int_{B_\rho^+(y^*)} \big|\h{1pt}\nabla \widetilde{u}^{\h{.8pt}*} \h{1pt}\big|^2  
			< \varepsilon \h{15pt}\text{for any $\rho \in (\h{0.5pt}0,\rho_0\h{0.5pt}]$.}
			\label{4170}
		\end{align}
		Here, $C>0$ is a universal constant induced by the Poincar\'e inequality. Recalling $\eps{5} $ and $\R{7} $ given in Lemma \ref{lem Partial Regularity}, we show in the following that $$\widehat{\mathcal{E}}^*_{\rho_0/2,\h{1.5pt}y^*;+}(\widetilde{u}_{\widetilde{a},r_k}) =\widehat{\mathcal{E}}^*_{r_k\h{0.2pt}\rho_0\h{0.2pt}/\h{0.2pt}2,\h{1.5pt}\widetilde{a}+r_ky^*;+}(\widetilde{u}) < \eps{5},$$ for any large enough $k\in \mathbb{N}$ with  $\widetilde{a}+r_k\h{0.5pt}y^*\in \p^0B^+_{R_0/4}$ and $r_k\h{0.5pt}\rho_0\in (0,\R{7}]$. The extension map constructed in Proposition \ref{lem ext from sphere to ball} will be used as a comparison map in the following argument. \vspace{0.4pc}
		
		\noindent{\bf Part 2.1. Verification of assumptions in Proposition \ref{lem ext from sphere to ball}:} By the compactness of the  Sobolev embedding and the trace operator, the weak convergence of $\big\{\widetilde{u}_{\widetilde{a},r_k}\big\}$ in \eqref{wk con} implies that $\widetilde{u}_{\widetilde{a},r_k} \to \widetilde{u}^{\h{.8pt}*}$ strongly in $L^2\big(B_{1}^+; \mathbb{S}^2\big)$ and $L^2\big(\p B_{1}^+; \mathbb{S}^2\big)$, as $k \to \infty$. Denoting by $q^k$ the average of $\widetilde{u}_{\widetilde{a},r_k}$ over $B_{\rho_0}^+(y^*)$, that is $$ q^k:=M_{B_{\rho_0}^+(y^*)}(\widetilde{u}_{\widetilde{a},r_k}),$$ we see that
		\begin{align}
			&W^{\h{.5pt}q^k}_{\widetilde{u}_{\widetilde{a},r_k},\,B_{\rho_0}^+(y^*)} \nonumber\\[1mm] &\h{10pt}\lesssim \rho_0^3 \left|\h{1pt} M_{B_{\rho_0}^+(y^*)}(\widetilde{u}^{\h{.8pt}*})
			-M_{B_{\rho_0}^+(y^*)}(\widetilde{u}_{\widetilde{a},r_k}) \h{1pt}
			\right| ^2  + \int_{B_{\rho_0}^+(y^*)} \left|\h{1pt} \widetilde{u}_{\widetilde{a},r_k}-\widetilde{u}^{\h{.8pt}*}\h{.5pt}\right| ^2
			+  \left|\h{1pt} \widetilde{u}^{\h{.8pt}*}-M_{B_{\rho_0}^+(y^*)}(\widetilde{u}^{\h{.8pt}*}) \h{1pt}\right| ^2 
			 \nonumber\\[1mm]
			&\h{10pt}\lesssim \int_{B_{\rho_0}^+(y^*)} \left|\h{1pt} \widetilde{u}_{\widetilde{a},r_k}-\widetilde{u}^{\h{.8pt}*} \h{.5pt}\right| ^2 
			+ \left|\h{1pt} \widetilde{u}^{\h{.8pt}*}-M_{B_{\rho_0}^+(y^*)}(\widetilde{u}^{\h{.8pt}*}) \h{1pt}\right| ^2  \lesssim \h{1pt}\varepsilon \h{0.5pt} \rho_0^3,
			\label{3713}
		\end{align}
		for any large enough $k\in \mathbb{N}$ depending on $\rho_0$ and $\varepsilon$. For the estimate of $	W_{\widetilde{u}_{\widetilde{a},r_k},\h{0.5pt}\varphi_k,\h{0.5pt}B_{\rho_0}^+(y^*)}$, we have
		\begin{align}
		W_{\widetilde{u}_{\widetilde{a},r_k},\h{0.5pt}\varphi_k,\h{0.5pt}B_{\rho_0}^+(y^*)} 
			&\lesssim_{\,\Omega\,} \int_{B_{\rho_0}^+(y^*)} 
			\left[ \textup{dist}_{\varphi^{-1}(\widetilde{a})}(\widetilde{u}^{\h{.8pt}*}(y)) \right] ^2
			+r^2_k+ \big|\h{1pt}\widetilde{u}_{\widetilde{a},r_k}(y)-\widetilde{u}^{\h{.8pt}*}(y) \h{1pt}\big|^2\h{1.5pt}\mathrm d\h{.5pt}y. 
			\label{4256}
		\end{align}
        By the definition of $\varphi_k$ in \eqref{defn of varphi_k} and the triangle inequality, it turns out $$\textup{dist}_{\varphi^{-1}(\widetilde{a})}\big(\widetilde{u}^{\h{.8pt}*}(y) \big)
		\lesssim_{\,\Omega\,} \textup{dist}_{\varphi_k(y)}\big(\widetilde{u}_{\widetilde{a},r_k}(y) \big)+r_k
		+ \big|\h{1pt}\widetilde{u}_{\widetilde{a},r_k}(y)-\widetilde{u}^{\h{.8pt}*}(y) \h{1pt}\big| \h{15pt}\text{on $\p^0B^+_{2}$. }$$ Since on $\p^0B^+_{2}$, it satisfies $\textup{dist}_{\varphi_k(\cdot)}\big(\widetilde{u}_{\widetilde{a},r_k}(\cdot) \big)=0$, and $\big|\h{1pt}\widetilde{u}_{\widetilde{a},r_k}(\cdot)-\widetilde{u}^{\h{.8pt}*}(\cdot) \h{1pt}\big|\to 0$ as $k \to \infty$, $\mathscr H^2$-almost everywhere, from the above estimate we induce that \begin{align}\label{tangential cond of widetldu}\textup{dist}_{\varphi^{-1}(\widetilde{a})}\big(\widetilde{u}^{\h{.8pt}*}(\cdot) \big)=0 \h{20pt}\text{$\mathscr H^2$-a.e. on $\p^0B^+_{2}$.}\end{align} Therefore, using the Poincar\'e type inequality, \eqref{4170}, and \eqref{3713}, we get
		\begin{align*}
			W_{\widetilde{u}_{\widetilde{a},r_k},\varphi_k,B_{\rho_0}^+(y^*)}
			\h{1pt}\lesssim_{\,\Omega\,} \rho_0^3 \h{1pt}r_k^2 + \rho_0^2\int_{B_{\rho_0}^+(y^*)} 
			\big|\h{.5pt} \nabla \widetilde{u}^{\h{.8pt}*}\big|^2
			+\int_{B_{\rho_0}^+(y^*)} \big|\h{1pt}\widetilde{u}_{\widetilde{a},r_k}-\widetilde{u}^{\h{.8pt}*} \h{1pt}\big|^2
			\h{1pt}\lesssim_{\,\Omega\,} \rho_0^3\left(\varepsilon+r_k^2\right),
		\end{align*}
		for any large enough $k\in \mathbb{N}$. \vspace{0.2pc}
        
        For $k\in\mathbb{N}$, we use \eqref{ineq. uniform energy bdd of rescaled map}, \eqref{3713}, and Fubini's theorem to find a $\rho_k\in [ \frac{3\rho_0}{4},\rho_0]$ such that
		\begin{align}
			\mathcal{E}^*_{\widetilde{u}_{\widetilde{a},r_k},\h{.5pt}\p^+B^+_{\rho_k}(y^*)}
			\h{0.5pt}\lesssim_{\,\Omega} 1,\h{20pt}
			W^{\h{.5pt}q^k}_{\widetilde{u}_{\widetilde{a},r_k},\h{.5pt}\p^+B^+_{\rho_k}(y^*)}
			\h{0.5pt}\lesssim_{\,\Omega} \varepsilon \h{1pt} \rho_k^2, \h{20pt}W_{\widetilde{u}_{\widetilde{a},r_k},\h{.5pt}\varphi_k,\h{.5pt}\p^+B^+_{\rho_k}(y^*)}
			\h{1pt}\lesssim_{\,\Omega} \rho_k^2\left(\varepsilon 
			+r_k^2 \right).
			\label{3850}
		\end{align}
		Meanwhile, $\rho_k$ can be chosen such that $\widetilde{u}_{\widetilde{a},r_k}\in H^1\big(\p^+
		B_{\rho_k}^+(y^*);\mathbb{S}^2\big)$, and \begin{align}\label{tan at equator}\widetilde{u}_{\widetilde{a},r_k}(\cdot) \in  T_{\varphi_k(\cdot)}\p\Omega \h{20pt}\text{ on $\p B_{\rho_k}(y^*)\cap \big\{y_3=0\big\}$, in the sense of trace.}\end{align}
        
        Recall $\epsilon_0$, $\delta'$, $\eps{7}$, and $\mu$ in Proposition \ref{lem ext from sphere to ball}. In addition, we fix $\epsilon \in (0, \epsilon_0)$ and choose $\delta \in (0, \delta')$ such that \begin{align}\label{eps and mu rela}
        2\h{0.5pt}\epsilon^{2-\mu}\sqrt{\delta} 
        < \sqrt{\eps{7}}.
        \end{align} 
        Note that $r_k \to 0$ as $k \to \infty$. By \eqref{3850}, we can take $k$ large enough and get \begin{align*}
            r_k^2\h{1.5pt}\epsilon^{2-\mu}\left(r_k^2\left( \epsilon^{2-\mu} +\epsilon^{-2}\right)
			+ \mathcal{E}^*_{\widetilde{u}_{\widetilde{a},r_k},\h{.5pt}\p^+B^+_{\rho_k}(y^*)}  
            + \rho_k^{-2}\h{1.5pt}\epsilon^{-2}   \left(
			\epsilon^{2-\mu}  
			+  \epsilon ^{-2} \right)
			W^{\h{.5pt}q^k}_{\widetilde{u}_{\widetilde{a},r_k},\h{.5pt}\varphi_k,\h{.5pt}\p^+B^+_{\rho_k}(y^*)}\right) \h{1.5pt}\leq\h{1.5pt} \delta^2.
        \end{align*}In light of \eqref{eps and mu rela}, it turns out 
		\begin{align*}
            \epsilon^{2-\mu}\sqrt{\delta}
        +4 \h{0.5pt}\epsilon^{3-\mu}\h{1pt}r_k \h{1pt}\rho_k
				\leq \sqrt{\eps{7}}, \h{20pt}\text{for large $k$.}
		\end{align*}Still using the estimates in \eqref{3850}, we can take $\varepsilon$ small enough such that \begin{align*}\mathcal{E}^*_{\widetilde{u}_{\widetilde{a},r_k},\h{.5pt}\p^+B^+_{\rho_k}(y^*)} \h{1.5pt}
			W^{\h{.5pt}q^k}_{\widetilde{u}_{\widetilde{a},r_k},\h{.5pt}\varphi_k ,\h{.5pt}\p^+B^+_{\rho_k}(y^*)}\leq \delta^2\h{.5pt}\rho_k^2\h{1.5pt}\epsilon^\mu.\end{align*}
		Therefore, the last three estimates and Proposition \ref{lem ext from sphere to ball} induce that there exists an extension, denoted by $\overline{u}^k\in H^1\big(B^+_{\rho_k}(y^*);\mathbb{S}^2\big)$, such that $\overline{u}^k=\widetilde{u}_{\widetilde{a},r_k}$ on $\p^+B^+_{\rho_k}(y^*)$ in the sense of trace. However, Proposition \ref{lem ext from sphere to ball} only implies that \begin{align}\overline{u}^k(y) \in U_{\delta_0/2}\big(\h{1pt}\mathbb{S}^2 \cap T_{\varphi^{-1}\left(\widetilde{a}+r_k\h{0.3pt}y\right)}\p \Omega \h{1pt}\big),\h{20pt}\text{for $y \in \p^0 B^+_{\rho_k}(y^*)$ in the sense of trace.}\end{align}  We must construct a new map based on $\overline{u}^k$ with the same trace on $\p^+ B_{\rho_k}^{+}(y^*)$, but taking values in $T_{\varphi^{-1}\left(\widetilde{a}+r_k\h{0.3pt}y\right)}\p \Omega$ for $ y \in \p^0B^+_{\rho_k}(y^*)$, in the sense of trace. \vspace{0.4pc}
		
		\noindent{\bf Part 2.2. Construction of comparison map:} Let $\psi_* \in C^\infty\big(\h{0.5pt}\mathbb{S}^2;[0,1]\h{0.5pt}\big)$ be such that
\[
\psi_*(y)=
\begin{cases}
0, & \text{if} \h{5pt}y \in \mathbb{S}^2 \setminus U_{3\delta_0}\! \h{0.5pt}\big(\h{0.5pt}\mathbb{S}^2 \cap T_{\varphi^{-1}(\widetilde{a})}\partial\Omega\h{0.5pt}\big),\\[6pt]
1, & \text{if} \h{5pt} y \in \mathbb{S}^2 \cap U_{2\delta_0}\!\h{0.5pt}\big(\h{0.5pt}\mathbb{S}^2 \cap T_{\varphi^{-1}(\widetilde{a})}\partial\Omega\h{0.5pt}\big).
\end{cases}
\]
Recalling $x_0$ given at the beginning of Section \ref{bdry flatten}, we have $$ \mathbb{S}^2 \cap U_{3\delta_0}
		\big(\mathbb{S}^2 \cap T_{\varphi^{-1}(\widetilde{a})}\p\Omega\big) \subseteq U_{4\delta_0}\big(\mathbb{S}^2 \cap T_{x_0}\p\Omega\big),$$ which ensures, by Lemma \ref{lem. regularity of proj. and dist}, that $\Pi(y^0,p)$ is well-defined for any $$(y,p)\in\overline{ B^+_{r_k \rho_k}\big(\widetilde{a} + r_k y^*\big)} \times U_{4\delta_0}\big(\mathbb{S}^2 \cap T_{x_0}\p\Omega\big).$$  Denote 
		\begin{align}
        &\zeta^y:=y^* + \rho_k \h{0.5pt}\dfrac{y-y^*}{|\h{0.5pt} y-y^* |};      \h{37pt}\Pi^k(y,p):=\Pi\left(\widetilde{a}+r_k\h{0.5pt}y^0,p\right);\nonumber\\[1mm]
			&\Phi^{v,k}(y,z):= v(z) + \psi_*\pigl(v(z)\pigr) \Big[ \h{1pt}\Pi^k\pigl(y^0, v(z)\pigr)   -   v(z) \h{1pt}\Big];\nonumber\\[1mm] 
            &\Psi^{v,k}(y,z):=				v(z)+ \psi_*\pigl(v(z)\pigr)
			\Big[ \h{1pt}\Pi^k\Big( y^0, \Pi^k\left( z^0, v(z)\right) 
			\Big)  - \Pi^k\left( z^0, v(z)\right) \Big]. \label{def of Phi^(v,k)}
		\end{align}
		Moreover, we simply write $\Phi^{\overline{u}^k}=\Phi^{\overline{u}^k,k}$, $\Psi^{\overline{u}^k}=\Psi^{\overline{u}^k,k}$, and define for a given $\gamma\in (0,1)$ the map   
		\begin{equation}
			u^{\gamma,k}(y) := 
			\begin{cases} 
				\dfrac{\rho_k - |y-y^*|}{(1-\gamma)\rho_k} \h{1pt}
				\Phi^{\overline{u}^k}(y,\zeta^y) 
				+ \dfrac{|y-y^*| - \gamma\rho_k}{(1-\gamma)\rho_k} \h{1pt}
				\Psi^{\overline{u}^k}(y,\zeta^y)
				& \text{on } B^+_{\rho_k}(y^*) \setminus B^+_{\gamma\h{0.2pt}\rho_k}(y^*), \\[15pt]
				\Phi^{\overline{u}^k}\left( y,y^*+ \gamma^{-1}(y-y^*)\right) 
				& \text{on } B^+_{\gamma\rho_k}(y^*).
			\end{cases}
			\label{def. of u^rho 2}
		\end{equation}  
		        
        We claim that $\widehat{u^{\gamma,k}}$ is well-defined on $B^+_{\rho_k}(y^*)$ when $k$ is taken large enough. First, we assume $y \in B_{\gamma\h{0.2pt}\rho_k}^+(y^*)$ and let $z = y^* + \gamma^{-1}(y - y^*)$. If $\psi_*\pigl(\overline{u}^k(z)\pigr) = 0$, then $\big| u^{\gamma, k} (y) \h{1pt}\big| = 1$. If $\psi_*\pigl(\overline{u}^k(z)\pigr)>0$, then for large $k$, it turns out that \begin{align}\label{incl of bar uk}\overline{u}^k(z)\in \mathbb{S}^2 \cap U_{3\delta_0}\big(\mathbb{S}^2\cap T_{\varphi^{-1}(\widetilde{a})}\p\Omega\big)\subseteq  U_{4\delta_0}\big(\mathbb{S}^2\cap T_{\varphi^{-1}\left(\widetilde{a}+r_k\h{0.2pt}y^0\right)}\p\Omega\big). \end{align} Hence $\big|u^{\gamma,k}(y)\h{1pt}\big|\geq 1-4\delta_0> \frac{3}{5}$. Assume $y \in B^+_{\rho_k}(y^*) \setminus B^+_{\gamma\h{0.5pt}\rho_k}(y^*)$ and let $z = \zeta^{y}$. If $\psi_*\pigl(\overline{u}^k(z)\pigr) = 0$, we still have $\big| u^{\gamma, k} (y) \h{1pt}\big| = 1$. If $\psi_*\pigl(\overline{u}^k(z)\pigr) > 0$, then \eqref{incl of bar uk} is still satisfied for large $k$. In addition, \begin{align}\label{set inclusion}\mathbb{S}^2\cap T_{\varphi^{-1}(\widetilde{a}+r_kz^0)}\p\Omega \subseteq U_{\delta_0/4}\big(\mathbb{S}^2\cap T_{\varphi^{-1}(\widetilde{a})}\p\Omega\big)\subseteq
		U_{\delta_0/2}\big(\mathbb{S}^2\cap T_{\varphi^{-1}(\widetilde{a}+r_k\h{0.2pt}y^0)}\p\Omega\big).\end{align} It also holds $\big|u^{\gamma,k}\big|\geq 1-4\delta_0>\frac{3}{5}$ on $B^+_{\rho_k}(y^*) \setminus B^+_{\gamma\h{0.5pt}\rho_k}(y^*)$.  Therefore, the map $\widehat{u^{\gamma,k}} \in H^1(B^+_{\rho_k}(y^*);\mathbb{S}^2)$ is well-defined for large $k$. \vspace{0.2pc}
		
		\sloppy In the remainder of this part, we verify the boundary conditions of $\widehat{u^{\gamma, k}}$.\vspace{0.2pc}
        
        On $\p^+ B^+_{\rho_k}(y^*)$, we have $ \zeta^y=y$. Hence
		$$
		\Pi^k\left( y^0, \Pi^k\pigl( \left(\zeta^y\right)^0, \overline{u}^k(\zeta^y)\pigr) 
		\right) 
		=\Pi^k\left( \left( \zeta^y\right)^0, \Pi^k\pigl( \left(\zeta^y\right)^0, \overline{u}^k(\zeta^y)\pigr) 
		\right) 
		=\Pi^k\left( \left( \zeta^y\right)^0, \overline{u}^k(\zeta^y)\right),$$ 
        leading to $$u^{\gamma,k}(y)=\overline{u}^k(\zeta^y)=\overline{u}^k(y)
		=\widetilde{u}_{\widetilde{a},r_k}(y) \h{15pt}\text{on $\p^+ B^+_{\rho_k}(y^*)$, in the sense of trace.}$$ 
        
        By \eqref{tan at equator} and \eqref{set inclusion},  $$\psi_*\pigl(\overline{u}^k(\zeta^y)\pigr)=1 \h{5pt}\text{and} \h{5pt} 
		\overline{u}^k(\zeta^y) = \Pi^k\left( \left( \zeta^y\right)^0, \overline{u}^k(\zeta^y)\right) \h{15pt}\text{on $\p^0 B^+_{\rho_k}(y^*)\setminus \p^0 B^+_{\gamma\h{0.2pt}\rho_k}(y^*)$,}$$ in the sense of trace. We then obtain $$\Psi^{\overline{u}^k}(y,\zeta^y) \in T_{\varphi^{-1}\left(\widetilde{a}+r_k\h{0.2pt}y\right)}\p \Omega \h{15pt}\text{on $\p^0 B^+_{\rho_k}(y^*)\setminus \p^0 B^+_{\gamma\h{0.2pt}\rho_k}(y^*)$, in the sense of trace.}$$ Using the fact that $\Pi^k\left( y^0, \overline{u}^k(\zeta^y)\right)\in T_{\varphi^{-1}(\widetilde{a}+r_ky)}\p \Omega$ deduces $\Phi^{\overline{u}^k}(y,\zeta^y) \in T_{\varphi^{-1}\left(\widetilde{a}+r_k\h{0.2pt}y\right)}\p \Omega$. Then it follows that $$\,u^{\gamma,k}(y) \in T_{\varphi^{-1}(\widetilde{a}+r_ky)}\p \Omega, \h{15pt}\text{for $y \in\p^0 B^+_{\rho_k}(y^*)\setminus \p^0 B^+_{\gamma\h{0.2pt}\rho_k}(y^*)$, in the sense of trace.}$$ 

On $\p^0 B^+_{\gamma\h{0.2pt}\rho_k}(y^*)$, we arrive at a similar result: $u^{\gamma,k}(y) \in T_{\varphi^{-1}\left(\widetilde{a}+r_k\h{0.2pt}y\right)}\p \Omega$, because the fact that $$\overline{u}^k(y) \in U_{\delta_0/2}\big(\mathbb{S}^2 \cap T_{\varphi_k(y)}\p \Omega\big)\subseteq U_{\delta_0}\big(\mathbb{S}^2\cap T_{\varphi^{-1}(\widetilde{a})}\p\Omega\big) \h{15pt}\text{for $y \in \p^0 B^+_{\rho_k}(y^*)$, in the sense of trace}$$  implies $\psi_*\left( \overline{u}^k\pigl(y^*+ \gamma^{-1}(y-y^*) \pigr)\right) =1$ on $\p^0 B^+_{\gamma\h{0.2pt}\rho_k}(y^*)$. \vspace{0.2pc}

From the above arguments, we obtain
        \begin{align*}\widehat{u^{\gamma,k}} &\in \mathscr{F}_{
			\widetilde{u}_{\widetilde{a},r_k},\,\varphi_k,\,B^+_{\rho_k}(y^*)}\\[1mm]
		&:=\Big\{v \in H^1(B^+_{\rho_k}(y^*);\mathbb{S}^2):
		v=\widetilde{u}_{\widetilde{a},r_k}\h{5pt} \text{on $\p^+ B^+_{\rho_k}(y^*)$ \h{8pt}and} \h{8pt}v\h{1pt}\big|_{\p^0 B^+_{\rho_k}(y^*)}\in T_{\varphi_k(\cdot)}\p\Omega\h{1pt}\Big\}.\end{align*}
		
		\noindent{\bf Part 2.3. Estimates of the energy of comparison map:} We introduce spherical coordinates $(r,\xi,\theta)$ centered at $y^*$, where $r$ is the radial variable. $\xi \in [\h{0.8pt}0, \pi\h{0.5pt}]$ is the polar angle. $\theta \in [\h{0.8pt}0, 2 \pi\h{0.5pt})$ is the azimuthal angle. \vspace{0.2pc}
        
        In the region $B^+_{\rho_k}(y^*) \setminus B^+_{\gamma\h{0.2pt}\rho_k}(y^*)$, the definition of $u^{\gamma,k}$ in \eqref{def. of u^rho 2} implies 
		\begin{align*}
			\p_r u^{\gamma,k} = 	 \dfrac{\Psi^{\overline{u}^k}(y,\zeta^y)
			-\Phi^{\overline{u}^k}(y,\zeta^y)}{(1-\gamma)\rho_k } +\dfrac{\rho_k-r}{(1-\gamma)\rho_k}
			\h{1pt}\p_{r}\pig[	\Phi^{\overline{u}^k}(y,\zeta^y)\pig]
			+ \dfrac{r - \gamma\rho_k}{(1-\gamma)\rho_k} \h{1pt}\p_{r}\pig[	\Psi^{\overline{u}^k}(y,\zeta^y)\pig].
		\end{align*}
        Simply denote by $w^k$ the $0$-homogeneous map $\overline{u}^k(\zeta^y)$. The triangle inequality implies
        \begin{align*}
			 \left| \p_r u^{\gamma,k}\right| 
			&\h{1.5pt}\leq\h{1.5pt} \dfrac{\psi_*\pigl(w^k\pigr)}{(1-\gamma)\rho_k} \h{1pt}	\bigg| \h{1pt} 			\Pi^k\pigl(y^0, w^k\pigr)
			- w^k 
			-   \Pi^k\Big( y^0, \Pi^k\left( \left(\zeta^y\right)^0, w^k\right) \Big) 
			+ \Pi^k\left( \left(\zeta^y\right)^0, w^k\right) 
			\bigg|\\[2mm]
			&\h{1.5pt}+\h{1.5pt}\psi_*\pigl(w^k\pigr) \h{1.5pt}
			\bigg| \h{1pt}\dfrac{\rho_k-r}{(1-\gamma)\rho_k} \h{1pt}
			\p_{y_j}\Pi^k\left( y^0, w^k\right) \p_{r}y_j
			+ \dfrac{r - \gamma\h{0.2pt}\rho_k}{(1-\gamma)\rho_k}\h{1pt}
			\p_{y_j}\Pi^k\Big( y^0, \Pi^k\left( (\zeta^y)^0, w^k\right)\Big) \h{1.5pt} \p_{r}y_j \h{1pt}\bigg|.
		\end{align*}
		The last inequality in \eqref{3850} implies that
		\begin{align*}
			&\int_{B^+_{\rho_k}(y^*) \h{1pt}\setminus \h{1pt} B^+_{\gamma\h{0.2pt}\rho_k}(y^*)}
\bigg|\h{1pt}\Pi^k\pigl(y^0, w^k\pigr)
			- w^k -   \Pi^k\Big( y^0, \Pi^k\left( \left(\zeta^y\right)^0, w^k\right) \Big) 
			+ \Pi^k\left( \left(\zeta^y\right)^0, w^k\right) 
			\bigg|^2\\	
			&\h{75pt}\lesssim_{\,\Omega\,} \int_{B^+_{\rho_k}(y^*) \h{1pt}\setminus\h{1pt} B^+_{\gamma\h{0.2pt}\rho_k}(y^*)}	\bigg|\h{1pt}w^k 
			- \Pi^k\left( \left(\zeta^y\right)^0, w^k\right) 
			\bigg|^2 \h{1.5pt}\lesssim_{\,\Omega\,}\h{1.5pt}\rho_k^3 \left(1-\gamma\right)\left(\varepsilon 
			+r_k^2 \right).
		\end{align*} 
		Using the last estimate and Item (1) in Lemma \ref{lem. regularity of proj. and dist} then yield
		\begin{align}
			\int_{B^+_{\rho_k}(y^*) \h{1pt}\setminus\h{1pt} B^+_{\gamma\h{0.2pt}\rho_k}(y^*)}\left|\h{1pt} \p_r u^{\gamma,k}\h{1pt}\right|^2  \,\lesssim_{\,\Omega\,}\,&  \dfrac{\rho_k \left(\varepsilon 
				+r_k^2 \right)}{ (1-\gamma)} 
			+ \left(1-\gamma\right)\rho_k^3\h{1pt}r_k^2.
			\label{4407}
		\end{align}	
        
		Differentiate with respect to the polar angle $\xi$. It turns out that 
		\begin{align*}
			\p_{\xi} \left[	\Phi^{\overline{u}^k}(y,\zeta^y)\right]
			=\,&\psi_*\pigl(w^k\pigr)
			\left(\p_\xi y \cdot\nabla_y\Pi^k\pigl(y^0, p\pigr) + \p_\xi w^k \cdot \nabla_p \Pi^k\pigl(y^0, p\pigr)\right)\bigg|_{p \h{1pt}=\h{1pt} w^k} \nonumber\\[1.5mm]
			&+ \left(\nabla \psi_* \h{1pt}\big|_{w^k} \cdot 
			\p_{\xi} w^k \right) 
    			\left(\Pi^k\pigl(y^0, w^k\pigr) - w^k \right) 
			+ \left[ 1-\psi_*\pigl(w^k\pigr)\right] \p_{\xi}w^k.
		\end{align*}
        As $\psi_*$ is smooth, we use the definition of $\Pi^k$ and Item (1) in Lemma \ref{lem. regularity of proj. and dist} to obtain
		\begin{align}
			\left|\h{1pt}\p_{\xi} \left[\h{1pt}	\Phi^{\overline{u}^k}(y,\zeta^y)\right]\h{1pt}\right|
			\h{1.5pt}\lesssim_{\,\Omega\,}  r\h{0.5pt}r_k + \big|\h{1pt}\p_\xi w^k \big|\h{15pt}\text{on $B^+_{\rho_k}(y^*) \setminus B^+_{\gamma\h{0.2pt}\rho_k}(y^*)$}.
			\label{3783}
		\end{align}
        
		We also compute
		\begin{align*}
			&\p_\xi \left[\h{1pt}	\Psi^{\overline{u}^k}(y,\zeta^y) \h{1pt}\right] = \p_{\xi}w^k +\p_{\xi}w^k \cdot \nabla \psi_*\h{.5pt}\big|_{w^k}   
			\left[\h{1pt} \Pi^k\Big( y^0, \Pi^k\left( \left(\zeta^y\right)^0, w^k\right) \Big) 
			- \Pi^k\left( \left(\zeta^y\right)^0, w^k\right) \h{1pt}
			\right]\nonumber\\[1.5mm] 
			&\h{30pt} +\psi_*\pigl(w^k\pigr)\bigg[ \p_\xi y \cdot\nabla_y \Pi^k\left( y^0, p \right) + \p_\xi \Pi^k\left( \left(\zeta^y\right)^0, w^k\right) \cdot \nabla_p \Pi^k\left(y^0, p\right) - \p_\xi \Pi^k\left( \left(\zeta^y\right)^0, w^k\right) \bigg], 
		\end{align*}where $p = \Pi^k\left( \left(\zeta^y\right)^0, \h{1pt}w^k\right)$.
        Similarly to \eqref{3783}, we deduce that
        \begin{align}
			\left|\h{1pt}\p_\xi \left[\h{1pt}	\Psi^{\overline{u}^k}(y,\zeta^y)\h{1pt}\right]\h{1pt}\right|
			\h{1pt}\lesssim_{\,\Omega\,} 
			r_k\left(r+\rho_k\right)+ \big|\h{1pt}\p_\xi w^k\h{1pt}\big|.
			\label{4417}
		\end{align}Similar estimates hold if we replace the $\partial_\xi$-derivative in \eqref{3783}-\eqref{4417} by the $\partial_\theta$-derivative. That is \begin{align}\label{partial theta}
        &\left|\h{1pt}\p_{\theta} \left[\h{1pt}	\Phi^{\overline{u}^k}(y,\zeta^y)\right]\h{1pt}\right|
			\h{1.5pt}\lesssim_{\,\Omega\,}  r\h{0.5pt}r_k\h{0.5pt}\sin \xi + \big|\h{1pt}\p_\theta w^k \big|,\nonumber\\[1.5mm]
            &\left|\h{1pt}\p_\theta \left[\h{1pt}	\Psi^{\overline{u}^k}(y,\zeta^y)\h{1pt}\right]\h{1pt}\right|
			\h{1pt}\lesssim_{\,\Omega\,} 
			r_k\left(r+\rho_k\right) \sin \xi + \big|\h{1pt}\p_\theta w^k\h{1pt}\big|, \h{25pt}\text{on $B^+_{\rho_k}(y^*) \setminus B^+_{\gamma\h{0.2pt}\rho_k}(y^*)$}.
        \end{align}Using \eqref{3783}-\eqref{partial theta}, together with the first estimate in \eqref{3850}, we obtain
		\begin{align}
			 &\int_{B^+_{\rho_k}(y^*) \h{1pt}\setminus\h{1pt} B^+_{\gamma \h{0.2pt}\rho_k}(y^*)}\dfrac{1}{r^2}
			\left|\h{1pt} \p_\xi u^{\gamma,k}\h{.5pt}\right|^2
			+\dfrac{1}{r^2\sin^2\xi}
			\left|\h{1pt} \p_\theta u^{\gamma,k} \h{.5pt}\right|^2 \nonumber\\[1.5mm] 
            &\h{20pt}\lesssim_{\,\Omega\,}
			\left(1-\gamma\right)r_k^2\h{1pt}\rho_k^3 
			+\int_{B^+_{\rho_k}(y^*) \h{1pt}\setminus\h{1pt} B^+_{\gamma\h{0.2pt}\rho_k}(y^*)}\dfrac{1}{r^2}
			\left|\h{1pt} \p_\xi w^k \h{.5pt}\right|^2
			+\dfrac{1}{r^2\sin^2\xi}
			\left|\h{1pt} \p_\theta w^k\h{0.5pt}\right|^2  \nonumber\\[1.5mm] 
            &\h{20pt}\lesssim_{\,\Omega\,}\left(1-\gamma\right)\rho_k\left(1 + r_k^2\h{0.5pt} \rho_k^2 \h{0.5pt}\right). 
			\label{3818}
		\end{align}
        
		Fixing $y \in B^+_{\gamma\h{0.5pt}\rho_k}(y^*)$, we define $z = y^* + \gamma^{-1}\left(y - y^*\right)$. Moreover, $\overline{u}^k$ is a map on $z$. Thus, \begin{align*}
		    \p_{y_j} u^{\gamma, k} = \gamma^{-1} \p_{z_j} \overline{u}^k &+ \gamma^{-1}\p_{z_j}\overline{u}^k \cdot \nabla \psi_*\h{1pt}\big|_{\overline{u}^k} \Big[ \h{1pt}\Pi^k\pigl(y^0, \overline{u}^k\pigr)   -  \overline{u}^k \h{1pt}\Big] \\[1.5mm]
            &+ \psi_*\pigl(\overline{u}^k\pigr) \Big[ \h{1pt}\p_{y_j} \Pi^k\pigl(y^0, p\pigr) + \gamma^{-1}\p_{z_j} \overline{u}^k \cdot \nabla_p \Pi^k\pigl(y^0, p\pigr)  -   \gamma^{-1} \p_{z_j} \overline{u}^k \h{1pt}\Big],
		\end{align*}where $p = \overline{u}^k$. Through the change of variables, we deduce that
		\begin{align}
			\int_{  B^+_{\gamma\h{0.2pt}\rho_k}(y^*)}\big|\h{.5pt} \nabla_y u^{\gamma,k}\h{.5pt}\big|^2
			&\lesssim_{\,\Omega\,} r_k^2\h{0.5pt}\gamma^3\h{0.5pt}\rho_k^3 + \gamma^{-2} \int_{  B^+_{\gamma\rho_k}(y^*)} \big|\h{.5pt} \nabla_z\overline{u}^k \h{.5pt}\big|^2 \h{1pt}\lesssim\h{1pt} r_k^2\h{0.5pt}\gamma^3 \h{0.5pt}\rho_k^3
			+\gamma\int_{  B^+_{\rho_k}(y^*)} 
			\big|\h{.5pt} \nabla\overline{u}^k \h{.5pt}\big|^2.
			\label{3771}
		\end{align}
        
		Recall the minimality of $\widetilde{u}_{\widetilde{a},r_k}$ and the fact that $\widehat{u^{\gamma,k}} \in \mathscr{F}_{ \widetilde{u}_{\widetilde{a},r_k},\,\varphi_k,\,B^+_{\rho_k}(y^*)}$. By \eqref{4407}, \eqref{3818}-\eqref{3771}, and the rescaled energy in \eqref{defn of sacal}, it follows that
		\begin{align*}
			\widetilde{\mathcal{E}}^{\h{0.7pt}k}_{ \widetilde{u}_{\widetilde{a},r_k},\h{0.5pt}B_{\rho_k}^{+}(y^*)} \,\leq\,\widetilde{\mathcal{E}}^{\h{0.7pt}k}_{\widehat{u^{\gamma,k}},\h{0.5pt}B^+_{\rho_k}(y^*)}
			\lesssim_{\,\Omega\,} 
			\dfrac{\rho_k \left(\varepsilon 
				+r_k^2 \right)}{ (1-\gamma)} 
			+\left(1-\gamma\right) \rho_k \left( 1 + \rho_k^2\h{1pt}r_k^2\right) 
			+r_k^2\h{1pt} \rho^3_k
			+ \int_{  B^+_{\rho_k}(y^*)} 
			\big|\h{0.5pt} \nabla\overline{u}^k\h{0.5pt}\big|^2.
		\end{align*} 
		Recall $\epsilon$ and $q^k$ given in Part 2.1 of this proof. By Items (2) in Proposition \ref{lem ext from sphere to ball},
        and \eqref{3850}, 
		\begin{align*}
			\mathcal{E}^*_{\overline{u}^k,\h{0.5pt}B^+_{\rho_k}(y^*)} &\lesssim_{\,\Omega\,} \rho_k\h{1.5pt}\epsilon\h{1.5pt} \mathcal{E}^*_{\widetilde{u}_{\widetilde{a},r_k},\h{.5pt}\p^+B^+_{\rho_k}(y^*)}
			+\rho_k^{-1}\epsilon^{-\mu}\h{1pt} W^{q^k}_{\widetilde{u}_{\widetilde{a},r_k},\h{0.5pt}	\varphi_k,\h{0.5pt}\p^+B^+_{\rho_k}(y^*) }
			+\rho_k\h{1pt}r_k^2\left( \epsilon^{-1}\dfrac{ \log \epsilon}{\log \alpha} +\epsilon^{-2}
			+ \epsilon^{2-\mu} \right) \nonumber\\[2mm] 
			&\lesssim_{\,\Omega\,} \rho_k \h{1pt} \epsilon 
			+\rho_k\h{1pt}\epsilon^{-\mu}\left(\varepsilon+r_k^2\right)
			+\rho_k \h{1pt}r_k^2\left( \epsilon^{-1}\dfrac{ \log \epsilon}{\log \alpha} +\epsilon^{-2}
			+ \epsilon^{2-\mu} \right). 
		\end{align*}
		Thus, if we let $1 - \gamma = \epsilon^\mu$, then it satisfies
		\begin{align*}
			\rho_k^{-1}\mathcal{E}^*_{ \widetilde{u}_{\widetilde{a},r_k},\h{0.5pt}B_{\rho_k}^{+}(y^*)} &\h{1pt}\lesssim_{\,\Omega}\h{1pt} \dfrac{\varepsilon 
				+r_k^2 }{\epsilon^\mu} 
			+\epsilon^\mu\left( 1 + \rho_k^2\h{1pt}r_k^2\right) 
			+r_k^2\h{1pt}\rho^2_k  + \epsilon 
			+ r_k^2\left(  \epsilon^{-2}
			+ \epsilon^{2-\mu} \right). 
		\end{align*} 
		Note that $\rho_k\in [ \frac{3\rho_0}{4},\rho_0]$. We first choose $\epsilon$ small. Then we take $\varepsilon$ small and $k$ large. Hence,
        $$\widehat{\mathcal{E}}^*_{3\rho_0/4, \h{1pt}y^*; \h{0.5pt}+}\left(\widetilde{u}_{\widetilde{a},r_k}\right) < \eps{5}.$$ Here $\eps{5}$ is the small constant in Lemma \ref{lem Partial Regularity}. Applying Lemma \ref{lem Partial Regularity}, we find an $\alpha_0 \in (0, 1)$ so that 
		\begin{align*}
\left[\h{1pt}\widetilde{u}_{\widetilde{a},r_k}\h{0.2pt}\right]_{\alpha_0, \h{1pt} \overline{B_{3\rho_0/32}^+(y^*)}} \leq C_{\Omega}, \h{20pt}\text{for any large enough $k\in \mathbb{N}$.}
		\end{align*}
		 Here, $[\h{1pt}\cdot\h{1pt}]_{\alpha_0, U}$ is the Hölder semi-norm on $U$ with exponent $\alpha_0$. By the Arzelà–Ascoli theorem, $$\widetilde{u}_{\widetilde{a},r_k} \to \widetilde{u}^{\h{.8pt}*} \h{15pt}\text{ uniformly on $\overline{B_{3\rho_0/32}^+(y^*)}$, as $k\to \infty$.}$$

		\noindent{\bf Proof of Item (3).} Note that Items (1) and (2) also hold on $\overline{B^+_{3/2}}$. Hence, for any $\varepsilon>0$, there is a finite cover, denoted by $\big\{B_{\tau_i/4}(x^i)\big\}_{i=1}^{N}$, 
		of $\mathscr{S}(\overline{B^+_{3/2}})$ such that $\sum^N_{i=1}\tau_i \leq \varepsilon$. The energy monotonicity in Lemma \ref{lem int energy mono} yields 
		\begin{align*}
			\int_{B^+_{\tau_i}(x^i)}
			\big|\h{0.5pt}\nabla \widetilde{u}_{\widetilde{a},r_k} \h{0.5pt}\big|^2
			=\dfrac{1}{r_k}\int_{B^+_{r_k\tau_i}\left(\widetilde{a} + r_k \h{0.2pt} x^i\right)} \big|\h{0.5pt}\nabla \widetilde{u}\h{0.5pt}\big|^2
			\h{1.5pt}\lesssim_{\,\Omega\,} \tau_i\left( R_2 + \dfrac{1}{R_2}\int_{B^+_{R_2}\left(\widetilde{a} + r_k \h{0.2pt} x^i\right)} \big|\h{0.5pt}\nabla \widetilde{u}\h{0.5pt}\big|^2
			\right)
			\h{1.5pt}\lesssim_{\,\Omega\,} \tau_i,
		\end{align*}
		which implies
		\begin{align}
			\int_{\bigcup_iB^+_{\tau_i}\left(x^i\right)}
			\big|\h{1pt}\nabla \widetilde{u}_{\widetilde{a},r_k}\h{.5pt}\big|^2\h{1pt}\leq\h{1pt}\sum^N_{i=1}	\int_{B^+_{\tau_i}\left(x^i\right)}
			\big|\h{0.5pt}\nabla \widetilde{u}_{\widetilde{a},r_k} \h{0.5pt}\big|^2
			\h{1pt}\lesssim_{\,\Omega\,} \varepsilon, \h{15pt}\text{for large $k \in \mathbb N$.}
			\label{3280}
		\end{align}
		
		Recall the extension of $\widetilde{u}$ from $B^+_{R_0}$ to $B_{R_0}$ in \eqref{def ext of tilde u}. After rescaling, we still use $\widetilde{u}_{\widetilde{a},r_k}$ to denote the map $\mathcal{U}\left(\widetilde{a} + r_k\h{0.5pt}\cdot\right)$. For large $k$, it is an extension of $\widetilde{u}_{\widetilde{a},r_k}$ from $B^+_2$ to $B_2$. \vspace{0.2pc}
        
        Let $\psi \in C_c^\infty\left(B_{3/2};[\h{0.5pt}0, 1\h{0.5pt}]\h{0.5pt}\right)$ be a function with the compact support contained in $$U_3:= \left(B_{3/2} \setminus \h{1pt} \bigcup_i \overline{B^+_{\tau_i/2}(x^i)}\right) \h{1pt}\bigcap \h{1pt} \left(B_{3/2} \setminus \h{1pt} \bigcup_i \h{1pt}\left\{y:y^\star\in \overline{B^+_{\tau_i/2}(x^i)}\h{1pt}\right\} \right).$$ We then define, for large $k, l \in \mathbb N$, the map $\Psi :=\left(\widetilde{u}_{\widetilde{a},r_k}
		-\widetilde{u}_{\widetilde{a},r_l} \right)\psi$. We use $\widetilde{a} + r_k \h{0.2pt}y$ and $\widetilde{a} + r_l\h{0.2pt}y$ to change the variables in equation \eqref{eq. of widetilde of u, whole ball} and then subtract the two resulting equations. Applying the estimate \eqref{J} in Lemma \ref{lem eq of extended u} and using $\Psi$ as the test function, we deduce that  
		\begin{align*}
			&\int_{B_{3/2}}\psi\h{1.5pt}\overline{a}_{ij}^{\h{1pt}k}\h{1.5pt}\p_{y_i} \pig[\h{0.5pt}\widetilde{u}_{\widetilde{a},r_k}
			-\widetilde{u}_{\widetilde{a},r_l} \pig]\cdot
			\p_{y_j}\pig[\h{0.5pt}\widetilde{u}_{\widetilde{a},r_k}
			-\widetilde{u}_{\widetilde{a},r_l} \pig]
			+ \psi \h{1.5pt}\pig[\h{0.5pt}\overline{a}_{ij}^{\h{1pt}k}-\overline{a}_{ij}^{\h{1pt}l}\h{0.5pt}\pig] \h{1.5pt}
			\p_{y_i}  \widetilde{u}_{\widetilde{a},r_l} \cdot
			\p_{y_j}\pig[\widetilde{u}_{\widetilde{a},r_k}
			-\widetilde{u}_{\widetilde{a},r_l} \pig]\\[1mm]
			&\h{15pt}+\int_{B_{3/2}} \overline{a}_{ij}^{\h{1pt}k}\h{1.5pt}\p_{y_j}\psi \h{1.5pt}\p_{y_i}\pig[\h{0.5pt}\widetilde{u}_{\widetilde{a},r_k}
			-\widetilde{u}_{\widetilde{a},r_l} \h{0.5pt}\pig]\cdot
			\pig[\h{0.5pt}\widetilde{u}_{\widetilde{a},r_k}
			-\widetilde{u}_{\widetilde{a},r_l} \h{0.5pt}\pig]
			+\pig[\overline{a}_{ij}^{\h{1pt}k}-\overline{a}_{ij}^{\h{1pt}l}\pig] \h{1pt} \p_{y_j}\psi\h{1.5pt}
			\p_{y_i} \widetilde{u}_{\widetilde{a},r_l} \cdot
			\pig[\h{0.5pt}\widetilde{u}_{\widetilde{a},r_k}
			-\widetilde{u}_{\widetilde{a},r_l} \pig]\\[1mm]
			&\h{15pt}\lesssim_{\,\Omega\,}
			r_k^2 + r_l^2 + \int_{B_{3/2}^+}
			\Big(\pig|\nabla \widetilde{u}_{\widetilde{a},r_k}\pig|^2
			+\pig|\nabla \widetilde{u}_{\widetilde{a},r_l} \pig|^2\Big)
			\pig(\h{1pt}\big|\h{0.5pt}\Psi \h{0.5pt}\big| + \big| \Psi(\h{0.5pt}\cdot^\star)\h{0.5pt}\big|\h{1.5pt}\pig).
		\end{align*}Here, $\overline{a}_{ij}^{\h{1pt}k} (\cdot):= \overline{a}_{ij}(\widetilde{a}+r_k \h{0.5pt}\cdot)$ with $\overline{a}_{ij}$ defined in Lemma \ref{lem eq of extended u}. By the energy bound \eqref{ineq. uniform energy bdd of rescaled map}, estimate \eqref{ineq. of widetilde of u whole ball<half ball}, H\"{o}lder's inequality and the definition of the extension map in \eqref{defn of ex}, it turns out that
		\begin{align*}
			\int_{U_3}  \psi \h{1pt}\pig|\h{0.5pt}\nabla\widetilde{u}_{\widetilde{a},r_k}
			-\nabla\widetilde{u}_{\widetilde{a},r_l}\pigr|^2
			\h{1pt}\lesssim_{\,\Omega,\h{1pt}\psi\,} r_k^2 + r_l^2 + 
			\sup_{B_{3/2}} \,\sup_{i, \,j}\, \pig|\h{1pt}\overline{a}_{ij}^{\h{1pt}k}
			-\overline{a}_{ij}^{\h{1pt}l}\h{0.5pt}\pig|
			+\sup_{U_3} \pig|\h{1pt}\widetilde{u}_{\widetilde{a},r_k}
			-\widetilde{u}_{\widetilde{a},r_l} \pig|.
		\end{align*}
		Letting $$U_4:= \left(B_{3/2} \setminus \h{1pt} \bigcup_i \overline{B^+_{\tau_i}\left(x^i\right)} \right) \h{1pt}
		\bigcap \h{1pt} \left(B_{3/2} \setminus \h{1pt} \bigcup_i \h{1pt}\left\{y:y^\star\in  \overline{B^+_{\tau_i}\left(x^i\right)} \h{1pt}\right\}\right)\subseteq U_3,$$ we choose $\psi$ such that $\psi \equiv 1$ in $B_1 \cap U_4$. Together with \eqref{3280}, we obtain
		\begin{align*}
			\int_{B_{1}} \pig|\h{1pt}\nabla\widetilde{u}_{\widetilde{a},r_k}
			-\nabla\widetilde{u}_{\widetilde{a},r_l} \h{0.5pt}\pigr|^2
			&\lesssim \int_{B_1 \setminus U_4}
			\pig|\h{1pt}\nabla \widetilde{u}_{\widetilde{a},r_k}\h{0.5pt}\pigr|^2+
			\pig|\h{1pt}\nabla \widetilde{u}_{\widetilde{a},r_l}\h{0.5pt}\pigr|^2 + 
			\int_{B_1 \h{1pt} \cap \, U_4} \pig|\h{1pt}\nabla\widetilde{u}_{\widetilde{a},r_k}
			-\nabla\widetilde{u}_{\widetilde{a},r_l} \h{0.5pt}\pigr|^2\\[1mm]
			&\lesssim_{\,\Omega\,} 	
			\varepsilon +  
			C_\psi\left(
			r_k^2 + r_l^2 + \sup_{B_{3/2}}\h{1pt}\sup_{i,\,j}\h{1pt}\pig|\overline{a}_{ij}^{\h{1pt}k}
			-\overline{a}_{ij}^{\h{1pt}l}\pig|
			+\sup_{U_3} \pig|\h{1pt}\widetilde{u}_{\widetilde{a},r_k}
			-\widetilde{u}_{\widetilde{a},r_l}\h{0.5pt}\pigr|\right).
		\end{align*}The desired result in Item (3) is obtained by taking $k$, $l$ to $\infty$ and taking $\varepsilon$ to $0$, successively.  
	\end{proof}

	The limiting map of the rescaled maps obtained in Lemma \ref{lem compact of rescaled map} is referred to as the tangent map in the literature. In the following, we characterize this tangent map. By choosing a suitable coordinate frame and $\varphi$ for the domain transformation, we assume that $a_{ij}(\widetilde{a})=\delta_{ij}$ for any $i, j = 1, 2, 3$.   
	\begin{lem}[\bf Minimizing tangent maps] \label{lem prop of tangent map}
		The tangent map $\widetilde{u}^{\h{.8pt}*}$ obtained in Lemma \ref{lem compact of rescaled map} is homogeneous of degree zero. Moreover, there is $\sigma\in[\h{0.5pt}\frac{1}{4}, \frac{1}{2}\h{0.5pt}]$ such that $\widetilde{u}^{\h{.8pt}*}$ minimizes the standard Dirichlet energy in the configuration space: $$\mathscr{F}_{\widetilde{u}^{\h{.8pt}*},\,\varphi^{-1}(\widetilde{a}),\,B_\sigma^+}:=\Big\{v \in H^1(B_\sigma^+;\mathbb{S}^2):v=\widetilde{u}^{\h{.8pt}*}\h{5pt} \text{on $\p^+B_\sigma^+$ \h{5pt}and} \h{10pt}v\h{0.5pt}\big|_{\p^0B^+_\sigma}\in T_{\varphi^{-1}(\widetilde{a})}\p\Omega\h{1pt}\Big\}.$$
	\end{lem}
	\begin{proof} We divide the proof into 3 parts. \vspace{0.2pc}\\
		\noindent{\bf Part 1.\h{4pt}$0$\h{0.2pt}-\h{0.2pt}homogeneity:}
		The energy monotonicity in Lemma \ref{lem. energy mono.} implies that $$\lim_{r\to0^+} r^{-1}\int_{B^+_{r}(\widetilde{a})}
		\big|\h{1pt}\nabla\widetilde{u} \h{1pt}\big|^2 = L^*,\h{15pt}\text{for some $L^* \geq 0$.}$$ In addition, still by Lemma \ref{lem. energy mono.}, 
		$$  
		L^* + \int^R_0\dfrac{e^{\C{2}  \tau}}{2\tau} \h{1.5pt}\mathrm d \tau\int_{\p^+ B^+_{\tau}(\widetilde{a})} \big|\h{0.5pt} \p_r \widetilde{u} \h{0.5pt}\big|^2  \h{1.5pt}\mathrm d\h{.5pt}\mathscr{H}^2  
		\h{1.5pt}\leq\h{1.5pt} \C{3}\h{0.5pt}R + 
		\dfrac{e^{\C{2} R}}{R}\int_{B^+_{R}(\widetilde{a})} \big|\h{0.5pt}\nabla\widetilde{u} \h{0.5pt}\big|^2 \h{15pt}\text{for any $R\in(0,\R{1}]$.}$$
		Replacing $R$ with $r_k$ and rescaling, we have
		$$  
		L^* + \int^1_0\dfrac{e^{\C{2} r_k \tau}}{2\tau} \h{1.5pt}\mathrm d \tau
		\int_{\p^+ B^+_{\tau}} \big|\h{0.5pt} \p_r \widetilde{u}_{\widetilde{a},r_k}\h{0.5pt}\big|^2  \h{1.5pt}\mathrm d\h{.5pt}\mathscr{H}^2 
		\h{1.5pt}\leq \h{1.5pt} \C{3}\h{0.5pt}r_k + 
		\dfrac{e^{\C{2} r_k}}{r_k}\int_{B^+_{r_k}(\widetilde{a})} \big|\h{0.5pt}\nabla\widetilde{u} \h{0.5pt}\big|^2.$$
		By Lemma \ref{lem compact of rescaled map}, the rescaled maps $\widetilde{u}_{\widetilde{a},r_k}$ strongly converge in $H^1(B^+_1;\mathbb{S}^2)$ to the tangent map $\widetilde{u}^{\h{.8pt}*}$. We therefore take $k \to \infty $ in the last estimate and deduce that
		$$\int^1_0\frac{\mathrm d \tau}{2\tau}\int_{\p^+ B^+_{\tau}} \big|\h{0.5pt} \p_r \widetilde{u}^{\h{.8pt}*} \h{0.5pt}\big|^2  \h{1.5pt}\mathrm d\h{.5pt}\mathscr{H}^2 = 0,$$ implying the $0$-homogeneity of $\widetilde{u}^*$. \vspace{0.4pc}
		
		\noindent{\bf Part 2. Uniform convergence of $\widetilde{u}_{\widetilde{a},r_k}$ away from the origin:} The rescaled map $\widetilde{u}_{\widetilde{a},r_k}$ solves
		\begin{equation}\label{2781}
			\int_{B^+_{1}} a_{ij}^k\h{1.5pt}\p_{y_i} \widetilde{u}_{\widetilde{a},r_k} \cdot
			\p_{y_j} \psi
			-a_{ij}^k\h{1.5pt}\p_{y_i} \widetilde{u}_{\widetilde{a},r_k} \cdot
			\p_{y_j} \widetilde{u}_{\widetilde{a},r_k} \left(\widetilde{u}_{\widetilde{a},r_k}
			\cdot\psi\right)=0,  
		\end{equation}
		for any $\psi \in H^1\left(B^+_{1};\mathbb{R}^3\right)
		\h{1pt}\bigcap\h{1pt} L^\infty\left(B^+_{1};\mathbb{R}^3\right)\h{1pt}\bigcap\h{1pt} \pig\{\h{1pt}\psi:\psi \h{1pt}\big|_{\p^0B^+_{1}}\in T_{\varphi_k(\cdot)}\p\Omega \h{6pt}\text{and}\h{6pt}\psi\h{1pt}\big|_{\p^+B^+_{1}}=0\h{1pt}\pig\}$.\vspace{0.2pc}
        
        In addition, we recall \eqref{tangential cond of widetldu}. It turns out $$\widetilde{u}^{\h{.8pt}*} \in \mathbb{S}^2 \cap T_{\varphi^{-1}(\widetilde{a})}\p\Omega  \h{20pt}\text{$\mathscr H^2$-a.e. on $\p^0B^+_{1}$.}$$ 
        By choosing a suitable coordinate frame, we assume that $T_{\varphi^{-1}(\widetilde{a})}\p\Omega = \big\{y_3=0\big\}$. For any 
        $$\phi\in C^1\left(\overline{B^+_{1}};\mathbb{R}^3\right)\h{1pt}\bigcap\h{1pt}\pig\{\psi:\psi\h{1pt}\big|_{\p^+B^+_{1}}=0\h{6pt} \text{and} \h{6pt} \psi_3\h{1pt}\big|_{\p^0B^+_{1}}=0\pig\},$$ 
        we find a sequence $\big\{\phi^k\big\}_{k\h{0.5pt}\in \h{0.5pt}\mathbb{N}}$ such that $\phi^k \to \phi$ in $C^1\left(\overline{B^+_{1}};\mathbb{R}^3\right)$ as $k \to \infty$ with $$\phi^k\in C^\infty\left(\overline{B^+_{1}};\mathbb{R}^3\right) \h{1pt}\bigcap \h{1pt} \pig\{\psi:\psi\h{1pt}\big|_{\p^+B^+_{1}}=0 \h{6pt}\text{and}\h{6pt}\psi \h{1pt}\big|_{\p^0B^+_{1}}\in T_{\varphi_k(\cdot)}\p\Omega\h{1pt}\pig\}.$$ 
        Substitute $\phi^k$ for $\psi$ in \eqref{2781} and reorganize the resulting equation. It turns out that
		\begin{align*}
			& \int_{B^+_{1}} \p_{y_i} \widetilde{u}^{\h{.8pt}*} \cdot
			\p_{y_i} \phi 
			-
			\big|\h{0.5pt}\nabla\widetilde{u}^{\h{.8pt}*}\h{0.5pt}\big|^2 \h{1.5pt} \widetilde{u}^{\h{.8pt}*} \cdot\phi  
			= \int_{B^+_{1}} \left(\delta_{ij}-a_{ij}^k\right)
			\p_{y_i} \widetilde{u}^{\h{.8pt}*} \cdot
			\p_{y_j} \phi 
			-\left(\delta_{ij}-a_{ij}^k\right)
			\p_{y_i} \widetilde{u}^{\h{.8pt}*} \cdot
			\p_{y_j} \widetilde{u}^{\h{.8pt}*} \left(\widetilde{u}^{\h{.8pt}*} \cdot\phi\right)  \\[1.5mm]
			&\h{63pt}+\int_{B^+_{1}}  a_{ij}^k \h{1.5pt} 
			\p_{y_i} \left(\widetilde{u}^{\h{.8pt}*} -\widetilde{u}_{\widetilde{a},r_k}\right)\cdot
			\p_{y_j} \phi 
			- a_{ij}^k\left(
			\p_{y_i} \widetilde{u}^{\h{.8pt}*} \cdot
			\p_{y_j} \widetilde{u}^{\h{.8pt}*} \h{1.5pt} \widetilde{u}^{\h{.8pt}*}
			-\p_{y_i} \widetilde{u}_{\widetilde{a},r_k} \cdot
			\p_{y_j} \widetilde{u}_{\widetilde{a},r_k} \h{1.5pt}\widetilde{u}_{\widetilde{a},r_k}  \right)
			\cdot\phi  \\[1.5mm]
			&\h{63pt}+\int_{B^+_{1}}  a_{ij}^k \h{1.5pt}
			\p_{y_i}  \widetilde{u}_{\widetilde{a},r_k} \cdot
			\p_{y_j} \left(\phi -\phi^k \right) 
			- a_{ij}^k \h{1.5pt}
			\p_{y_i} \widetilde{u}_{\widetilde{a},r_k} \cdot
			\p_{y_j} \widetilde{u}_{\widetilde{a},r_k} \h{1.5pt}\widetilde{u}_{\widetilde{a},r_k}  
			\cdot\left(\phi -\phi^k \right).
		\end{align*}
		In light of the consequences in Lemma \ref{lem compact of rescaled map}, we pass $k \to \infty$ in the last equality and deduce that
		\begin{align*}
			 \int_{B^+_{1}} \p_{y_i} \widetilde{u}^{\h{.8pt}*} \cdot
			\p_{y_i} \phi 
			-
			\big|\h{0.5pt}\nabla\widetilde{u}^{\h{.8pt}*} \h{0.5pt}\big|^2 \left(\widetilde{u}^{\h{.8pt}*} \cdot\phi \right)  = 0.
		\end{align*}
		
		We extend $\widetilde{u}^{\h{.8pt}*} = \big(\widetilde{u}^{\h{.8pt}*}_1, \widetilde{u}^{\h{.8pt}*}_2, \widetilde{u}^{\h{.8pt}*}_3 \big)^\top$ from $B_1^+$ to $B_1$ such that $\widetilde{u}^{\h{.8pt}*}_1$, $\widetilde{u}^{\h{.8pt}*}_2$ are even and $\widetilde{u}^{\h{.8pt}*}_3$ is odd across the plane $\big\{y_3=0\big\}$. Denote the extension by $\widetilde{\mathcal{U}}^*$ and choose $\Phi \in C^\infty_0\big(B_{1};\mathbb{R}^3\big)$. Direct computations yield
		 \begin{align*}
			\int_{B_1} \p_{y_i} \widetilde{\mathcal{U}}^* \cdot
			\p_{y_i} \Phi
			-
			\big|\h{.5pt}\nabla \widetilde{\mathcal{U}}^*\h{.5pt}\big|^2    \h{1.5pt}\widetilde{\mathcal{U}}^* \cdot\Phi  
			=\int_{B^+_{1}} \p_{y_i} \widetilde{u}^{\h{.8pt}*} \cdot
			\p_{y_i} \widetilde{\Phi} 
			-
			\big|\h{.5pt}\nabla\widetilde{u}^{\h{.8pt}*}\h{.5pt}\big|^2 \h{1.5pt}\widetilde{u}^{\h{.8pt}*} \cdot\widetilde{\Phi}=0,
		\end{align*}
		where $\widetilde{\Phi}_j = \Phi_j+\Phi_j (\cdot^\star)$ for $j=1,2$ and $\widetilde{\Phi}_3 = \Phi_3-\Phi_3 (\cdot^\star)$. We take $\Phi=\Phi_1(r)\Phi_2(\omega)$, where $\Phi_1 \in C^\infty_c\big((0,1);\mathbb{R}\big)$ satisfies $\displaystyle \int_0^1\Phi_1 \neq 0$. $\Phi_2 \in C^\infty\big(\mathbb{S}^{2};\mathbb{R}^3\big)$ is arbitrary. It follows from the last equation and the $0$-homogeneity of $\widetilde{\mathcal U}^*$ that \begin{align*}
		    \int_{\mathbb S^2} \nabla_{\omega}\h{1pt} \widetilde{\mathcal{U}}^* \cdot
			\nabla_{\omega} \Phi_2
			-
			\big|\h{.5pt}\nabla_\omega\h{1pt} \widetilde{\mathcal{U}}^*\h{.5pt}\big|^2    \h{1.5pt}\widetilde{\mathcal{U}}^* \cdot\Phi_2 = 0. 
		\end{align*} The inner product above is defined in terms of the spherical metric. We then conclude that $\widetilde{\mathcal{U}}^*$ is a weak harmonic map on $\mathbb{S}^2$, which is smooth on $\mathbb S^2$ by \cite{H90}. \vspace{0.2pc}
		
		Let $z_0 \in \overline{B^+_{1/2}} \setminus B^+_{1/4}$. There is $r_0 \in \big(0, \frac{1}{8}\big)$, depending only on $\Omega$ and $z_0$, such that 
        $$\frac{1}{r_0}\int_{B^+_{r_0}(z_0)} \big|\h{0.5pt}\nabla \widetilde{u}^{\h{.8pt}*} \h{0.5pt}\big|^2 < \frac{\eps{5}}{4}.$$ 
        Item (3) in Lemma \ref{lem compact of rescaled map} implies that 
		$$\frac{1}{r_0}\int_{B^+_{r_0}(z_0)} \big|\h{0.5pt}\nabla \widetilde{u}_{\widetilde{a},r_k} \h{0.5pt}\big|^2
		\h{1pt}\leq\h{1pt}\dfrac{\eps{5}}{2} + \frac{2}{r_0}\int_{B^+_{r_0}(z_0)} \big|\h{0.5pt}\nabla \widetilde{u}^{\h{.8pt}*} \h{0.5pt}\big|^2
		 \h{1pt}<\h{1pt} \eps{5}\h{20pt}\text{for large  $k\in \mathbb{N}$.}$$
		 By Lemma \ref{lem Partial Regularity} and standard interior $\varepsilon$-regularity, we deduce, using the Arzelà–Ascoli theorem, that $\widetilde{u}_{\widetilde{a},r_k} \to \widetilde{u}^{\h{.8pt}*}$ uniformly on $\overline{B^+_{r_0/8}(z_0)}$, which further infers $\widetilde{u}_{\widetilde{a},r_k} \to \widetilde{u}^{\h{.8pt}*}$ uniformly on $\overline{B^+_{1/2}} \setminus B^+_{1/4}$ as $z_0$ is arbitrary.\vspace{0.4pc}  
		
		\noindent{\bf Part 3. Energy minimizing property:} \sloppy In light of Item (3) in Lemma \ref{lem compact of rescaled map}, there is a $\sigma \in [\h{0.5pt}\frac{1}{4}, \frac{1}{2}\h{0.5pt}]$ and a subsequence, still denoted by $\big\{ \widetilde{u}_{\widetilde{a}, r_k}\big\}$, such that
		\begin{align}
			& \mathcal{E}^*_{\widetilde{u}_{\widetilde{a},r_k} - \widetilde{u}^*,\h{1pt}\p^+B^+_{\sigma }} + 
			\int_{\p^+B^+_{\sigma}}  \pig|\widetilde{u}^{\h{.8pt}*}
			-\widetilde{u}_{\widetilde{a},r_k}\pigr|^2 \longrightarrow 0 \h{15pt}\text{as $k \to \infty$.}
			\label{4777}
		\end{align}
		Meanwhile, for large $k$, the H\"{o}lder regularity of $\widetilde{u}_{\widetilde{a}, r_k}$ and $\widetilde{u}^*$ on $\overline{B^+_{1/2}}\setminus B^+_{1/4}$ infers that  $$\widetilde{u}_{\widetilde{a},r_k}(\cdot) \in T_{\varphi_k(\cdot)}\p\Omega  \h{15pt}\text{and}\h{15pt} \widetilde{u}^{\h{.8pt}*}(\cdot) \in T_{\varphi^{-1}(\widetilde{a})}\p\Omega  \h{20pt}\text{on $\p B_{\sigma}\cap\big\{y_3=0\big\}$.}$$ Recalling $\psi_*$, $\Pi^k $, $\Psi^{v,k}$ defined in the proof of Lemma \ref{lem compact of rescaled map}, Part 2.2, we introduce\begin{align}\label{def of Upsilon^(v,k)}
		    \Upsilon^{v,k}\left(y,z\right) := v(z) + \psi_*\big(v(z)\big)
		\left[\h{1pt} \Pi^k\left(y^0, \Pi\big(\widetilde{a}, v(z)\big) \right) 
		-\Pi\left(\widetilde{a}, v(z)\right)\h{1pt}\right]
		\end{align} and
		\begin{equation}
			v^{\gamma,k}(z) := 
			\begin{cases} 
				\dfrac{\sigma  - |z |}{(1-\gamma)\sigma } \h{1pt}
				\Upsilon^{\widetilde{u}^{\h{.8pt}*},\h{1pt}k}\left(z,\sigma \widehat{z}\h{1pt}\right)
				+ \dfrac{|z | -  \gamma\h{0.4pt}\sigma }{(1-\gamma)\sigma } \h{1pt}
				\Psi^{\widetilde{u}^{\h{.8pt}k}}\left(z,\sigma \h{0.2pt}\widehat{z} \h{1pt}\right)
				& \text{if } z \in B^+_{\sigma } \setminus B^+_{\gamma\h{0.2pt}\sigma }, \\[5mm]
				\Upsilon^{v^*,k}\left( z, \gamma^{-1}z\h{1pt}\right) 
				& \text{if } z \in B^+_{\gamma\h{0.2pt}\sigma }.
			\end{cases}
			\label{def. of u^rho 3}
		\end{equation} Here, $\gamma \in (\frac{1}{2}, 1)$. The map $\widetilde{u}_{\widetilde{a},r_k}$ is simply denoted by $\widetilde{u}^{\h{.8pt}k}$. $\Psi^{\widetilde{u}^{\h{.8pt}k}}$ is a simple notation for $\Psi^{\widetilde{u}^{\h{.8pt}k}, k}$. In addition, $v^*$ is a map arbitrarily chosen from  $\mathscr{F}_{\widetilde{u}^{\h{.8pt}*},\,\varphi^{-1}(\widetilde{a}),\,B_\sigma^+}$. The uniform convergence of  $\big\{\widetilde{u}_{\widetilde{a},r_k} \big\}$ obtained in Part 2 and arguments similar to those in the proof of Lemma \ref{lem compact of rescaled map},  Part 2.2, deduce that for large $k$, the normalized map of $v^{\gamma,k}$ is well-defined and satisfies $$\widehat{v^{\gamma,k}} \in \mathscr{F}_{\widetilde{u}^{\h{.8pt}k},\,\varphi_k,\,B_\sigma^+}  :=\Big\{v \in H^1(B_\sigma^+;\mathbb{S}^2):v=\widetilde{u}^{\h{.8pt}k}\h{5pt} \text{on $\p^+B_\sigma^+$ \h{5pt}and} \h{10pt}v\h{0.5pt}\big|_{\p^0B^+_\sigma}\in T_{\varphi_k(\cdot)}\p\Omega\h{1pt}\Big\}.$$
		
        
        We consider the energy of $\widehat{v^{\gamma, k}}$. Without ambiguity, we still denote by $(r,\xi,\theta)$ the spherical coordinates centered at the origin.
        
        First, we compute the following radial derivative at $z \in B^+_{\sigma} \setminus \overline{B^+_{\gamma\h{0.3pt}\sigma}}$:  
		\begin{align*}
			\p_r v^{\gamma,k} = \dfrac{\Psi^{\widetilde{u}^{\h{.8pt}k}}\left(z,\sigma\widehat{z}\h{1pt}\right)
			-\Upsilon^{\widetilde{u}^{\h{.8pt}*},k}\left(z,\sigma\widehat{z}\h{1pt}\right)}{(1-\gamma)\sigma } 
			+\dfrac{\sigma-|z|}{(1-\gamma)\sigma}\h{1.5pt}
			\p_{r}\left[	\Upsilon^{\widetilde{u}^{\h{.8pt}*},k}\left(z,\sigma\widehat{z}\h{1pt}\right)\h{1pt}\right]  
			+ \dfrac{|z| - \gamma\h{0.2pt}\sigma}{(1-\gamma)\sigma}\h{1.5pt}\p_{r}\left[	\Psi^{\widetilde{u}^{\h{.8pt}k}}\left(z,\sigma\widehat{z}\h{1pt}\right)\h{1pt}\right]. 
		\end{align*}
Using the definitions of $\Psi^{\widetilde{u}^{\h{.8pt}k}}$ and $\Upsilon^{\widetilde{u}^{\h{.8pt}*},k}$ in \eqref{def of Phi^(v,k)} and \eqref{def of Upsilon^(v,k)}, we have
        \begin{align*}
			\big|\h{1pt} \p_r v^{\gamma,k} \h{1pt} \big| &\,\leq\, \dfrac{\left|\Psi^{\widetilde{u}^{\h{.8pt}k}}\left(z,\sigma\widehat{z}\h{1pt}\right)
			-\Upsilon^{\widetilde{u}^{\h{.8pt}*},k}\left(z,\sigma\widehat{z}\h{1pt}\right)\h{1pt}\right|}{(1-\gamma)\sigma } \nonumber\\[1.5mm] 
			&\,+ \,\left|\h{1pt}\p_r z \cdot  
			\nabla_z \Pi^k\left(z^0, \Pi\big(\widetilde{a}, \widetilde{u}^{\h{.8pt}*}\left(\sigma\widehat{z}\h{1pt}\right)\big) \h{1pt} \right) \h{1pt} \right| + \left|\h{1pt}  
			\p_r z \cdot \nabla_z \Pi^k\left(z^0, \Pi^k\left(\left(\sigma\widehat{z}\h{1pt}\right)^0, \widetilde{u}^{\h{.8pt}k}\left(\sigma\widehat{z}\h{1pt}\right)\right) \right) \h{1pt}  \right|.
		\end{align*}
		By Item (1) in Lemma \ref{lem. regularity of proj. and dist}, it follows that
		\begin{align*}
			\left|   \Psi^{\widetilde{u}^{\h{.8pt}k}}\left(z,\sigma \widehat{z}\h{1pt}\right) - \Upsilon^{\widetilde{u}^{\h{.8pt}*},k}\left(z,\sigma \widehat{z}\h{1pt}\right) \h{1pt} \right| \,\lesssim_{\,\Omega}\,  &\left| \h{1pt}\widetilde{u}^{\h{.8pt}k}\left(\sigma \widehat{z}\h{1pt}\right) - \widetilde{u}^{\h{.8pt}*}\left(\sigma \widehat{z}\h{1pt}\right) \h{1pt}\right| + 
			\left|\h{1pt} \Pi^k\pigl(z^0, \Pi\pigl(\widetilde{a}, \widetilde{u}^{\h{.8pt}*}(\sigma \widehat{z}\h{1pt})\pigr) \pigr) 
			-\Pi\pigl(\widetilde{a}, \widetilde{u}^{\h{.8pt}*}(\sigma \widehat{z}\h{1pt})\pigr) \h{1pt} \right|  \\[1.5mm]
			& \h{-11pt}+\left|  
			\h{1pt}\Pi^k\left( z^0, \Pi^k\left( (\sigma \widehat{z}\h{1pt})^0, \widetilde{u}^{\h{.8pt}k}(\sigma \widehat{z}\h{1pt})\right) 
			\right)  
			- \Pi^k\left( (\sigma \widehat{z}\h{1pt})^0, \widetilde{u}^{\h{.8pt}k}(\sigma \widehat{z}\h{1pt})\right) \h{1pt} \right|\\[1.5mm]
			&\h{-22pt}\,\lesssim_{\,\Omega}\, r_k
			+ \left| \h{1pt}\widetilde{u}^{\h{.8pt}k}\left(\sigma \widehat{z}\h{1pt}\right) - \widetilde{u}^{\h{.8pt}*}\left(\sigma \widehat{z}\h{1pt}\right) \h{1pt}\right|.
		\end{align*} 
		Therefore, 
		$$
		\big|\h{1pt} \p_r v^{\gamma,k} \h{1pt}\big| 
		\,\lesssim_{\,\Omega}\, \frac{r_k
			+ \left| \h{1pt}\widetilde{u}^{\h{.8pt}k}\left(\sigma \widehat{z}\h{1pt}\right) - \widetilde{u}^{\h{.8pt}*}\left(\sigma \widehat{z}\h{1pt}\right) \h{1pt}\right|}{(1-\gamma)\sigma } \h{20pt}\text{on $B^+_{\sigma} \setminus \overline{B^+_{\gamma\h{0.3pt}\sigma}},$} 
            $$ 
            which, using \eqref{4777}, infers that
		\begin{align}
			\int_{B^+_{\sigma} \setminus B^+_{\gamma\h{0.2pt}\sigma}}\big|\h{1pt} \p_r v^{\gamma,k}\h{1pt}\big|^2
			\,\lesssim_{\,\Omega}\, \frac{r_k^2}{1-\gamma } + \frac{1}{1-\gamma} 
			\int_{\p^+B^+_{\sigma}}  \pig|\h{.5pt}\widetilde{u}^{\h{.8pt}k} - \widetilde{u}^{\h{.8pt}*}
			\h{.5pt}\pigr|^2
			\longrightarrow 0 \h{20pt}\text{as $k \to \infty$.}
			\label{4848}
		\end{align}
        
		Denoting by $\widetilde{U}^*$ the $0$-homogeneous map $\widetilde{u}^*\left(\sigma \widehat{z}\h{1pt}\right)$, we have the following estimate on $B^+_{\sigma} \setminus B^+_{\gamma\h{0.2pt}\sigma}$:
		\begin{align*}
			& \left|\h{1pt}\p_{\xi} \left[\h{1pt}	\Upsilon^{\widetilde{u}^{\h{.8pt}*},k}\left(z,\sigma \widehat{z}\h{1pt}\right)\right]\h{1pt}\right| \,\leq\,\left|\h{1pt}\p_{\xi}\widetilde{U}^{\h{.8pt}*} 
			+ \p_\xi \widetilde{U}^* \cdot \nabla \psi_* \h{1pt}\big|_{\widetilde{U}^*}
			\left[\h{1pt} \Pi^k\left( z^0, \Pi\pigl( \widetilde{a}, \widetilde{U}^{\h{.8pt}*}\pigr) \right) 
			- \Pi\pigl( \widetilde{a}, 
			\widetilde{U}^{\h{.8pt}*}\pigr)\h{1pt} 
			\right] \h{1pt}\right|\\[1.5mm]
			&\h{10pt} 
			+ \bigg|\h{1pt} \p_\xi z \cdot \nabla_{z}\Pi^k\left( z^0, p \right) +  \p_\xi \Pi\pigl( \widetilde{a}, 
			\widetilde{U}^{\h{.8pt}*}\pigr) \cdot \nabla_p \Pi^k\left( z^0, p \right) - \p_\xi \widetilde{U}^{\h{0.8pt}*} \cdot \nabla_q \Pi\big(\h{0.5pt}\widetilde{a}, q\big) \h{1pt}\bigg| 
			\,\lesssim_{\,\Omega}\, 
			r_k|z|+ \pigl|\h{1pt}\p_\xi \widetilde{U}^{\h{.8pt}*} \h{.5pt}\pigr|. 
		\end{align*}Here, we evaluate at $p=\Pi \big(\h{0.5pt}\widetilde{a}, \h{1pt}\widetilde{U}^{*} \big)$ and $q=\widetilde{U}^{\h{0.8pt}*}$. Similarly, \begin{align*}
			& \left|\h{1pt}\p_{\theta} \left[\h{1pt}	\Upsilon^{\widetilde{u}^{\h{.8pt}*},k}\left(z,\sigma \widehat{z}\h{1pt}\right)\right]\h{1pt}\right| 
			\,\lesssim_{\,\Omega}\, 
			r_k|z| \sin \xi + \pigl|\h{1pt}\p_\theta \widetilde{U}^{\h{.8pt}*} \h{.5pt}\pigr| \h{20pt}\text{on $B^+_{\sigma} \setminus B^+_{\gamma\h{0.2pt}\sigma}$}. 
		\end{align*}
		 If we denote by $\widetilde{u}^{\h{.8pt}k}$ the $0$-homogeneous map $\widetilde{u}^{\h{.8pt}k} \left(\sigma \widehat{z}\h{1pt}\right)$, then the same arguments for \eqref{4417}-\eqref{partial theta} yield the following:\begin{align*}&\left|\h{1pt}\p_\xi \left[\h{1pt}	\Psi^{\widetilde{u}^{\h{.8pt}k}}\left(z, \sigma\h{0.2pt}\widehat{z}\h{1pt}\right)\h{1pt}\right]\h{1pt}\right|
			\h{1pt}\lesssim_{\,\Omega\,} 
			r_k\left(\h{0.5pt} |z | + \sigma\right)+ \big|\h{1pt}\p_\xi \widetilde{U}^{\h{0.8pt}k}\h{.5pt}\big|,\nonumber\\[1.5mm]
            &\left|\h{1pt}\p_\theta \left[\h{1pt}	\Psi^{\widetilde{u}^{\h{.8pt}k}}\left(z, \sigma\h{0.2pt}\widehat{z}\h{1pt}\right)\h{1pt}\right]\h{1pt}\right|
			\h{1pt}\lesssim_{\,\Omega\,} 
			r_k\left(\h{0.5pt}|z| + \sigma\right) \sin \xi + \big|\h{1pt}\p_\theta \widetilde{U}^{\h{0.8pt}k}\h{.5pt}\big|, \h{25pt}\text{on $B^+_{\sigma} \setminus B^+_{\gamma\h{0.2pt}\sigma}$}. \end{align*}
		Hence, \eqref{4777} and the above estimates deduce that
		\begin{align}
			&\int_{B^+_{\sigma} \setminus B^+_{\gamma\sigma}}
			\dfrac{1}{r^2 }\left| \p_\xi v^{\gamma,k}\right|^2
			+	\dfrac{1}{r^2\sin^2\xi}\left| \p_\theta v^{\gamma,k}\right|^2 \,\lesssim_{\,\Omega}\, 	r_k^2 + \left(1 - \gamma\right) \sup_{k \h{1pt}\in\h{1pt}\mathbb N} \mathcal{E}^*_{\widetilde{u}_{\widetilde{a},r_k},\h{1pt}\p^+B^+_{\sigma }}\longrightarrow 0,
			\label{4887}
		\end{align} as we pass $k \to \infty$ and $\gamma \to 1^-$.

        On $B^+_{\gamma\h{0.2pt}\sigma}$, we observe that
		\begin{align*}
			\p_{z_j}\Big[\h{1pt}\Upsilon^{v^*,k}\left( z,\gamma^{-1}z\right)\h{.5pt}\Big] &=   \gamma^{-1}\p_{j}v^*\h{1pt}\Big|_{\gamma^{-1} z}+ \gamma^{-1} \p_{j}v^* \h{1pt}\Big|_{\gamma^{-1} z} \cdot \nabla_q \psi_*   
			\left( \Pi^k\left(z^0, p \right) 
			- p \right) + \psi_*
            \h{1pt}
			\p_{z_j} \Pi^k\pigl(z^0, p \pigr)  \nonumber\\[2mm]
			&+ \gamma^{-1}\psi_* \left[ \h{1pt} \left(\p_j v^* \Big|_{\gamma^{-1}z} \cdot \nabla_q \Pi \h{0.5pt}\pigl(\widetilde{a}, q\pigr) \right) \cdot \nabla_p \Pi^k\pigl(z^0, p \pigr) 
			- \p_j v^* \h{1pt}\Big|_{\gamma^{-1}z} \cdot \nabla_q \Pi \h{0.5pt}\pigl( \widetilde{a}, q \pigr) \h{1pt} \right].
		\end{align*}
        Here, $p$ is evaluated at $\Pi\pigl(\widetilde{a}, v^*\left(\gamma^{-1} z \right)\pigr)$. $\psi_* = \psi_*(q)$ with $q$ evaluated at $v^*\left(\gamma^{-1} z \right)$. Note that $$ \Pi^k \big(z^0, p\big) \to p = \Pi\pigl(\widetilde{a}, v^*\left(\gamma^{-1} z \right)\pigr) \h{20pt}\text{ uniformly on $B_{\gamma\h{0.2pt}\sigma}^+$, as $k \to \infty$.}$$
        In addition, we also find that  $$ \left(\p_j v^* \Big|_{\gamma^{-1}z} \cdot \nabla_q \Pi \h{0.5pt}\pigl(\widetilde{a}, q\pigr) \right) \cdot \nabla_p \Pi^k\pigl(z^0, p \pigr) \to \gamma \h{1pt}\p_{z_j} \Pi \pigl(\widetilde{a}, \Pi \pigl(\widetilde{a}, v^*\left(\gamma^{-1} z \right) \pigr) \pigr) = \p_j v^* \h{1pt}\Big|_{\gamma^{-1}z} \cdot \nabla_q  \Pi \pigl(\widetilde{a}, q \pigr), $$ uniformly on $B_{\gamma\h{0.2pt}\sigma}^+$, as $k \to \infty$. Therefore, 
		\begin{align*}
			\lim_{k \,\to\, \infty}
			\int_{B^+_{\gamma\sigma }}
			a^k_{ij} \h{2pt}\p_i v^{\gamma, k}  \cdot \p_j v^{\gamma, k}
			 =  \gamma^{-2}
			\int_{B^+_{\gamma \h{0.2pt}\sigma }} 
			\big|\h{1pt}\nabla v^*\h{0.5pt}\big|^2\big(\gamma^{-1}z\big)  \h{2pt}\mathrm d\h{.5pt}z =\gamma
			\int_{B^+_{\sigma }} \big|\h{0.5pt}\nabla v^* \h{0.5pt}\big|^2.		\end{align*}
		By this convergence, together with \eqref{4848}-\eqref{4887}, it follows that
		\begin{align*}
			\lim_{\gamma\to 1^-}	\lim_{k\to \infty}	\int_{B^+_{\sigma } }a^k_{ij}
			\h{2pt}\p_{y_i} \widehat{v^{\gamma,k}}\cdot
			\p_{y_j} \widehat{v^{\gamma,k}}
			=	\int_{B^+_{\sigma }} 
			\big|\h{0.5pt}\nabla v^* \h{0.5pt}\big|^2.
		\end{align*}
		Since both $\widetilde{u}^{\h{.8pt}k}$ and $\widehat{v^{\gamma,k}}$ belong to $\mathscr{F}_{\widetilde{u}^{\h{.8pt}k},\,\varphi_k,\,B_\sigma^+}$, the desired result is obtained by the minimality of $\widetilde{u}^{\h{.8pt}k}$ and Fatou's lemma. 
	\end{proof}

	Following the standard dimension reduction arguments and higher-order estimates, Item (3) in Theorem \ref{thm main result} follows. 

\subsection{Singularity structure}\label{sin str}
 
	We now prove Item (4) in Theorem \ref{thm main result}, with regard to the structure of the boundary singularities.

    Suppose $u^{*}$ is a $0$-homogeneous map on $B^+_1$ minimizing the Dirichlet energy over  \[
	  \mathscr F_{u^{*}, \h{1pt}B^+_1}
	  :=
	  \Bigl\{
	  v\in H^1(B_1^+; \mathbb{S}^2):
	  v=u^{*}\ \text{on }\partial^+B_1^+,\;
	  v_3=0\ \text{on }\partial^0B_1^+
	  \Bigr\}.
	  \] We extend $u^{*}_1$, $u^{*}_2$ evenly, and $u^{*}_3$ oddly with respect to $y_3$ and still denote the extension by $u^{*}$. Then we have 
      \begin{lem}\label{lem deg<4} The topological degree of $u^{*} : \p B_1 \to \mathbb S^2$ satisfies $ \big|\deg u^{*} \big| \leq 2$.
        \end{lem} 
\begin{proof} We divide the proof into 3 steps. \\[2mm]   \noindent{\bf Step 1. Energy comparison and stability inequality:}
Denote by $\widetilde \Pi_{e_1}$ the stereographic projection from $e_1 := (1, 0, 0)^\top$. That is  
\begin{align}\label{def stereo proj e_1}
	\widetilde \Pi_{e_1}(z) := \frac{1}{|z|^2 + 1}\left(\h{1pt}|z|^2-1,2 z_1,2 z_2\h{0.5pt}\right)^\top \qquad \text{for $z=z_1+i z_2$.}
\end{align}
	Moreover, we define $\widetilde{\Pi}_{e_1}(\infty) := e_1$. Let $$\mathbb{H}:= \big\{z\in\mathbb{C}:z_2>0 \h{1pt}\big\} \cup \big\{\infty\big\} \h{15pt}\text{and}\h{15pt} \p \mathbb{H}:= \big\{z\in\mathbb{C}:z_2=0 \h{1pt}\big\}\cup \big\{\infty\big\}.$$ Then, it turns out that
	 	\begin{align}\label{inclusion}
	 	\widetilde \Pi_{e_1}\big(\big\{z_2 > 0 \big\}\big)=\mathbb{S}^2\cap \big\{y_3>0\big\}
	 	\h{15pt} \text{and} \h{15pt}
	 	\widetilde \Pi_{e_1}(\p \mathbb{H})=\mathbb{S}^2\cap \big\{y_3=0\big\}.
	 	\end{align}  
Given a small parameter \(\varepsilon\in(0,1)\), we fix a smooth function $\theta$ on $[\h{0.2pt}0, 1\h{0.2pt}]$ that satisfies
	 	\[
	 	\theta \equiv 0 \quad \text{on}\h{3pt}  [\h{0.5pt}0,\varepsilon\h{0.5pt}],\qquad \theta > 0 \quad
	 	 \text{on} \h{3pt}(\h{0.5pt}\varepsilon,1\h{0.5pt}], \qquad \theta(1)=1.
	 	\] Letting $R\in \mathrm{SO}(3)$ be an arbitrary constant rotation about the $y_3$-axis, we then define the comparison map $\phi_R$ as follows:
	 	\[
	 	\phi_R(x)
	 	:=
	 	\widetilde \Pi_{e_1}\left(
	 	\frac{1}{\theta\big(\h{0.5pt}|x|\h{0.5pt}\big)}\,\widetilde \Pi_{e_1}^{-1}\circ R\h{.5pt}u^{*}\!\left(x\right)
	 	\right)
	 	\qquad \text{for }x\in B_1.
	 	\]
	 	When $\theta\big(\h{0.5pt}|x|\h{0.5pt}\big)=0$, we set $\phi_R(x)=\widetilde \Pi_{e_1}(\infty)=e_1$. \vspace{0.2pc}
        
        Since \(\theta(1)=1\), it holds
	 	$\phi_R = R \h{.5pt} u^{*}$ on \(\partial^+ B_1^+\). By the second equality in \eqref{inclusion}, the image of $
	 	\widetilde \Pi_{e_1}^{-1}\circ R \h{0.5pt} u^{*}$ is equal to $\p \mathbb{H}$ when the map is restricted on $\p^0 B^+_1 \setminus \big\{0\big\}$. Still by the second equality in \eqref{inclusion}, the image of $\phi_R$ is equal to $\mathbb{S}^2\cap \big\{y_3=0 \big\}$ when $\phi_R$ is restricted on $\p^0 B_1^+ \setminus \big\{0\big\}$, verifying that $\phi_R\in \mathscr F_{R \h{0.5pt}u^{*}, \h{1pt}B_1^+}$. \vspace{0.2pc} 
        
       Define $m_*:= \big|\deg u^{*} \big|$ and $f_R:= \widetilde \Pi_{e_1}^{-1}\circ R\h{.5pt}u^{*} \circ \widetilde \Pi_{e_1}$. By comparing the energies between $R u^*$ and $\phi_R$, the following stability inequality holds: \begin{align}\label{sta ine}
      \left(\dfrac{\pi m_*}{2}\right)^{\frac{1}{2}} \h{1pt}\leq\h{1pt}  \h{1pt}F_{f_R}, \h{15pt}\text{where} \h{3pt} F_{f_R} := \int_0^1\!\mathrm d\h{.5pt}s\left\{\h{1pt}\int_{\mathbb{R}^2}
	\dfrac{\big|\h{1pt}f_R(z) \h{1pt}\big|^2\,\,\mathrm d\h{.5pt}z_1\,\mathrm d\h{.5pt}z_2}{\big(s^2+ \big|\h{1pt}f_R(z) \h{1pt}\big|^2\h{1pt}\big)^2 \left(1+|z|^2\right)^2}\h{1pt}\right\}^{\!\frac{1}{2}}.
       \end{align} See \cite[(7.44)]{BCL86} for the details of the proof of this stability inequality. Note that we fix $\theta$ as in \cite{BCL86} specifically and take $\varepsilon$ to $0$. \vspace{0.4pc}
       
\noindent\textbf{Step 2. Integration over $\mathrm{SO}(2)$:} We identify the group of all the rotations about the $y_3$-axis with $\mathrm{SO}(2)$. The Haar measure on $\mathrm{SO}(2)$ is denoted by $\,\mathrm d\h{.5pt}h$. Integrate the inequality in \eqref{sta ine} over $\mathrm{SO}(2)$ and apply Cauchy's inequality. It can be deduced that
\begin{equation*}\label{eq:minkowski}
 \left(\dfrac{\pi  m_* }{2}\right)^{\frac{1}{2}}
 \,\le\, \int_{\mathrm{SO}(2)} F_{f_R}\,\mathrm d\h{.5pt}h_R
\le\int_0^1 \mathrm d\h{.5pt}s \!\, \left\{\int_{\mathbb{R}^2}
\frac{\,\mathrm d\h{.2pt}z_1\,\mathrm d\h{.2pt}z_2}{\left(1+|z|^2\right)^2}\h{1pt}\mathcal I_1\pigl(s, u^{*} \circ \widetilde \Pi_{e_1}(z)\pigr)\right\}^{\!\frac{1}{2}},
\end{equation*}
where
\begin{equation}\label{eq:I-def}
\mathcal I_1(s,b):=\int_{\mathrm{SO}(2)} \dfrac{\big|\h{1pt}\widetilde \Pi_{e_1}^{-1}\big(R\h{0.5pt}b\big)\h{0.5pt} \big|^2}
{\pig[s^2+\big|\h{1pt}\widetilde \Pi_{e_1}^{-1}\big(R\h{0.5pt}b\big)\h{0.5pt}\big|^2\pigr]^2} \,\,\mathrm d\h{.5pt}h_R\qquad
\text{for  $b\in\mathbb{S}^2$ and $s\in(0,1)$.}
\end{equation}
Applying Cauchy's inequality again, we then obtain
	\begin{align}\label{3066}
m_* 
		&\leq
		\frac{4}{\pi}\int_0^1s^{\frac{1}{2}} \h{2pt}\mathrm d\h{.5pt}s\int_{\mathbb{R}^2}
		\frac{\,\mathrm d\h{.5pt}z_1\,\mathrm d\h{.5pt}z_2}{\left(1+|z|^2\right)^2} \h{1.5pt}\mathcal I_1\pigl(s, u^{*} \circ\widetilde \Pi_{e_1}(z) \pigr).
	\end{align}
By \eqref{def stereo proj e_1}, it satisfies $\big|\h{1pt}\widetilde \Pi_{e_1}^{-1}(y)\h{0.5pt}\big|^2=\frac{1+y_1}{1-y_1}$ for any $y\in\mathbb{S}^2$.
Therefore, $\mathcal I_1(s,b)$ is equal to
\begin{align*}
\int_{\mathrm{SO}(2)}
\frac{1- \big(R\h{0.2pt}b\big)_1^2}{\left[\h{1pt}s^2+1+(1-s^2)\big(R\h{0.2pt}b\big)_1\right]^2}
\,\mathrm d\h{.5pt}h_R =
\frac{1}{\pi}\int_0^{\pi}
\frac{1-\left(1-b_3^2\h{0.5pt}\right)\cos^2\theta}{\left[\h{1pt}s^2+1 +\left(1-s^2\right)\sqrt{1-b_3^2}\h{1pt}\cos\theta\h{1pt}\right]^2}
\,\mathrm d\h{.5pt}\theta.
\end{align*}  Standard integrals yield the closed form: 
\begin{align}\label{eq:I-closedform}
\mathcal I_1(s,b)=\widetilde{\mathcal I}_1\left(s, b_3 \h{1pt}\right):=
	\frac{1}{\left(1-s^2\right)^2}
	\left[
	\frac{
		2\left(1+s^2\right)\left(2s^2+(1-s^2)^2\h{1pt}b_3^2\h{1pt}\right)
	}{
		\left(4s^2+(1-s^2)^2\h{1pt}b_3^2\h{1pt}\right)^{\frac{3}{2}}
	}
	-1
	\right].
\end{align}
We claim that 
\begin{align}\label{3097}
		\max_{b_3\h{1pt}\in \h{1pt}[-1,1\h{0.5pt}]} \h{2pt}\int_0^1 \widetilde{\mathcal I}_1\left(s, b_3 \right) s^{\frac{1}{2}} \,\mathrm d\h{.5pt}s
		\h{1.5pt}\leq\h{1.5pt}
		\frac14+\frac{\pi}{8}.
\end{align}
     Applying the last estimate to \eqref{3066} implies
	\[
	\begin{aligned}
		m_* 
		&\leq
		\frac{4}{\pi}  
		\left(\frac14+\frac{\pi}{8}\right)
		\int_{\mathbb{R}^2}
		\frac{\,\mathrm d\h{.5pt}z_1\,\mathrm d\h{.5pt}z_2}{(1+|z|^2)^2}
			= 1 + \frac{\pi}{2},
	\end{aligned}
	\] 
	which proves the lemma. \vspace{0.4pc}

\noindent{\bf Step 3. Proof of \eqref{3097}:} Define
		\[
		A_{b_3,\h{1pt}s}:=\frac{|\h{0.5pt}b_3\h{0.5pt}|\left(1-s^2\right)}{2s}.
		\] 
		A direct computation shows that $\widetilde{\mathcal I}_1(s, b_3)
			-\widetilde{\mathcal I}_1(s,0)$ is equal to 
		\[
		\begin{aligned}
			\frac{1}{\left(1-s^2\right)^2}
			\left[
			\frac{
				2\left(1+s^2\right)\cdot 2s^2\left(1+2A_{b_3,s}^2\right)
			}{
				\left(4s^2\left(1+A_{b_3,\h{1pt}s}^2\right)\right)^{\frac{3}{2}}
			}
			-\frac{2 \left(1+s^2\right)\left(2s^2\right)}{\left(4s^2\right)^{\frac{3}{2}}}
			\right] =
			\frac{1+s^2}{2s\left(1-s^2\right)^2}
			\Big( g\big(A_{b_3,s}\big)
			-1
			\Big), 
		\end{aligned}
		\]
	where $
		g(x):=
		\dfrac{1+2x^2}{(1+x^2)^{\frac{3}{2}}}$. Thus,
		\[
		\begin{aligned}
			&\int_0^1
			\left[\, \widetilde{\mathcal I}_1(s, b_3 )-\widetilde{\mathcal I}_1(s,0)
			\right] s^{\frac{1}{2}}
			\h{1.5pt}\mathrm d\h{.5pt}s 
			=
			\int_0^1
			\frac{(1+s^2)}{2s^{\frac{1}{2}}\left(1-s^2\right)^2}
			\Big(g\big(A_{b_3,s}\big)-1\Big)
			\,\mathrm d\h{.5pt}s. 
		\end{aligned}
		\]
	Note that
		\[
		\frac{\mathrm d\h{.5pt}A_{b_3,s}}{\mathrm d\h{.5pt}s}
		=
		-\frac{|b_3|\left(1+s^2\right)}{2s^2}\,\leq\, 0 \h{20pt}\text{which induces} \h{6pt} \frac{(1+s^2)}{2s^{\frac{1}{2}}\left(1-s^2\right)^2}\,\mathrm d\h{.5pt}s
		=
		-\frac{|b_3|}{4}\frac{s^{-\frac{1}{2}}}{A_{b_3,s}^2}\,\mathrm d\h{.5pt}A_{b_3,s}.
		\]Use this change of variable and let $B_{b_3,r}:=\dfrac{\sqrt{r^2+b_3^2}-r}{|b_3|}$. It turns out that 
		\[
		\begin{aligned}
			&\int_0^1
			\left[\, \widetilde{\mathcal I}_1(s, b_3 )-\widetilde{\mathcal I}_1(s,0)
			\right] s^{\frac{1}{2}}
			\h{1.5pt}\mathrm d\h{.5pt}s 
			=	\frac{|b_3|}{4}
			\int_0^\infty
			\frac{g(r)-1}{r^2}
			B_{b_3,r}^{- \frac{1}{2}} 
			\,\mathrm d\h{.5pt}r. 
		\end{aligned}
		\]
        Since we have 
        $$\dfrac{\mathrm d\h{.5pt}}{\mathrm d\h{.5pt}r}
        \frac{\sqrt{1+r^2}-1}{r\sqrt{1+r^2}}
        =\frac{1}{r^2}
		\left[
		\frac{1+2r^2}{(1+r^2)^{\frac{3}{2}}}
		-1
		\right]
        =\frac{g(r)-1}{r^2},
		$$ 
        integrating by parts yields
        \begin{align*}
			&\int_0^1
			\left[\, \widetilde{\mathcal I}_1(s, b_3 )-\widetilde{\mathcal I}_1(s,0)
			\right] s^{\frac{1}{2}}
			\h{1.5pt}\mathrm d\h{.5pt}s =	- \frac{| b_3 |^{\frac{1}{2}}}{8}
            \int_0^\infty
			 \frac{\sqrt{1+r^2}-1}{r\sqrt{1+r^2}}\h{1pt}
			 \left(\frac{1}{\sqrt{r^2 + b_3^2}} + \frac{r}{r^2 + b_3^2}\right)^{\frac{1}{2}}\h{1pt}\mathrm d r.
		\end{align*}
        It turns out that
        $$\int_0^1
			\widetilde{\mathcal I}_1(s, b_3) \h{1pt} s^{\frac{1}{2}}   
			\h{1.5pt}\mathrm d\h{.5pt}s
            \h{1pt}\leq\h{1pt} 
           \int_0^1
			\widetilde{\mathcal I}_1(s,0) \h{1pt} s^{\frac{1}{2}}  
			\h{1.5pt}\mathrm d\h{.5pt}s 
            \h{1pt}=\h{1pt}	\frac{1}{2}\int_0^1
			\frac{\mathrm d\h{.5pt}s}{s^{\frac{1}{2}}(1+s)^2} =\frac14+\frac{\pi}{8}. $$
    The proof is complete. 
\end{proof}

To rule out the degree 2 bubbles and prove the radial symmetry of the degree 1 bubble, we need the following first-order expansion of the energy $\mathcal E^*_{U_{\lambda p},\h{0.5pt}B_1}$ with respect to the parameter $\lambda$. Here, $U_{\lambda p}$ is a perturbation of $u^*$ constructed by translating the singularity of $u^*$ away from the origin.     
\begin{lem}[\bf Moving-center variation]\label{lem energy of moving centre}  For any $ p \in \mathbb{S}^2$, $\lambda\in(\h{0.5pt}0, \frac{1}{2}\h{1pt}]$ and $x\in \overline{B_1} \setminus \big\{\lambda p\big\}$, we define $\Phi_{\lambda p}(x)$ to be the unique point on $\mathbb S^2$, taking the form
	\[
	\Phi_{\lambda p}(x):=\lambda p+t\left(x-\lambda p \right),\quad \text{for some $t>0$}.
	\] Then the map $$U_{\lambda p}(x):=u^{*}\circ\Phi_{\lambda p}(x),\quad\quad\text{where $x\in B_1 \setminus \big\{\lambda p\big\}$}$$ 
    satisfies $U_{\lambda p}=u^{*}$ on $\p B_1$. In addition, we have  
	\begin{align}\label{energy dec}
	\Big|\h{1pt} \mathcal{E}^*_{U_{\lambda p},\h{0.5pt}B_1} - \mathcal{E}^*_{u^{*},\h{0.5pt}B_1} 
    +\lambda \h{1pt} p\cdot N(u^{*})\h{1pt}\Big|
	\h{1.5pt}\lesssim \h{1.5pt} \lambda^2 \h{1pt}\mathcal{E}^*_{u^{*},\h{0.5pt}B_1}. 
	\end{align}
Here, $N(u^{*}):= \displaystyle \int_{\mathbb{S}^2} \big|\h{1pt}\nabla^{\textup{tan}} u^{*}(y) \h{1pt}\big|^2\, y \,\, \mathrm d\h{.5pt}\mathscr{H}^2_y$ is the center of mass of $\big|\h{.5pt}\nabla^{\textup{tan}} u^{*}\h{.5pt}\big|^2$.
 \end{lem}
The above lemma shows that if the center of mass is non-zero, then there exists a $p \in \mathbb R^3$ such that we can strictly reduce the energy by translating the singularity along the direction $p$. The proof of this lemma is simply to expand the energy $\mathcal{E}^*_{U_{\lambda p},\h{0.5pt}B_1}$ with respect to $\lambda$. We omit it here. The reader may refer to \cite[B. Proof of Theorem 7.3]{BCL86}, in which a special case where $p = e_3$ is discussed.\vspace{0.2pc}

In the next two lemmas, we prove $|\deg u^*| \neq 2$. Firstly, we verify two standard forms for $u^*$, up to rotations.

            \begin{lem}\label{lem: deg 2 explicit form}
        If $\deg u^* = 2$, then either one of the following holds:\begin{itemize}
            \item[$\mathrm{(1).}$] There exist $c\in \mathbb{C}$ with $|\h{0.5pt}c\h{0.5pt}| \neq 1$ and a rotation $R\in \mathrm{SO}(3)$ about the $y_3$-axis such that
        
        $$\widetilde \Pi_{-e_3}^{-1}\circ R\h{.5pt}u^* \circ \widetilde \Pi_{-e_3}(z) 
= \dfrac{z^2-c^2}{1-\overline{c}^2 z^2} \h{20pt}\text{for any $z \in \mathbb C \cup\{\infty\}$.}$$ Here, $\widetilde \Pi_{-e_3}$ is the stereographic projection with respect to the south pole $- e_3$. \vspace{0.2pc}
\item[$\mathrm{(2).}$] There exist $\alpha_0, \gamma_0 \in \mathbb{R}$ with $\alpha_0 \neq 0$ and a rotation $R\in \mathrm{SO}(3)$ about the $y_3$-axis such that \begin{align*} 
\widetilde \Pi_{-e_3}^{-1}\circ R\h{.5pt}u^* \circ \widetilde \Pi_{-e_3}(e^{i \h{0.5pt} \gamma_0} z) = \dfrac{w(z) - i \h{0.5pt} \alpha_0}{w(z) + i \h{0.5pt}\alpha_0},\h{20pt}\text{where $w(z) := z + \frac{1}{z}$.}
\end{align*}
        \end{itemize} 
      \end{lem}
      \begin{proof}The proof is divided into 4 steps. \vspace{0.2pc}

      \noindent\textbf{Step 1.} We claim that the center of mass $N(u^*) = 0$. Indeed, using the symmetry of $u^*$ on $B_1$, we know that $N_3(u^*) = 0$. If $N(u^*) \neq 0$, then the normalized vector $p = \widehat{N(u^*)}$ must be on the equatorial plane. Therefore, $U_{\lambda p}$ defined in Lemma \ref{lem energy of moving centre} has the same symmetry as $u^*$.  
      Since $u^*$ is a Dirichlet energy minimizer in the configuration space $\mathscr F_{u^{*}, \h{1pt}B^+_1}$, it induces that $ \mathcal{E}^*_{u^{*},\h{0.5pt}B_1} \leq \mathcal{E}^*_{U_{\lambda p},\h{0.5pt}B_1}$ for any $\lambda \in (0, \frac{1}{2} \h{0.5pt}]$. However, this contradicts \eqref{energy dec} for sufficiently small $\lambda$.   \vspace{0.2pc}

      \noindent\textbf{Step 2.} Note that $u^*$ is a harmonic map from $\mathbb S^2$ to $\mathbb S^2$. Moreover, $u^*$ maps the equator to itself. We can find two complex numbers $a_1$ and $a_2$ and a planar rotation $R \in \mathrm{SO}(3)$ about the $y_3$-axis such that 
\begin{align}\label{3290}
\widetilde \Pi_{-e_3}^{-1}\circ R\h{.5pt}u^* \circ \widetilde \Pi_{-e_3}(z) 
= B(z) := \dfrac{z-a_1}{1-\overline{a_1}z}\h{1pt}
\dfrac{z-a_2}{1-\overline{a_2}z} \h{20pt}\text{for any $z \in \mathbb C \cup\{\infty\}$.}
\end{align}
Through the change of variables, it follows from Step 1 that  
\begin{align} \label{3275 1}
N_j(R\h{0.5pt}u^{*})
=
16\int_{\mathbb R^2}
\dfrac{z_j}{1+|z|^2} \h{2pt} \dfrac{ |\h{.5pt}B'\h{.5pt}|^2}{\big(1+  |\h{.5pt}B \h{.5pt}|^2\h{1pt}\big)^2}   
\,\,\mathrm d\h{.5pt}z_1\,\mathrm d\h{.5pt}z_2=0, \h{15pt} j = 1, 2.
\end{align}Next, we identify all such $a_1$, $a_2 \in \mathbb{C}$. \vspace{0.4pc}

\noindent \textbf{Step 3.} We prove $a_1 + a_2 = 0$, using \eqref{3275 1} and the assumption that $|\h{0.5pt} a_1 a_2 \h{0.5pt}| \neq 1$. \vspace{0.2pc}

Given \(z \in \mathbb C\), we define 
\begin{align}\label{ito fun}\iota(z):=\dfrac{\pigl(S - P \h{0.5pt}\overline{S} \h{0.5pt}\pigr)
- \left(1-|P|^2\right)z}
{\big(1-|P|^2\big)
- \pigl(\overline{S} - \overline{P} \h{0.5pt}S\pigr)z}, \h{20pt}\text{where $P = a_1 a_2$, $S = a_1 + a_2$.}\end{align}
If $$\Lambda
:=\dfrac{S - P \h{.5pt}\overline{S}}
{1- |\h{1pt}P\h{1pt}|^2}$$ has modulus $1$, then $\left(1-|a_1|^2\right)\left(1-|a_2|^2\right)\big|\h{0.5pt}1-a_1\overline{a_2}\h{0.5pt}\big|^2=0$. Since $|\h{0.5pt}a_1a_2\h{0.5pt}|\neq 1$, we have $|\h{0.5pt}a_1\h{0.5pt}|=1$ or $|\h{0.5pt}a_2\h{0.5pt}|=1$. $B(z)$ is a degree-one Möbius map or a constant if $|\h{0.5pt}\Lambda\h{0.5pt}| = 1$. However, this is impossible. Therefore, $|\h{0.5pt}\Lambda\h{0.5pt}|\neq1$ and $\iota(z)$ is not a constant map. Furthermore, direct computations yield $B(\iota(z))=B(z)$, which induces $
B'\left(\iota(z)\right)\iota'(z)=B'(z)$.\vspace{0.2pc}

Changing the variables by letting $w=\iota(z)$ infers
\[
\int_{\mathbb R^2}\frac{z}{1+|z|^2}
\h{1pt}\dfrac{\big|\h{1pt}B'(z) \h{1pt}\big|^2}{\big(1+ \big|\h{1pt}B(z) \h{1pt}\big|^2\big)^2}
\, \mathrm d z_1 \mathrm d z_2
=
\int_{\mathbb R^2}
\frac{\iota^{-1}(w)}{1+ \big|\h{0.5pt} \iota^{-1}(w) \h{0.5pt}\big|^2} \h{2pt}
\dfrac{ \big|\h{1pt}B'(w) \h{1pt}\big|^2}{\big(1+ \big|\h{1pt}B(w)\h{1pt}\big|^2 \big)^2}
\, \mathrm d w_1 \mathrm d w_2.
\]
Note that $\iota\circ\iota(z)=z$. Therefore, 
\begin{align}\label{3315}
\int_{\mathbb R^2}\frac{z}{1+|z|^2}
\h{2pt}\dfrac{ \big|\h{1pt}B'(z)\h{1pt}\big|^2}{\big(1+ \big|\h{1pt}B(z)\h{1pt}\big|^2\h{0.5pt}\big)^2}
\, \mathrm d z_1 \mathrm d z_2
=
\int_{\mathbb R^2}
\frac{\iota(w)}{1+|\h{1pt}\iota(w)\h{1pt}|^2} \h{2pt}
\dfrac{\big|\h{1pt}B'(w) \h{1pt}\big|^2}{\big(1+ \big|\h{1pt}B(w)\h{1pt}\big|^2\big)^2}
\, \mathrm d w_1 \mathrm d w_2.
\end{align}From \eqref{3275 1}-\eqref{3315}, it turns out that
\begin{align}\label{com cen}
N_1(R u^{*})
=
8\int_{\mathbb R^2}
\textup{Re} \left(\dfrac{z}{1+|z|^2}   
+\dfrac{\iota(z)}{1+|\h{1pt}\iota(z)\h{1pt}|^2}   \right)
\dfrac{ \big|\h{1pt}B'(z)\h{1pt}\big|^2}{\big(1+\big|\h{1pt}B(z)\h{1pt}\big|^2\big)^2}   
\,\mathrm d\h{.5pt}z_1\,\mathrm d\h{.5pt}z_2 = 0.
\end{align}

For any $\gamma \in \mathbb R$, it satisfies
\begin{align}\label{rotation sy}e^{-2 \h{.5pt}i\h{.5pt}\gamma} B\big(e^{i\h{0.5pt}\gamma}z)
= e^{-2\h{.5pt}i\h{.5pt}\gamma}\h{2pt}\dfrac{e^{i\h{.5pt}\gamma}z-a_1}{1-\overline{a_1}\h{1pt}e^{i\h{.5pt}\gamma}\h{.2pt}z}
\h{2pt}\dfrac{e^{i\h{.5pt}\gamma}z-a_2}{1-\overline{a_2}\h{1pt}e^{i\h{.5pt}\gamma}\h{.2pt}z}
=\dfrac{z-e^{-i\h{.5pt}\gamma}\h{.5pt}a_1}{1-\overline{e^{-i\h{.5pt}\gamma}\h{.5pt}a_1}\h{.5pt} z}
\h{2pt}\dfrac{ z-e^{-i\h{.5pt}\gamma}\h{.5pt}a_2}{1-\overline{e^{-i\h{.5pt}\gamma}\h{.5pt}a_2}\h{.8pt} z}.\end{align}
Due to \eqref{3275 1}, the condition $ N( R u^{*})=0$ is invariant under complex rotation. By the last equality and this property of invariance, we can assume that  $$\Lambda
=\dfrac{S - P \h{.5pt}\overline{S}}
{1- |\h{1pt}P\h{1pt}|^2}
=\dfrac{a_1+a_2- a_1a_2\h{1pt}\big(\h{0.5pt}\overline{a_1}+\overline{a_2}\big)\h{0.5pt}}
{1-|\h{0.5pt}a_1 a_2\h{0.5pt}|^2}
\geq 0. $$
A direct computation shows that
\begin{align*}
\textup{Re}\left(
\frac{z}{1+|z|^2}
+
\frac{\iota(z)}{1+|\h{1pt}\iota(z)\h{1pt}|^2}
\right)
=
\textup{Re}\left(
\frac{z}{1+|z|^2}
+
\dfrac{\frac{\Lambda-z}{1-\Lambda z}}{1+\left|\frac{\Lambda-z}{1-\Lambda z}\right|^2}
\right)
=
\frac{
	\Lambda \h{1pt}\big|\h{.5pt}1-z^2 \h{.5pt}\big|^2
}{
	\big(1+|z|^2\big)
	\big(\h{1pt}
	|\h{0.5pt}1-\Lambda z\h{0.5pt}|^2+ |\h{0.5pt}\Lambda-z\h{0.5pt}|^2
	\big)
}.
\end{align*}
By this equality and \eqref{com cen}, we have $\Lambda = 0$, which implies $a_1 + a_2 = 0$. The proof of this step is complete by setting $a_1 = c$ for some $c \in \mathbb C$. Note that $|\h{0.5pt}c\h{0.5pt}|\neq 1$. Otherwise, $B(z)$ is constant.\vspace{0.4pc}

\noindent\textbf{Step 4.} Suppose that $|\h{0.5pt} a_1 a_2 \h{0.5pt}| = 1$. Using \eqref{3290} and \eqref{rotation sy}, we can find another planar rotation $R \in \mathrm{SO}(3)$ about the $y_3$-axis and $\gamma_0 \in \mathbb R$ such that \begin{align}\label{tild B} 
\widetilde \Pi_{-e_3}^{-1}\circ R\h{.5pt}u^* \circ \widetilde \Pi_{-e_3}(e^{i \h{0.5pt}\gamma_0}z) 
= e^{- 2\h{0.5pt}i\h{0.5pt}\gamma_0}B(e^{i \h{0.5pt} \gamma_0} z) = \widetilde{B}(z) := \dfrac{z-b_1}{1-\overline{b_1}z}\h{1pt}
\dfrac{z-b_2}{1-\overline{b_2}z},\h{15pt}\text{where $b_1 b_2 = 1$.}
\end{align} We write $b_1 = r_0\h{0.5pt}e^{i\h{0.5pt}\theta_0}$ and $b_2=r_0^{-1}e^{-i\h{0.5pt}\theta_0}$ for some $r_0>0$ and $\theta_0 \in[\h{0.5pt}0,2\pi\h{0.5pt})$. Thus, \eqref{3275 1} induces \begin{align}\label{cen 0}
	\int_{\mathbb R^2}
\dfrac{\alpha_0^2 \h{1.5pt}\sin^2 \theta_0 \h{1pt}\big|\h{0.5pt}1-z^2\h{0.5pt}\big|^2}
{\Big[\h{1pt} \big|z^2- \beta_0 \cos \theta_0 \h{1pt} z +1\big|^2 +\alpha_0^2 \h{1pt}\sin^2 \theta_0 \h{1pt}|z|^2 \h{1pt}\Big]^2} \h{2pt}   
\dfrac{z_j}{1+|z|^2}\,\,\mathrm d\h{.5pt}z_1\,\mathrm d\h{.5pt}z_2 = 0 \h{15pt}\text{for $j = 1, 2$.}
\end{align}
Here, $\alpha_0 = r_0 - r_0^{-1}$ and $\beta_0 = r_0 + r_0^{-1}$. Defining 
$$D_\pm(z):=  \big|z^2 \pm \beta_0\h{.5pt}\cos \theta_0\h{.5pt} z +1\big|^2 
+ \alpha_0^2\h{0.5pt}\sin^2 \theta_0 \h{0.5pt}|z|^2,$$ we then rewrite \eqref{cen 0} as follows: 
\begin{align*}
 \alpha_0^2\h{0.5pt}\sin^2 \theta_0\int_{\mathbb R^2}
\left(\dfrac{1}{D_-(z)^2}
-\dfrac{1}{D_+(z)^2}\right)   
\dfrac{z_j \h{0.5pt} \big|1-z^2\big|^2}{1+|z|^2}\,\mathrm d\h{.5pt}z_1\,\mathrm d\h{.5pt}z_2 = 0 \h{20pt}\text{for $j = 1, 2$.}
\end{align*}Direct computation shows that
$$ \dfrac{1}{D_-(z)^2}
-\dfrac{1}{D_+(z)^2}=4\beta_0\cos \theta_0 \h{1pt}z_1 \big(1+|z|^2\big) \h{1.5pt} \dfrac{D_+(z)+D_-(z)}{D_+(z)^2D_-(z)^2}.$$ Taking $j = 1$ then yields \begin{align*} 
\beta_0\cos\theta_0 \left(\alpha_0 \sin\theta_0 \right)^2
\int_{\mathbb R^2}
z_1^2 \, \big| 1-z^2 \big|^2 \, \frac{
D_+(z)+D_-(z) 
}
{D_+(z)^2D_-(z)^2}
\,\mathrm d z_1\,\mathrm d z_2 = 0,\end{align*} which implies $r_0 = 1$, $\sin \theta_0 = 0$, or $\cos \theta_0 = 0$. $r_0$ cannot be $1$. Otherwise, $b_1$ and $b_2$ are on the unit circle and $\widetilde{B}$ is a constant mapping. $\sin \theta_0$ cannot be 0. Otherwise, $\widetilde{B}$ is also a constant mapping by \eqref{tild B}. Hence, we can only have $\cos \theta_0 = 0$. The proof is complete.
 \end{proof}

In the next, we prove    \begin{lem} The degree of $u^{*}$ satisfies $ \big|\deg u^{*} \big| = 1$.\label{lem polynomial m=1}
        \end{lem} 
\begin{proof}
 We only need to rule out the case in which $| \deg u^* | = 2$. Without loss of generality, we assume $\deg u^* = 2$. The following proof is divided into two cases, based on the two standard forms obtained in Lemma \ref{lem: deg 2 explicit form}.  \vspace{0.4pc}

\noindent \textbf{Case 1. Item (1) in Lemma \ref{lem: deg 2 explicit form}:} Write $c=\rho\h{0.5pt} e^{i\h{0.2pt}\gamma}$ and denote by $R_s \in \mathrm{SO}(3)$ the planar rotation by the angle $s$ about the $y_3$-axis. The standard form in Item (1) of Lemma \ref{lem: deg 2 explicit form} can be rewritten as follows
        $$\widetilde \Pi_{-e_3}^{-1}\circ R_{\pi-2\gamma}\h{.5pt} R\h{.5pt}u^* \circ \widetilde \Pi_{-e_3}\big(e^{i\gamma}z\big) 
        = -\dfrac{z^2-\rho^2}{1-\rho^2 z^2}.$$
Note that $\widetilde \Pi_{-e_3}^{-1}\circ \widetilde \Pi_{e_1}(z)=\dfrac{\overline{z}+i}{\overline{z}-i}$ and $\widetilde \Pi_{e_1}^{-1}\circ \widetilde \Pi_{-e_3}(z)=-i\h{0.8pt}\dfrac{\overline{z}+1}{\overline{z}-1}$. It turns out that
        \begin{align*}&\widetilde \Pi_{e_1}^{-1} \circ R_{\pi-2\gamma}\h{0.5pt} R\h{.5pt}u^* \circ \widetilde \Pi_{-e_3}\big(e^{i\gamma}z\big)\\[1mm] &\h{15pt}=\widetilde \Pi_{e_1}^{-1}
        \circ \widetilde \Pi_{-e_3}\circ
        \left(\h{1pt} \widetilde \Pi_{-e_3}^{-1}\circ R_{\pi-2\gamma}\h{0.5pt} R\h{.5pt}u^* \circ \widetilde \Pi_{-e_3}\right)\big(e^{i\gamma}z\big) =  i\h{1pt}\dfrac{\rho^2+1}{\rho^2 - 1}\h{2pt}
\dfrac{\overline{z}^2-1}{\overline{z}^2+1}.\end{align*}
Here, we use $\rho \neq 1$. Consequently,
	  \begin{align} \label{3454}
	  	\widetilde \Pi_{e_1}^{-1}\circ R_{\pi-2\gamma}\h{0.5pt}Ru^*
	  	\circ\widetilde \Pi_{e_1}(z) 
	  	&= \widetilde \Pi_{e_1}^{-1}\circ R_{\pi-2\gamma}\h{0.5pt}Ru^*
	  	\circ \widetilde \Pi_{-e_3} \circ \left(\widetilde \Pi^{-1}_{-e_3}\circ\widetilde \Pi_{e_1}(z) \right) \\[1.5mm]
	  	&=\widetilde \Pi_{e_1}^{-1}\circ R_{\pi-2\gamma}\h{0.5pt}Ru^*
	  	\circ \widetilde \Pi_{-e_3}
	  	\left(
	  	e^{i\gamma}\h{1pt}
	  	e^{-i\gamma}\h{1pt}\frac{\overline{z}+i}{\overline{z}-i}
	  	\right)
	  	=
	  	i\h{1pt}\frac{\rho^2+1}{\rho^2 - 1}
	  	\h{2pt}\frac{
	  		 \zeta(z)^2-1
	  	}{
	  		\zeta(z)^2+1
	  	}. \nonumber
	  \end{align}The function $\zeta$ is given by $ \zeta(z) := e^{i\gamma}\h{1pt}\dfrac{z-i}{z+i}.$ \vspace{0.2pc}

Denoting by $R^*$ the rotation matrix $R_{\pi-2\gamma}\h{0.5pt}R$, we estimate $F_{f_{R^*}}$ in \eqref{sta ine}. Through the change of variables $\zeta = \zeta(z)$ and $t = \frac{|\h{.5pt} \rho^2 - 1\h{.5pt}|}{\rho^2 + 1} \h{1pt}s$, it follows that
\begin{align}\label{3471}
F_{f_{R^*}} \leq \int_0^{1}\!\mathrm d\h{.5pt}t\left\{\h{1pt}\int_{\mathbb{R}^2}
	\dfrac{
	  		 \left|\h{0.5pt} \zeta^2 + 1 \h{0.5pt}\right|^{2} \left|\h{0.5pt} \zeta^2-1\h{0.5pt}\right|^2}{\left( \h{1pt}t^2  \left|\h{0.5pt} \zeta^2 + 1 \h{0.5pt}\right|^{2} +  
	  		\left|\h{0.5pt} \zeta^2 - 1 \h{0.5pt}\right|^{2}\h{1pt}\right)^2 }\h{1pt} \h{2pt}\frac{\mathrm d\h{.5pt}\zeta_1\,\mathrm d\h{.5pt}\zeta_2}{\big(1+|\h{0.5pt}\zeta\h{0.5pt}|^2\big)^2}\right\}^{\!\frac{1}{2}}.
\end{align}
Here, we also use $ \frac{|\h{0.5pt}\rho^2 - 1 \h{0.5pt}|}{\rho^2 + 1} \leq1$. 
Let $\zeta = r\h{0.5pt}e^{i\theta}$. Direct calculations yield
 \begin{align*}
 	\left|\h{0.5pt}\zeta^2 - 1\h{0.5pt}\right|^2 \left|\h{0.5pt}\zeta^2 + 1\h{0.5pt}\right|^2
 	=\left(r^4+1\right)^2-4\h{0.5pt}r^4\cos^2 2\theta 
 \end{align*}
 and
  \begin{align*} 
 	t^2 \left|\h{0.5pt}\zeta^2 + 1 \h{0.5pt}\right|^2+ \left|\h{0.5pt}\zeta^2 - 1\h{0.5pt}\right|^2
 	&=\left(t^2+1\right)\left(r^4+1\right)+2\left(t^2-1\right)r^2\cos2\theta. 
 \end{align*}
Therefore, 
 \[
\mathcal I_2\left(t,r\right) := \int_0^{2\pi}
 \frac{\left|\h{0.5pt}\zeta^2 + 1\h{0.5pt}\right|^2 \left|\h{0.5pt}\zeta^2 - 1\h{0.5pt}\right|^2}
	  {\Big(\h{1pt}t^2 \left|\h{0.5pt}\zeta^2+1\h{0.5pt}\right|^2+\left|\h{0.5pt}\zeta^2-1\h{0.5pt}\right|^2\Big)^2}	\,\,\mathrm d\h{.5pt}\theta
 =\int_0^{2\pi}\frac{\left(r^4+1\right)^2-4\h{0.5pt}r^4\cos^2\phi}{\Big[ \left(t^2+1\right)\left(r^4 + 1\right)+2\left(t^2-1\right)r^2\cos \phi\h{1pt}\Big]^2}\,\,\mathrm d \phi.
 \]
For any $t\in(0,1)$, it satisfies $\left(t^2+1\right)^2\left(r^4 + 1\right)^2>4\left(t^2-1\right)^2r^4$. Set
 \begin{equation*} 
 	D:=\sqrt{\left(t^2+1\right)^2\left(r^4+1\right)^2-4\left(t^2-1\right)^2r^4}.
 \end{equation*}
We can write $\mathcal I_2\left(t, r\right)$ in a closed form as follows:
 \begin{align*}
 	\mathcal I_2\left(t,r\right)=2\pi\left[\h{1pt}\frac{\left(t^2+1\right)\left(r^4+1\right)^3}{D^3} -\frac{1}{\left(t^2-1\right)^2}\left(1-\frac{2\left(t^2+1\right)\left(r^4+1\right)}{D}+\frac{\left(t^2+1\right)^3\left(r^4+1\right)^3}{D^3}\right)\h{1pt}\right].
 \end{align*}
Thus, the inequality in \eqref{sta ine} implies that 
	\begin{align}\label{contradiction}1.7724 \h{.5pt}<\h{.5pt}\sqrt{\pi}
	\h{1pt}\le\h{1pt} \int^1_0\mathrm d\h{.5pt}t \h{1pt}
    \left\{ \int_0^{\infty}\frac{r}{\left(r^2+1\right)^2}\,\,\mathcal I_2\left(t,r\right)\h{1pt}\mathrm d \h{.5pt} r\right\}^{\frac{1}{2}}\h{0.5pt}<\h{0.5pt}1.5519.
\end{align}
We arrive at a contradiction. The standard form in Item (1) of Lemma \ref{lem: deg 2 explicit form} is ruled out.\vspace{0.4pc}

 \noindent \textbf{Case 2. Item (2) in Lemma \ref{lem: deg 2 explicit form}:} We derive the same estimate as in \eqref{3471}. Hence, we can arrive at the same contradiction as in \eqref{contradiction}.\vspace{0.2pc}
 
 Using the notation in Item (2) of Lemma \ref{lem: deg 2 explicit form}, we compute, similarly to \eqref{3454}, the following representation of $f_R$:
	  \[
	  f_R(z):=\widetilde \Pi_{e_1}^{-1}\circ R u^*
	  \circ \widetilde \Pi_{e_1}(z)
	  =-\dfrac{1}{\alpha_0}\h{1pt} w \left(e^{i\h{0.2pt} \gamma_0}
      \h{1.5pt}\dfrac{z-i}{z+i}\right).
	  \]
      From \eqref{sta ine} and applying the change of variables $ \zeta=e^{i\h{0.2pt}\gamma_0}\h{1.5pt}\dfrac{z-i}{z+i}$, we obtain
\begin{align*}
F_{f_R} 
    &= \int_0^1\!\mathrm d\h{.5pt}s\left\{\h{1pt}\int_{\mathbb{R}^2}
	\dfrac{\alpha_0^{-2}
	  		\big|\h{.5pt}w(\zeta)\h{.5pt}\big|^2}{\left(s^2+ \alpha_0^{-2}
	  		\big|\h{0.5pt}w(\zeta)\h{0.5pt}\big|^2\h{1pt}\right)^2 }\h{2pt}\frac{\mathrm d\h{.5pt}\zeta_1\,\mathrm d\h{.5pt}\zeta_2}{\big(1+|\h{0.5pt}\zeta\h{0.5pt}|^2\big)^2}\right\}^{\!\frac{1}{2}}.
\end{align*}
If we further change the variables by $\zeta= \dfrac{\upsilon-i}{\upsilon+i}$, then $w(\zeta)=\dfrac{\upsilon-i}{\upsilon+i}+\dfrac{\upsilon+i}{\upsilon-i}=2\h{0.5pt}\dfrac{\upsilon^2-1}{\upsilon^2+1}$. Hence,
\begin{align*}
F_{f_R} 
    &= \int_0^1\!\mathrm d\h{.5pt}s\left\{\h{1pt}\int_{\mathbb{R}^2}
	\dfrac{\frac{4}{\alpha_0^2} \h{1pt}\big| \upsilon^2+1 \big|^2 
	  		\big|\upsilon^2-1 \big|^2}{\left(s^2 \h{1pt} \big| \upsilon^2+1 \big|^2 + \frac{4}{\alpha_0^2}
	  		\h{1pt}\big| \upsilon^2-1 \big|^2\h{1pt}\right)^2 }\h{2pt}\frac{\mathrm d\h{.5pt}\upsilon_1\,\mathrm d\h{.5pt}\upsilon_2}{\big(1+|\upsilon|^2\big)^2}\right\}^{\!\frac{1}{2}}.
\end{align*}
If $|\alpha_0|\leq 2$, then \eqref{3471} is valid with $R^*$ replaced with $R$. The contradiction in \eqref{contradiction} still follows.\vspace{0.2pc}

We are left to study the case \(|\alpha_0|>2\). Note that
	\[
	\widetilde \Pi_{-e_3}^{-1}\circ R_\pi R u^*
	\circ \widetilde \Pi_{-e_3}(e^{i\gamma_0}z)
	=
	-\frac{w(z)-i\alpha_0}{w(z)+i\alpha_0}.
	\]
	It then turns out that
	\[
\begin{aligned}
\widetilde \Pi_{e_1}^{-1}\circ R_\pi R u^*
\circ \widetilde \Pi_{-e_3}(e^{i\gamma_0}z)
=\frac{\alpha_0}{\overline{w(z)}},
\end{aligned}\] which furthermore induces   	\[
	f_{R_\pi R}(z):=\widetilde \Pi_{e_1}^{-1}\circ R_\pi R u^*
\circ \widetilde \Pi_{e_1}(z)
	=
	 \frac{\alpha_0}{
	w\Big(
	e^{i\h{0.3pt}\gamma_0}\h{1pt}\frac{z-i}{z+i}
	\Big)}.
	\]
     From \eqref{sta ine} and applying the change of variables $ \zeta=e^{i\gamma_0}\dfrac{z-i}{z+i}$, we obtain
\begin{align*}
F_{f_{R_\pi R}} 
    &= \int_0^1\!\mathrm d\h{.5pt}s\left\{\h{1pt}\int_{\mathbb{R}^2}
	\dfrac{\alpha_0^2
	  		\h{0.5pt}\big|\h{0.5pt}w(\zeta)\h{0.5pt}\big|^{-2}}{\Big(s^2+ \alpha_0^2
	  		\h{.5pt}\big|\h{.5pt}w(\zeta)\h{.5pt}\big|^{-2}\h{1pt}\Big)^2 }\h{1pt} \h{2pt}\frac{\mathrm d\h{.5pt}\zeta_1\,\mathrm d\h{.5pt}\zeta_2}{\big(1+|\h{.5pt}\zeta\h{.5pt}|^2\big)^2}\right\}^{\!\frac{1}{2}}.
\end{align*}
If we further change the variables by  $\zeta= \dfrac{\upsilon-1}{\upsilon+1}$, then $w(\zeta) =\dfrac{\upsilon-1}{\upsilon+1}+\dfrac{\upsilon+1}{\upsilon-1}=2\h{1pt}\dfrac{\upsilon^2+1}{\upsilon^2-1}$. Hence,
\begin{align*}
F_{f_{R_\pi R}}  
    = \int_0^1\!\mathrm d\h{.5pt}s\left\{\h{1pt}\int_{\mathbb{R}^2}
	\dfrac{\frac{\alpha_0^2}{4}\h{0.5pt}\big|\h{0.5pt}\upsilon^2+1 \big|^2\h{0.5pt}
	  		\big|\h{0.5pt} \upsilon^2-1 \h{0.5pt} \big|^2}{\Big(s^2\h{0.5pt}\big|\h{0.5pt}\upsilon^2+1\h{0.5pt}\big|^2 + \frac{\alpha_0^2}{4} \h{0.5pt}
	  		\big|\h{0.5pt} \upsilon^2-1 \h{0.5pt}\big|^2\h{1pt}\Big)^2 }\h{1pt}\h{2pt}\frac{\mathrm d\h{.5pt}\upsilon_1\,\mathrm d\h{.5pt}\upsilon_2}{\big(1+|\upsilon|^2\big)^2}\right\}^{\!\frac{1}{2}}.
\end{align*}
As $|\alpha_0|> 2$, we then obtain the same contradiction as in \eqref{contradiction}. The proof is complete. \end{proof}
    
It remains to prove that $u^{*}$ is a rotation of $\widehat y$ if $| \deg u^* | = 1$.

\begin{lem}\label{lem:nonrotation-moment}
	Suppose that $| \deg u^{*} | = 1$. Then 
\[
u^{*}(y)= \mathrm{sign}\left(\deg u^*\right) R \h{1pt} \widehat y
\qquad\text{for all }y\in B_1 \setminus \big\{0\big\}.
\]Here, $R$ is some constant rotation in $\mathrm{SO}(3)$ about the $y_3$-axis. $\mathrm{sign}$ is the sign function. 
\end{lem}

\begin{proof} By the symmetry of $u^*$, there exists $z_0 \in \mathbb C$ with $|z_0| \neq 1$ and a planar rotation $R \in \mathrm{SO}(3)$ about the $y_3$-axis such that 
\begin{align*} 
\widetilde \Pi_{-e_3}^{-1}\circ R\h{.5pt}u^* \circ \widetilde \Pi_{-e_3}(z) 
= \dfrac{z-z_0}{1-\overline{z_0}z}\h{20pt}\text{for any $z \in \mathbb C$.}
\end{align*}Since $N(u^*) = 0$, as in \eqref{3275 1}, we have \begin{align*} 
\int_{\mathbb R^2}
\dfrac{z}{1+|z|^2} \h{2pt} \dfrac{\mathrm d\h{.5pt}z_1\,\mathrm d\h{.5pt}z_2}{\big(\h{1pt}|\h{0.5pt}1-\overline{z_0}z\h{0.5pt}|^2+  |\h{.5pt}z - z_0 \h{.5pt}|^2\h{1pt}\big)^2} = - \int_{\mathbb R^2} \dfrac{z}{1+|z|^2} \h{2pt} \dfrac{\mathrm d\h{.5pt}z_1\,\mathrm d\h{.5pt}z_2}{\big(\h{1pt}|\h{0.5pt}1 + \overline{z_0}z\h{0.5pt}|^2+  |\h{.5pt}z + z_0 \h{.5pt}|^2\h{1pt}\big)^2}    = 0.
\end{align*}Here, we also use $|z_0| \neq 1$. Thus, \begin{align*}
    \int_{\mathbb R^2}
\dfrac{z\h{1pt}\mathrm{Re}\left(z \h{0.5pt}\overline{z_0}\right)}{1+|z|^2} \h{2pt} \dfrac{ |\h{0.5pt}1 + \overline{z_0}z\h{0.5pt}|^2+  |\h{.5pt}z + z_0 \h{.5pt}|^2 + |\h{0.5pt}1-\overline{z_0}z\h{0.5pt}|^2+  |\h{.5pt}z - z_0 \h{.5pt}|^2 }{\big(\h{1pt}|\h{0.5pt}1-\overline{z_0}z\h{0.5pt}|^2+  |\h{.5pt}z - z_0 \h{.5pt}|^2\h{1pt}\big)^2 \big(\h{1pt}|\h{0.5pt}1 + \overline{z_0}z\h{0.5pt}|^2+  |\h{.5pt}z + z_0 \h{.5pt}|^2\h{1pt}\big)^2} \h{2pt} \mathrm d\h{.5pt}z_1\,\mathrm d\h{.5pt}z_2 = 0.
\end{align*} Multiply the identity by $\overline{z_0}$ and take the real part. It turns out that $\mathrm{Re}\left( z \h{0.5pt}\overline{z_0} \right) \equiv 0$ for any $z \in \mathbb C$, which induces $z_0 = 0$ by taking $z = z_0$. The proof is finished. 
\end{proof}

\noindent\textbf{Acknowledgments:} We are grateful to Prof. Eugene C. Gartland for insightful clarification of the nondimensionalization issue. The authors also thank Prof. Radu Ignat and Prof. Nils Schopohl for pointing out important references. H. M. Tai is partially supported by the Australian Research Council Discovery Project DP240100781. He also thanks Prof. Yong Yu for the kind invitation to The Chinese University of Hong Kong in January 2026. Y. Yu is partially supported by Hong Kong RGC grants Nos. 14303723, 14306622, and 14304821.

        \appendix

    \section{Energy monotonicity and partial regularity}
Energy monotonicity at the boundary under tangential boundary conditions is proved in the more general form of \cite[Theorem 4.2]{S06tangentplane}. For the reader’s convenience, we restate it here:
    \begin{lem}[\bf Energy monotonicity on boundary]\label{lem. energy mono.}
		There are $\C{2}, \h{1pt}\C{3}>0$ and $\R{1}\in (0,R_0]$, depending only on $\Omega$, such that for any Dirichlet energy minimizer $u$ in $H_T^1(\Omega; \mathbb S^2)$, we have
		$$ \int^R_r\dfrac{e^{\C{2} \tau}}{2\tau} \h{1.5pt}\mathrm d \tau\int_{\p^+ B^+_{\tau}(y_0)} \big|\h{1pt}\p_r \widetilde{u} \h{1pt}\big|^2  \h{1.5pt}\mathrm d\mathscr{H}^2
		+\dfrac{e^{\C{2} r}}{r}\int_{B^+_{r}(y_0)} \big|\h{1pt}\nabla\widetilde{u} \h{1pt}\big|^2
        \leq 
		\dfrac{e^{\C{2} R}}{R}\int_{B^+_{R}(y_0)} \big|\h{1pt}\nabla\widetilde{u}\h{1pt}\big|^2
		+\C{3}(R-r).$$ Here, $\widetilde{u}$ is defined in \eqref{def transformed u}. Moreover, $y_0 \in \p^0B^+_{R_0}$ is arbitrary. $r,R \in (0,\R{1}]$ satisfy $ r\leq R$. 
	\end{lem}

    	We now combine the energy monotonicity result above with the standard interior monotonicity from \cite[\h{0.5pt}Proposition 2.4 or (2.5)\h{1pt}]{SU82} to establish an energy inequality for balls whose centers do not lie on the boundary.  Note that this is not a monotonicity formula owing to the presence of a factor greater than one on the right-hand side. Nonetheless, it is sufficient for our analysis.
	\begin{lem}[\bf  Energy inequality] 
		There are $\C{4}>0$ and $\R{2}\in (0,\R{1}]$, depending only on $\Omega$, such that for any Dirichlet energy minimizer $u$ in $H_T^1\big(\Omega; \mathbb S^2\big)$ with $\widetilde{u}$ defined in \eqref{def transformed u}, it holds that
		$$\dfrac{1}{r}\left(\int_{B^+_r(a)} \big|\h{1pt}\nabla \widetilde{u} \h{1pt}\big|^2 +\int_{B^+_r(a^\star)} \big|\h{1pt}\nabla \widetilde{u} \h{1pt}\big|^2  \right)
		\,\leq\, \C{4}R + \dfrac{\C{4}}{R}\left(\int_{B^+_R(a)} \big|\h{1pt}\nabla \widetilde{u} \h{1pt}\big|^2 +\int_{B^+_R(a^\star)} \big|\h{1pt}\nabla \widetilde{u} \h{1pt}\big|^2  \right)
		.$$ Here, $a \in B_{R_0}$ is arbitrary. The positive numbers $r,R$ satisfy $r\leq R\leq \R{2}$.
		\label{lem int energy mono}
	\end{lem}
	\begin{proof}
		It suffices to prove the result for $a\in B^+_{R_0}$. Since the above estimate follows if $r > \frac{R}{4}$, we consider only the case where $r \leq \frac{R}{4}$. The remaining proof is divided into 3 cases, depending on the relationship between $r$, $R$, and the location of $a = (a_1, a_2, a_3)$. \vspace{0.4pc}
		
		\noindent{\bf Case 1. $a_3 \geq \frac{R}{4}$:}
		In this case, we have $B^+_{r}(a^\star)=\varnothing$ and $B_r(a) \subset B_{R/4}(a)\subset B^+_{R}(a)$. The desired estimate can be obtained by the standard result of interior regularity in \cite[Proposition 2.4]{SU82}. \\[2mm]
		\noindent{\bf Case 2. $a_3 < \frac{R}{4}$ and $a_3\leq r$:}
		Taking $\tau:=\big(r^2-a_3^2\big)^{\frac{1}{2}}$, we have $$B^+_{r}(a)\subset B^+_{2r}(a^0) \subseteq B^+_{R/2}(a^0) \subset B^+_{R}(a) \h{15pt} \text{and} \h{15pt} B_r^+(a^\star) \subset B_{\tau}^+(a^0)\subseteq B^+_{R/4}(a^0)\subset B^+_{R}(a^\star).$$ Lemma \ref{lem. energy mono.} then implies
		\begin{align}
			\dfrac{1}{r}\int_{B^+_r(a)} \big|\h{1pt}\nabla \widetilde{u} \h{1pt} \big|^2  
			\h{1pt}\leq\h{1pt} \dfrac{1}{r} \int_{B^+_{2r}(a^0)} \big|\h{1pt}\nabla \widetilde{u}  \h{1pt}\big|^2   
			&\h{1pt}\lesssim_{\,\Omega,R_1\,} R  + \dfrac{2}{R} \int_{B^+_{R/2}(a^0)} \big|\h{1pt}\nabla \widetilde{u}  \h{1pt}\big|^2    \h{1pt}\lesssim\h{1pt}    R + \dfrac{1}{R} \int_{B^+_{R}(a)} \big|\h{1pt}\nabla \widetilde{u} \h{1pt} \big|^2, 
			\label{1416}
		\end{align}
		and
		\begin{align*}
			\dfrac{1}{r}\int_{B^+_r(a^\star)} \big|\h{1pt}\nabla \widetilde{u}  \h{1pt}\big|^2  
			\h{1pt}\leq\h{1pt} \dfrac{1}{\tau} \int_{B^+_{\tau}(a^0)} \big|\h{1pt}\nabla \widetilde{u}  \h{1pt}\big|^2   
			&\h{1pt}\lesssim_{\,\Omega,R_1\,} R+ \dfrac{4}{R} \int_{B^+_{R/4}(a^0)} \big|\h{1pt}\nabla \widetilde{u}  \h{1pt}\big|^2 \h{1pt}\lesssim_{\,\Omega,R_1\,}  R + \dfrac{1}{R} \int_{B^+_{R}(a^\star)} \big|\h{1pt}\nabla \widetilde{u}  \h{1pt}\big|^2.
		\end{align*}The estimate in this lemma still holds by adding the above. \\[2mm]
		\noindent{\bf Case 3. $r < a_3 < \frac{R}{4}$:}
		As $r < a_3$, we have $$B_{r}^+(a)=B_{r}(a)\subset B_{a_3}(a)= B_{a_3}^+(a).$$ Applying the interior energy inequality \cite[Proposition 2.4]{SU82} deduces
		\begin{align*}
			\dfrac{1}{r}\int_{B^+_r(a)} \big|\h{1pt}\nabla \widetilde{u} \h{1pt} \big|^2  
			\h{1pt}\lesssim_{\,\Omega,R_1}\, a_3 + \dfrac{1}{a_3}\int_{B_{a_3}^+(a)}\left|\nabla \widetilde{u}  \right|^2.
		\end{align*}
		Since $a_3< \frac{R}{4}$, we substitute $r=a_3$ in \eqref{1416} and obtain 
		\begin{align*}
			\dfrac{1}{a_3}\int_{B^+_{a_3}(a)} \big|\h{1pt}\nabla \widetilde{u} \h{1pt}  \big|^2  
			\,\lesssim_{\,\Omega,R_1}\,   R + \dfrac{1}{R} \int_{B^+_{R}(a)} \big|\h{1pt}\nabla \widetilde{u}  \h{1pt}\big|^2.
		\end{align*}The desired estimate in the lemma also follows from the last two estimates.
	\end{proof}

    The following partial regularity theorem can be found in \cite[\h{0.5pt}Theorem 5.3\h{1pt}]{S06tangentplane}.

    	\begin{lem}[\bf Partial regularity]\label{lem Partial Regularity}
		There are $\alpha_0\in(0,1)$, $\eps{5}\in(0,\eps{4}]$, $\C{14}>0$ and $\R{7}\in(0,\R{6}]$ depending only on $\Omega$ such that for any $z \in \p^0B_{R_0/4}^+$, any $R\in(0,\R{7}]$ and any Dirichlet energy minimizer $u$ in $H_T^1(\Omega; \mathbb S^2)$, if it satisfies $\widehat{\mathcal{E}}^*_{R,\h{1pt} z \h{0.5pt};\h{0.5pt}+}\left(\widetilde{u}\right) < \eps{5}$, then it holds that $$\|\h{1pt}\widetilde{u} \h{1pt}\|_{C^{\alpha_0}\pig(\overline{B_{R/8}^+(z)}\pig)}\leq \C{14}.$$ Here, $\widetilde{u}$ is defined in \eqref{def transformed u}.
	\end{lem} 

			\bibliographystyle{abbrv} 
			\bibliography{bio}

		\end{document}